\documentclass[11pt]{article}

\usepackage{amsmath, amsthm}
\usepackage{amsmath, amsfonts}
\usepackage{amsmath, amssymb}
\usepackage{amsmath}
\usepackage{graphics}
\usepackage{graphicx}
\usepackage{xcolor}
\usepackage{hyperref}
\usepackage[all]{xy}

\newtheorem{theorem}{Theorem}[subsection]
\newtheorem{definition}[theorem]{Definition}
\newtheorem{definition-lemma}[theorem]{Definition/Lemma}
\newtheorem{definition-explanation}[theorem]{Definition/Explanation}
\newtheorem{explanation-definition}[theorem]{Explanation/Definition}
\newtheorem{definition-fact}[theorem]{Definition/Fact}
\newtheorem{definition-notation}[theorem]{Definition/Notation}
\newtheorem{definition-conjecture}[theorem]{Definition/Conjecture}
\newtheorem{lemma}[theorem]{Lemma}
\newtheorem{lemma-definition}[theorem]{Lemma/Definition}
\newtheorem{proposition}[theorem]{Proposition}
\newtheorem{corollary}[theorem]{Corollary}
\newtheorem{remark}[theorem]{\it Remark}
\newtheorem{remark-notation}[theorem]{\it Remark/Notation}

\newtheorem{application-lemma}[theorem]{Application/Lemma}

\newtheorem{convention}[theorem]{\it Convention}

\newtheorem{example}[theorem]{Example}
\newtheorem{example-definition}[theorem]{Example/Definition}

\newtheorem{definition-prototype}[theorem]{Definition-Prototype}

\newtheorem{terminology}[theorem]{\it Terminology}

\numberwithin{equation}{subsection}

\newtheorem{stheorem}{Theorem}[section]
\newtheorem{sdefinition}[stheorem]{Definition}
\newtheorem{sdefinition-lemma}[stheorem]{Definition/Lemma}
\newtheorem{sdefinition-explanation}[stheorem]{Definition/Explanation}
\newtheorem{sexplanation-definition}[stheorem]{Explanation/Definition}
\newtheorem{sdefinition-fact}[stheorem]{Definition/Fact}
\newtheorem{sdefinition-notation}[stheorem]{Definition/Notation}
\newtheorem{sdefinition-conjecture}[stheorem]{Definition/Conjecture}
\newtheorem{slemma}[stheorem]{Lemma}
\newtheorem{slemma-definition}[stheorem]{Lemma/Definition}

\newtheorem{sremark}[stheorem]{\it Remark}
\newtheorem{sremark-notation}[stheorem]{\it Remark/Notation}

\newtheorem{sapplication-lemma}[stheorem]{Application/Lemma}

\newtheorem{sexample-definition}[stheorem]{Example/Definition}

\newtheorem{sdefinition-prototype}[stheorem]{Definition-Prototype}

\newtheorem{ssdefinition-lemma}[sstheorem]{Definition/Lemma}
\newtheorem{ssdefinition-explanation}[sstheorem]{Definition/Explanation}
\newtheorem{ssexplanation-definition}[sstheorem]{Explanation/Definition}
\newtheorem{ssdefinition-fact}[sstheorem]{Definition/Fact}
\newtheorem{ssdefinition-notation}[sstheorem]{Definition/Notation}
\newtheorem{ssdefinition-conjecture}[sstheorem]{Definition/Conjecture}

\newtheorem{sslemma-definition}[sstheorem]{Lemma/Definition}

\newtheorem{ssremark-notation}[sstheorem]{\it Remark/Notation}

\newtheorem{ssapplication-lemma}[sstheorem]{Application/Lemma}

\newtheorem{ssexample-definition}[sstheorem]{Example/Definition}

\newtheorem{ssdefinition-prototype}[sstheorem]{Definition-Prototype}

\newcommand{\Ad}{{\mbox{\it Ad}\,}}

\newcommand{\Cl}{{\mbox{\it Cl}\,}}

\newcommand{\Der}{\mbox{\it Der}\,}

  \newcommand{\tinyDouble}{\mbox{\it\tiny Double}}

\newcommand{\Endsheaf}{{\mbox{\it ${\cal E}\!$nd}\,}}

\newcommand{\Hom}{\mbox{\it Hom}\,}

\newcommand{\Imaginary}{\mbox{\it Im}\,}

  \newcommand{\Modscriptsize}{{\scriptsize\mbox{\it Mod}\,}}

\newcommand{\Pev}{{{\cal P}^{\,ev}}}

\newcommand{\Real}{\mbox{\it Re}\,}

\newcommand{\SL}{\mbox{\it SL}}

\newcommand{\SO}{\mbox{\it SO}\,}

\newcommand{\Spin}{\mbox{\it Spin}\,}

\newcommand{\Tr}{\mbox{\it Tr}\,}

  \newcommand{\tinyWZ}{{\mbox{\tiny\it WZ}\,}}

\newcommand{\scriptsizeach}{{\mbox{\scriptsize\it ach}}}

\newcommand{\anticommuting}{{\mbox{\scriptsize\it anti-c}}}

\newcommand{\chd}{\mbox{\it c.h.d}\,}
  
\newcommand{\scriptsizech}{{\mbox{\scriptsize\it ch}}}

\newcommand{\coordinates}{{\mbox{\scriptsize\it coordinates}}}

\newcommand{\even}{{\mbox{\scriptsize\rm even}}}
  
\newcommand{\exotic}{{\mbox{\scriptsize\it exotic}\,}}  
\newcommand{\field}{{\mbox{\scriptsize\it field}}}

\newcommand{\gaugescriptsize}{{\mbox{\scriptsize\it gauge}\,}}

\newcommand{\hol}{{\mbox{\it\scriptsize hol}}}
  \newcommand{\ahol}{{\mbox{\it\scriptsize ahol}}}

\newcommand{\mediumscriptsize}{{\mbox{\it\scriptsize medium}\,}}

\newcommand{\odd}{{\mbox{\scriptsize\rm odd}}}

\newcommand{\parameter}{{\mbox{\scriptsize\it parameter}}}
\newcommand{\physics}{\mbox{\scriptsize\it physics}}

\newcommand{\smallscriptsize}{{\mbox{\scriptsize\it small}\,}}
  \newcommand{\smalltiny}{{\mbox{\tiny\it small}\,}}

\newcommand{\spinor}{{\mbox{\scriptsize\it spinor}}}

\newcommand{\tamescriptsize}{{\mbox{\scriptsize\it tame}\,}}

\newcommand{\trivial}{{\mbox{\scriptsize\it trivial}}} 

\newcommand{\vectorr}{{\mbox{\scriptsize\it vector}}}

\newcommand{\bolda}{\mbox{\boldmath $a$}}
\newcommand{\boldd}{\mbox{\boldmath $d$}}
  \newcommand{\scriptsizeboldd}{\mbox{\scriptsize\boldmath $d$}}

\newcommand{\boldu}{\mbox{\boldmath $u$}}

\newcommand{\boldy}{\mbox{\boldmath $y$}}

\newcommand{\boldz}{\mbox{\boldmath $z$}}

\newcommand{\bdsigma}{\boldsymbol{\sigma}}

\newcommand{\squareflat}{{\square\hspace{-.98ex}\flat}\,}

\newcommand{\tinybullet}{{\raisebox{.2ex}{\tiny $\bullet$}}}							
																				
\begin{document}

\enlargethispage{24cm}

\begin{titlepage}

$ $

\vspace{-1.5cm} 

\noindent\hspace{-1cm}
\parbox{6cm}{\small December 2019}\
   \hspace{7cm}\
   \parbox[t]{6cm}{\small
                arXiv:yymm.nnnnn [math.AG] \\
                SUSY(2.1): $d=3+1$, $N=1$ SQFT, \\ $\mbox{\hspace{5em}}$
				                       trivialized spinor bundle
				}

\vspace{2em}

\centerline{\large\bf
 Grothendieck meeting [Wess \& Bagger]:\hspace{1em} }
\vspace{1ex}
\centerline{\bf
 Wess \& Bagger [Supersymmetry and supergravity: Chs.\,IV, V, VI, VII, XXII] ({\scriptsize 2nd ed.})}
\vspace{1ex}
\centerline{\large\bf
 reconstructed in complexified ${\Bbb Z}/2$-graded $C^\infty$-Algebraic Geometry}
\vspace{1ex}
\centerline{\large\bf
  I. The construction under a trivialization of the Weyl spinor bundles}
\vspace{1ex}
\centerline{\large\bf
   by complex conjugate pairs of covariantly constant sections}

\bigskip

\vspace{3em}

\centerline{\large
  Chien-Hao Liu
            \hspace{1ex} and \hspace{1ex}
  Shing-Tung Yau
}

\vspace{3em}

\begin{quotation}
\centerline{\bf Abstract}
\vspace{0.3cm}

\baselineskip 12pt  
{\small
  Forty-six years after the birth of supersymmetry in 1973 from works of {\it Julius Wess} and {\it Bruno Zumino},    
   the standard quantum-field-theorists and particle physicists' language of
     `superspaces', `supersymmetry', and `supersymmetric action functionals in superspace formulation'
	 as given in Chapters IV, V, VI, VII, XXII of the classic on supersymmetry and supergravity:
  {\it Julius Wess \& Jonathan Bagger:} {\sl Supersymmetry and Supergravity} (2nd ed.),	
  is finally polished, with only minimal mathematical patches added for consistency and accuracy
     in dealing with nilpotent objects from the Grassmann algebra involved,
  to a precise setting in the language of complexified ${\Bbb Z}/2$-graded $C^\infty$-Algebraic Geometry.
 This is completed after the lesson learned from  D(14.1) (arXiv:1808.05011 [math.DG])
   and the notion of `$d=3+1$, $N=1$ towered superspaces' as complexified ${\Bbb Z}/2$-graded $C^\infty$-schemes, 
    their distinguished sectors, and purge-evaluation maps 
	first developed in SUSY(1) (= D(14.1.Supp.1)) (arXiv:1902.06246 [hep-th]) and further polished in the current work.
 While the construction depends on a choice of a trivialization of the spinor bundle by covariantly constant sections,
  as long as
     the transformation law and the induced isomorphism
        under a change of trivialization of the spinor bundle by covariantly constant sections
	 are understood,
  any object or structure thus defined or constructed is mathematically well-defined.        
 The construction can be generalized to all other space-time dimensions with simple or extended  supersymmetries.
 This is part of the mathematical foundation required to study fermionic D-branes in the Ramond-Neveu-Schwarz formulation.
} 
\end{quotation}

\bigskip

\baselineskip 12pt
{\footnotesize
} 

\medskip

\noindent {\small MSC number 2020:  14A22, 16S38, 58A50; 14M30, 81T60, 83E50
} 

\bigskip

\baselineskip 10pt
{\scriptsize
\noindent{\bf Acknowledgements.}
We thank
 Andrew Strominger, Cumrun Vafa
   for influence to our understanding of strings, branes, and gravity.
C.-H.L.\ thanks in addition
 Daniel Freed,
   whose works (alone or collaborative) almost 20 years ago have occupied his thoughts for the last three years;
 Girma Hailu, Albrecht Klemm, Pei-Ming Ho, Jesse Thaler (time-ordered)
   who together boosted his understanding of superspaces and supersymmetry on the physics side
    in one year (fall 2018 - fall 2019, cf.\:footnote~1 {\it Special acknowledge});
 Enno Ke{\ss}ler
   for enrichment of his understanding of supergeometry on the mathematical side;
 Dennis Borisov, Tristan Collins, Artan Sheshmani, Cumrun Vafa, Kai Xu, Yang Zhou for discussions/consultations on issues beyond;
 Eric Sharpe, Wei Gu for communication;
 Yng-Ing Lee and Department of Mathematics and Ho and Department of Physics, National Taiwan University,
   for hospitality and discussions, May 2019;
 Karsten Gimre, Vafa, Sheshmani, and
 Nima Arkani-Hamed,  Christopher Gerig, Matthew Reece, Yum-Tung Siu, Thaler
   for topic courses in spring 2019 and fall 2019 respectively;
 Ling-Miao Chou
   for comments that improve the illustrations and moral support.
The project is supported by NSF grants DMS-9803347 and DMS-0074329.
} 

\end{titlepage}

\newpage

\enlargethispage{24cm}
\begin{titlepage}

$ $


\centerline{\small\it
 Chien-Hao Liu dedicates this work to the memory of}
\centerline{\small\it
 two founding fathers of Supersymmetry$^{\,\sharp}$}
\centerline{\small\it
 Prof.\:Bruno Zumino (1923 - 2014),}
\centerline{\small\it
 whom he came across in his topic course on {\sl Quantum Groups} during U.C.\,Berkeley years,}
\centerline{\small\it
 \hspace{-4em}and\;\;\;\;\;\;\;\;Prof.\:Julius Wess (1934 - 2007),}
\centerline{\small\it
 whose classic book$^{\,\natural}$ with Jonathan Bagger on Supersymmetry and Supergravity motivates the current work.}

\bigskip

\centerline{\includegraphics[width=0.6\textwidth]{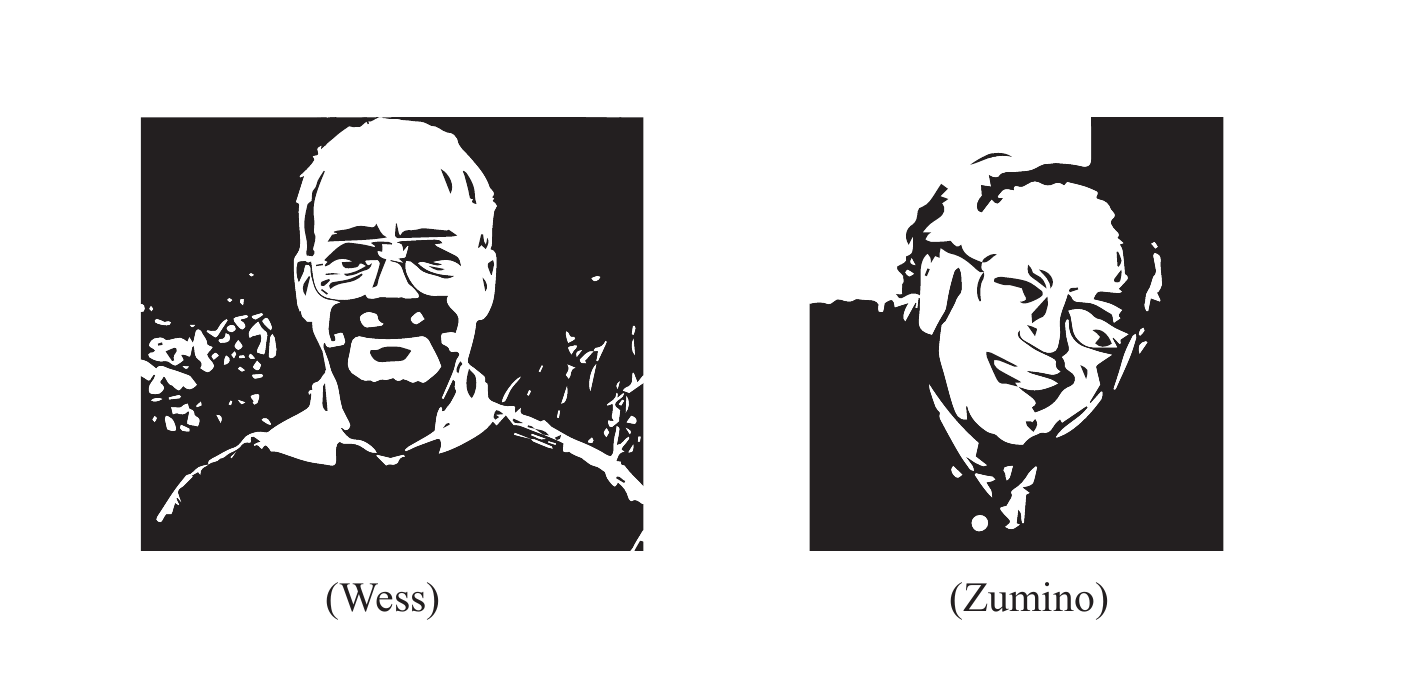}}

\bigskip

\centerline{\includegraphics[width=0.8\textwidth]{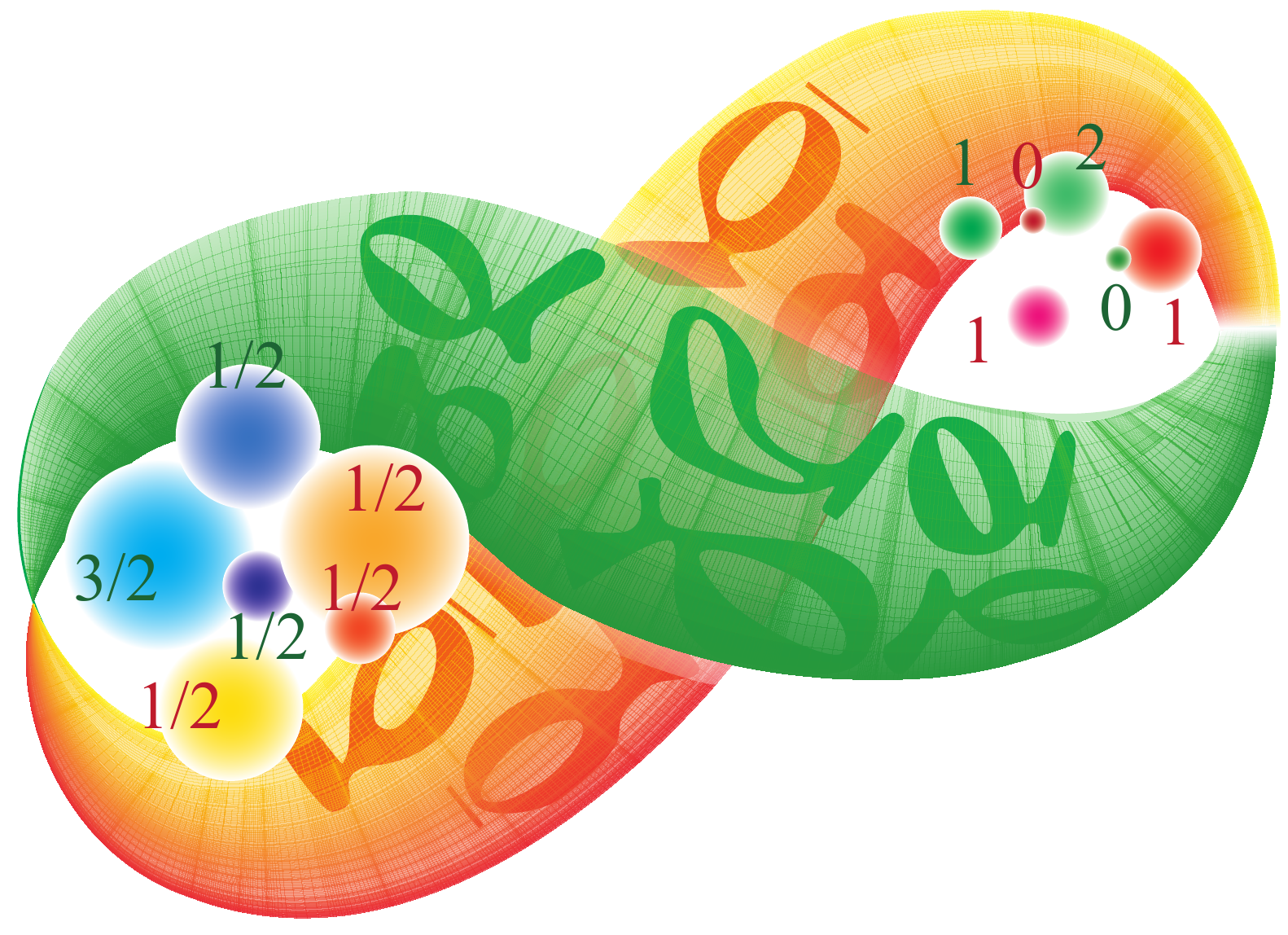}}
\centerline{[\, \textcolor{blue}{{$\cal S$}upersymmetry}:
                           \textcolor{green}{$Q$} and \textcolor{red}{$\bar{Q}$}\;]$^\flat$}
 
\vspace{2em}

\noindent
{\footnotesize
$^\sharp$Works
of {\it Julius Wess} and {\it Bruno Zumino} (1973):

 \medskip

 [W-Z1]\hspace{1em}\parbox[t]{44em}{J.\ Wess and B.\ Zumino,
 {\it A Lagrangian model invariant under supergauge transformations},\\
 {\sl Phys.\ Lett.}\ {\bf 49B} (1974), 52--54.}

 \medskip

 [W-Z2]\hspace{1em}\parbox[t]{44em}{--------,
 {\it Supergauge transformations in four-dimensions},
 {\sl Nucl.\ Phys.}\ {\bf 70} (1974), 39--50.}

 \smallskip

 [W-Z3]\hspace{1em}\parbox[t]{44em}{--------,
 {\it Supergauge invariant extension of quantum electrodynamics},
 {\sl Nucl.\ Phys.}\ {\bf B78} (1974), 1--13. }

\bigskip

\noindent
$^\natural$The
classic book by {\it Julius Wess} and  {\it Jonathan Bagger}:

\medskip
 
[Wess \& Bagger] J.\:Wess and J.\:Bagger,
 {\sl Supersymmetry and supergravity}, 2nd ed., revised and expanded,\\ $\mbox{\hspace{5.2em}}$
  Princeton Univ.\ Press, 1992.

\bigskip

\noindent
$^\flat$Illustration
inspired by particle physics and art works of {\sl Maurits Cornelis Escher} (1898-1972):\\
 $\mbox{\hspace{.2ex}}$
 {\it Day and Night}, (1938), woodcut,   and
 {\it Swans}, (1956), wood engraving on thin Japanese paper.
} 

%
%

\end{titlepage}


\newpage
$ $

\vspace{-3em}

\centerline{\sc
 Grothendieck Meeting [Wess \& Bagger], I
 } %

\vspace{2em}


\begin{flushleft}
{\Large\bf 0. Introduction and outline}
\end{flushleft}
Following
 [L-Y1] (D(1)) and [L-Y2] (D(11.1)),
one naturally comes to the realization that
 \begin{itemize}
  \item[\LARGE $\cdot$]  \parbox[t]{38em}{From
   the aspect of modern Algebraic Geometry in the spirit of Alexander Grothendieck,
   a {\it fermionic D-brane in the Ramond-Neveu-Schwarz (RNS) formulation} is described
   by a differentiable map
   $$
      \widehat{f}\; :\;
      (\widehat{X}^{\!A\!z}, \widehat{\cal E}; \widehat{\nabla})\;
      \longrightarrow\; Y
   $$
   from an Azumaya/matrix supermanifold/superscheme with a fundamental module
      $\widehat{X}^{\!A\!z}
	     := (\widehat{X},
		         \widehat{\cal O}_X^{A\!z}
				   := \Endsheaf_{\Modscriptsize\mbox{\scriptsize -}{\cal O}_X^{\Bbb C}}(\widehat{\cal E}))$
	   with a connection $\widehat{\nabla}$ on $\widehat{\cal E}$
   (cf.\:the world-sheet of a fermionic D-brane with the Chan-Paton bundle with a connection)
      to a smooth manifold $Y$ (cf.\:the target space-time).
    }
 \end{itemize}
Cf.\:[L-Y3] (D(11.2)) and [L-Y4] (D(14.1)).
The structure sheaf $\widehat{\cal O}_X^{A\!z}$ of the Azumaya superscheme $\widehat{X}^{\!A\!z}$
 is a sheaf of Azumaya algebras over a complexified ${\Bbb Z}/2$-graded $C^\infty$-scheme $\widehat{X}$.
Thus,
 any construct of $\widehat{X}$
 would give rise to a notion/definiton of fermionic D-branes in RNS formulation in the end.
The question is
 \begin{itemize}
  \item[\bf Q.]   \parbox[t]{38em}{\it
    Which construct of $\widehat{X}$, or its enhancement if necessary,
	 truly reflects particle physicists' perception and practical language
	  (albeit possibly requiring some mathematical patches to make everything precise)
	 of superspaces and supersymmetries,
	 as in, e.g.
	 {\rm\footnotesize
	 \begin{itemize}
	  \item[]\hspace{-2em}
	   \parbox[t]{5em}{ [G-G-R-S]}
       \parbox[t]{38em}{S.J.\:Gates, Jr., M.T.\:Grisaru, M.\:Roc\u{c}ek, and W.\:Siegel,
             {\sl Superspace -- one thousand and one lessons in supersymmetry},
               Frontiers Phys.\ Lect.\ Notes Ser.\ 58,
               Benjamin/Cummings Publ.\ Co., Inc., 1983.
			   }

	  \item[]\hspace{-2em}
       \parbox[t]{5em}{[We]}
	   \parbox[t]{38em}{ P.\:West,
          {\sl Introduction to supersymmetry and supergravity},
            extended 2nd ed.,
            World Scientific, 1990.
			}

	  \item[]\hspace{-2em}
	   \parbox[t]{5em}{$\hspace{1em}$}
	   \parbox[t]{38em}{\hspace{-5.6em}
	      [Wess \& Bagger]\hspace{.4em}
	      J.\:Wess and J.\:Bagger,
           {\sl Supersymmetry and supergravity}, 2nd ed.,
             Princeton Univ.\ Press, 1992.
		 }	
	 \end{itemize}}
	 ?
	  }
 \end{itemize}
 so that the corresponding notion of `fermionic D-branes' rings with the fermionic D-branes from superstring theory.
(Cf.\ See footnote\:1 for more words on what we are looking for here.)
In this way, our pursuit in the study of fermionic D-branes brings us back to a more fundamental question we have to resolve first
 before going on.

\bigskip
 
The goal of the current work is to provide a construct of superspaces that answers the above question.
(The construct is in the $d=3+1$, $N=1$ case
    but similar arguments extend it to other space-time dimensions
	with simple (i.e.\:$N=1$) or extended (i.e.\:$N\ge 2$) supersymmetries.)
With lessons learned from
  [L-Y4] (D(14.1), arXiv:1808.05011 [math.DG])  and
  [L-Y5] (SUSY(1) (= D(14.1.Supp.1)),  arXiv:1902.06246 [hep-th]) as the foundation,
 we re-do one of the classics on supersymmetry and supergravity for quantum-field-theorists and particle physicists:
   \begin{itemize}
    \item[]
	[Wess \& Bagger] \hspace{.4em}
	 \parbox[t]{30em}{Julius Wess \& Jonathan Bagger: {\sl Supersymmetry and Supergravity}\\
	    2nd ed., revised and expanded,  Princeton Univ.\ Press, 1992,}
	 \begin{itemize}
      \item[]
	  \parbox[t]{7em}{Chapter IV}
	   {\it Superfields}
	
	  \item[]
	  \parbox[t]{7em}{Chapter V}
	   {\it Chiral superfields}
	
	  \item[]
	  \parbox[t]{7em}{Chapter VI}
	   {\it Vector superfields}
	
	  \item[]
	  \parbox[t]{7em}{Chapter VII}
	  {\it Gauge invariant interactions},\;
	  ($U(1)$ part)
	
	  \item[]
	  \parbox[t]{7em}{Chapter XXII}
	  {\it Chiral models and \"{a}hler geometry}	
	\end{itemize}
   \end{itemize}
  from the aspect of complexified ${\Bbb Z}/2$-graded $C^\infty$-Algebraic
     Geometry.\footnote{\makebox[17.6em][l]{\it Special acknowledgements from C.-H.L.}
                 The time-and-place was some year in 1990s
				   at the Department of Mathematics in Evens Hall, University of California, Berkeley.
				  Motivated by the emerging dominating mathematical topic at that time:
				     {\sl Quantum invariants of low-dimensional manifolds},
				   which had arisen from supersymmetry and topological field theories on the physics side
				    --- particularly works of {\it Edward Witten} ---,
                  Prof.\:{\it Nicolai Reshetikhin}, a then young star on quantum invariants in low dimensional topology,
				  organized 	a seminar on {\sl Quantum invariants}.
                 In one of the beginning meetings of the seminar, Prof.\:Reshetikhin started to explain {\sl Supersymmetry}
				   to a group of serious mathematicians, including {\it Alexander Givental} and {\it Alan Weinstein},
				   and enthusiastic graduate students.
				 Alas! Not far into the intended lecture, tons of questions and puzzles already arose from the audience
				   that in the end Prof.\:Reshetikhin had to quit the lecture he had prepared
				   and rather announced that he would try to invite someone from the Department of Physics to come to the seminar
				   to explain supersymmetry to mathematicians.
				 Unfortunately, no physicists showed up to give a lecture on supersymmetry for the mathematical mind   and
				  that intended introduction of supersymmetry to mathematicians was thus ended in the middle and never finished.
				
				 Almost twenty-five years later,
				  in fall 2017 when the D-project stepped into her second decade and the focus was turned to fermionic D-branes,
				  one would expect that after all these years' joint effort
				    from both physics-friendly mathematicians and mathematics-friendly physicists,
				  there should be some mathematical work on supersymmetry and superspaces
				    that one could take as the clean, clear, and solid starting point to address fermionic D-branes
				    in parallel to Ramond-Neveu-Schwarz fermionic strings.
                 To make sure such a study of fermionic D-branes is linked to superstring theory,
				  two minimal requirements on such a sought-for mathematical work on supersymmetry are
                  \begin{itemize}
				   \item[(1)]
				    Since
					 \begin{itemize}
					  \item[\LARGE $\cdot$]
					  in addition to the ordinary commuting coordinate functions,
					  a superspace has also\\ fermionic/anticommuting coordinate functions,
					 which are nilpotent,
					
					  \item[\LARGE $\cdot$]
					  the underlying topology of a superspace remains a smooth manifold,
					
					  \item[\LARGE $\cdot$]
					  either complex spinors or complexified spinors were involved,
					 \end{itemize}
					this sought-for work should be
					 {\it in the realm of complexified ${\Bbb Z}/2$-graded $C^\infty$-Algebraic Geometry},
					 in the same spirit as how {\it Alexander Grothendieck} built up the modern language of Algebraic Geometry.

                   \item[(2)]			
                    Since [Wess \& Bagger: {\sl Supersymmetry and supergravity}] (2nd ed.)
					  by {\it Julius Wess} and {\it Jonathan Bagger}
				     is one of the classical textbooks that were taken by string theorists as the standard language for supersymmetry,
				     the sought-for work should {\it reproduce chapters in {\sl [Wess \& Bagger]}
					  of classical} (as opposed to quantum) {\it level in a direct, fluent, and effortless way}.	
					In particular, all the expressions in these chapters should have  a well-defined mathematical meaning and
					  all the formulas and computations are re-derivable in this sought-for mathematical work.
				  \end{itemize}
				 Painfully, up to summer 2018,
				   (i.e.\ 47 years since the appearance of supersymmetry in string theory in 1971   and
				              45 years since a $4$-dimensional supersymmetric quantum field theory was constructed in 1973),
				 a mathematical work that met the above minimal requirements
				   had not yet been in existence.
                 It was when I was in such an embarrassing situation that a train of lucks were given to us.
				 \begin{itemize}
				   \item[$(a)$]
				   Through the {\it Topic course in Supersymmetry}
				      given by {\it Girma Hailu} at the Department of Physics, Harvard University,
					  and related discussions with him, fall 2018,
				     it was realized that to construct the physically correct function-ring of a superspace,
					a Grassmann-algebra tower must be built over the ordinary superspace.
				   It is because of this tower that physicists can take only even objects but still with anticommuting fields included. 				
				   With this as the starting point and taking into account all its mathematical consequences, we were led
				    to the ``(tower construction) + (purge-evaluations in the end)" way
					to produce results in [Wess \& Bagger]
					in the realm of  complexified ${\Bbb Z}/2$-graded $C^\infty$-Algebraic Geometry;
				   cf.\:[L-Y5] (SUSY(1)).		
				   Through SUSY(1), one also learns why and how physicists could so cleverly bypass all the sign-factors
				    in ${\Bbb Z}/2$-graded geometry, as compared to [L-Y4] (D(14.1)).

                   \item[$(b)$]		
				   In March 2019, after a highlight of the construction to {\it Albrecht Klemm} at the Center,
				    he pointed out to me that many of the multiple-spinors type expressions
					 that appear in supersymmetric quantum field theories
					 are already contracted in some way in physicists' interpretation of these expressions.
                   This led us to a guiding question:
					  \begin{itemize}
					   \item[{\bf Q.}]
					   {\it Should one enact a purge-evaluation map not at the end of making sense of
					             the supersymmetric action functional, as did in SUSY (1), but rather much earlier before that?
								If so, then how?}                               								 					 
                      \end{itemize}                   				
				   				
				   \item[$(c)$]
				   In May 2019, the two days' intensive discussions with {\it Pei-Ming Ho} at the Department of Physics,
				    National Taiwan University, and the presentation of SUSY (1)  in the String Theory Seminar there
					left me with other input, including operational interpretation of these bi-spinor or triple-spinor type expressions.
                   After all, for physicists, it is the quantum theory that matters most. 	
				   Expressions that appear in a supersymmetric action functional may or may not have easy explanation classically
				    but as long as their meaning is clear at the quantum level, it remains a very good expression.
                   Indeed, in one aspect of superspace, there is hiddenly/implicitly already an infinite tower
				     over the ordinary superspace.
                   Such quantum-level picture of a superspace from physicists' view
				     could be difficult to realize through $C^\infty$-Algebraic Geometry alone but may serve as a final goal.
				   				
				   \item[$(d)$]
				    Finally, in fall 2019, {\it Jesse Thaler} gave a topic course on {\sl Supersymmetric quantum field theories}
    			 	  at the Department of Physics, Massachusetts Institute of Technology.
				    This is a SQFT course given by a superspace advocate.
				    From the very first lecture and through the seven lectures in September 2019,
				      he explained carefully what a $d=3+1$, $N=1$ superspace is from a physicist' eye
				      and how it can be used in a very elegant manner to construct supersymmetric action functionals.
                    After the exercises done in D(14.1) and SUSY~(1) and inputs from Girma, Albrecht, and Pei-Ming,
				      this gave me a rare chance to re-examine and compare step by step
					   physicists' way of working with supersymmetry and superspace
				       and what's enforced by complexified ${\Bbb Z}/2$-graded Algebraic Geometry.
					His very careful counting of independent degrees of freedom (either on-shell or off-shell)
					  whenever a new type of superfield is introduced indicates very clearly
					  that something crucial is still missing in the setting of SUSY (1).
                    There, as a mathematical must, additional independent degrees of freedom have to be thrown in
					  to keep the collection of superfields (resp.\; chiral superfields) to honestly form a ring.
                    It is through such re-checking,  cross examinations,	and his generous answer to my questions
				       that it becomes clear that there is still room for improvement beyond SUSY (1)
				       to realize [Wess \& Bagger] in complexified ${\Bbb Z}/2$-graded $C^\infty$-Algebraic Geometry.				
				 \end{itemize}
			  It is through the above sequence of unexpected occurrences ---
			    with right people, at right time, in right order, and at right place --- and these physicists' input along the way
				that the current work appears.
			  As if these lucks were not enough, while the current work is in the intensive editing stage,
               {\it Enno Ke{\ss}ler} suggested weekly Friday meetings on Supersymmetry
			    at the Center of Mathematical Sciences and Applications,
			    spring semester 2020, in which we teach each other
				our different, ``orthogonal-to-each-other", tower-vs-base aspects of supergeometry and supersymmetry.
              The various questions he raised about [Wess \& Bagger] in the meetings were turned to part of the motivations to improve
			    the writing of the current notes.             		
              His insight on supergeometry as explained in his book [Ke],
			    {\sl Supergeometry, super Riemann surfaces and the superconformal action functional},
			   particularly the emphasis on the {\it construction over bases}
			     (as opposed to the towered construction we are advocating in [L-Y5] (SUSY (1)) and the current work,
				     which looks most natural from the particle physics aspect),			
			   the notion of (relative) `underlying even manifolds', and the various setup/design in [Ke: Part II]
			    are sure to have a significant influence to one's understanding of supergeometry. 			  
				 },  
In this first part of the work,
we present the construction as close as can be to the original presentation
  except mathematically more favored notations and necessary mathematical patches for accuracy and consistency.
In particular,
 the construction of the full function-ring of a superspace, i.e.\:the function ring of a towered superspace,
  will depend on a choice of a trivialization of the spinor bundle by covariantly constant sections.
As long as
     the transformation law and/or the induced isomorphism
        under a change of trivialization of the spinor bundle by covariantly constant sections
	 is understood,
  any object or structure thus defined or constructed is mathematically well-defined.
 In this way
   --- and 46 years after the birth of supersymmetry in 1973 from works of Julius Wess and Bruno Zumino ---
 the standard physics language of
     `superspaces', `supersymmetry', and `supersymmetric action functionals in superspace formulation'
  is finally polished, with only minimal mathematical patches added for consistency and accuracy
     in dealing with nilpotent objects from the Grassmann algebra involved,
  to a precise setting in the language of complexified ${\Bbb Z}/2$-graded $C^\infty$-Algebraic Geometry.
   
Once a mathematical presentation of superspace and supersymmetry that matches [Wess \& Bagger] is completed,
  the immediate next question is:
  What is the intrinsic description of the same,
    without resorting to a choice of a trivialization of the spinor bundles by covariantly constant sections?
This and the similar construction for supersymmetric non-Abelian gauge theory
    ([Wess \& Bagger: Chapter VII, non-Abelian part])
   are the themes for the sequels.

\bigskip
\bigskip

\noindent
{\bf Convention.}
 References for standard notations, terminology, operations and facts are:\; \\
 (1)
     algebraic geometry: [Hart];\; $C^{\infty}$-algebraic geometry: [Jo];\; supergeometry: [Ke]; \;\; \\
 (2)
     spinors and supersymmetry (mathematical aspect):  [Ch], [De], [D-F1], [D-F2], [Fr], [Harv], [S-W];\;\;	
 (3)
     supersymmetry (physical aspect, especially $d=4$, $N=1$ case): [Wess \& Bagger; W-B], [G-G-R-S], [We];\,
            also [Argu], [Argy],  [Bi], [B-T-T], [RR-vN], [St], [S-S].   			
 \begin{itemize}
    \item[$\cdot$]
    For clarity, the {\it real line} as a real $1$-dimensional manifold is denoted by ${\Bbb R}^1$,
      while the {\it field of real numbers} is denoted by ${\Bbb R}$.
     Similarly, the {\it complex line} as a complex $1$-dimensional manifold is denoted by ${\Bbb C}^1$,
      while the {\it field of complex numbers} is denoted by ${\Bbb C}$.
  
   \item[$\cdot$]	
    The inclusion `${\Bbb R}\subset{\Bbb C}$' is referred to the {\it field extension
     of ${\Bbb R}$ to ${\Bbb C}$} by adding $\sqrt{-1}$, unless otherwise noted.
  

   \item[$\cdot$]
    All manifolds are paracompact, Hausdorff, and admitting a (locally finite) partition of unity.
    We adopt the {\it index convention for tensors} from differential geometry.
     In particular, the tuple coordinate functions on an $n$-manifold is denoted by, for example,
     $(y^1,\,\cdots\,y^n)$.
    However,  all summations are expressed explicitly; no up-low index summation convention is used.
  
  
  
  \item[$\cdot$]
   `{\it differentiable}', `{\it smooth}', and $C^{\infty}$ are taken as synonyms.

  %
  
  \item[$\cdot$]
   {\it section} $s$ of a sheaf or vector bundle vs.\ dummy labelling index $s$

  %
  %
  %
  %
  %
  %
  %
  
  \item[$\cdot$]
   group {\it action} vs.\  {\it action} functional in a quantum field theory


  \item[$\cdot$]
   {\it sheaves} ${\cal F}$, ${\cal G}$ vs.\
    {\it curvature tensor} $F_{\nabla}$,
	{\it gauge-symmetry group} ${\cal G}_{\gaugescriptsize}$

  %
  %
  %
  
  \item[$\cdot$]
   {\it coordinate-function index}, e.g.\ $(y^1,\,\cdots\,,\, y^n)$ for a real manifold
      vs.\  the {\it exponent of a power},
	  e.g.\  $a_0y^r+a_1y^{r-1}+\,\cdots\,+a_{r-1}y+a_r\in {\Bbb R}[y]$.

  \item[$\cdot$]
   {\it dimension} $n$ or $2n$ vs.\ {\it nilpotent} component $n$ of an element of a ring.
	
  %
  %
  %
   
  \item[$\cdot$]
   {\it Various brackets}\,:\;\;
    \parbox[t]{30em}{$[A,B]:= AB-BA$,\; $\{A,B\}:= AB+BA$,\\
	 $[A,B\}:= AB-(-1)^{p(A)p(B)}BA$,
	   where $p(\,\tinybullet\,)$ is the parity of $\tinybullet$\,.
	                  }
					
  \item[$\cdot$]
   We adopt the following convention as in the work of Deligne and Freed [D-F2: \S 6]:
    \begin{itemize}	
     \item[] {\it Convention\; $[$cohomological degree vs.\ parity\,$]$}\;
      We treat elements $f$ of ${\Bbb Z}/2$-graded ring as of cohomological degree $0$
        and the exterior differential operator $d$ as of cohomological degree $1$ and {\it even}.
        In notation, $\chd(f)=0$ and $\chd(d)=1$, $p(d)=0$.
      Under such $({\Bbb Z}\times ({\Bbb Z}/2))$-bi-grading,
        $$
          ab = (-1)^{c.h.d(a)\,c.h.d(b)}(-1)^{p(a)p(b)}ba
        $$
        for objects $a, b$ homogeneous with respect to the bi-grading.     														 
       Here, $a$ and $b$ are not necessarily of the same type.
    \end{itemize}
					
  \item[$\cdot$]
   The current SUSY(2.1) continues the study in
	  \begin{itemize}	
	   \item[]  \hspace{-2em} [L-Y4]\hspace{1em}\parbox[t]{34em}{{\it
	    $N=1$  fermionic D3-branes in RNS formulation I.
         $C^\infty$-Algebrogeometric foundations of $d=4$, $N=1$ supersymmetry,
         SUSY-rep compatible hybrid connections, and
         $\widehat{D}$-chiral maps from a $d=4$ $N=1$ Azumaya/matrix superspace},
         arXiv:1808.05011 [math.DG]. (D(14.1))
	    }
	
       \smallskip	
	   \item[]  \hspace{-2em} [L-Y5]\hspace{1em}\parbox[t]{34em}{{\it
	   Physicists' $d=3+1$, $N=1$ superspace-time and supersymmetric QFTs
         from a tower construction in complexified ${\Bbb Z}/2$-graded $C^\infty$-Algebraic Geometry
         and a purge-evaluation/index-contracting map},
       arXiv:1902.06246 [hep-th]. (D(14.1.Supp.1)=SUSY(1))   		
		 }
	  \end{itemize}  	
   Notations and conventions follow ibidem whenever applicable.
 \end{itemize}

\newpage

\begin{flushleft}
{\bf Outline}
\end{flushleft}
\nopagebreak
{\small\raggedright
\baselineskip 12pt  
\begin{itemize}
 \item[0]
  Introduction

 \item[1]
  The function-ring of a $d=3+1$, $N=1$ towered superspace and its distinguished even subrings
  (under a trivialization of the spinor bundle by covariantly constant sections)
   \vspace{-.6ex}
   \begin{itemize}	 	
    \item[1.1]	
     Basics of the Clifford algebra, the spin group, Clifford modules, and Weyl spinors\\ behind
     $d = 3 + 1$, $N = 1$ supersymmetry\;
    (cf.\:[Wess \& Bagger: Appendix A])
 
    \item[1.2]
	The function-ring of a $d=3+1$, $N=1$ towered superspace and its distinguished even subrings\\
   (under a trivialization of the spinor bundle by covariantly constant sections)\\	
    (cf.\:[Wess \& Bagger: Chapter IV and Appendices A \& B])
		
    \item[1.3]
    The chiral and the antichiral condition on  $C^\infty(\widehat{X}^{\widehat{\boxplus}})$
   \end{itemize}	

 \item[2]
  Purge-evaluation maps and
  the Fundamental Theorem on supersymmetric action functionals\\ via a superspace formulation
   
 \item[3]
  The small function-ring of $\widehat{X}^{\widehat{\boxplus}}$ and the Wess-Zumino model\\
 (cf.\:[Wess \& Bagger: Chapter V])
  \vspace{-.6ex}
  \begin{itemize}	
   \item[3.1]
    The small function-ring
	  $C^\infty(\widehat{X}^{\widehat{\boxplus}})^\smalltiny$
	and its chiral and antichiral sectors
 
   \item[3.2]
    The Wess-Zumino model
  \end{itemize}
   
  \item[4]
  Supersymmetric $U(1)$ gauge theory with matter on $X$ in terms of $\widehat{X}^{\widehat{\boxplus}}$\\
 (cf.\:[Wess \& Bagger: Chapter VI and Chapter VII, $U(1)$ part])
  \vspace{-.6ex}
  \begin{itemize}	 	
	\item[4.1]
	 Vector superfields and their associated (even left) connection
	
	\item[4.2]
	 Vector superfields in Wess-Zumino gauge
		
	\item[4.3]
     Supersymmetry transformations of a vector superfield in Wess-Zumino gauge	
		
	\item[4.4]	
	 Supersymmetric $U(1)$ gauge theory with matter on $X$ in terms of $\widehat{X}^{\widehat{\boxplus}}$
  \end{itemize}
		
  \item[5]
	$d=3+1$, $N=1$ nonlinear sigma models\;
    (cf.\:[Wess \& Bagger: Chapter XXII])	
  \vspace{-.6ex}
  \begin{itemize}	 	
	\item[5.1]
	 Smooth maps from $\widehat{X}^{\widehat{\boxplus}, \smalltiny}$ to a smooth manifold $Y$
	
	\item[5.2]
	 Chiral maps from $\widehat{X}^{\widehat{\boxplus}, \smalltiny}$ to a complex manifold $Y$
	
	\item[5.3]
	 The action functional for chiral maps\\ --- $d=3+1$, $N=1$ nonlinear sigma models with a superpotential
  \end{itemize}

%
%
%

\end{itemize}
} 

\newpage

\section{The function-ring of a $d=3+1$, $N=1$ towered superspace and its distinguished even subrings
  {\large (under a trivialization of the spinor bundle by covariantly constant sections)}}
  
In this section we review and upgrade the notion of `towered superspace' from [L-Y5: Sec.\:1] (SUSY(1)).
The representation theory of the Lorentz group, particularly the vector representation and the spinor representations,
 is among the key constituents in how particle physicists think about a superspace and superfields,
 thus in Sec.\:1.1 we recall some basics of the Clifford algebra, the spin group, Clifford modules, and Weyl spinors behind
   $d = 3 + 1$, $N = 1$ supersymmetry and set up the notations.
In Sec.\:1.2 we recall why one needs the notion of a towered superspace and how they are constructed.
Three distinguished sectors thereof that are related to physicists' notion of scalar superfields are identified.
In Sec.\:1.3 we discuss the chiral conditions and the antichiral conditions on these sectors and
 compare the corresponding chiral superfields to
 chiral multiplets from representations of supersymmetry algebra.

\bigskip
  
\subsection{Basics of the Clifford algebra, the spin group, Clifford modules, and Weyl spinors behind
     $d = 3 + 1$, $N = 1$ supersymmetry\;\\
    {\normalsize (cf.\:[Wess \& Bagger: Appendix A])}}

In this subsection
 we highlight basics of the Clifford algebra, the spin group, Clifford modules, and Weyl spinors behind
   behind the $d = 3 + 1$, $N = 1$ supersymmetry     and
 introduce some notations used in [Wess \& Bagger: Appendix A] and the current notes.
Readers are referred to, e.g., [Harv], [L-M], [Mo] for details.

\bigskip

\begin{flushleft}
{\bf The Clifford algebra and the spin group}
\end{flushleft}
Let
 $V\simeq {\Bbb R}^4$ be a $4$-dimensional vector space over ${\Bbb R}$
   with an inner product $\langle\,,\,\rangle$ of signature type $(-+++)$,  and
 \begin{itemize}
  \item[\LARGE $\cdot$]
   $T^\bullet V := \oplus_{k\ge 0} V^{\otimes k}$ be the {\it tensor algebra} associated to $V$,

  \item[\LARGE $\cdot$]
   $\bigwedge^\bullet V:= \oplus_{k\ge 0}\bigwedge^kV
     = T^\bullet V/(v\otimes v: v\in V)$  be the {\it Grassmann algebra}
   (synonymously, {\it exterior algebra}) associated to $V$,
   where $(v\otimes v: v\in V)$ is the bi-ideal in $T^\bullet V$ generated by elements as indicated,      and

  \item[\LARGE $\cdot$]
   $\Cl(V, \langle\,,\,\rangle)
      := T^\bullet V/ (v\otimes v +  \langle v, v\rangle\cdot 1: v\in V)$
     be the {\it Clifford algebra} associated to $(V, \langle\,,\,\rangle)$,
     where $(v\otimes v+  \langle v, v\rangle: v\in V)$
      is the bi-ideal in $T^\bullet V$ generated by elements as indicated.
  Denote the multiplicative group of multiplicatively invertible elements of $\Cl(V, \langle\,,\,\rangle)$
     by $\Cl^\ast(V, \langle\,,\,\rangle)$. 	
  \end{itemize}
The tensor algebra $T^\bullet V$ associated to $V$ is naturally
 ${\Bbb Z}/2$-graded: $T^\bullet V= T^{\even}V\oplus T^{\odd}V
    : =  (\oplus_{k\ge 0, even}V^{\otimes k})
	          \oplus (\oplus_{k\ge 1, odd}V^{\otimes k})  $.
After passing to quotients,
 all $T^\bullet V$, $\bigwedge^\bullet V$, and $\Cl(V, \langle\,,\,\rangle)$
 are (unital, associative) ${\Bbb Z}/2$-graded algebras over ${\Bbb R}$.
The quotient maps $T^\bullet V\rightarrow \bigwedge^\bullet V$ and
  $T^\bullet V\rightarrow \Cl(V, \langle\,,\,\rangle)$
 together induce a canonical isomorphism $\bigwedge^\bullet V \simeq \Cl(V, \langle\,,\,\rangle)$
 as vector spaces over ${\Bbb R}$.
The inner product $\langle\,,\,\rangle$ on $V$ induces canonically an inner product on $\bigwedge^\bullet V$,
 which in turn induces an inner product on $\Cl(V,\langle\,,\,\rangle)$ via the above vector-space isomorphism.
This renders both $\bigwedge^\bullet V$ and $\Cl(V,\langle\,,\,\rangle)$ algebras over ${\Bbb R}$
 with {\it norm squared},\footnote{Here,
                                                    for any vector space $W$ with an inner product,
													 (i.e.\ nondegenerate pairing)    $\langle\,,\,\rangle$,
													we call $\langle w, w\rangle$ the {\it norm squared} of $w\in W$.
                                                                } 
 denoted $\|\cdot\|^2$.
By definition, the {\it spin group} $\SO(1,3)$ is the subgroup of the multiplicative group
 $\Cl^\ast(V,\langle\,,\,\rangle)^{\even}$ of $\Cl(V,\langle\,,\,\rangle)^{\even}$ generated
 by elements in $V$ of norm squared $\pm1$ (i.e.\ the unit ``sphere" in $(V,\langle\,,\,\rangle)$).
Its connected component that contains the identity element will be denoted $\Spin^0(1,3)$ and
 called the {\it identity-component} of the spin group.
The isometry on $(V, \langle\,,\,\rangle)$, $v\mapsto -v$,
  induces an algebra-automorphism $^\sim:  \Cl(V,\langle\,,\,\rangle)\rightarrow \Cl(V, \langle\,,\,\rangle)$,
  which in turn defines a {\it twisted Adjoint representation} $\widetilde{\Ad}$  of $\Cl^\ast(V,\langle\,,\,\rangle)$
  on $\Cl(V, \langle\,,\,\rangle)$ (as a representation of a Lie group on a vector space)
  $$
    \widetilde{\Ad}_a b\;: =\; \tilde{a} b a^{-1}\hspace{2em}
	 \mbox{for $a\in \Cl^\ast(V,\langle\,,\,\rangle)$ and $b\in \Cl(V,\langle\,,\,\rangle)$}\,.
  $$
In terms of $\widetilde{\Ad}$, the spin group is characterized by
 $$
   \begin{array}{lcl}
    \Spin(1,3)     & =
	  & \{a\in \Cl^\ast(V,\langle\,,\,\rangle)^{\even}\,|\, \widetilde{\Ad}_a(V)\subset V\,,\;  \|a\|^2=\pm 1 \}
	                                                                     \\[.8ex]
	\Spin^0(1,3) & =
	  & \{a\in \Cl^\ast(V,\langle\,,\,\rangle)^{\even}\,|\, \widetilde{\Ad}_a(V)\subset V\,,\;  \|a\|^2= 1 \} \,.
   \end{array}
 $$
Furthermore, the $\Spin(1,3)$-action on $V$ via the twisted Adjoint representation
  preserves the inner product $\langle\,,\,\rangle$.
This gives rise to the double covers
 $$
   \Spin(1,3)\;\longrightarrow\; \SO(1,3) \hspace{2em}\mbox{and}\hspace{2em}
   \Spin^0(1,3)\;\longrightarrow\; \SO^\uparrow(1,3)
 $$
 with kernel $\{1, -1\}\simeq{\Bbb Z}/2$.
Here,
 $\SO(1,3)$ is the isometry group of $(V, \langle\,,\,\rangle)$ that preserves a fixed orientation of $V$;
 $\SO^\uparrow(1,3)\subset \SO(1,3)$ is its subgroup that preserves in addition a specified time direction on $V$,
 which is identical to the connected component of $\SO(1,3)$ that contains the identity element.

\bigskip

\begin{flushleft}
{\bf Weyl spinors with a symplectic pairing $\varepsilon$}
\end{flushleft}
Up to equivalences,
 $\Cl(V, \langle\,,\,\rangle)$ has a unique irreducible complex representation $S$,
    of complex dimension $4$ (i.e.\:the Dirac spinors),
while
 $\Cl(V, \langle\,,\,\rangle)^{\even}$
 has two inequivalent irreducible complex representations $S^\prime$ and $S^{\prime\prime}$,
 both of complex dimension $2$, that are complex conjugate to each other.
$S\simeq S^\prime\oplus S^{\prime\prime}$ as $\Cl(V, \langle\,,\,\rangle)^{\even}$-modules
 under the inclusion $\Cl(V, \langle\,,\,\rangle)^{\even}\subset \Cl(V, \langle\,,\,\rangle)$.
The application of $\Cl(V, \langle\,,\,\rangle)^{\odd}$ on $S$ exchanges $S^\prime$ and $S^{\prime\prime}$.
Elements in $S^\prime$ and $S^{\prime\prime}$ are called {\it Weyl spinors} in physics literature.
Let
 $\hat{\hspace{1ex}}: \Cl(V, \langle\,,\,\rangle)\rightarrow \Cl(V, \langle\,,\,\rangle)$
 be an involution on $\Cl(V, \langle\,,\,\rangle)$
 generated by the correspondence
  $v_1\otimes\cdots\otimes v_p \mapsto (-v_p)\otimes\cdots \otimes(-v_1)$.
Then, up to a complex constant,
  there is a unique complex symplectic bilinear form $\varepsilon$ on $S^\prime$ (resp.\:$S^{\prime\prime}$)
 such that
  $\varepsilon(as_1, s_2)= \varepsilon(s_1, \hat{a}s_2)$
 for $a\in \Cl(V, \langle\,,\,\rangle)^{\even}$ and $s_1, s_2\in S^{\prime}$ (resp.\:$S^{\prime\prime}$).
The restriction of $\Cl(V, \langle\,,\,\rangle)^{\even}$ to $\Spin^0(1,3)$
 gives a group-isomorphism $\Spin^0(1,3)\simeq \SL(2,{\Bbb C})$.

\bigskip

\begin{remark} $[$choice of $\varepsilon]$\; {\rm
Through the above highlights, note that, up to an equivalence of representations,
 the above construction is canonically and uniquely associated to $(V, \langle\,,\,\rangle)$,
 except for the choice of the complex symplectic bilinear form $\varepsilon$ on $S^\prime$ and $S^{\prime\prime}$.
The symplectic complex bilinear form $\varepsilon$ on $S$ can be chosen to be real and hence pass to $\varepsilon$
 on $S^{\prime\prime}$.
Under the reality constraint, the ambiguity remains is a positive constant in ${\Bbb R}_{> 0}$.
The isomorphisms $S^\prime\stackrel{\sim}{\rightarrow}S^{\prime,\vee}$
   and $S^{\prime\prime}\stackrel{\sim}{\rightarrow}S^{\prime\prime,\vee}$
   determined by $\varepsilon$ via the ${\Bbb C}$-linear map\footnote{{\it Note for mathematicians}\hspace{1em}
                                                     There can be three other conventions in physics literature to set the isomorphisms
													   $S^\prime\stackrel{\sim}{\rightarrow}S^{\prime,\vee}$ and
													   $S^{\prime\prime}\stackrel{\sim}{\rightarrow} S^{\prime\prime,\vee}$
													   via $\varepsilon$\,:\;
													  (2) $s\mapsto \varepsilon(s,\cdot)$
																     for $s\in S^{\prime}$ or $S^{\prime\prime}$;
                                                      (3) $s^\prime\mapsto \varepsilon(\cdot, s^\prime)$ and
																$s^{\prime\prime}\mapsto \varepsilon(s^{\prime\prime}, \cdot)$
																     for $s^\prime\in S^{\prime}$ and
																	      $s^{\prime\prime}\in S^{\prime\prime}$; and
                                                      (4) $s^\prime\mapsto \varepsilon(s^\prime, \cdot)$ and
																$s^{\prime\prime}\mapsto \varepsilon(\cdot, s^{\prime\prime})$
																     for $s^\prime\in S^{\prime}$ and
																	      $s^{\prime\prime}\in S^{\prime\prime}$.
                                                     Since $\varepsilon$ is symplectic, different conventions lead to a discrepancy
													   by  factors of $-1$.
                                                     Here, we follow the convention of [Wess \& Bagger: Appendix Eq.\:(A.9)].
                                                                } 
       $s\mapsto   s^\vee =\varepsilon(\,\cdot\,, s)$
   can thus be canonically specified only up to a complex constant.
The induced symplectic form, still denoted by $\varepsilon$, on $S^{\prime,\vee}$ is defined by requiring
 $\varepsilon(s^\vee, t^\vee)= - \varepsilon(s, t)$ for $s, t\in S^\prime$.\footnote{One
                                                     may define the induced symplectic form $\varepsilon$ on $S^{\prime\vee}$
													    (resp.\ $S^{\prime,\vee}$) by requiring
													  $\varepsilon(s^\vee, t^\vee)=\varepsilon(s, t)$ for $s, t\in S^\prime$
													    (resp.\ $S^{\prime\prime}$).			
                                                     However, since we choose to define $s^\vee$ as $\varepsilon(\cdot, s)$,
													   rather than $\varepsilon(s, \cdot)$, for $s\in S^\prime$,
													   it is more natural to require $\varepsilon(s^\vee, t^\vee)= - \varepsilon(s,t) $
													    ($=\varepsilon(t,s)$).
                                                     This is also the convention used in [Wess \& Bagger].
                                                                }  
Similarly for the induced symplectic form, also denoted by $\varepsilon$, on $S^{\prime\prime, \vee}$.
}\end{remark}

\bigskip

\begin{flushleft}
{\bf The Clifford multiplication and the $\Spin^0(1,3)$-module-isomorphism
         $V^\vee_{\Bbb C}\simeq S^\prime \otimes_{\Bbb C}S^{\prime\prime}$}
\end{flushleft}
The representation of $\Cl(V, \langle\,,\,\rangle)$ on $S\simeq S^\prime\oplus S^{\prime\prime}$
 (i.e.\:the Clifford multiplication)
 and the inclusion $V\subset \Cl(V, \langle\,,\,\rangle)^\odd$
 induces a homomorphism
 $$
  V_{\Bbb C}:= V\otimes_{\Bbb R}{\Bbb C}\; \longrightarrow\;
  \Hom_{\Bbb C}(S^\prime, S^{\prime\prime})
    \simeq S^{\prime\,\vee}\otimes_{\Bbb C} S^{\prime\prime}
 $$
 as $\Spin^0(1,3)$-modules.
This homomorphism turns out to be an isomorphism.
Together with the $\Spin^0(1,3)$-module isomorphisms
 $S^\prime\simeq S^{\prime\,\vee}$ and $S^{\prime\prime}\simeq S^{\prime\prime\,\vee}$,
one has the $\Spin^0(1,3)$-module isomorphism
  $V^\vee_{\Bbb C}\simeq S^\prime \otimes_{\Bbb C}S^{\prime\prime}$.

\bigskip

\begin{flushleft}
{\bf The explicit presentation in [Wess \& Bagger: Appendix A]}
\end{flushleft}
An explicit presentation of the above is given by [Wess \& Bagger: Appendix A],
  which we recall here and will used in the current notes.

\bigskip

\begin{definition} {\bf [Weyl spinors with a symplectic pairing $\varepsilon$]}\; {\rm
 Under the isomorphism $\Spin^0(1,3)\simeq \SL(2,{\Bbb C})$,
  the two inequivalent Weyl spinor representations $S^\prime$ and $S^{\prime\prime}$ of $\Spin^0(1,3)$ are given
  respectively by the fundamental representation ${\Bbb C}^2$  of $\SL(2,{\Bbb C})$ and its complex conjugate.
 The representation of $\SL(2,{\Bbb C})$ on $S^{\prime,\vee}$ (resp.\ $S^{\prime\prime,\vee}$)
   is equivalent to that on $S^\prime$ (resp.\ $S^{\prime\prime}$)
   and is given by $\eta \mapsto  (m^{-1})^t \eta$ for $m\in \SL(2,{\Bbb C})$
   (resp.\ $\bar{\eta} \mapsto  (\overline{m}^{-1})^t\bar{\eta}$) for $m\in \SL(2,{\Bbb C})$.
 Here $(\,\cdot\,)^t$ is the transpose of a matrix $(\,\cdot\,)$.
 In terms of components of Weyl spinors, these representations are given explicitly
   by\footnote{{\it Note for mathematicians}\hspace{1em}
                         Here we respect physicists' standard notation that
						  the conjugate Weyl spinor in $S^{\prime\prime}$ is denoted with a
						   bar $\bar{\hspace{1ex}}$ together with a dotted spinor index $\dot{\beta}$.
						 Though this may look redundant for mathematicians, there is a good reason for this:
                           In situations physicists want to take a $2$-component Weyl spinor as a whole,
						     unbar versus bar distinguishes the two Weyl spinors in inequivalent representations of the Lorentz group,
						  for example,
							  $\theta\theta\;  (:= \sum_\alpha \theta^\alpha\theta_\alpha)$
							      versus $\bar{\theta}\bar{\theta}\;
								                (:=  \sum_{\dot{\beta}}\bar{\theta}_{\dot{\beta}}\bar{\theta}^{\dot{\beta}})$,
						   while in situations
						     that involve only spinorial indices,
							    undotted versus dotted distinguishes the two inequivalent spinor representations the expression must transform
								accordingly under the covering {\it Spin} group of the Lorentz group,
						     for example,
							   (1) $\varepsilon^{\alpha\gamma}$ versus $\varepsilon^{\dot{\beta}\dot{\delta}}$ and
							   (2) $\sigma^\mu_{\alpha\dot{\beta}}$.
                           Employing both notational conventions on a spinor component reinforces the distinction of the two inequivalent Weyl spinors.
                         }
     \marginpar{\vspace{1em}\raggedright\tiny
         \raisebox{-1ex}{\hspace{-2.4em}\LARGE $\cdot$}Cf.\,[\,\parbox[t]{20em}{Wess
		\& Bagger:\\ Eq.\:(A.1)].}}
  $$
   \begin{array}{llccll}
    \mbox{on $S^\prime$}:
	  &  (\psi_\alpha)_{\alpha} \mapsto
	         \mbox{\Large(} \mbox{$\sum$}_{\beta}\,m_\alpha\,\!^\beta \psi_\beta
			    \mbox{\Large)}_{\alpha}\,,
	  &&
	    & \mbox{on $S^{\prime\prime}$}:
		&   (\bar{\psi}_{\dot{\alpha}})_{\dot{\alpha}} \mapsto
	         \mbox{\Large(} \mbox{$\sum$}_{\dot{\beta}}\,
			        \overline{m}_{\dot{\alpha}}\,\!^{\dot{\beta}} \bar{\psi}_{\dot{\beta}}
			    \mbox{\Large)}_{\dot{\alpha}}\,,
		\\[2ex]
     \mbox{on $S^{\prime,\vee}$}:
	  & (\psi^\alpha)_{\alpha} \mapsto
	         \mbox{\Large(} \mbox{$\sum$}_{\beta}\, m^{-1}\,\!_\beta\,\!^\alpha \psi^\beta
			    \mbox{\Large)}_{\alpha}\,,
	  &&
	    & \mbox{on $S^{\prime\prime, \vee}$}:
		&     (\bar{\psi}^{\dot{\alpha}})_{\dot{\alpha}} \mapsto
	         \mbox{\Large(} \mbox{$\sum$}_{\dot{\beta}}\,
			        \overline{m}^{-1}\,\!_{\dot{\beta}}\,\!^{\dot{\alpha}} \bar{\psi}^{\dot{\beta}}
			    \mbox{\Large)}_{\dot{\alpha}}\,,
  \end{array}
 $$
 for $m=(m_\alpha\,\!^\beta)_{\alpha\beta}\in \SL(2,{\Bbb C})$;
 cf.\;[W-B: Appendix A, Eq.\,(A.1)].

 We fix the symplectic pairing $\varepsilon$
   on $S^\prime$ and $S^{\prime\prime}$ to be the standard/defining symplectic pairing   and
   on $S^{\prime\,\vee}$ and $S^{\prime\,\vee}$ the negative standard symplectic pairing:
(as anticommuting $2$-tensors)
     \marginpar{\vspace{1em}\raggedright\tiny
         \raisebox{-1ex}{\hspace{-2.4em}\LARGE $\cdot$}Cf.\,[\,\parbox[t]{20em}{Wess
		\& Bagger:\\ above Eq.\:(A.8)].}}
 $$
  \varepsilon^{12}= - \varepsilon^{21}
	    = \varepsilon^{\dot{1}\dot{2}}= - \varepsilon^{\dot{2}\dot{1}}\;
		=\; 1\,,\hspace{2em}
     \varepsilon_{12}= - \varepsilon_{21}
	    = \varepsilon_{\dot{1}\dot{2}}= - \varepsilon_{\dot{2}\dot{1}}\;
		=\; -1\,.
 $$		
 By construction,
  $\varepsilon$ on $S^\prime$, $S^{\prime\prime}$, $S^{\prime,\vee}$, and $S^{\prime\prime,\vee}$
  are invariant under the $\SL(2,{\Bbb C})$-action.
 Also note that
   $$
    \mbox{$\sum$}_\beta \varepsilon_{\alpha\beta}\varepsilon^{\beta\gamma}\;
	  =\; \delta_\alpha^{\;\gamma}
	 \hspace{2em}\mbox{and}\hspace{2em}
	\mbox{$\sum$}_{\dot{\beta}}
	   \varepsilon_{\dot{\alpha}\dot{\beta}}\varepsilon^{\dot{\beta}\dot{\gamma}}\;
	  =\; \delta_{\dot{\alpha}}^{\;\dot{\gamma}}	
   $$
  for our choice of $\varepsilon$.
  Here, $\delta_\tinybullet^{\,\tinybullet}$ is the Kronecker delta:
    $\delta_\alpha^{\;\gamma}=1$ for $\alpha=\gamma$, else $=0$;
	similarly for $\delta_{\dot{\alpha}}^{\;\dot{\gamma}}$.
}\end{definition}

\medskip

\begin{definition} {\bf [raising/lowering Weyl spinorial indices via $\varepsilon$]}\; {\rm
 Continuing Definition~1.1.2,
we adopt the following rule to raise or lower a spinorial index:
     \marginpar{\vspace{1em}\raggedright\tiny
         \raisebox{-1ex}{\hspace{-2.4em}\LARGE $\cdot$}Cf.\,[\,\parbox[t]{20em}{Wess
		\& Bagger:\\ Eq.\:(A.9)].}}
  $$
   \begin{array}{c}	
     \psi^\alpha\;=\; \sum_{\beta=1,2}\varepsilon^{\alpha\beta}\psi_{\beta}\,,\hspace{2em}
     \psi_\alpha\;=\; \sum_{\beta=1,2}\varepsilon_{\alpha\beta}\psi^{\beta}\,, \\[1.2ex]
	 \bar{\psi}^{\dot{\alpha}}\;
	    =\; \sum_{\dot{\beta}=\dot{1},\dot{2}}
		    \varepsilon^{\dot{\alpha}\dot{\beta}}\bar{\psi}_{\dot{\beta}}\,,\hspace{2em}
     \bar{\psi}_{\dot{\alpha}}\;
	    =\; \sum_{\dot{\beta}=\dot{1},\dot{2}}
		    \varepsilon_{\dot{\alpha}\dot{\beta}}\bar{\psi}^{\dot{\beta}}
   \end{array}	
  $$
 and similarly for any object with upper or lower, undotted or dotted spinorial indices.
} \end{definition}
 
\bigskip
 
Note that, by convention from differential geometry,
 tensorial indices from $V$ or $V_{\Bbb C}$ can be raised or lowered via the inner product $\langle\,,\,\rangle$.
(Note that the inner product on $V$ extends to a complex inner product, still denoted $\langle\,,\,\rangle$,
    on $V_{\Bbb C}$ via ${\Bbb C}$-bi-linearity.)
 
\bigskip

\begin{remark} $[S^{\prime\,\vee}\simeq S^\prime$ and $S^{\prime\prime\,\vee}\simeq S^{\prime\prime}]$\;
{\rm
 Let
  $$
   m\;=\;\left(\!\! \begin{array}{rr} a & b \\ c & d \end{array}\!\!\right)  \;\in\; \SL(2,{\Bbb C})\,.
  $$
 Then
  \begin{eqnarray*}
   (m^{-1})^t & = &
      \left(\!\! \begin{array}{rr} d & -c \\ -b & a \end{array} \!\!\right)\;\;
	 =\;\;   \left(\!\! \begin{array}{rr} 0 & -1 \\ 1 & 0 \end{array} \!\!\right)
	   \left(\!\! \begin{array}{rr} a & b \\ c & d \end{array} \!\!\right)
	     \left(\!\! \begin{array}{rr} 0 & -1 \\ 1 & 0 \end{array} \!\!\right)^{-1}\,;
		 \\[.6ex]
  (\overline{m}^{\,-1})^t & = &
      \left(\!\! \begin{array}{rr} \bar{d} & -\bar{c} \\ -\bar{b} & \bar{a} \end{array} \!\!\right)\;\;
	 =\;\;   \left(\!\! \begin{array}{rr} 0 & -1 \\ 1 & 0 \end{array} \!\!\right)
	   \left(\!\! \begin{array}{rr} \bar{a} & \bar{b} \\ \bar{c} & \bar{d} \end{array} \!\!\right)
	     \left(\!\! \begin{array}{rr} 0 & -1 \\ 1 & 0 \end{array} \!\!\right)^{-1}\,.
  \end{eqnarray*}
  This realizes the equivalences of $\Spin^0(1,3)$-modules
   $S^{\prime\,\vee}\simeq S^\prime$ and $S^{\prime\prime\,\vee}\simeq S^{\prime\prime}$
   that respect the rule of raising and lowering spinorial indices via $\varepsilon$.
}\end{remark}
   
\bigskip

Let $M_2({\Bbb C})$ be the algebra of $2\times 2$ matrices over ${\Bbb C}$.
The dual vector space with inner product $(V^\vee, \langle\,,\,\rangle)$ is realized
 as the ${\Bbb R}$-subspace of Hermitian matrices in $M_2({\Bbb C})$
 $$
  p=(p_0, p_1, p_2, p_3)^t\in V^\vee \;
    \longmapsto\;
	\tilde{p}\;
	 :=\;
	    \left(
		  \begin{array}{cc}
		    -p_0  + p_3   & p_1 -\sqrt{-1}\,p_2 \\
			p_1 + \sqrt{-1}\,p_2 & -p_0 - p_3
		  \end{array}
	     \right)
  $$		
  with the inner product realized by the quadratic form $\langle \tilde{p}, \tilde{p}\rangle:= - \det(\tilde{p})$.
Under this isomorphism, $\SL(2,{\Bbb C})$ acts on $V^\vee$ by
 $$
   \tilde{p}\; \longmapsto\; m\,\tilde{p}\, \overline{m}^t \; =:\; m\,\tilde{p}\,m^\dagger
 $$
 for $m\in \SL(2,{\Bbb C})$.
It preserves $\langle\tilde{p}, \tilde{p} \rangle$   and
 realizes the double covering map $\SL(2,{\Bbb C})\rightarrow \SO^\uparrow(1,3)$.
The fact that $\SL(2,{\Bbb C})$ acts on $\tilde{p}$
       from the left by multiplication by $m$ and from the right by multiplication by $\overline{m}^t$
 implies that
  the correspondence $p\mapsto \tilde{p}$ induces an isomorphism
  $$
    V^{\vee}_{\Bbb C}
	  \stackrel{\sim}{\rightarrow}\;  S^\prime\otimes S^{\prime\prime}
  $$
  as (left) $\SL(2,{\Bbb C})$-modules.
This realizes the Clifford multiplication.
Denote
 $$
    \sigma^0 := \left(\!\!\begin{array}{rr} -1 & 0  \\ 0 & -1\end{array}\!\!\right)\,, \hspace{2em}
    \sigma^1 := \left(\!\!\begin{array}{rr}  0 & 1  \\ 1 & 0\end{array}\!\!\right)\,, \hspace{2em}
	\sigma^2 := \left(\!\!\begin{array}{rr}  0 &  -\sqrt{-1}  \\ \sqrt{-1} & 0\end{array}\!\!\right)\,,
	                \hspace{1em}
	\sigma^4 := \left(\!\!\begin{array}{rr}  1 & 0  \\ 0 & -1\end{array}\!\!\right)	
 $$
(the {\it Pauli matrices}). Then, one may write
 $$
  \tilde{p}\;=\; p_0\,\sigma^0 + p_1\,\sigma^1 + p_2\,\sigma^2 + p_3\,\sigma^3 \;
  =:\; (p_{\alpha\dot{\beta}})_{\alpha\dot{\beta}}\,.
 $$
The double indices $\alpha\dot{\beta}$ for entries of $\tilde{p}$ is justified by the isomorphism
  $V^\vee_{\Bbb C}\simeq S^\prime\times S^{\prime\prime}$.
The conversion rules are
     \marginpar{\vspace{1em}\raggedright\tiny
         \raisebox{-1ex}{\hspace{-2.4em}\LARGE $\cdot$}Cf.\,[\,\parbox[t]{20em}{Wess
		\& Bagger:\\ Eq.\:(A.13)].}}
 $$
  p_{\alpha\dot{\beta}}\;=\; \mbox{$\sum$}_{\mu=0}^3 \sigma^\mu_{\alpha\dot{\beta}}p_\mu
   \hspace{2em}\mbox{and}\hspace{2em}
  p^\mu\;
    =\; -\,\mbox{\large $\frac{1}{2}$}\,
               \sum_{\alpha=1,2;\,\dot{\beta}=\dot{1}, \dot{2}}
			     \bar{\sigma}^{\mu\dot{\beta}\alpha}\,p_{\alpha\dot{\beta}}\,,
 $$
 where
 $\bar{\sigma}^{\mu\dot{\beta}\alpha}
     := \sum_{\gamma\dot{\delta}}
	      \varepsilon^{\alpha\gamma}\varepsilon^{\dot{\beta}\dot{\delta}}
		    \sigma^\mu_{\gamma\dot{\delta}}$.

\bigskip

\subsection{The function-ring of a $d=3+1$, $N=1$ towered superspace and its distinguished subrings
   {\normalsize (under a trivialization of the spinor bundle by covariantly constant sections)\;
    {\normalsize (cf.\:[Wess \& Bagger: Chapter IV and Appendices A \& B])}}
	                       }

We recall in this subsection how we come to the notion of `{\it towered superspace}' from [L-Y5: Sec.\:1] (SUSY(1))
 for terminology and notations    and
 explain some upgrade on the notion of `{\it physics sector}'.\footnote{In
                                                                [L-Y5: Sec.\:1] (SUSY(1)) we define the function ring of the physics sector as
																 subring generated by the chiral superfield and the antichiral superfield.
																Then in [L-Y5: Sec.\:3] (SUSY(1))
																  we define `vector superfield' from this physics sector to construct
																  a supersymmetric $U(1)$ gauge theory coupled with matters.
																It will turn out that to fit [Wess \& Bagger] exactly,
																  one has to look for the notion of `vector superfields' {\it outside} this subring!
                                                                It is for this reason that in the current new work
																 we abandon the name ``{\it physics sector}" assigned to this subring
																 since that would give a wrong impression or implication that
																 elements outside this subring is not physics-relevant.
																In the current work, this subring is given a new name
																  the `{\it small function-ring}' of the towered superspace in question.
																It is the subring that is involved in
																  Wess-Zumino models (cf.\:Sec.\:3)  and
																  $d=3+1$, $N=1$ nonlinear sigma models (cf.\:Sec.\:5).
											          } 
Readers are referred to ibidem for more details.

\bigskip

\begin{flushleft}
{\bf $d=3+1$, $N=1$ towered superspaces and their function ring}
\end{flushleft}
Let
 \begin{itemize}
  \item[\LARGE $\cdot$]
   $X={\Bbb R}^{3+1}$ be the $3+1$-dimensional Minkowski space-time;
     as a $C^\infty$-scheme $(X, {\cal O}_X)$,
	 where ${\cal O}_X$ is the sheaf of smooth functions on $X$;
   with coordinate functions $(x^\mu)_\mu = (x^0, x^1, x^2, x^3)$ and
   the metric $ds^2 = -(dx^0)^2 + (dx^1)^2+ (dx^2)^2+ (dx^3)^3$;
   tangent bundle $T_X$, cotangent bundle $T^\ast_X$,
   their corresponding sheaf of sections ${\cal T}_X$, ${\cal T}_X^\ast$  and
   complexification
	  ${\cal T}_X^{\,{\Bbb C}}:= {\cal T}_X\otimes_{{\cal O}_X}\!\!{\cal O}_X^{\,\Bbb C}$ and
	  ${\cal T}_X^{\ast, {\Bbb C}}:= {\cal T}_X^\ast\otimes_{{\cal O}_X}\!\!{\cal O}_X^{\,\Bbb C}$, 	
	all equipped with the Levi-Civita connection;
  here ${\cal O}_X^{\,\Bbb C}$ is the sheaf of complex-valued smooth functions on $X$ and
    $(X, {\cal O}_X^{\,\Bbb C})$ as a complexified $C^\infty$-scheme
	  with $C^\infty$ hull ${\cal O}_X\subset {\cal O}_X^{\,\Bbb C}$;
   
  \item[\LARGE $\cdot$]
   $P_X$ the principal Lorentz-frame bundle over $X$ with the Levi-Civita connection,\\
   ${\cal P}_X$ the corresponding principal sheaf of Lorentz frames with a connection;
   
  \item[\LARGE $\cdot$]
   $S^\prime$, $S^{\prime\prime}$ be the Weyl spinor bundles and
   ${\cal S}^\prime$, ${\cal S}^{\prime\prime}$ the corresponding sheaf of sections,
      all equipped with the induced connection and covariantly constant symplectic pairing $\varepsilon$
	 that respect the complex conjugation
       $\bar{\hspace{1ex}}: S^\prime \leftrightarrow S^{\prime\prime}$,
	   ${\cal S}^\prime \leftrightarrow {\cal S}^{\prime\prime}$;
   $S^{\prime\,\vee}$, $S^{\prime\prime\,\vee}$	 the dual spinor bundle of $S^\prime$ and $S^{\prime\prime}$
     respectively; \
	${\cal S}^{\prime\,\vee}$ and ${\cal S}^{\prime\prime\,\vee}$ the corresponding sheaves.	
	
  \item[\LARGE $\cdot$]	
   Recall Definition~1.1.3 and Remark~1.1.4.
   Let $(\theta_1, \theta_2)$ be a pair of independent covariantly constant sections of $S^\prime$.
   This then gives rise to pairs of covariantly constant sections
     $(\theta^1, \theta^2)$ in $S^{\prime\,\vee}$,
	 $(\bar{\theta}_{\dot{1}}, \bar{\theta}_{\dot{2}})$ in $S^{\prime\prime}$,  and
	 $(\bar{\theta}^{\dot{1}}, \bar{\theta}^{\dot{2}})$ in $S^{\prime\prime\,\vee}$.
	
  \item[]
   When in need of taking a copy of spinor bundle in a construction,
       i.e.\:complex rank-$2$ bundle on $X$ with a fixed isomorphism to $S^\prime$,
    $(\theta_1, \theta_2)$ then passes to $(\vartheta_1, \vartheta_2)$ in the copy,
	which in turn gives rise to
	 $(\vartheta^1, \vartheta^2)$,
	 $(\bar{\vartheta}_{\dot{1}}, \bar{\vartheta}_{\dot{2}})$,  and
	 $(\bar{\vartheta}^{\dot{1}}, \bar{\vartheta}^{\dot{2}})$	
	 respectively in the various corresponding copies of spinor bundles via taking dual or complex conjugation.
 \end{itemize}

\bigskip

From the very start, physicists' design of superfields
  ({\it Abdus Salam} and {\it John Strathdee}:  [S-S: Eqs.\:(I.15) \& (I.16)] (1978))
  wants them to form a ring,
 i.e.\:addition and multiplication of superfields make sense and are still superfields.
In essence, a `towered superspace' is the `space' that takes this ring as its function ring.
From the fact that a single `superfield' on $X$ combines both bosonic field(s) and fermionic field(s) on $X$ into its components
 (and hence ``super") and the number and details of bosonic fields and fermionic fields depend
 on the supersymmetric quantum-field-theory model one wants to construct,
this towered superspace depends on the supersymmetric quantum-field-theory model one wants to construct as well.
As a complexified ${\Bbb Z}/2$-graded $C^\infty$-scheme, this towered superspace can be described level-by-level
 as follows:
(The presentation here is for the case of $d=3+1$, $N=1$ supersymmetric quantum field theories and
     assuming that all fermionic fields are described by sections of Weyl spinor bundles $S^\prime$, $S^{\prime\prime}$.
	Other situations are similar.)
\begin{itemize}
 \item[(1)] [\,{\it fundamental/ground level}\,]\hspace{1em}
  This is the super homogeneous space $\widehat{X}$
    from the quotient of the super Poincar\'{e} group by the Lorentz subgroup.
  As a complexified ${\Bbb Z}/2$-graded $C^\infty$-scheme, it has
    the underlying topology $X$   and
    the structure sheaf
    $$
     {\cal O}_{\widehat{X}}:= \widehat{\cal O}_X
	 :=  \mbox{$\bigwedge$}^\tinybullet_{{\cal O}_X^{\,\Bbb C}}
	         ({\cal S}^{\prime\,\vee} \oplus {\cal S}^{\prime\prime\,\vee})\,,
    $$
     with the ${\Bbb Z}/2$-grading given by
    $$
     \mbox{$\bigwedge$}^\tinybullet_{{\cal O}_X^{\,\Bbb C}}
	    ({\cal S}^{\prime\,\vee} \oplus {\cal S}^{\prime\prime\,\vee})\;
      =\; \mbox{$\bigwedge$}^\even_{{\cal O}_X^{\,\Bbb C}}
	         ({\cal S}^{\prime\,\vee} \oplus {\cal S}^{\prime\prime\,\vee})
	       \oplus
          \mbox{$\bigwedge$}^\odd_{{\cal O}_X^{\,\Bbb C}}
		     ({\cal S}^{\prime\,\vee} \oplus {\cal S}^{\prime\prime\,\vee})\;
	 =:\; \widehat{\cal O}_X^{\,\even}\oplus \widehat{\cal O}_X^{\,\odd}
    $$
	and
    the $C^\infty$-hull given by
     $$
      {\cal O}_X\oplus
	    \mbox{$\bigwedge$}^{\even,\, \ge 2}_{{\cal O}_X^{\,\Bbb C}}
		  ({\cal S}^{\prime\,\vee} \oplus {\cal S}^{\prime\prime\,\vee})\,.
     $$
   This is the sheaf of Grassmann algebras generated by the ${\cal O}_X^{\,\Bbb C}$-module
    ${\cal S}^{\prime\,\vee}\oplus {\cal S}^{\prime\prime\,\vee}$.
   Covariantly constant sections of ${\cal S}^{\prime\,\vee}\oplus {\cal S}^{\prime\prime\,\vee}$	
     provide fermionic/anticommuting coordinate functions on $\widehat{X}$ and
   the action of the super Poincar\'{e} group on $\widehat{X}$ is realized as automorphisms
    on the function ring $C^\infty(\widehat{X})= \widehat{\cal O}_X(X)$ of $X$.
  This gives a representation of the super Poincar\'{e} algebra
    in the ${\Bbb Z}/2$-graded Lie algebra of derivations on $C^\infty(\widehat{X})$.	
  In particular, the infinitesimal supersymmetry generators $Q_\alpha$, $\bar{Q}_{\dot{\beta}}$ and
   the supersymmetrically-invariant-flow generators $D_\alpha$, $\bar{D}_{\dot{\beta}}$
   of [Wess \& Bagger: Eqs.\:(4.4) and (4.5)] are specific derivations on $C^\infty(\widehat{X})$.
 (Cf.\:Example~1.2.5.)
   
 \item[(2)] [\,{\it field/upper level(s)}\,]\hspace{1em}
 The Spin-Statistics Theorem from Quantum Statistical Mechanics and Quantum Field Theory
  says that fermionic fields must be anticommuting by nature.
 Thus, every time a fermionic field appears in the problem, we have to ask:
   \begin{itemize}
    \item[{\bf Q.}]
	  {\it Does it already lie in the existing sheaf of Grassmann algebras of the problem?}
   \end{itemize}
  If not, then we have to enlarge the generating ${\cal O}_X^{\,\Bbb C}$-module of the existing sheaf of Grassmann algebras
    to include the new anticommuting field(s).	
 This gives rise to an inclusion system of function rings or, contravariantly equivalently
  a projection system of complexified ${\Bbb Z}/2$-graded $C^\infty$-schemes over $\widehat{X}$.
 This is why and how a towered superspace appears.
 The new generators to the enlarged function ring are themselves fermionic fields over $X$.
 Specifically, when these fermionic fields are sections of different copies of spinor sheaves,
  we distinguish them by a subscript ${\cal S}^\prime_{\field, i}\oplus {\cal S}^{\prime\prime}_{\field, i}$,
    $i=1,\,\ldots\,, l$, and
  construct the towered superspace as a complexified ${\Bbb Z}/2$-graded $C^\infty$-scheme
   $$
     \widehat{X}^{\widehat{\boxplus}_l}\;
	  :=\; (X, \widehat{\cal O}_X^{\,\widehat{\boxplus}_l})\;
	  :=\; (X, \mbox{$\bigwedge$}_{{\cal O}_X^{\,\Bbb C}}^{\tinybullet}{\cal F})
   $$
  over $\widehat{X}$ with
   $$
    {\cal F}\;
	    :=\;  ({\cal S}^{\prime\,\vee}_{\coordinates}\oplus {\cal S}^{\prime\prime\,\vee}_{\coordinates})		
			    \oplus
				  \mbox{$\bigoplus$}_{i=1}^l
				     ({\cal S}^\prime_{\field, i}\oplus {\cal S}^{\prime\prime}_{\field, i})\,.
   $$
  We say that
    ${\cal S}^\prime_{\field, i}
          \oplus {\cal S}^{\prime\prime}_{\field, i}$
	 contributes to the {\it $i$-th field level} of $\widehat{X}^{\widehat{\boxplus}_l}$.			
 The total level number $l$ is the number of distinct types/species/generations of fermionic fields
   in a $d=3+1$, $N=1$ supersymmetric field theory one wants to construct.
 It can be different theory-by-theory. 	

 \item[(3)] [\,{\it parameter/basement level}\,]\hspace{1em}
  Finally, when physicists working on supersymmetry introduce `Grassmann number' parameter
    $(\eta,\bar{\eta}):= (\eta^1, \eta^2, \bar{\eta}^{\dot{1}}, \bar{\eta}^{\dot{2}})$
    in their computation,
   these `Grassmann number' parameter are meant to be independent of anything else.
  Thus, they should be thought of as constant sections
    of another copy of ${\cal S}^{\prime\,\vee}\oplus {\cal S}^{\prime\prime\,\vee}$.
 Since they are used in the computation, it is appealing in practice (if not in concept) to incorporate them
  into the function ring of the towered superspace and think of
  $\widehat{X}^{\widehat{\boxplus}}$ as over this basic complexified ${\Bbb Z}/2$-graded $C^\infty$-scheme.	
 Thus, we make a final adjustment to $\widehat{X}^{\widehat{\boxplus}}$ by redefining
  $$
     \widehat{X}^{\widehat{\boxplus}_l}\;
	  :=\; (X, \widehat{\cal O}_X^{\,\widehat{\boxplus}_l})\;
	  :=\; (X, \mbox{$\bigwedge$}_{{\cal O}_X^{\,\Bbb C}}^{\tinybullet}{\cal F})
   $$
 with
   $$
    {\cal F}\;
	    :=\;
	    ({\cal S}^{\prime\,\vee}_{\parameter}\oplus {\cal S}^{\prime\prime\,\vee}_{\parameter})
          \oplus
		({\cal S}^{\prime\,\vee}_{\coordinates}\oplus {\cal S}^{\prime\prime\,\vee}_{\coordinates})		
		  \oplus
		  \mbox{$\bigoplus$}_{i=1}^l
		   ({\cal S}^\prime_{\field, i}\oplus {\cal S}^{\prime\prime}_{\field, i})\,.
   $$	
 We say that
    ${\cal S}^{\prime\,\vee}_{\parameter}
          \oplus {\cal S}^{\prime\prime\,\vee}_{\parameter}$
	 contributes to the {\it Grassmann parameter level} of $\widehat{X}^{\widehat{\boxplus}_l}$.
 Similar to the fundamental level $\widehat{X}$ of the tower,
  this level depends only on the supersymmetry in question (here, $d=3+1$, $N=1$).
 Together they form the universal base for all $d=3+1$, $N=1$ towered superspaces.
\end{itemize}

\bigskip

These motivate the following definitions in [L-Y5: Sec.\:1.3] (SUSY(1)):

\bigskip

\begin{definition}
{\bf [$d=4$, $N=1$ towered superspace $\widehat{X}^{\widehat{\boxplus}_l}$ with $l$ field-theory levels]}\;
{\rm
 The complexified ${\Bbb Z}/2$-graded $C^\infty$-scheme
   $$
     \widehat{X}^{\widehat{\boxplus}_l}\;
	  :=\; (X, \widehat{\cal O}_X^{\,\widehat{\boxplus}_l})\;
	  :=\; (X, \mbox{$\bigwedge$}_{{\cal O}_X^{\,\Bbb C}}^{\tinybullet}{\cal F})
   $$
 with
   $$
    {\cal F}\;
	    :=\;
	  ({\cal S}^{\prime\,\vee}_{\parameter}\oplus {\cal S}^{\prime\prime\,\vee}_{\parameter})
				\oplus
		({\cal S}^{\prime\,\vee}_{\coordinates}\oplus {\cal S}^{\prime\prime\,\vee}_{\coordinates})		       
			    \oplus
				  \mbox{$\bigoplus$}_{i=1}^l
				     ({\cal S}^\prime_{\field, i}\oplus {\cal S}^{\prime\prime}_{\field, i})
   $$
  is called the {\it $d=4$, $N=1$ towered superspace with $l$ field-theory levels}.
 Here,
   all ${\cal S}^\prime_{\tinybullet}$
    (resp.\ ${\cal S}^{\prime\prime}_{\tinybullet}$,
	               ${\cal S}^{\prime\, \vee}_{\tinybullet}$, ${\cal S}^{\prime\prime\,\vee}_{\tinybullet}$)
   are copies\footnote{Mathematically
                                           this means that
                                          ${\cal S}^{\prime\,\vee}_{\mbox{\tiny\it coordinates }}$ is isomorphic to
										  ${\cal S}^{\prime\,\vee}$ with a fixed isomorphism; and similarly for all other spinor sheaves
										  that appear as direct summands of ${\cal F}$.
										      } 
      of ${\cal S}^\prime$
    (resp.\  ${\cal S}^{\prime\prime}$,
	               ${\cal S}^{\prime\,\vee}$, ${\cal S}^{\prime\prime\,\vee}$).
 When $l$ is implicit in the problem, we will denote
  $\widehat{X}^{\widehat{\boxplus}_l}$ simply by $\widehat{X}^{\widehat{\boxplus}}$.
 By convention, we will keep the parameter level suppressed when not activated for use in a discussion.
}\end{definition}

\bigskip

\noindent
Figure~1-2-1. (Cf.\:[L-Y5: Definition/Explanation 1.3.2] (SUSY(1)).)
  \begin{figure}[htbp]
 \bigskip
  \centering
  \includegraphics[width=0.80\textwidth]{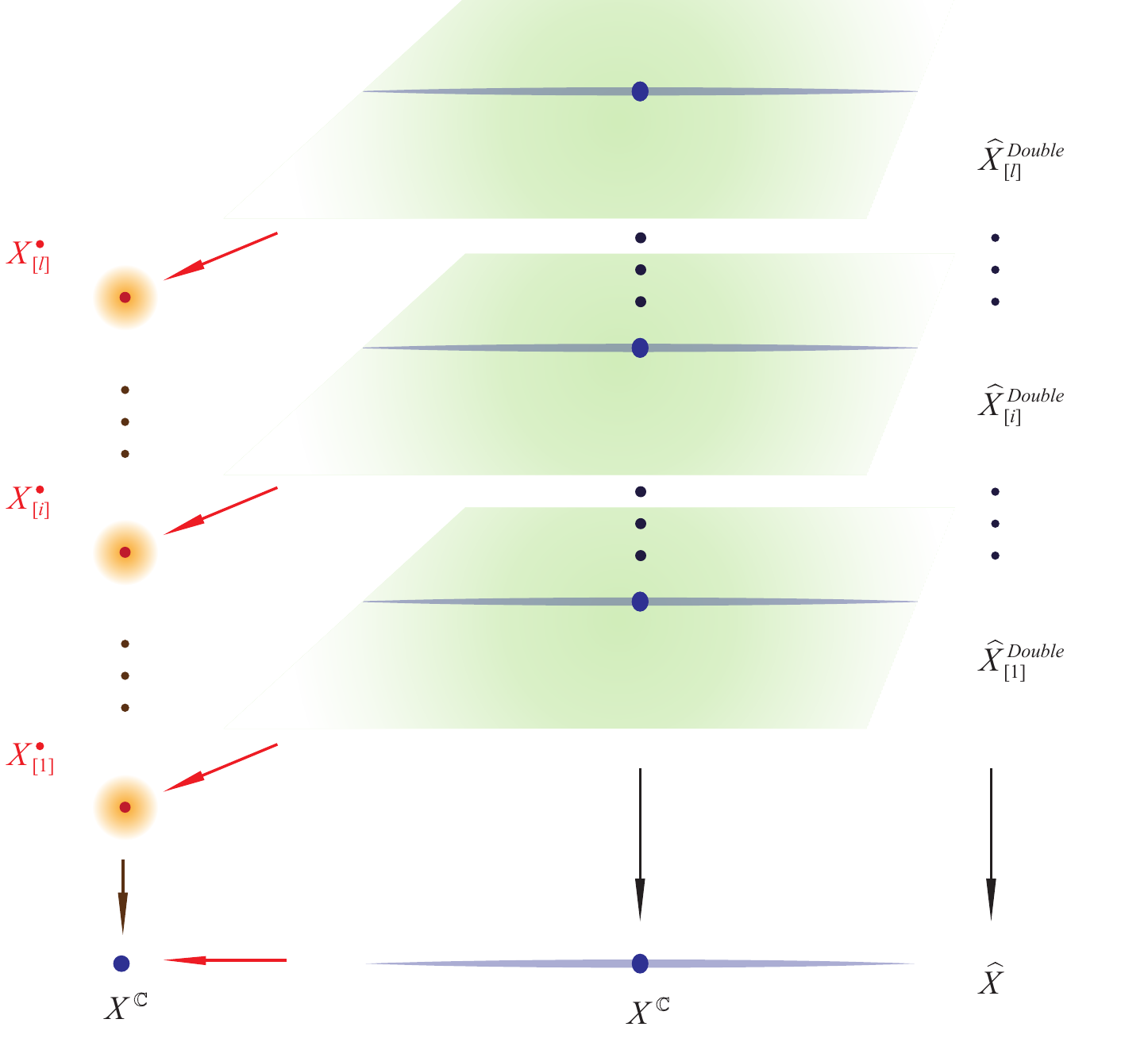}

  \bigskip
  \bigskip
 \centerline{\parbox{13cm}{\small\baselineskip 12pt
  {\sc Figure}~1-2-1.
  (Cf.\:[L-Y5: {\sc Figure} 1-4-1].)
  The space-time coordinate functions $x^\mu$, $\mu=0,1,2,3$,
     and the fermionic coordinate functions $\theta^\alpha$, $\bar{\theta}^{\dot{\beta}}$,
     $\alpha=1,2$, $\dot{\beta}=\dot{1}, \dot{2}$, 	
   generate the function ring of the fundamental superspace $\widehat{X}$
     as a complexified ${\Bbb Z}/2$-graded $C^\infty$-scheme.
  Over it sits a supertower with Grassmann-number level and other field-theory levels that are needed
    for the construction of supersymmetric quantum field theories.
  From the direct-sum expression of the generating sheaf 	
	\begin{eqnarray*}
    {\cal F}
	   & :=\:
	      &  ({\cal S}^{\prime\,\vee}_{\coordinates}\oplus {\cal S}^{\prime\prime\,\vee}_{\coordinates})
		        \oplus
		({\cal S}^{\prime\,\vee}_{\parameter}\oplus {\cal S}^{\prime\prime\,\vee}_{\parameter})\\
        &&\hspace{2em} 		
			    \oplus\,
				  \mbox{$\bigoplus$}_{i=1}^l
				     ({\cal S}^\prime_{\field, i}\oplus {\cal S}^{\prime\prime}_{\field, i})\\[1ex]
      & =
	      &  ({\cal S}^{\prime\,\vee}_{\coordinates}\oplus {\cal S}^{\prime\prime\,\vee}_{\coordinates})
		        \oplus
		({\cal S}^{\prime\,\vee}_{\parameter}\oplus {\cal S}^{\prime\prime\,\vee}_{\parameter})\\[.6ex]
        &&\hspace{2em} 		
			    \oplus\,
				  \mbox{$\bigoplus$}_{i=1}^l				
				    \frac{
					 ({\cal S}^{\prime\,\vee}_{\coordinates}\oplus {\cal S}^{\prime\prime\,\vee}_{\coordinates}
					       \oplus  {\cal S}^\prime_{\field, i}\oplus {\cal S}^{\prime\prime}_{\field, i})}
					{{\cal S}^{\prime\,\vee}_{\coordinates}\oplus {\cal S}^{\prime\prime\,\vee}_{\coordinates}}
   \end{eqnarray*}
   of the structure sheaf $\widehat{\cal O}_X^{\,\widehat{\boxplus}}$
     of $\widehat{X}^{\widehat{\boxplus}}$,
   one may think of each field-theory level as contributing
   a floor-$[i]$
    $$
	   \widehat{X}_{[i]}^{\tinyDouble}\; :=\;
	     \mbox{\Large $($}
		  X, \mbox{$\bigwedge$}_{{\cal O}_X^{\,\Bbb C}}^{\tinybullet}
		     ({\cal S}^{\prime\,\vee}_{\coordinates}\oplus {\cal S}^{\prime\prime\,\vee}_{\coordinates}
					       \oplus  {\cal S}^\prime_{\field, i}\oplus {\cal S}^{\prime\prime}_{\field, i})
		 \mbox{\Large $)$}
	$$
   over $\widehat{X}$  and
 these field-theory floors are glued by the ${\Bbb Z}/2$-graded version of
    fibered product over $\widehat{X}$ to give $\widehat{X}^{\widehat{\boxplus}}$.	
 Each field-theory floor $\widehat{X}_{[i]}^{\tinyDouble}$ has some distinguished sectors
   $X_{[i]}^\tinybullet$ that are purely even.
 They generate some distinguished sectors $\widehat{X}^{\widehat{\boxplus}, \tinybullet}$	
   of $\widehat{X}^{\widehat{\boxplus}}$
    that are also purely even.
 These distinguished sectors are where physics-relevant superfields on $X$ lie.
  }}
\end{figure}

\bigskip

Note that, as an ${\cal O}_X^{\,\Bbb C}$-modules,
 \begin{eqnarray*}
  \lefteqn{
   \widehat{\cal O}_X^{\,\widehat{\boxplus}}\;
    =\;   \mbox{$\bigwedge$}_{{\cal O}_X^{\,\Bbb C}}^{\tinybullet}
	             ({\cal S}^{\prime\,\vee}_{\coordinates}
				         \oplus{\cal S}^{\prime\prime\,\vee}_{\coordinates})   }\\
   && \hspace{3em}						
			  \otimes_{{\cal O}_X^{\,\Bbb C}}
              \mbox{$\bigwedge$}_{{\cal O}_X^{\,\Bbb C}}^{\tinybullet}
	             ({\cal S}^{\prime\,\vee}_{\parameter}
				         \oplus{\cal S}^{\prime\prime\,\vee}_{\parameter})
			  \otimes_{{\cal O}_X^{\,\Bbb C}}
			   \mbox{$\bigotimes_{{\cal O}_X^{\,\Bbb C}}$}_{i=1}^l
                 \mbox{$\bigwedge$}_{{\cal O}_X^{\,\Bbb C}}^{\tinybullet}
	             ({\cal S}^\prime_{\field, i}\oplus{\cal S}^{\prime\prime}_{\field, i})\,.
 \end{eqnarray*}
Each factor $\bigwedge^\tinybullet_{{\cal O}_X^{\,\Bbb C}}(\cdots)$
  in the $\otimes_{{\cal O}_X^{\,\Bbb C}}$-decomposition
contributes to a level/layer/floor of the towered superspace $\widehat{X}^{\widehat{\boxplus}}$.
  
\bigskip

\begin{definition}
{\bf [derivation on $\widehat{X}$ applied to $C^\infty(\widehat{X}^{\widehat{\boxplus}})$]}\;
{\rm
 Let
   $\xi\in \Der_{\Bbb C}(\widehat{X})$  be a derivation on $\widehat{X}$ over ${\Bbb C}$ and
   $\breve{f}\in C^\infty(\widehat{X}^{\widehat{\boxplus}})$.
 Then we define $\xi\breve{f}\in C^\infty(\widehat{X}^{\widehat{\boxplus}})$
   via the built-in inclusion
   $\Der_{\Bbb C}(\widehat{X})
              \hookrightarrow \Der_{\Bbb C}(\widehat{X}^{\widehat{\boxplus}})$.
}\end{definition}

\medskip

\begin{definition} {\bf [complex conjugation vs.\;twisted complex conjugation]}\; {\rm
 The complex conjugation
    $ \bar{\hspace{1.2ex}}:{\cal O}_X^{\,\Bbb C}\rightarrow {\cal O}_X^{\,\Bbb C}$ and
		${\cal S}^\prime\rightarrow {\cal S}^{\prime\prime}\,,\;
	       {\cal S}^{\prime\prime}\rightarrow {\cal S}^\prime$,
    of Weyl spinors
   extends canonically to a {\it complex conjugation}
   $$
     \bar{\hspace{1.2ex}}\;:\; \widehat{\cal O}_X^{\,\widehat{\boxplus}}\;
	    \longrightarrow\; \widehat{\cal O}_X^{\,\widehat{\boxplus}}\,,
   $$
   by setting
   \begin{itemize}	
	\item[(1)]
	 $\overline{\breve{f}+\breve{g}}\;=\; \bar{\breve{f}} + \bar{\breve{g}}$\,;
	
    \item[(2)]  $\overline{\breve{f}\breve{g}}\;=\; \bar{\breve{g}}\bar{\breve{f}}$\,.
   \end{itemize}
  and a	{\it twisted complex conjugation}
  $$
    ^\dag\;:\;  \widehat{\cal O}_X^{\,\widehat{\boxplus}}\;
	                    \longrightarrow\; \widehat{\cal O}_X^{\,\widehat{\boxplus}}\,,
  $$
  by setting
   \begin{itemize}
    \item[(0$^\prime$)]
	 $^\dag = \bar{\hspace{1.2em}}:
	       {\cal O}_X^{\,\Bbb C}\rightarrow {\cal O}_X^{\,\Bbb C}\,;
	      {\cal S}^\prime \rightarrow {\cal S}^{\prime\prime}\,,\;
	      {\cal S}^{\prime\prime}\rightarrow {\cal S}^\prime$\,;
	
	\item[(1$^\prime$)]
	 $(\breve{f}+\breve{g})^\dag\;=\; \breve{f}^\dag + \breve{g}^\dag$\,;
	
    \item[(2$^\prime$)]  $(\breve{f}\breve{g})^\dag\;=\; \breve{g}^\dag \breve{f}^\dag$\,.
   \end{itemize}
 Caution that the order of multiplication is preserved under the complex conjugate $\bar{\hspace{1.2ex}}$
   but is reversed under the twisted complex conjugate $^\dag$.
}\end{definition}
 
\medskip

\begin{definition} {\bf [standard coordinate functions on $\widehat{X}^{\widehat{\boxplus}}$]}\;
{\rm
 Let $\widehat{X}^{\widehat{\boxplus}} = \widehat{X}^{\widehat{\boxplus}_l}$.
 Then,
 the standard coordinate functions $(x,\theta,\bar{\theta})$ on $\widehat{X}$
 extends uniquely to a tuple of coordinate functions
 $$
  (x^\mu, \theta^{\alpha},
      \bar{\theta}^{\dot{\beta}};
	  \eta^{\alpha^\prime}, \bar{\eta}^{\dot{\beta}^\prime};
	  \vartheta^1_{\gamma_1}, \bar{\vartheta}^1_{\dot{\delta}_1} ;\,
	   \cdots\,;
	  \vartheta^l_{\gamma_l}, \bar{\vartheta}^l_{\dot{\delta}_l}  )\;
	  =:\;   (x, \theta, \bar{\theta}, \eta, \bar{\eta},
	               \vartheta, \bar{\vartheta})\;
 $$
 on $\widehat{X}^{\widehat{\boxplus}}$ via
    the $\varepsilon$-tensor
     $\varepsilon:
 	    {\cal S}^\prime\otimes_{{\cal O}_X^{\,\Bbb C}}{\cal S}^\prime
 		      \rightarrow {\cal O}_X^{\,\Bbb C} $,\,
        ${\cal S}^{\prime\prime}\otimes_{{\cal O}_X^{\,\Bbb C}}{\cal S}^{\prime\prime}
 		      \rightarrow {\cal O}_X^{\,\Bbb C} $, 			
   and
   the fixed isomorphisms
  $ {\cal S}^\prime_{\tinybullet}\simeq {\cal S}^\prime$,
  ${\cal S}^{\prime\prime}_{\tinybullet}\simeq {\cal S}^{\prime\prime}$.

 Explicitly, regard
  ${\cal S}^{\prime\,\vee}_{\parameter}$ as a copy of ${\cal S}^{\prime\,\vee}_{\coordinates}$,
  ${\cal S}^{\prime\prime\,\vee}_{\parameter}$
      as a copy of ${\cal S}^{\prime\prime\,\vee}_{\coordinates}$,
  ${\cal S}^\prime_{\field, i}$
    as a copy of $({\cal S}^{\prime\,\vee}_{\coordinates})^\vee = {\cal S}^\prime_{\coordinates}$, and
  ${\cal S}^{\prime\prime}_{\field, i}$
   as a copy of $({\cal S}^{\prime\prime\,\vee}_{\coordinates})^\vee
      = {\cal S}^{\prime\prime}_{\coordinates}$
  under the fixed isomorphisms.
 Then,
   $(\eta^{\alpha^\prime},\bar{\eta}^{\dot{\beta}^\prime})
     = (\theta^{\alpha^\prime},\bar{\theta}^{\dot{\beta}^\prime})$ and
   $(\vartheta^i_{\alpha_i},\bar{\vartheta}^i_{\dot{\beta}_i})
       = (\theta_{\alpha_i}, \bar{\theta}_{\dot{\beta}_i})$
    for all $i$, 	 	
 where $\theta_\alpha = \sum_{\gamma}\varepsilon_{\alpha\gamma}\theta^\gamma$,  	
	         $\theta_{\dot{\beta}}
			   = \sum_{\dot{\delta}}\varepsilon_{\dot{\beta}\dot{\delta}}\bar{\theta}^{\dot{\delta}}$\,.
			
 We shall call $(x, \theta, \bar{\theta}, \eta, \bar{\eta}, \vartheta, \bar{\vartheta})$		
   the {\it standard coordinate functions} on $\widehat{X}^{\widehat{\boxplus}}$.
}\end{definition}	

\bigskip

In terms of this,
 $$
   C^\infty(\widehat{X}^{\widehat{\boxplus}})\;
     =\;  C^\infty(X)^{\Bbb C}
	        [\theta, \bar{\theta}, \eta, \bar{\eta}, \vartheta, \bar{\vartheta}]^{\anticommuting}
 $$
and an $\breve{f}\in C^\infty(\widehat{X}^{\widehat{\boxplus}})$
 has a $(\theta,\bar{\theta})$-expansion
   \begin{eqnarray*}
     \breve{f} & = &  \breve{f}_{(0)}\,
	        +\, \sum_{\alpha}\theta^\alpha \breve{f}_{(\alpha)}\,
			+\, \sum_{\dot{\beta}} \bar{\theta}^{\dot{\beta}} \breve{f}_{(\dot{\beta})}\,
			+\, \theta^1\theta^2  \breve{f}_{(12)}\,
			+\, \sum_{\alpha,\dot{\beta}}
			       \theta^\alpha\bar{\theta}^{\dot{\beta}} \breve{f}_{(\alpha\dot{\beta})}\,  \\
        && 				
            +\, \bar{\theta}^{\dot{1}}\bar{\theta}^{\dot{2}}  \breve{f}_{(\dot{1}\dot{2})}\,
			+\, \sum_{\dot{\beta}}
			        \theta^1\theta^2\bar{\theta}^{\dot{\beta}}  \breve{f}_{(12\dot{\beta})}\,
			+\, \sum_\alpha
			        \theta^\alpha\theta^{\dot{1}}\theta^{\dot{2}}  \breve{f}_{(\alpha\dot{1}\dot{2})}\,
		    +\, \theta^1\theta^2\theta^{\dot{1}}\theta^{\dot{2}} \breve{f}_{(12\dot{1}\dot{2})}			
   \end{eqnarray*}
 with coefficients
   $\breve{f}_{(\tinybullet)}\in
       C^\infty(X)^{\Bbb C}[\eta,\bar{\eta},\vartheta,\bar{\vartheta}]^{\anticommuting}$.

\bigskip

\begin{example} {\bf [special derivations on $C^\infty(\widehat{X})$]}\; {\rm
 Recall from [L-Y4: Sec.\,1.4] (D(14.1))
  the standard {\it infinitesimal supersymmetry generators} on $\widehat{X}$
     \marginpar{\vspace{1em}\raggedright\tiny
         \raisebox{-1ex}{\hspace{-2.4em}\LARGE $\cdot$}Cf.\,[\,\parbox[t]{20em}{Wess
		\& Bagger:\\ Eq.\:(4.4)].}}
 {\small
  $$
   Q_{\alpha}\;
    =\;  \frac{\partial}{\partial \theta^{\alpha}}
            -\, \sqrt{-1}\,\sum_{\mu=0}^3 \sum_{\dot{\beta}=\dot{1}}^{\dot{2}}
			          \sigma^{\mu}_{\alpha\dot{\beta}}\bar{\theta}^{\dot{\beta}}
					   \frac{\partial}{\partial x^{\mu}}
    \hspace{2em}\mbox{and}\hspace{2em}
  \bar{Q}_{\dot{\beta}}	\;
    =\; -\, \frac{\partial}{\rule{0ex}{.8em}\partial \bar{\theta}^{\dot{\beta}}}\,
	      +\, \sqrt{-1}\sum_{\mu=0}^3 \sum_{\alpha=1}^2
		              \theta^{\alpha} \sigma^\mu_{\alpha\dot{\beta}}\frac{\partial}{\partial x^{\mu}}
  $$}
 and {\it derivations that are invariant under the flow that generate supersymmetries} on $\widehat{X}$
     \marginpar{\vspace{1em}\raggedright\tiny
         \raisebox{-1ex}{\hspace{-2.4em}\LARGE $\cdot$}Cf.\,[\,\parbox[t]{20em}{Wess
		\& Bagger:\\ Eq.\:(4.6)].}}
 {\small
 $$
   e_{\alpha^{\prime}}\;
    =\;  \frac{\partial}{\partial \theta^{\alpha}}
            +\, \sqrt{-1}\,\sum_{\mu=0}^3 \sum_{\dot{\beta}=\dot{1}}^{\dot{2}}
			          \sigma^{\mu}_{\alpha\dot{\beta}}\bar{\theta}^{\dot{\beta}}
					   \frac{\partial}{\partial x^{\mu}}
    \hspace{2em}\mbox{and}\hspace{2em}
  e_{\beta^{\prime\prime}}	\;
    =\; -\, \frac{\partial}{\rule{0ex}{.8em}\partial \bar{\theta}^{\dot{\beta}}}\,
	      -\, \sqrt{-1}\sum_{\mu=0}^3 \sum_{\alpha=1}^2
		              \theta^{\alpha} \sigma^\mu_{\alpha\dot{\beta}}\frac{\partial}{\partial x^{\mu}}\,.
 $$}
Then one can check directly that
  they satisfy the following anticommutator relations
     \marginpar{\vspace{1em}\raggedright\tiny
         \raisebox{-1ex}{\hspace{-2.4em}\LARGE $\cdot$}Cf.\,[\,\parbox[t]{20em}{Wess
		\& Bagger:\\ Eqs.\:(4.5), (4.7),\\ (4.8)].}}
  $$
   \begin{array}{c}
    \begin{array}{lcl}
     \{ Q_\alpha, \bar{Q}_{\dot{\beta}}\}\;
	  =\; 2\sqrt{-1}\,\sum_\mu
	           \sigma^\mu_{\alpha\dot{\beta}}
			   \mbox{\Large $\frac{\partial}{\partial x^\mu}$}\,,
	    && \{Q_\alpha, Q_\beta\}\;
	             =\; \{\bar{Q}_{\dot{\alpha}}, \bar{Q}_{\dot{\beta}}\}\;
				 =\; 0\,;  \\[1.2ex]
      \{ e_{\alpha^\prime},  e_{\dot{\beta}^{\prime\prime}}\}\;
	    =\; -\,2\sqrt{-1}\,\sum_\mu
	         \sigma^\mu_{\alpha\dot{\beta}}
			   \mbox{\Large $\frac{\partial}{\partial x^\mu}$}\,,
	    && \{e_{\alpha^\prime}, e_{\beta^\prime}\}\;
	             =\; \{e_{\dot{\alpha}^{\prime\prime}}, e_{\dot{\beta}^{\prime\prime}}\}\; =\; 0\,,
    \end{array}\\[4ex]
	 \{e_{\alpha^\prime}, Q_\beta\}\;
	 =\;  \{e_{\alpha^\prime}, \bar{Q}_{\dot{\beta}}\}\;
	 =\; \{e_{\alpha^{\prime\prime}}, Q_\beta\}\;
	 =\; \{e_{\alpha^{\prime\prime}}, \bar{Q}_{\dot{\beta}}\}\;
     =\; 0\,,
	\end{array}
  $$
 for
  $\alpha, \beta=1, 2$; $\dot{\alpha}, \dot{\beta}=\dot{1}, \dot{2}$;
  $\alpha^\prime, \beta^\prime=1^\prime, 2^\prime$; and
  $\alpha^{\prime\prime}, \beta^{\prime\prime}=1^{\prime\prime}, 2^{\prime\prime}$.
 See ibidem for more details.
 (Cf.\:[Wess \& Bagger: Eqs.\:(4.4), (4.5), (4.6), (4.7), (4.8)] and related discussion therein.)
}\end{example}

\bigskip

\begin{flushleft}
{\bf A subtlety:
  The ${\cal P}_X$-module structure on $\widehat{\cal O}_X^{\,\widehat{\boxplus}}$
  from a partial twisting by ${\cal T}_X^{\ast,{\Bbb C}}$}
\end{flushleft}
While under the fixed trivialization of ${\cal S}^{\prime\, \vee}$ by the covariantly constant sections $\theta^\alpha$
  and of ${\cal S}^{\prime\prime\,\vee}$ by their complex conjugate $\bar{\theta}^{\dot{\beta}}$
 the structure sheaf $\widehat{\cal O}_X^{\,\widehat{\boxplus}}$ of $\widehat{X}^{\widehat{\boxplus}}$,
    as a sheaf of complexified ${\Bbb Z}/2$-graded $C^\infty$-rings,
   is isomorphic to the sheaf of Grassmann algebras
  $\bigwedge_{{\cal O}_X^{\,\Bbb C}}^{\tinybullet}{\cal F}$, where\\
  ${\cal F}
	   =  ({\cal S}^{\prime\,\vee}_\coordinates \oplus {\cal S}^{\prime\prime\,\vee}_\coordinates)
		        \oplus
				({\cal S}^{\prime\,\vee}_\parameter \oplus {\cal S}^{\prime\prime\,\vee}_\parameter)
			    \oplus
				  \mbox{$\bigoplus$}_{i=1}^l
				     ({\cal S}^\prime_{\field, i}\oplus {\cal S}^{\prime\prime}_{\field, i})$
  for some $l$, 	
this isomorphism does {\it not} respect the ${\cal P}_X$-module structure physicists intended for
 $\widehat{\cal O}_X^{\,\widehat{\boxplus}}$!
This is what makes an intrinsic construction of the function ring
 $C^\infty(\widehat{X}^{\widehat{\boxplus}})= \Gamma (\widehat{\cal O}_X^{\,\widehat{\boxplus}})$
 that matches [Wess \& Bagger] subtle.
While the built-in/natural ${\cal P}_X$-module structure on
 $\bigwedge_{{\cal O}_X^{\,\Bbb C}}^{\tinybullet}{\cal F})$
 is induced by the spinor representations of the Lorentz group on the spinor bundles $S^\prime$ and $S^{\prime\prime}$,
  \begin{itemize}
   \item[\LARGE $\cdot$]
   {\it the ${\cal P}_X$-module structure physicists use for $\widehat{\cal O}_X^{\,\widehat{\boxplus}}$
      is somewhat
       the ``sheaf of Grassmann algebras
       $\bigwedge_{{\cal O}_X^{\,\Bbb C}}^{\tinybullet}{\cal F}$
       partially twisted by the complexified cotangent sheaf ${\cal T}^{\ast,{\Bbb C}}_X$ of $X$"}.
  \end{itemize}
Since ${\cal T}_X^{\ast, {\Bbb C}}$ is itself a nontrivial ${\cal P}_X$-module,
 this creates a difference between
   the ${\cal P}_X$-module structure on $\widehat{\cal O}_X^{\,\widehat{\boxplus}}$   and
   that on  $\bigwedge_{{\cal O}_X^{\,\Bbb C}}^{\tinybullet}{\cal F}$.
For the current work, we will explicitly spell out such a partial-twisting-by-${\cal T}_X^{\ast,{\Bbb C}}$
 in the next theme for all the superfields that will be used to reconstruct [Wess \& Bagger].

Before proceeding to the next theme, it is important to note that
it is with respect to this partially twisted ${\cal P}$-module structure on $\widehat{\cal O}_X^{\,\widehat{\boxplus}}$,
  not the natural ${\cal P}_X$-module structure on $\bigwedge_{{\cal O}_X^{\,\Bbb C}}^{\tinybullet}{\cal F}$
 that physicists defines the following notions:

\bigskip

\begin{lemma-definition} {\bf [(Lorentz-)scalar superfield, spinor superfield, ... on $X$]}\;
 The partially twisted ${\cal P}_X$-structure on $\widehat{\cal O}_X^{\,\widehat{\boxplus}}$
  induces a direct-sum decomposition as a ${\cal P}_X$-module
  $$
   \widehat{\cal O}_X^{\,\widehat{\boxplus}}\;
    =\; {\cal M}_\trivial + {\cal M}_\spinor + {\cal M}_\vectorr + \cdots\,,
  $$
 where
  ${\cal M}_\trivial$ is a trivial ${\cal P}_X$-module,
  ${\cal M}_\spinor$ is a direct sum of ${\cal P}_X$-modules associated to spinor representations of the Lorentz group,
  ${\cal M}_\vectorr$ is a direct sum of ${\cal P}_X$-modules associated to the vector representation of the Lorentz group,
  $\cdots$, etc..
 {\rm
  A section of ${\cal M}_\trivial$ is called a {\it Lorentz-scalar superfield} or simply {\it scalar superfield}  on $X$;
  a section of ${\cal M}_\spinor$ is called a {\it spinor superfield} on $X$;
  a section of ${\cal M}_\vectorr$ is called a {\it vector superfield} on $X$;
  $\cdots$, etc..
   }
\end{lemma-definition}

\bigskip

\noindent
The details of the decomposition require a further study of the partial twist of
 $\bigwedge_{{\cal O}_X^{\,\Bbb C}}^{\tinybullet}{\cal F}$ by ${\cal T}_X^{\ast,{\Bbb C}}$,
 which we won't pursue in the current work.

\bigskip

\begin{terminology} $[$scalar superfield, vector superfield vs.\:Lemma/Definition 1.2.6\,$]$\;
{\rm
 In addition to Lemma/Definition~1.2.6,
  there is a second naming system\footnote{\makebox[12em][l]{\it Note for mathematicians}
                                                     The naming system by the ${\cal P}_X$-module structure
													   is favored from the perspective of
													   complexified ${\Bbb Z}/2$-graded $C^\infty$-Algebraic Geometry
													 while the naming by the representation of the lowest dynamical component(s)
													   is favored from the perspective of representation theory of supersymmetry algebra.
													 One should tell exactly which sense the term is in from the context.	
                                                           } 
     assigned to the term `scalar superfield', `vector superfield', $\cdots$, etc.\:that physicists also use:
 {\it naming by the lowest dynamical component fields of a superfield}.
 E.g.\
  a Lorentz-scalar superfield in the sense of Lemma/Definition~1.2.6
      that has its lowest dynamical component(s) in the trivial representation of the Lorentz group
   	     is called a {\it scalar superfield};
  a Lorentz-scalar superfield in the sense of Lemma/Definition~1.2.6
      that has its lowest dynamical components in the vector representation of the Lorentz group
	  is called a {\it vector superfield}; $\cdots$, etc.. 	
 It is this second sense that is used in [Wess \& Bagger: Chapter VI] when defining `vector superfield',
  which we will follow in Sec.\:4.
}\end{terminology}
 
\bigskip

\begin{convention} $[$towered superspace$]$\; {\rm
 To keep the notation simple and being enough to demonstration the reconstruction of [Wess \& Bagger],
   for the rest of the work
  we will assume
    $\widehat{X}^{\widehat{\boxplus}}= \widehat{X}^{\widehat{\boxplus}_1}$
    unless otherwise noted.
 Also,
  the parameter level of $\widehat{X}^{\widehat{\boxplus}}$ will be suppressed
  when not in use for a discussion, computation, or expression.
}\end{convention}

\bigskip

\begin{flushleft}
{\bf Tame, medium, and small scalar superfields on $X$}
\end{flushleft}
As a polynomial in $C^\infty(X)^{\Bbb C}[\theta, \bar{\theta}, \vartheta, \bar{\vartheta}]^\anticommuting$,
 a general element in $C^\infty(\widehat{X}^{\widehat{\boxplus}})$ has $2^8=256$ monomial summands.
Some special classes of elements, with much fewer monomial summands, play more prominent role
  in the construction of supersymmetric quantum field theories.
In this theme we spell out a few such classes.
Each class forms an even subring of $C^\infty(\widehat{X}^{\widehat{\boxplus}})$.

\bigskip

\begin{definition} {\bf [tame (Lorentz-)scalar superfield]}\;
{\rm
 An $\breve{f}\in C^{\infty}(\widehat{X}^{\widehat{\boxplus}})$ is called
   a {\it tame scalar superfield} on\footnote{Since $\breve{f}$
                                                                                 is a combination of both bosonic an fermionic fields on $X$,
                                                                                 it is conceptually more accurate to call $\breve{f}$ a superfield on $X$,
																				 rather than on $\widehat{X}$ or
																				 $\widehat{X}^{\widehat{\boxplus}}$.
																		} 
   $X$   																		                                                                  																				
  if it is a Lorentz-scalar superfield on $X$ in the sense of Lemma/Definition~1.2.6    and,
 as a polynomial in $C^\infty(X)^{\Bbb C}[\theta, \bar{\theta}, \vartheta, \bar{\vartheta}]^{\anticommuting}$,
   it satisfies the following property
  \begin{itemize}
   \item[\LARGE $\cdot$]
    $(\vartheta, \bar{\vartheta})$-degree $\le$  $(\theta, \bar{\theta})$-degree
	for every summand of $\breve{f}$.
  \end{itemize}
}\end{definition}
   
\bigskip
   
Explicitly, a tame scalar superfield $\breve{f}$ on $X$ is of the following form
 {\small
 \begin{eqnarray*}
  \breve{f}
   & =  &
   \breve{f}_{(0)}
   + \sum_{\alpha}\theta^\alpha\breve{f}_{(\alpha)}
   + \sum_{\dot{\beta}}\bar{\theta}^{\dot{\beta}}\breve{f}_{(\dot{\beta})}
   + \theta^1\theta^2 \breve{f}_{(12)}
   + \sum_{\alpha,\dot{\beta}} \theta^\alpha\bar{\theta}^{\dot{\beta}}
           \breve{f}_{(\alpha\dot{\beta})}
   + \bar{\theta}^{\dot{1}}\bar{\theta}^{\dot{2}} \breve{f}_{(\dot{1}\dot{2})} \\
 && \hspace{4em}
   + \sum_{\dot{\beta}}\theta^1\theta^2\bar{\theta}^{\dot{\beta}}
           \breve{f}_{(12\dot{\beta})}
   + \sum_\alpha \theta^\alpha\bar{\theta}^{\dot{1}}\bar{\theta}^{\bar{\dot{2}}}
           \breve{f}_{(\alpha\dot{1}\dot{2})}
   + \theta^1\theta^2\bar{\theta}^{\dot{1}}\bar{\theta}^{\dot{2}}
           \breve{f}_{(12\dot{1}\dot{2})} \\
  &= &
    f_{(0)}
	+ \sum_{\alpha}\theta^\alpha\vartheta_\alpha f_{(\alpha)}
	+ \sum_{\dot{\beta}}
	       \bar{\theta}^{\dot{\beta}}\bar{\vartheta}_{\dot{\beta}} f_{(\dot{\beta})}
		  \\
  && \hspace{1em}			
    +\; \theta^1\theta^2
	       \mbox{\large $($}
		     f^\prime_{(0)}
	         + \vartheta_1\vartheta_2 f_{(12)}
		   \mbox{\large $)$}
	+ \sum_{\alpha,\dot{\beta}}\theta^\alpha \bar{\theta}^{\dot{\beta}}
	      \left(\rule{0ex}{1.2em}\right.\!
		    \sum_\mu \sigma^\mu_{\alpha\dot{\beta}} f_{[\mu]}\,
			 +\, \vartheta_\alpha\bar{\vartheta}_{\dot{\beta}} f_{(\alpha\dot{\beta})}
		  \!\left.\rule{0ex}{1.2em}\right)
    + \bar{\theta}^{\dot{1}}\bar{\theta}^{\dot{2}}
	     \mbox{\large $($}
		   f^{\prime\prime}_{(0)}
	       + \bar{\vartheta}_{\dot{1}}\bar{\vartheta}_{\dot{2}} f_{(\dot{1}\dot{2})}
		 \mbox{\large $)$}
		   \\
  && \hspace{1em}		
	+ \sum_{\dot{\beta}}\theta^1\theta^2\bar{\theta}^{\dot{\beta}}
	     \left(\rule{0ex}{1.2em}\right.\!
		   \sum_{\alpha,\mu }
		          \vartheta_\alpha\,\sigma^{\mu \alpha}_{\;\;\;\;\dot{\beta}} f^\prime_{[\mu]}\,
			+\, \bar{\vartheta}_{\dot{\beta}} f^\prime_{(\dot{\beta})}
		    +\, \vartheta_1\vartheta_2\bar{\vartheta}_{\dot{\beta}} f_{(12\dot{\beta})}
		 \!\left.\rule{0ex}{1.2em}\right)
		   \\
  && \hspace{1em}			
    + \sum_\alpha \theta^\alpha\bar{\theta}^{\dot{1}}\bar{\theta}^{\dot{2}}
	   \left(\rule{0ex}{1.2em}\right.\!
	   	 \vartheta_\alpha f^{\prime\prime}_{(\alpha)}
	    +\, \sum_{\dot{\beta}, \mu}    		
		     \bar{\vartheta}_{\dot{\beta}}
			  \sigma^{\mu\dot{\beta}}_\alpha  f^{\prime\prime}_{[\mu]}\,				
        +\, \vartheta_\alpha \bar{\vartheta}_{\dot{1}}\bar{\vartheta}_{\dot{2}}	
		           f_{(\alpha\dot{1}\dot{2})}
	   \!\left.\rule{0ex}{1.2em}\right)
	   \\
  && \hspace{1em}
	+\; \theta^1\theta^2\bar{\theta}^{\dot{1}}\bar{\theta}^{\dot{2}}
	     \left(\rule{0ex}{1.2em}\right.\!
		  f^\sim_{(0)}
		  + \vartheta_1\vartheta_2 f^\sim_{(12)}
		  + \sum_{\alpha,\dot{\beta}, \mu} \vartheta_\alpha\bar{\vartheta}_{\dot{\beta}}
		         \bar{\sigma}^{\mu\dot{\beta}\alpha} f^{\sim}_{[\mu]}
		  + \bar{\vartheta}_{\dot{1}} \bar{\vartheta}_{\dot{2}}
		       f^\sim_{(\dot{1}\dot{2})}
		  + \vartheta_1\vartheta_2\bar{\vartheta}_{\dot{1}}\bar{\vartheta}_{\dot{2}}
		        f_{(12\dot{1}\dot{2})}
		 \!\left.\rule{0ex}{1.2em}\right)\,,
 \end{eqnarray*}
  }
  where
  $\alpha=1,2$; $\dot{\beta}=\dot{1}, \dot{2}$; $\mu=0,1,2,3$;
  $$
    \begin{array}{c}
    \sigma^{\mu \alpha}_{\;\;\;\;\dot{\beta}}
      = \sum_{\gamma}\varepsilon^{\alpha\gamma} \sigma^\mu_{\gamma\dot{\beta}} \,,\;\;\;\;
    \sigma^{\mu\dot{\beta}}_\alpha
      = \sum_{\dot{\delta}}
	        \varepsilon^{\dot{\beta}\dot{\delta}}\sigma^\mu_{\alpha\dot{\delta}}\,,\;\;\;\;
   \bar{\sigma}^{\mu\dot{\beta}\alpha}
     = \sum_{\gamma\dot{\delta}}
	      \varepsilon^{\alpha\gamma}\varepsilon^{\dot{\beta}\dot{\delta}}
		    \sigma^\mu_{\gamma\dot{\delta}}\,;
    \end{array}			
  $$
  and,
 for the {\it forty-one} coefficients $f^{\tinybullet}_{\tinybullet}$
   of the $(\theta,\bar{\theta},\vartheta,\bar{\vartheta})$-monomial summands of $\breve{f}$,
 {\small
  \begin{eqnarray*}
   \lefteqn{
     f_{(0)};\;\;
     f_{(\alpha)};\;\;
	 f_{(\dot{\beta})};\;\;
     f^\prime_{(0)},\, f_{(12)};\;\;
     f_{(\alpha\dot{\beta})};\;\;
     f^{\prime\prime}_{(0)},\,  f_{(\dot{1}\dot{2})};
      }\\
  &&
      f^\prime_{(\dot{\beta})},\,  f_{(12\dot{\beta})};\;\;
           f^{\prime\prime}_{(\alpha)},\,  f_{(\alpha\dot{1}\dot{2})};\;\;
     f^\sim_{(0)},\,
        f^\sim_{(12)},\, f^\sim_{(\dot{1}\dot{2})},\,
        f_{(12\dot{1}\dot{2})}\;
      \in\;  C^\infty(X)^{\Bbb C}\;\; \mbox{as a trivial ${\cal P}_X$-module}\,,
   \end{eqnarray*}}
 while
 {\small
 $$
  (f_{[\mu]})_\mu,\,
    (f^\prime_{[\mu]})_\mu,\,
    (f^{\prime\prime}_{[\mu]})_\mu,\,
	(f^\sim_{[\mu]})_\mu\;
    \in\; (C^\infty(X)^{\Bbb C})^{\oplus 4}
	         \simeq {\cal T}_X^{\ast\,{\Bbb C}}
			 \;\; \mbox{as a trivialized nontrivial ${\cal P}_X$-module}\,.
 $$
 Here the trivialization
   $(C^\infty(X)^{\Bbb C})^{\oplus 4}
	         \simeq {\cal T}_X^{\ast\,{\Bbb C}}$ as ${\cal O}_X^{\,\Bbb C}$-modules
   is induced by the fixed isomorphism\\
   ${\cal T}_X^{\ast,{\Bbb C}}
       \simeq {\cal S}^\prime\otimes_{{\cal O}_X^{\,{\Bbb C}}} {\cal S}^{\prime\prime}$
     as ${\cal P}_X$-modules and
   the trivialization of ${\cal S}^{\prime\,\vee}$ by $\theta^\alpha$'s and
       ${\cal S}^{\prime\prime\,\vee}$ by $\bar{\theta}^{\dot{\beta}}$'s.
This is the partial twist of
 $\bigwedge_{{\cal O}_X^{\,\Bbb C}}^{\tinybullet}{\cal F}$ by ${\cal T}_X^{\ast,{\Bbb C}}$
 we refer to in the previous theme.
This partial twist justifies the $\breve{f}$ as expressed above to be a Lorentz-scalar superfield
 in the sense of Lemma/Definition~1.2.6.

Note that	among the forty-one coefficients $f^\tinybullet_\tinybullet$,
 {\it twenty-nine}
 {\small
  $$
     f_{(0)};\;\;
     f^\prime_{(0)},\, f_{(12)};\;\;
     f_{[\mu]},\,  f_{(\alpha\dot{\beta})};\;\;
     f^{\prime\prime}_{(0)},\,  f_{(\dot{1}\dot{2})};\;\;
     f^\prime_{[\mu]};\;\;
     f^{\prime\prime}_{[\mu]};\;\;
     f^\sim_{(0)},\,
        f^\sim_{(12)},\, f^\sim_{[\mu]},\, f^\sim_{(\dot{1}\dot{2})},\,
        f_{(12\dot{1}\dot{2})}
   $$}
  are related to bosonic fields on $X$ and {\it twelve}
 {\small
  $$
     f_{(\alpha)};\;\;  f_{(\dot{\beta})};\;\;	
     f^\prime_{(\dot{\beta})},\,  f_{(12\dot{\beta})};\;\;
     f^{\prime\prime}_{(\alpha)},\,  f_{(\alpha\dot{1}\dot{2})};\;\;\
  $$}
  to fermionic fields on $X$.

\bigskip

\begin{lemma-definition} {\bf [tame sector $\widehat{X}^{\widehat{\boxplus}, \tamescriptsize}$
                                            of $\widehat{X}^{\widehat{\boxplus}}$]}\;
 The collection of  tame scalar superfields on $X$
   as defined in Definition~1.2.9
  is an even subring of the complexified ${\Bbb Z}/2$-graded $C^\infty$-ring
  $C^\infty(\widehat{X}^{\widehat{\boxplus}})$.
 Denote this subring
   (also  a $C^{\infty}(X)^{\Bbb C}$-subalgebra
   of $C^{\infty}(\widehat{X}^{\widehat{\boxplus}})$)
   by $C^\infty(\widehat{X}^{\widehat{\boxplus}})^\tamescriptsize$.
 Then, the $C^\infty$-hull of $C^\infty(\widehat{X}^{\widehat{\boxplus}})$ restricts to the $C^\infty$-hull
   of $C^\infty(\widehat{X}^{\widehat{\boxplus}})^\tamescriptsize$, which is given by
   $$
     \mbox{$C^\infty$-hull}\,(C^\infty(\widehat{X}^{\widehat{\boxplus}})^\tamescriptsize)\;
	   =\; \{\breve{f}\in C^\infty(\widehat{X}^{\widehat{\boxplus}})^\tamescriptsize
	                                                                            \,|\, \breve{f}_{(0)}\in C^\infty(X)\}\,.
   $$																					
 {\rm
 Denote by $\widehat{X}^{\widehat{\boxplus}, \tamescriptsize}$
   the complexified ${\Bbb Z}/2$-graded $C^\infty$-scheme with the underlying topology $X$ and function ring
   $C^\infty(\widehat{X}^{\widehat{\boxplus}})^\tamescriptsize$.
 Then there is a built-in dominant morphism
   $\widehat{X}^{\widehat{\boxplus}}\rightarrow \widehat{X}^{\widehat{\boxplus}, \tamescriptsize}$.
 We will call $\widehat{X}^{\widehat{\boxplus}, \tamescriptsize}$ the {\it tame sector}
   of $\widehat{X}^{\widehat{\boxplus}}$.
   }
\end{lemma-definition}

\begin{proof}
 This follows from the fact that
  both the conditions on tame scalar superfields
   (1)  Lorentz scalar and
   (2)  $(\vartheta, \bar{\vartheta})$-degree $\le$  $(\theta, \bar{\theta})$-degree
	           for every summand
 are closed under the multiplication in $C^\infty(\widehat{X}^{\widehat{\boxplus}})$.

\end{proof}
 
\bigskip

Explicitly, let
 {\small
 \begin{eqnarray*}
  \breve{A}
  &= &
    A_{(0)}
	+ \sum_{\alpha}\theta^\alpha\vartheta_\alpha A_{(\alpha)}
	+ \sum_{\dot{\beta}}
	       \bar{\theta}^{\dot{\beta}}\bar{\vartheta}_{\dot{\beta}} A_{(\dot{\beta})}
		  \\
  && \hspace{1em}			
    +\; \theta^1\theta^2
	       \mbox{\large $($}
		     A^\prime_{(0)}
	         + \vartheta_1\vartheta_2 A_{(12)}
		   \mbox{\large $)$}
	+ \sum_{\alpha,\dot{\beta}}\theta^\alpha \bar{\theta}^{\dot{\beta}}
	      \left(\rule{0ex}{1.2em}\right.\!
		    \sum_\mu \sigma^\mu_{\alpha\dot{\beta}} A_{[\mu]}\,
			 +\, \vartheta_\alpha\bar{\vartheta}_{\dot{\beta}} A_{(\alpha\dot{\beta})}
		  \!\left.\rule{0ex}{1.2em}\right)
    + \bar{\theta}^{\dot{1}}\bar{\theta}^{\dot{2}}
	     \mbox{\large $($}
		   A^{\prime\prime}_{(0)}
	       + \bar{\vartheta}_{\dot{1}}\bar{\vartheta}_{\dot{2}} A_{(\dot{1}\dot{2})}
		 \mbox{\large $)$}
		   \\
  && \hspace{1em}		
	+ \sum_{\dot{\beta}}\theta^1\theta^2\bar{\theta}^{\dot{\beta}}
	     \left(\rule{0ex}{1.2em}\right.\!
		   \sum_{\alpha,\mu }
		          \vartheta_\alpha\,\sigma^{\mu \alpha}_{\;\;\;\;\dot{\beta}} A^\prime_{[\mu]}\,
			+\, \bar{\vartheta}_{\dot{\beta}} A^\prime_{(\dot{\beta})}
		    +\, \vartheta_1\vartheta_2\bar{\vartheta}_{\dot{\beta}} A_{(12\dot{\beta})}
		 \!\left.\rule{0ex}{1.2em}\right)
		   \\
  && \hspace{1em}			
    + \sum_\alpha \theta^\alpha\bar{\theta}^{\dot{1}}\bar{\theta}^{\dot{2}}
	   \left(\rule{0ex}{1.2em}\right.	
	     \vartheta_\alpha f^{\prime\prime}_{(\alpha)}
	     + \sum_{\dot{\beta}, \mu}
		      \bar{\vartheta}_{\dot{\beta}} \sigma^{\mu\dot{\beta}}_\alpha
		      A^{\prime\prime}_{[\mu]}\,		
         +\, \vartheta_\alpha \bar{\vartheta}_{\dot{1}}\bar{\vartheta}_{\dot{2}}	
		           A_{(\alpha\dot{1}\dot{2})}
	   \left.\rule{0ex}{1.2em}\right)
	   \\
  && \hspace{1em}
	+\; \theta^1\theta^2\bar{\theta}^{\dot{1}}\bar{\theta}^{\dot{2}}
	     \left(\rule{0ex}{1.2em}\right.\!
		  A^\sim_{(0)}
		  + \vartheta_1\vartheta_2 A^\sim_{(12)}
		  + \sum_{\alpha,\dot{\beta}, \mu} \vartheta_\alpha\bar{\vartheta}_{\dot{\beta}}
		         \bar{\sigma}^{\mu\dot{\beta}\alpha} A^{\sim}_{[\mu]}
		  + \bar{\vartheta}_{\dot{1}} \bar{\vartheta}_{\dot{2}}
		       A^\sim_{(\dot{1}\dot{2})}
		  + \vartheta_1\vartheta_2\bar{\vartheta}_{\dot{1}}\bar{\vartheta}_{\dot{2}}
		        A_{(12\dot{1}\dot{2})}
		 \!\left.\rule{0ex}{1.2em}\right)
 \end{eqnarray*}
  }
 and
 {\small
 \begin{eqnarray*}
  \breve{B}
  &= &
    B_{(0)}
	+ \sum_{\alpha}\theta^\alpha\vartheta_\alpha B_{(\alpha)}
	+ \sum_{\dot{\beta}}
	       \bar{\theta}^{\dot{\beta}}\bar{\vartheta}_{\dot{\beta}} B_{(\dot{\beta})}
		  \\
  && \hspace{1em}			
    +\; \theta^1\theta^2
	       \mbox{\large $($}
		     B^\prime_{(0)}
	         + \vartheta_1\vartheta_2 B_{(12)}
		   \mbox{\large $)$}
	+ \sum_{\alpha,\dot{\beta}}\theta^\alpha \bar{\theta}^{\dot{\beta}}
	      \left(\rule{0ex}{1.2em}\right.\!
		    \sum_\mu \sigma^\mu_{\alpha\dot{\beta}} B_{[\mu]}\,
			 +\, \vartheta_\alpha\bar{\vartheta}_{\dot{\beta}} B_{(\alpha\dot{\beta})}
		  \!\left.\rule{0ex}{1.2em}\right)
    + \bar{\theta}^{\dot{1}}\bar{\theta}^{\dot{2}}
	     \mbox{\large $($}
		   B^{\prime\prime}_{(0)}
	       + \bar{\vartheta}_{\dot{1}}\bar{\vartheta}_{\dot{2}} B_{(\dot{1}\dot{2})}
		 \mbox{\large $)$}
		   \\
  && \hspace{1em}		
	+ \sum_{\dot{\beta}}\theta^1\theta^2\bar{\theta}^{\dot{\beta}}
	     \left(\rule{0ex}{1.2em}\right.\!
		   \sum_{\alpha,\mu }
		          \vartheta_\alpha\,\sigma^{\mu \alpha}_{\;\;\;\;\dot{\beta}} B^\prime_{[\mu]}\,
			+\, \bar{\vartheta}_{\dot{\beta}} f^\prime_{(\dot{\beta})}
		    +\, \vartheta_1\vartheta_2\bar{\vartheta}_{\dot{\beta}} B_{(12\dot{\beta})}
		 \!\left.\rule{0ex}{1.2em}\right)
		   \\
  && \hspace{1em}			
    + \sum_\alpha \theta^\alpha\bar{\theta}^{\dot{1}}\bar{\theta}^{\dot{2}}
	   \left(\rule{0ex}{1.2em}\right.\!
	     \vartheta_\alpha B^{\prime\prime}_{(\alpha)}
	     \sum_{\dot{\beta}, \mu}
		   \bar{\vartheta}_{\dot{\beta}} \sigma^{\mu\dot{\beta}}_\alpha
		     B^{\prime\prime}_{[\mu]}\,	
         +\, \vartheta_\alpha \bar{\vartheta}_{\dot{1}}\bar{\vartheta}_{\dot{2}}	
		           B_{(\alpha\dot{1}\dot{2})}
	   \!\left.\rule{0ex}{1.2em}\right)\\
  && \hspace{1em}
	+\; \theta^1\theta^2\bar{\theta}^{\dot{1}}\bar{\theta}^{\dot{2}}
	     \left(\rule{0ex}{1.2em}\right.\!
		  B^\sim_{(0)}
		  + \vartheta_1\vartheta_2 B^\sim_{(12)}
		  + \sum_{\alpha,\dot{\beta}, \mu} \vartheta_\alpha\bar{\vartheta}_{\dot{\beta}}
		         \bar{\sigma}^{\mu\dot{\beta}\alpha} B^{\sim}_{[\mu]}
		  + \bar{\vartheta}_{\dot{1}} \bar{\vartheta}_{\dot{2}}
		       B^\sim_{(\dot{1}\dot{2})}
		  + \vartheta_1\vartheta_2\bar{\vartheta}_{\dot{1}}\bar{\vartheta}_{\dot{2}}
		        B_{(12\dot{1}\dot{2})}
		 \!\left.\rule{0ex}{1.2em}\right)
       \\
   & \in &	   C^\infty(\widehat{X}^{\widehat{\boxplus}})^\tamescriptsize
 \end{eqnarray*}
  }
be tame scalar superfields on $X$.
Here,
 $A^{\tinybullet}_{\,\tinybullet}= A^{\tinybullet}_{\,\tinybullet}(x)$ and
 $B^{\tinybullet}_{\,\tinybullet}= A^{\tinybullet}_{\,\tinybullet}(x)\,\in C^\infty(X)^{\Bbb C}$.
Then, their product is given by

{\small
\begin{eqnarray*}
 && \breve{A}\breve{B}\; =\; \breve{B}\breve{A}
   \\
 && \hspace{1em}
  =\;
   A_{(0)} B_{(0)}
     + \sum_\alpha \theta^\alpha\vartheta_\alpha
	      \mbox{\Large $($}
	        A_{(\alpha)}B_{(0)} + A_{(0)} B_{(\alpha)}
		  \mbox{\Large $)$}
	 + \sum_{\dot{\beta}} \bar{\theta}^{\dot{\beta}}\bar{\vartheta}_{\dot{\beta}}
	       \mbox{\Large $($}
	        A_{(\dot{\beta})} B_{(0)} + A_{(0)} B_{(\dot{\beta})}
		   \mbox{\Large $)$}
	 \\
   && \hspace{2em}	
	 +\, \theta^1\theta^2
	        \left\{\rule{0ex}{1.2em}\right.			
			   \mbox{\Large $($}
			    A^\prime_{(0)}B_{(0)} + A_{(0)}B^\prime_{(0)}
			   \mbox{\Large $)$}
	         + \vartheta_1\vartheta_2
	          \mbox{\Large $($}
   	            A_{(12)}B_{(0)} - A_{(1)}B_{(2)}
				  - A_{(2)} B_{(1)} + A_{(0)} B_{(12)}
			  \mbox{\Large $)$}
		   \left.\rule{0ex}{1.2em}\right\}
	  \\
   && \hspace{2em}	
	 + \sum_{\alpha, \dot{\beta}}
	      \left\{\rule{0ex}{1.2em}\right.
		    \sum_\mu \sigma^\mu_{\alpha\dot{\beta}}
			  \mbox{\Large $($}
			   A_{[\mu]}B_{(0)} + A_{(0)} B_{[\mu]}
			  \mbox{\Large $)$}
	     + \vartheta_\alpha \bar{\vartheta}_{\dot{\beta}}
		      \mbox{\Large $($}
		      A_{(\alpha \dot{\beta})} B_{(0)} - A_{(\alpha)}B_{(\dot{\beta})}
			  - A_{(\dot{\beta})} B_{(\alpha)} + A_{(0)} B_{(\alpha\dot{\beta})}			
			  \mbox{\Large $)$}
		  \left.\rule{0ex}{1.2em}\right\}
	  \\
   && \hspace{2em}	  		
	 +\, \bar{\theta}^{\dot{1}} \bar{\theta}^{\dot{2}}
	        \left\{\rule{0ex}{1.2em}\right.
	          \mbox{\Large $($}
			   A^{\prime\prime}_{(0)}B_{(0)} + A_{(0)}B^{\prime\prime}_{(0)}
	          \mbox{\Large $)$}
	         + \bar{\vartheta}_{\dot{1}} \bar{\vartheta}_{\dot{2}}
		          \mbox{\Large $($}
		            A_{(\dot{1}\dot{2})} B_{(0)} - A_{(\dot{1})}B_{(\dot{2})}
			          - A_{(\dot{2})}B_{(\dot{1})} + A_{(0)} B_{(\dot{1}\dot{2})}
                  \mbox{\Large $)$}			
            \left.\rule{0ex}{1.2em}\right\}		
	   \\
   && \hspace{2em}
	 + \sum_{\dot{\beta}} \theta^1\theta^2 \bar{\theta}^{\dot{\beta}}
	      \left\{\rule{0ex}{1.2em}\right.
		   \sum_{\alpha, \mu}
		     \vartheta_\alpha \sigma^{\mu\alpha}_{\;\;\;\;\dot{\beta}}
			    \mbox{\Large $($}
			      A^\prime_{[\mu]}B_{(0)}  + A_{[\mu]}B_{(\alpha)} +A_{(\alpha)} B_{[\mu]}
				   + A_{(0)}B^\prime_{[\mu]}
			    \mbox{\Large $)$}	
	         \\
		&& \hspace{6em}
		  +\, \bar{\vartheta}_{\dot{\beta}}
		         \mbox{\Large $($}
				  A^\prime_{(\dot{\beta})} B_{(0)} + A^\prime_{(0)}B_{(\dot{\beta})}
				   + A_{(\dot{\beta})} B^\prime_{(0)} + A_{(0)} B^\prime_{(\dot{\beta})}
				 \mbox{\Large $)$}
			 \\
        && \hspace{6em}	  				
	       +\, \vartheta_1\vartheta_2 \bar{\vartheta}_{\dot{\beta}}
		          \mbox{\Large $($}
  		             A_{(12\dot{\beta})} B_{(0)} + A_{(12)}B_{(\dot{\beta})}
					  + A_{(1\dot{\beta})}B_{(2)}+ A_{(2\dot{\beta})} B_{(1)}
	                      \\[-1ex]
                         && \hspace{12em}	  					
					  +\, A_{(1)} B_{(2\dot{\beta})} + A_{(2)} B_{(1\dot{\beta})}
					  + A_{(\dot{\beta})} B_{(12)} + A_{(0)} B_{(12\dot{\beta})}
				  \mbox{\Large $)$}
		  \left.\rule{0ex}{1.2em}\right\}
	  \\
   && \hspace{2em}	  		
     + \sum_\alpha \theta^\alpha \bar{\theta}^{\dot{1}} \bar{\theta}^{\dot{2}}
	      \left\{\rule{0ex}{1.2em}\right.
		    \vartheta_\alpha
			 \mbox{\Large $($}
			  A^{\prime\prime}_{(\alpha)}B_{(0)} + A^{\prime\prime}_{(0)}B_{(\alpha)}
			   + A_{(\alpha)}B^{\prime\prime}_{(0)} + A_{(0)}B^{\prime\prime}_{(\alpha)}
			 \mbox{\Large $)$}
		     	  \\
              && \hspace{6em}	  		
		    +\, \sum_{\dot{\beta}, \mu}
			   \bar{\vartheta}_{\dot{\beta}} \sigma^{\mu\dot{\beta}}_\alpha
			     \mbox{\Large $($}
			      A^{\prime\prime}_{[\mu]}B_{(0)} + A_{[\mu]} B_{(\dot{\beta})}
				     + A_{(\dot{\beta})} B_{[\mu]} + A_{(0)}B^{\prime\prime}_{[\mu]}
				 \mbox{\Large $)$}
           	  \\
              && \hspace{6em}	  			
		    +\, \vartheta_\alpha \bar{\vartheta}_{\dot{1}} \bar{\vartheta}_{\dot{2}}
			       \mbox{\Large $($}
			         A_{(\alpha\dot{1}\dot{2})} B_{(0)}  + A_{(\alpha\dot{1})}B_{(\dot{2})}
					  + A_{(\alpha\dot{2})}B_{(\dot{1})}  + A_{(\dot{1}\dot{2})}B_{(\alpha)}
					  	  \\[-1ex]
                     && \hspace{12em}	  					
					  +\, A_{(\alpha)}B_{(\dot{1}\dot{2})}+ A_{(\dot{1})}B_{(\alpha\dot{2})}
					  + A_{(\dot{2})} B_{(\alpha\dot{1})} + A_{(0)}B_{(\alpha\dot{1}\dot{2})}
				   \mbox{\Large $)$}
		  \left.\rule{0ex}{1.2em}\right\}
		\\
  && \hspace{2em}
      +\, \theta^1\theta^2\bar{\theta}^{\dot{1}} \bar{\theta}^{\dot{2}}
	       \left\{\rule{0ex}{1.2em}\right.
		       \mbox{\Large $($}
		     A^\sim_{(0)}B_{(0)}    + A^\prime_{(0)}B^{\prime\prime}_{(0)}
			   + 2 \sum_{\mu, \nu}\eta^{\mu\nu} A_{[\mu]}B_{[\nu]}
			   + A^{\prime\prime}_{(0)}B^\prime_{(0)}+ A_{(0)}B^\sim_{(0)}
			   \mbox{\Large $)$}
			  \\
           && \hspace{6em}	
		     +\, \vartheta_1\vartheta_2
                  \mbox{\Large $($}
				    A^\sim_{(12)}B_{(0)}
					 - A^{\prime\prime}_{(1)}B_{(2)} - A^{\prime\prime}_{(2)}B_{(1)}
					 + A^{\prime\prime}_{(0)}B_{(12)}
					  	  \\[-.8ex]
                       && \hspace{14em}	  				
					 +\, A_{(12)}B^{\prime\prime}_{(0)}
					 - A_{(1)}B^{\prime\prime}_{(2)} - A_{(2)}B^{\prime\prime}_{(1)}
					 + A_{(0)}B^\sim_{(12)}
                  \mbox{\Large $)$}				
        	  \\
           && \hspace{6em}	  				
			  +\, \sum_{\alpha, \dot{\beta}}
			        \vartheta_\alpha \bar{\vartheta}_{\dot{\beta}}\,
					   \bar{\sigma}^{\mu\dot{\beta}\alpha}
					 \mbox{\Large $($}
					   A^\sim_{[\mu]}B_{(0)} + A^\prime_{[\mu]} B_{(\dot{\beta})}
					   - A^{\prime\prime}_{[\mu]}B_{(\alpha)}- A_{(\alpha)}B^{\prime\prime}_{[\mu]}
					   + A_{(\dot{\beta})} B^\prime_{[\mu]} + A_{(0)}B^\sim_{[\mu]}
					 \mbox{\Large $)$}
        	  \\
		    && \hspace{6em}	
		    +\, \bar{\vartheta}_{\dot{1}}\bar{\vartheta}_{\dot{2}}
                  \mbox{\Large $($}
				    A^\sim_{(\dot{1}\dot{2})}B_{(0)}
					- A^\prime_{(\dot{1})}B_{(\dot{2})} - A^\prime_{(\dot{2})}B_{(\dot{1})}
					+ A^\prime_{(0)}B_{(\dot{1}\dot{2})}
					   \\[-.8ex]
                       && \hspace{14em}	  	
					+\, A_{(\dot{1}\dot{2})}B^\prime_{(0)}
					- A_{(\dot{1})}B^\prime_{(\dot{2})}
					    - A_{(\dot{2})}B^\prime_{(\dot{1})}
					+ A_{(0)}B^\sim_{(\dot{1}\dot{2})}
                  \mbox{\Large $)$}				
        	  \\
            && \hspace{6em}	  						
			  +\, \vartheta_1\vartheta_2 \bar{\vartheta}_{\dot{1}} \bar{\vartheta}_{\dot{2}}
			        \mbox{\Large $($}
			          A_{(12\dot{1}\dot{2})}B_{(0)} - A_{(12\dot{1})}B_{(\dot{2})}
					  - A_{(12\dot{2})}B_{(\dot{1})} - A_{(1\dot{1}\dot{2})}B_{(2)}
					  - A_{(2\dot{1}\dot{2})}B_{(1)}
					  \\
                   && \hspace{8em}	  					
					  +\, A_{(12)}B_{(\dot{1}\dot{2})}
					  + A_{(1\dot{1})}B_{(2\dot{2})} + A_{(1\dot{2})}B_{(2\dot{1})}
					  + A_{(2\dot{1})}B_{(1\dot{2})} + A_{(2\dot{2})}B_{(1\dot{1})}
 	                   \\[-1ex]
                   && \hspace{8em}	  										
					  +\, A_{(\dot{1}\dot{2})}B_{(12)} - A_{(1)}B_{(2\dot{1}\dot{2})}
					  - A_{(2)}B_{(1\dot{1}\dot{2})} -A_{(\dot{1})}B_{(12\dot{2})}
					  - A_{(\dot{2})}B_{(12\dot{1})} + A_{(0)}B_{(12\dot{1}\dot{2})}
				   \mbox{\Large $)$}
		   \left.\rule{0ex}{1.2em}\right\}
	        \\
     && \hspace{1em}	  							
       \in\; C^\infty(\widehat{X}^{\widehat{\boxplus}})^\tamescriptsize\,.
\end{eqnarray*}
}

\bigskip

\begin{definition} {\bf [medium (Lorentz-)scalar superfield]}\;
{\rm
 A tame scalar superfield\\ $\breve{f}\in C^{\infty}(\widehat{X}^{\widehat{\boxplus}})$
  is called a {\it medium scalar superfield} on $X$   	
  if in addition
   \begin{itemize}
    \item[\LARGE $\cdot$]
	 $f^\prime_{(0)}= f^{\prime\prime}_{(0)}=0$
   \end{itemize}
  as a polynomial in $C^\infty(X)^{\Bbb C}[\theta, \bar{\theta}, \vartheta, \bar{\vartheta}]^{\anticommuting}$.
}\end{definition}
   
\bigskip
   
Explicitly, a medium scalar superfield $\breve{f}$ on $X$ is of the following form
 {\small
 \begin{eqnarray*}
  \breve{f}
  &= &
    f_{(0)}
	+ \sum_{\alpha}\theta^\alpha\vartheta_\alpha f_{(\alpha)}
	+ \sum_{\dot{\beta}}
	       \bar{\theta}^{\dot{\beta}}\bar{\vartheta}_{\dot{\beta}} f_{(\dot{\beta})}
		  \\
  && \hspace{1em}			
    +\; \theta^1\theta^2 \vartheta_1\vartheta_2 f_{(12)}		  		
	+ \sum_{\alpha,\dot{\beta}}\theta^\alpha \bar{\theta}^{\dot{\beta}}
	      \left(\rule{0ex}{1.2em}\right.\!
		    \sum_\mu \sigma^\mu_{\alpha\dot{\beta}} f_{[\mu]}\,
			 +\, \vartheta_\alpha\bar{\vartheta}_{\dot{\beta}} f_{(\alpha\dot{\beta})}
		  \!\left.\rule{0ex}{1.2em}\right)
    + \bar{\theta}^{\dot{1}}\bar{\theta}^{\dot{2}}
	      \bar{\vartheta}_{\dot{1}}\bar{\vartheta}_{\dot{2}} f_{(\dot{1}\dot{2})}  		 		
		   \\
  && \hspace{1em}		
	+ \sum_{\dot{\beta}}\theta^1\theta^2\bar{\theta}^{\dot{\beta}}
	     \left(\rule{0ex}{1.2em}\right.\!
		   \sum_{\alpha,\mu }
		          \vartheta_\alpha\,\sigma^{\mu \alpha}_{\;\;\;\;\dot{\beta}} f^\prime_{[\mu]}\,
			+\, \bar{\vartheta}_{\dot{\beta}} f^\prime_{(\dot{\beta})}
		    +\, \vartheta_1\vartheta_2\bar{\vartheta}_{\dot{\beta}} f_{(12\dot{\beta})}
		 \!\left.\rule{0ex}{1.2em}\right)
		   \\
  && \hspace{1em}			
    + \sum_\alpha \theta^\alpha\bar{\theta}^{\dot{1}}\bar{\theta}^{\dot{2}}
	   \left(\rule{0ex}{1.2em}\right.\!
	   	 \vartheta_\alpha f^{\prime\prime}_{(\alpha)}
	    +\, \sum_{\dot{\beta}, \mu}    		
		     \bar{\vartheta}_{\dot{\beta}}
			  \sigma^{\mu\dot{\beta}}_\alpha  f^{\prime\prime}_{[\mu]}\,				
        +\, \vartheta_\alpha \bar{\vartheta}_{\dot{1}}\bar{\vartheta}_{\dot{2}}	
		           f_{(\alpha\dot{1}\dot{2})}
	   \!\left.\rule{0ex}{1.2em}\right)
	   \\
  && \hspace{1em}
	+\; \theta^1\theta^2\bar{\theta}^{\dot{1}}\bar{\theta}^{\dot{2}}
	     \left(\rule{0ex}{1.2em}\right.\!
		  f^\sim_{(0)}
		  + \vartheta_1\vartheta_2 f^\sim_{(12)}
		  + \sum_{\alpha,\dot{\beta}, \mu} \vartheta_\alpha\bar{\vartheta}_{\dot{\beta}}
		         \bar{\sigma}^{\mu\dot{\beta}\alpha} f^{\sim}_{[\mu]}
		  + \bar{\vartheta}_{\dot{1}} \bar{\vartheta}_{\dot{2}}
		       f^\sim_{(\dot{1}\dot{2})}
		  + \vartheta_1\vartheta_2\bar{\vartheta}_{\dot{1}}\bar{\vartheta}_{\dot{2}}
		        f_{(12\dot{1}\dot{2})}
		 \!\left.\rule{0ex}{1.2em}\right)
		 \\
   &\in &	C^\infty(X)^{\Bbb C}[\theta, \bar{\theta}, \vartheta, \bar{\vartheta}]^\anticommuting\,.
 \end{eqnarray*}
  }

\bigskip

\begin{lemma-definition} {\bf [medium sector $\widehat{X}^{\widehat{\boxplus}, \mediumscriptsize}$
                                            of $\widehat{X}^{\widehat{\boxplus}}$]}\;
 The collection of  medium scalar superfields on $X$
   as defined in Definition~1.2.11
  is an even subring of $C^\infty(\widehat{X}^{\widehat{\boxplus}})$.
 Denote this subring
   (also  a $C^{\infty}(X)^{\Bbb C}$-subalgebra
   of $C^{\infty}(\widehat{X}^{\widehat{\boxplus}})$)
   by $C^\infty(\widehat{X}^{\widehat{\boxplus}})^\mediumscriptsize$.
 Then, the $C^\infty$-hull of $C^\infty(\widehat{X}^{\widehat{\boxplus}})$ restricts to the $C^\infty$-hull
   of $C^\infty(\widehat{X}^{\widehat{\boxplus}})^\mediumscriptsize$,
   which is given by
   $$
     \mbox{$C^\infty$-hull}\,(C^\infty(\widehat{X}^{\widehat{\boxplus}})^\mediumscriptsize)\;
	   =\; \{\breve{f}\in C^\infty(\widehat{X}^{\widehat{\boxplus}})^\mediumscriptsize
	                                                                            \,|\, \breve{f}_{(0)}\in C^\infty(X)\}\,.
   $$																					
 {\rm
 Denote by $\widehat{X}^{\widehat{\boxplus}, \mediumscriptsize}$
   the complexified ${\Bbb Z}/2$-graded $C^\infty$-scheme with the underlying topology $X$ and
   function ring $C^\infty(\widehat{X}^{\widehat{\boxplus}})^\mediumscriptsize$.
 Then there is a built-in dominant morphism
   $\widehat{X}^{\widehat{\boxplus}, \tamescriptsize}
      \rightarrow \widehat{X}^{\widehat{\boxplus}, \mediumscriptsize}$.
 We will call $\widehat{X}^{\widehat{\boxplus}, \mediumscriptsize}$
   the {\it medium sector} of $\widehat{X}^{\widehat{\boxplus}}$.
   }
\end{lemma-definition}

\begin{proof}
 This follows from the fact that
   the conditions on medium scalar superfields
   (1)  Lorentz scalar,
   (2)  $(\vartheta, \bar{\vartheta})$-degree $\le$  $(\theta, \bar{\theta})$-degree
	           for every summand, and
   (3)	 $f^\prime_{(0)}= f^{\prime\prime}_{(0)}=0$
            as a polynomial in
			$C^\infty(X)^{\Bbb C}[\theta, \bar{\theta}, \vartheta, \bar{\vartheta}]^{\anticommuting}$
 are closed under the multiplication in $C^\infty(\widehat{X}^{\widehat{\boxplus}})$.

\end{proof}

\bigskip

Explicitly, let
 {\small
 \begin{eqnarray*}
  \breve{A}
  &= &
    A_{(0)}
	+ \sum_{\alpha}\theta^\alpha\vartheta_\alpha A_{(\alpha)}
	+ \sum_{\dot{\beta}}
	       \bar{\theta}^{\dot{\beta}}\bar{\vartheta}_{\dot{\beta}} A_{(\dot{\beta})}
		  \\
  && \hspace{1em}			
    +\; \theta^1\theta^2  \vartheta_1\vartheta_2 A_{(12)}		 		
	+ \sum_{\alpha,\dot{\beta}}\theta^\alpha \bar{\theta}^{\dot{\beta}}
	      \left(\rule{0ex}{1.2em}\right.\!
		    \sum_\mu \sigma^\mu_{\alpha\dot{\beta}} A_{[\mu]}\,
			 +\, \vartheta_\alpha\bar{\vartheta}_{\dot{\beta}} A_{(\alpha\dot{\beta})}
		  \!\left.\rule{0ex}{1.2em}\right)		
    + \bar{\theta}^{\dot{1}}\bar{\theta}^{\dot{2}}
	     \bar{\vartheta}_{\dot{1}}\bar{\vartheta}_{\dot{2}} A_{(\dot{1}\dot{2})}  		
		   \\
  && \hspace{1em}		
	+ \sum_{\dot{\beta}}\theta^1\theta^2\bar{\theta}^{\dot{\beta}}
	     \left(\rule{0ex}{1.2em}\right.\!
		   \sum_{\alpha,\mu }
		          \vartheta_\alpha\,\sigma^{\mu \alpha}_{\;\;\;\;\dot{\beta}} A^\prime_{[\mu]}\,
			+\, \bar{\vartheta}_{\dot{\beta}} A^\prime_{(\dot{\beta})}
		    +\, \vartheta_1\vartheta_2\bar{\vartheta}_{\dot{\beta}} A_{(12\dot{\beta})}
		 \!\left.\rule{0ex}{1.2em}\right)
		   \\
  && \hspace{1em}			
    + \sum_\alpha \theta^\alpha\bar{\theta}^{\dot{1}}\bar{\theta}^{\dot{2}}
	   \left(\rule{0ex}{1.2em}\right.	
	     \vartheta_\alpha f^{\prime\prime}_{(\alpha)}
	     + \sum_{\dot{\beta}, \mu}
		      \bar{\vartheta}_{\dot{\beta}} \sigma^{\mu\dot{\beta}}_\alpha
		      A^{\prime\prime}_{[\mu]}\,		
         +\, \vartheta_\alpha \bar{\vartheta}_{\dot{1}}\bar{\vartheta}_{\dot{2}}	
		           A_{(\alpha\dot{1}\dot{2})}
	   \left.\rule{0ex}{1.2em}\right)
	   \\
  && \hspace{1em}
	+\; \theta^1\theta^2\bar{\theta}^{\dot{1}}\bar{\theta}^{\dot{2}}
	     \left(\rule{0ex}{1.2em}\right.\!
		  A^\sim_{(0)}
		  + \vartheta_1\vartheta_2 A^\sim_{(12)}
		  + \sum_{\alpha,\dot{\beta}, \mu} \vartheta_\alpha\bar{\vartheta}_{\dot{\beta}}
		         \bar{\sigma}^{\mu\dot{\beta}\alpha} A^{\sim}_{[\mu]}
		  + \bar{\vartheta}_{\dot{1}} \bar{\vartheta}_{\dot{2}}
		       A^\sim_{(\dot{1}\dot{2})}
		  + \vartheta_1\vartheta_2\bar{\vartheta}_{\dot{1}}\bar{\vartheta}_{\dot{2}}
		        A_{(12\dot{1}\dot{2})}
		 \!\left.\rule{0ex}{1.2em}\right)
 \end{eqnarray*}
  }
 and
 {\small
 \begin{eqnarray*}
  \breve{B}
  &= &
    B_{(0)}
	+ \sum_{\alpha}\theta^\alpha\vartheta_\alpha B_{(\alpha)}
	+ \sum_{\dot{\beta}}
	       \bar{\theta}^{\dot{\beta}}\bar{\vartheta}_{\dot{\beta}} B_{(\dot{\beta})}
		  \\
  && \hspace{1em}			
    +\; \theta^1\theta^2
	        \vartheta_1\vartheta_2 B_{(12)}		  		   		
	+ \sum_{\alpha,\dot{\beta}}\theta^\alpha \bar{\theta}^{\dot{\beta}}
	      \left(\rule{0ex}{1.2em}\right.\!
		    \sum_\mu \sigma^\mu_{\alpha\dot{\beta}} B_{[\mu]}\,
			 +\, \vartheta_\alpha\bar{\vartheta}_{\dot{\beta}} B_{(\alpha\dot{\beta})}
		  \!\left.\rule{0ex}{1.2em}\right)
    + \bar{\theta}^{\dot{1}}\bar{\theta}^{\dot{2}}
	     \bar{\vartheta}_{\dot{1}}\bar{\vartheta}_{\dot{2}} B_{(\dot{1}\dot{2})}  	
		   \\
  && \hspace{1em}		
	+ \sum_{\dot{\beta}}\theta^1\theta^2\bar{\theta}^{\dot{\beta}}
	     \left(\rule{0ex}{1.2em}\right.\!
		   \sum_{\alpha,\mu }
		          \vartheta_\alpha\,\sigma^{\mu \alpha}_{\;\;\;\;\dot{\beta}} B^\prime_{[\mu]}\,
			+\, \bar{\vartheta}_{\dot{\beta}} f^\prime_{(\dot{\beta})}
		    +\, \vartheta_1\vartheta_2\bar{\vartheta}_{\dot{\beta}} B_{(12\dot{\beta})}
		 \!\left.\rule{0ex}{1.2em}\right)
		   \\
  && \hspace{1em}			
    + \sum_\alpha \theta^\alpha\bar{\theta}^{\dot{1}}\bar{\theta}^{\dot{2}}
	   \left(\rule{0ex}{1.2em}\right.\!
	     \vartheta_\alpha B^{\prime\prime}_{(\alpha)}
	     \sum_{\dot{\beta}, \mu}
		   \bar{\vartheta}_{\dot{\beta}} \sigma^{\mu\dot{\beta}}_\alpha
		     B^{\prime\prime}_{[\mu]}\,	
         +\, \vartheta_\alpha \bar{\vartheta}_{\dot{1}}\bar{\vartheta}_{\dot{2}}	
		           B_{(\alpha\dot{1}\dot{2})}
	   \!\left.\rule{0ex}{1.2em}\right)
	   \\
  && \hspace{1em}
	+\; \theta^1\theta^2\bar{\theta}^{\dot{1}}\bar{\theta}^{\dot{2}}
	     \left(\rule{0ex}{1.2em}\right.\!
		  B^\sim_{(0)}
		  + \vartheta_1\vartheta_2 B^\sim_{(12)}
		  + \sum_{\alpha,\dot{\beta}, \mu} \vartheta_\alpha\bar{\vartheta}_{\dot{\beta}}
		         \bar{\sigma}^{\mu\dot{\beta}\alpha} B^{\sim}_{[\mu]}
		  + \bar{\vartheta}_{\dot{1}} \bar{\vartheta}_{\dot{2}}
		       B^\sim_{(\dot{1}\dot{2})}
		  + \vartheta_1\vartheta_2\bar{\vartheta}_{\dot{1}}\bar{\vartheta}_{\dot{2}}
		        B_{(12\dot{1}\dot{2})}
		 \!\left.\rule{0ex}{1.2em}\right)
		 \\
	& \in & C^\infty(\widehat{X}^{\widehat{\boxplus}})^\mediumscriptsize
 \end{eqnarray*}
  }
be medium scalar superfields on $X$.
Then, their product is given by

{\small
\begin{eqnarray*}
 && \breve{A}\breve{B}\; =\; \breve{B}\breve{A}
   \\
 && \hspace{1em}
  =\;
   A_{(0)} B_{(0)}
     + \sum_\alpha \theta^\alpha\vartheta_\alpha
	      \mbox{\Large $($}
	        A_{(\alpha)}B_{(0)} + A_{(0)} B_{(\alpha)}
		  \mbox{\Large $)$}
	 + \sum_{\dot{\beta}} \bar{\theta}^{\dot{\beta}}\bar{\vartheta}_{\dot{\beta}}
	       \mbox{\Large $($}
	        A_{(\dot{\beta})} B_{(0)} + A_{(0)} B_{(\dot{\beta})}
		   \mbox{\Large $)$}
	 \\
   && \hspace{2em}	
	 +\, \theta^1\theta^2 \vartheta_1\vartheta_2
	          \mbox{\Large $($}
   	            A_{(12)}B_{(0)} - A_{(1)}B_{(2)}
				  - A_{(2)} B_{(1)} + A_{(0)} B_{(12)}
			  \mbox{\Large $)$}
	  \\
   && \hspace{2em}	
	 + \sum_{\alpha, \dot{\beta}}
	      \left\{\rule{0ex}{1.2em}\right.
		    \sum_\mu \sigma^\mu_{\alpha\dot{\beta}}
			  \mbox{\Large $($}
			   A_{[\mu]}B_{(0)} + A_{(0)} B_{[\mu]}
			  \mbox{\Large $)$}
	     + \vartheta_\alpha \bar{\vartheta}_{\dot{\beta}}
		      \mbox{\Large $($}
		      A_{(\alpha \dot{\beta})} B_{(0)} - A_{(\alpha)}B_{(\dot{\beta})}
			  - A_{(\dot{\beta})} B_{(\alpha)} + A_{(0)} B_{(\alpha\dot{\beta})}			
			  \mbox{\Large $)$}
		  \left.\rule{0ex}{1.2em}\right\}
	  \\
   && \hspace{2em}	  		
	 +\, \bar{\theta}^{\dot{1}} \bar{\theta}^{\dot{2}} 	
	         \bar{\vartheta}_{\dot{1}} \bar{\vartheta}_{\dot{2}}
		          \mbox{\Large $($}
		            A_{(\dot{1}\dot{2})} B_{(0)} - A_{(\dot{1})}B_{(\dot{2})}
			          - A_{(\dot{2})}B_{(\dot{1})} + A_{(0)} B_{(\dot{1}\dot{2})}
                  \mbox{\Large $)$}			
	   \\
   && \hspace{2em}
	 + \sum_{\dot{\beta}} \theta^1\theta^2 \bar{\theta}^{\dot{\beta}}
	      \left\{\rule{0ex}{1.2em}\right.
		   \sum_{\alpha, \mu}
		     \vartheta_\alpha \sigma^{\mu\alpha}_{\;\;\;\;\dot{\beta}}
			    \mbox{\Large $($}
			      A^\prime_{[\mu]}B_{(0)}  + A_{[\mu]}B_{(\alpha)} +A_{(\alpha)} B_{[\mu]}
				   + A_{(0)}B^\prime_{[\mu]}
			    \mbox{\Large $)$}	
		  + \bar{\vartheta}_{\dot{\beta}}
		         \mbox{\Large $($}
				  A^\prime_{(\dot{\beta})} B_{(0)}	+ A_{(0)} B^\prime_{(\dot{\beta})}
				 \mbox{\Large $)$}
			 \\
        && \hspace{6em}	  				
	       +\, \vartheta_1\vartheta_2 \bar{\vartheta}_{\dot{\beta}}
		          \mbox{\Large $($}
  		             A_{(12\dot{\beta})} B_{(0)} + A_{(12)}B_{(\dot{\beta})}
					  + A_{(1\dot{\beta})}B_{(2)}+ A_{(2\dot{\beta})} B_{(1)}
	                      \\[-1ex]
                         && \hspace{12em}	  					
					  +\, A_{(1)} B_{(2\dot{\beta})} + A_{(2)} B_{(1\dot{\beta})}
					  + A_{(\dot{\beta})} B_{(12)} + A_{(0)} B_{(12\dot{\beta})}
				  \mbox{\Large $)$}
		  \left.\rule{0ex}{1.2em}\right\}
	  \\
   && \hspace{2em}	  		
     + \sum_\alpha \theta^\alpha \bar{\theta}^{\dot{1}} \bar{\theta}^{\dot{2}}
	      \left\{\rule{0ex}{1.2em}\right.
		    \vartheta_\alpha
			 \mbox{\Large $($}
			  A^{\prime\prime}_{(\alpha)}B_{(0)} + A_{(0)}B^{\prime\prime}_{(\alpha)}
			 \mbox{\Large $)$}
		    + \sum_{\dot{\beta}, \mu}
			   \bar{\vartheta}_{\dot{\beta}} \sigma^{\mu\dot{\beta}}_\alpha
			     \mbox{\Large $($}
			      A^{\prime\prime}_{[\mu]}B_{(0)} + A_{[\mu]} B_{(\dot{\beta})}
				     + A_{(\dot{\beta})} B_{[\mu]} + A_{(0)}B^{\prime\prime}_{[\mu]}
				 \mbox{\Large $)$}
           	  \\
              && \hspace{6em}	  			
		    +\, \vartheta_\alpha \bar{\vartheta}_{\dot{1}} \bar{\vartheta}_{\dot{2}}
			       \mbox{\Large $($}
			         A_{(\alpha\dot{1}\dot{2})} B_{(0)}  + A_{(\alpha\dot{1})}B_{(\dot{2})}
					  + A_{(\alpha\dot{2})}B_{(\dot{1})}  + A_{(\dot{1}\dot{2})}B_{(\alpha)}
					  	  \\[-1ex]
                     && \hspace{12em}	  					
					  +\, A_{(\alpha)}B_{(\dot{1}\dot{2})}+ A_{(\dot{1})}B_{(\alpha\dot{2})}
					  + A_{(\dot{2})} B_{(\alpha\dot{1})} + A_{(0)}B_{(\alpha\dot{1}\dot{2})}
				   \mbox{\Large $)$}
		  \left.\rule{0ex}{1.2em}\right\}
		\\
  && \hspace{2em}
      +\, \theta^1\theta^2\bar{\theta}^{\dot{1}} \bar{\theta}^{\dot{2}}
	       \left\{\rule{0ex}{1.2em}\right.
		       \mbox{\Large $($}
		     A^\sim_{(0)}B_{(0)}    			
			   + 2 \sum_{\mu, \nu}\eta^{\mu\nu} A_{[\mu]}B_{[\nu]}			
			   + A_{(0)}B^\sim_{(0)}
			   \mbox{\Large $)$}
			  \\
           && \hspace{6em}	
		     +\, \vartheta_1\vartheta_2
                  \mbox{\Large $($}
				    A^\sim_{(12)}B_{(0)}
					 - A^{\prime\prime}_{(1)}B_{(2)} - A^{\prime\prime}_{(2)}B_{(1)}					 
					 - A_{(1)}B^{\prime\prime}_{(2)} - A_{(2)}B^{\prime\prime}_{(1)}
					 + A_{(0)}B^\sim_{(12)}
                  \mbox{\Large $)$}				
        	  \\
           && \hspace{6em}	  				
			  +\, \sum_{\alpha, \dot{\beta}}
			        \vartheta_\alpha \bar{\vartheta}_{\dot{\beta}}\,
					   \bar{\sigma}^{\mu\dot{\beta}\alpha}
					 \mbox{\Large $($}
					   A^\sim_{[\mu]}B_{(0)} + A^\prime_{[\mu]} B_{(\dot{\beta})}
					   - A^{\prime\prime}_{[\mu]}B_{(\alpha)}- A_{(\alpha)}B^{\prime\prime}_{[\mu]}
					   + A_{(\dot{\beta})} B^\prime_{[\mu]} + A_{(0)}B^\sim_{[\mu]}
					 \mbox{\Large $)$}
        	  \\
		    && \hspace{6em}	
		    +\, \bar{\vartheta}_{\dot{1}}\bar{\vartheta}_{\dot{2}}
                  \mbox{\Large $($}
				    A^\sim_{(\dot{1}\dot{2})}B_{(0)}
					- A^\prime_{(\dot{1})}B_{(\dot{2})} - A^\prime_{(\dot{2})}B_{(\dot{1})}						
					- A_{(\dot{1})}B^\prime_{(\dot{2})}
					    - A_{(\dot{2})}B^\prime_{(\dot{1})}
					+ A_{(0)}B^\sim_{(\dot{1}\dot{2})}
                  \mbox{\Large $)$}				
        	  \\
            && \hspace{6em}	  						
			  +\, \vartheta_1\vartheta_2 \bar{\vartheta}_{\dot{1}} \bar{\vartheta}_{\dot{2}}
			        \mbox{\Large $($}
			          A_{(12\dot{1}\dot{2})}B_{(0)} - A_{(12\dot{1})}B_{(\dot{2})}
					  - A_{(12\dot{2})}B_{(\dot{1})} - A_{(1\dot{1}\dot{2})}B_{(2)}
					  - A_{(2\dot{1}\dot{2})}B_{(1)}
					  \\
                   && \hspace{8em}	  					
					  +\, A_{(12)}B_{(\dot{1}\dot{2})}
					  + A_{(1\dot{1})}B_{(2\dot{2})} + A_{(1\dot{2})}B_{(2\dot{1})}
					  + A_{(2\dot{1})}B_{(1\dot{2})} + A_{(2\dot{2})}B_{(1\dot{1})}
 	                   \\[-1ex]
                   && \hspace{8em}	  										
					  +\, A_{(\dot{1}\dot{2})}B_{(12)} - A_{(1)}B_{(2\dot{1}\dot{2})}
					  - A_{(2)}B_{(1\dot{1}\dot{2})} -A_{(\dot{1})}B_{(12\dot{2})}
					  - A_{(\dot{2})}B_{(12\dot{1})} + A_{(0)}B_{(12\dot{1}\dot{2})}
				   \mbox{\Large $)$}
		   \left.\rule{0ex}{1.2em}\right\}
	        \\
     && \hspace{1em}	  							
       \in\; C^\infty(\widehat{X}^{\widehat{\boxplus}})^\mediumscriptsize\,.
\end{eqnarray*}
}

\bigskip

\begin{definition} {\bf [small (Lorentz-)scalar superfield]}\;
{\rm
 A medium scalar superfield\\
  $\breve{f}\in C^{\infty}(\widehat{X}^{\widehat{\boxplus}})^\mediumscriptsize$
  is called a {\it small\footnote{The
                                          term `{\it short}' is better reserved for the future when we address {\it BPS short multiplet}
		                                  in the situation with extended (i.e.\ $N\ge 2$) supersymmetries.
										 Conceptually one may better call a small scalar superfield on $X$
										   a `{\it curtailed}' or '{\it trimmed}' or `{\it diminished}' scalar superfield on $X$
										   since this is exactly what it is.
										 However, when naming the subring they form in
										  $C^\infty(\widehat{X}^{\widehat{\boxplus}})$,
										  it seems more appealing to call it the {\it small function-ring}
										     of the towered superspace $\widehat{X}^{\widehat{\boxplus}}$
										   than any other name.
                                         It is for this reason, we choose the term `{\it small scalar superfield}'	here.
                                         Such superfields  are called
										   {\it superfields in the physical sector} of $\widehat{X}^{\widehat{\boxplus}}$
										   or synonymously {\it physical superfields} in [L-Y5: Definition 1.4.1] (SUSY(1))
										   as they form the subring in $C^\infty(\widehat{X}^{\widehat{\boxplus}})$
										    generated by chiral superfields and antichiral superfields; cf.\:[L-Y5: Sec.\:1.4] (SUSY(1)).
										 However, during the topic course {\sl Supersymmetric quantum field theories}
										    given by {\it Jesse Thaler} at MIT, fall 2019,
										  we gradually realized that the best candidate in the setting of [L-Y5] (SUSY(1))
										    for the notion of `{\it vector superfield}' in [Wess \& Bagger: Chapter VI]
										    is {\it not} in this subring;
										    cf.\:Sec.\:4.1 of the current notes versus [L-Y5: Sec.\:3] (SUSY(1)).
										 With this upgraded understanding and
										   since the name `physical' will almost for sure mislead one to think
										   that elements in $C^\infty(\widehat{X}^{\widehat{\boxplus}})$
										       that do not lie in this subring
										     are ``not physical" or ``not relevant to physics",
										   we now fix ourselves to call this subring the ``{\it small function-ring}" or the ``{\it small sector}"
										   of the towered superspace $\widehat{X}^{\widehat{\boxplus}}$;
                                         cf.\:Lemma/Definition~1.2.14.
                                                           }
						   scalar superfield} on $X$   	
  if in addition
   \begin{itemize}
    \item[\LARGE $\cdot$]
	 $f^\prime_{(\dot{\beta})}= f^{\prime\prime}_{(\alpha)}
	   = f^\sim_{(12)}= f^\sim_{(\dot{1}\dot{2})}=0$,
	 for $\alpha=1,2$, $\dot{\beta}=\dot{1}, \dot{2}$.
   \end{itemize}
  as a polynomial in $C^\infty(X)^{\Bbb C}[\theta, \bar{\theta}, \vartheta, \bar{\vartheta}]^{\anticommuting}$.
}\end{definition}
   
\bigskip
   
Explicitly, a small scalar superfield $\breve{f}$ on $X$ is of the following form
 {\small
 \begin{eqnarray*}
  \breve{f}
 &= &
    f_{(0)}
	+ \sum_{\alpha}\theta^\alpha\vartheta_\alpha f_{(\alpha)}
	+ \sum_{\dot{\beta}}
	       \bar{\theta}^{\dot{\beta}}\bar{\vartheta}_{\dot{\beta}} f_{(\dot{\beta})}  \\
  && \hspace{1em}			
    +\; \theta^1\theta^2\vartheta_1\vartheta_2 f_{(12)}
	+ \sum_{\alpha,\dot{\beta}}\theta^\alpha \bar{\theta}^{\dot{\beta}}
	      \left(\rule{0ex}{1.2em}\right.\!
		    \sum_\mu \sigma^\mu_{\alpha\dot{\beta}} f_{[\mu]}\,
			 +\, \vartheta_\alpha\bar{\vartheta}_{\dot{\beta}} f_{(\alpha\dot{\beta})}
		  \!\left.\rule{0ex}{1.2em}\right)
    + \bar{\theta}^{\dot{1}}\bar{\theta}^{\dot{2}}
	   \bar{\vartheta}_{\dot{1}}\bar{\vartheta}_{\dot{2}} f_{(\dot{1}\dot{2})}  \\
  && \hspace{1em}		
	+ \sum_{\dot{\beta}}\theta^1\theta^2\bar{\theta}^{\dot{\beta}}
	     \left(\rule{0ex}{1.2em}\right.\!
		   \sum_{\alpha,\mu }
		          \vartheta_\alpha\,\sigma^{\mu \alpha}_{\;\;\;\;\dot{\beta}} f^\prime_{[\mu]}\,
		    +\, \vartheta_1\vartheta_2\bar{\vartheta}_{\dot{\beta}} f_{(12\dot{\beta})}
		 \!\left.\rule{0ex}{1.2em}\right)
    + \sum_\alpha \theta^\alpha\bar{\theta}^{\dot{1}}\bar{\theta}^{\dot{2}}
	   \left(\rule{0ex}{1.2em}\right.\!
	     \sum_{\dot{\beta}, \mu}
		  \sigma^{\mu\dot{\beta}}_\alpha
		     \bar{\vartheta}_{\dot{\beta}} f^{\prime\prime}_{[\mu]}\,
         +\, \vartheta_\alpha \bar{\vartheta}_{\dot{1}}\bar{\vartheta}_{\dot{2}}	
		           f_{(\alpha\dot{1}\dot{2})}
	   \!\left.\rule{0ex}{1.2em}\right)\\
  && \hspace{1em}
	+\; \theta^1\theta^2\bar{\theta}^{\dot{1}}\bar{\theta}^{\dot{2}}
	     \left(\rule{0ex}{1.2em}\right.\!
		  f^\sim_{(0)}
		  + \sum_{\alpha,\dot{\beta}, \mu} \vartheta_\alpha\bar{\vartheta}_{\dot{\beta}}
		         \bar{\sigma}^{\mu\dot{\beta}\alpha} f^{\sim}_{[\mu]}
		  + \vartheta_1\vartheta_2\bar{\vartheta}_{\dot{1}}\bar{\vartheta}_{\dot{2}}
		        f_{(12\dot{1}\dot{2})}
		 \!\left.\rule{0ex}{1.2em}\right)
	 \\
   & \in &	C^\infty(X)^{\Bbb C}[\theta, \bar{\theta}, \vartheta, \bar{\vartheta}]^\anticommuting\,.		 
 \end{eqnarray*}
  }
    
\bigskip

\begin{lemma-definition} {\bf [small sector $\widehat{X}^{\widehat{\boxplus},\smallscriptsize}$
                and small function-ring of $\widehat{X}^{\widehat{\boxplus}}$]}\;
 The collection of  small scalar superfields on $X$
   as defined in Definition~1.2.13
  is an even subring of the complexified ${\Bbb Z}/2$-graded $C^\infty$-ring
  $C^\infty(\widehat{X}^{\widehat{\boxplus}})$.
 Denote this subring
   (also  a $C^{\infty}(X)^{\Bbb C}$-subalgebra
   of $C^{\infty}(\widehat{X}^{\widehat{\boxplus}})$)
   by $C^\infty(\widehat{X}^{\widehat{\boxplus}})^\smallscriptsize$.
 Then, the $C^\infty$-hull of $C^\infty(\widehat{X}^{\widehat{\boxplus}})$ restricts to the $C^\infty$-hull
   of $C^\infty(\widehat{X}^{\widehat{\boxplus}})^\smallscriptsize$,
   which is given by
   $$
     \mbox{$C^\infty$-hull}\,(C^\infty(\widehat{X}^{\widehat{\boxplus}})^\smallscriptsize)\;
	   =\; \{\breve{f}\in C^\infty(\widehat{X}^{\widehat{\boxplus}})^\smallscriptsize
	                                                                            \,|\, \breve{f}_{(0)}\in C^\infty(X)\}\,.
   $$								
{\rm
 Denote by $\widehat{X}^{\widehat{\boxplus}, \smallscriptsize}$
   the complexified ${\Bbb Z}/2$-graded $C^\infty$-scheme with the underlying topology $X$ and
   function ring $C^\infty(\widehat{X}^{\widehat{\boxplus}})^\smallscriptsize$.
 Then there is a built-in dominant morphism
   $\widehat{X}^{\widehat{\boxplus}, \mediumscriptsize}
      \rightarrow \widehat{X}^{\widehat{\boxplus}, \smallscriptsize}$.
 We will call $\widehat{X}^{\widehat{\boxplus}, \smallscriptsize}$
   the {\it small sector} of $\widehat{X}^{\widehat{\boxplus}}$  and
   $C^\infty(\widehat{X}^{\widehat{\boxplus}})^\smallscriptsize$
   the {\it small function-ring} of the towered superspace $\widehat{X}^{\widehat{\boxplus}}$.
   }
\end{lemma-definition}

\begin{proof}
 This follows from the fact that
  the conditions on small scalar superfields
   (1)  Lorentz scalar,
   (2)  $(\vartheta, \bar{\vartheta})$-degree $\le$  $(\theta, \bar{\theta})$-degree
	           for every summand,
   (3)	 $f^\prime_{(0)}= f^{\prime\prime}_{(0)}
                 = f^\prime_{(\dot{\beta})}= f^{\prime\prime}_{(\alpha)}
				 = f^\sim_{(12)} = f^\sim_{(\dot{1}\dot{2})}= 0$
            as a polynomial in
			$C^\infty(X)^{\Bbb C}[\theta, \bar{\theta}, \vartheta, \bar{\vartheta}]^{\anticommuting}$
 are closed under the multiplication in $C^\infty(\widehat{X}^{\widehat{\boxplus}})$.
 Cf.\:[L-Y5: Lemma 1.4.2] (SUSY(1)).

\end{proof}
 
\bigskip

Explicitly, let
{\small
\begin{eqnarray*}
 && \breve{A}\; =\;
   A_{(0)}
     + \sum_\alpha \theta^\alpha\vartheta_\alpha A_{(\alpha)}
	 + \sum_{\dot{\beta}} \bar{\theta}^{\dot{\beta}}\vartheta_{\dot{\beta}} A_{(\dot{\beta})}
	 \\
   && \hspace{2em}	
	 +\, \theta^1\theta^2 \vartheta_1\vartheta_2\, A_{(12)}
	 + \sum_{\alpha, \dot{\beta}}
	      \mbox{\Large $($}
		    \sum_\mu \sigma^\mu_{\alpha\dot{\beta}}A_{[\mu]}
			+ \vartheta_\alpha \bar{\vartheta}_{\dot{\beta}} A_{(\alpha \dot{\beta})}
		  \mbox{\Large $)$}
	 + \bar{\theta}^{\dot{1}} \bar{\theta}^{\dot{2}}
	       \bar{\vartheta}_{\dot{1}} \bar{\vartheta}_{\dot{2}}\, A_{(\dot{1}\dot{2})}
	   \\
   && \hspace{2em}
	 + \sum_{\dot{\beta}} \theta^1\theta^2 \bar{\theta}^{\dot{\beta}}
	      \mbox{\Large $($}
		   \sum_{\alpha, \mu}
		     \vartheta_\alpha \sigma^{\mu\alpha}_{\;\;\;\;\dot{\beta}}\,A^\prime_{[\mu]}
			+ \vartheta_1\vartheta_2 \bar{\vartheta}_{\dot{\beta}}\,  A_{(12\dot{\beta})}
		  \mbox{\Large $)$}
     + \sum_\alpha \theta^\alpha \bar{\theta}^{\dot{1}} \bar{\theta}^{\dot{2}}
	      \mbox{\Large $($}
		    \sum_{\dot{\beta}, \mu}
			   \bar{\vartheta}_{\dot{\beta}} \sigma^{\mu\dot{\beta}}_\alpha A^{\prime\prime}_{[\mu]}
		    + \vartheta_\alpha \vartheta_{\dot{1}} \vartheta_{\dot{2}}\, A_{(\alpha\dot{1}\dot{2})}
		  \mbox{\Large $)$}
		\\
  && \hspace{2em}
      +\, \theta^1\theta^2\bar{\theta}^{\dot{1}} \bar{\theta}^{\dot{2}}
	       \mbox{\Large $($}
		     A^\sim_{(0)}
			  + \sum_{\alpha, \dot{\beta}}
			        \vartheta_\alpha \bar{\vartheta}_{\dot{\beta}}\,
					   \bar{\sigma}^{\mu\dot{\beta}\alpha}\, A^\sim_{[\mu]}
			  + \vartheta_1\vartheta_2 \bar{\vartheta}_{\dot{1}} \bar{\vartheta}_{\dot{2}}
			       A_{(12\dot{1}\dot{2})}
		   \mbox{\Large $)$}
\end{eqnarray*}
}and 
{\small
\begin{eqnarray*}
 && \breve{B}\;  =\;
   B_{(0)}
     + \sum_\gamma \theta^\gamma \vartheta_\gamma B_{(\gamma)}
	 + \sum_{\dot{\delta}} \bar{\theta}^{\dot{\delta}}\vartheta_{\dot{\delta}} B_{(\dot{\delta})}
	 \\
   && \hspace{2em}	
	 +\, \theta^1\theta^2 \vartheta_1\vartheta_2\, B_{(12)}
	 + \sum_{\gamma, \dot{\delta}}
	      \mbox{\Large $($}
		    \sum_\nu \sigma^\nu_{\gamma\dot{\delta}}B_{[\nu]}
			+ \vartheta_\gamma \bar{\vartheta}_{\dot{\delta}} B_{(\gamma \dot{\delta})}
		  \mbox{\Large $)$}
	 + \bar{\theta}^{\dot{1}} \bar{\theta}^{\dot{2}}
	       \bar{\vartheta}_{\dot{1}} \bar{\vartheta}_{\dot{2}}\, B_{(\dot{1}\dot{2})}
	   \\
   && \hspace{2em}
	 + \sum_{\dot{\delta}} \theta^1\theta^2 \bar{\theta}^{\dot{\delta}}
	      \mbox{\Large $($}
		   \sum_{\gamma, \nu}
		     \vartheta_\gamma \sigma^{\nu\gamma}_{\;\;\;\;\dot{\delta}}\,B^\prime_{[\nu]}
			+ \vartheta_1\vartheta_2 \bar{\vartheta}_{\dot{\delta}}\,  B_{(12\dot{\delta})}
		  \mbox{\Large $)$}
     + \sum_\gamma \theta^\gamma \bar{\theta}^{\dot{1}} \bar{\theta}^{\dot{2}}
	      \mbox{\Large $($}
		    \sum_{\dot{\delta}, \nu}
			   \bar{\vartheta}_{\dot{\delta}} \sigma^{\nu\dot{\delta}}_\gamma B^{\prime\prime}_{[\nu]}
		    + \vartheta_\gamma \vartheta_{\dot{1}} \vartheta_{\dot{2}}\, B_{(\gamma\dot{1}\dot{2})}
		  \mbox{\Large $)$}
		\\
  && \hspace{2em}
      +\, \theta^1\theta^2\bar{\theta}^{\dot{1}} \bar{\theta}^{\dot{2}}
	       \mbox{\Large $($}
		     B^\sim_{(0)}
			  + \sum_{\gamma, \dot{\delta}}
			        \vartheta_\gamma \bar{\vartheta}_{\dot{\delta}}\,
					   \bar{\sigma}^{\nu\dot{\delta}\gamma}\, B^\sim_{[\nu]}
			  + \vartheta_1\vartheta_2 \bar{\vartheta}_{\dot{1}} \bar{\vartheta}_{\dot{2}}
			       B_{(12\dot{1}\dot{2})}
		   \mbox{\Large $)$}
\end{eqnarray*}
}be 
elements in $C^\infty(\widehat{X}^{\widehat{\boxplus}})^\smallscriptsize$.
Here,
 $A^{\tinybullet}_{\,\tinybullet}= A^{\tinybullet}_{\,\tinybullet}(x)$ and
 $B^{\tinybullet}_{\,\tinybullet}= A^{\tinybullet}_{\,\tinybullet}(x)\,\in C^\infty(X)^{\Bbb C}$.
Then,

{\small
\begin{eqnarray*}
 && \breve{A}\breve{B}\; =\; \breve{B}\breve{A}
   \\
 && \hspace{1em}
  =\;
   A_{(0)} B_{(0)}
     + \sum_\alpha \theta^\alpha\vartheta_\alpha
	      \mbox{\Large $($}
	        A_{(\alpha)}B_{(0)} + A_{(0)} B_{(\alpha)}
		  \mbox{\Large $)$}
	 + \sum_{\dot{\beta}} \bar{\theta}^{\dot{\beta}}\vartheta_{\dot{\beta}}
	       \mbox{\Large $($}
	        A_{(\dot{\beta})} B_{(0)} + A_{(0)} B_{(\dot{\beta})}
		   \mbox{\Large $)$}
	 \\
   && \hspace{2em}	
	 +\, \theta^1\theta^2 \vartheta_1\vartheta_2
	        \mbox{\Large $($}
   	          A_{(12)}B_{(0)} - A_{(1)}B_{(2)} - A_{(2)} B_{(1)} + A_{(0)} B_{(12)}
			\mbox{\Large $)$}
	  \\
   && \hspace{2em}	
	 + \sum_{\alpha, \dot{\beta}}
	      \left\{\rule{0ex}{1.2em}\right.
		    \sum_\mu \sigma^\mu_{\alpha\dot{\beta}}
			  \mbox{\Large $($}
			   A_{[\mu]}B_{(0)} + A_{(0)} B_{[\mu]}
			  \mbox{\Large $)$}
	     + \vartheta_\alpha \bar{\vartheta}_{\dot{\beta}}
		      \mbox{\Large $($}
		      A_{(\alpha \dot{\beta})} B_{(0)} - A_{(\alpha)}B_{\dot{\beta}}
			  - A_{(\dot{\beta})} B_{(\alpha)} + A_{(0)} B_{(\alpha\dot{\beta})}			
			  \mbox{\Large $)$}
		  \left.\rule{0ex}{1.2em}\right\}
	  \\
   && \hspace{2em}	  		
	 +\, \bar{\theta}^{\dot{1}} \bar{\theta}^{\dot{2}}
	       \bar{\vartheta}_{\dot{1}} \bar{\vartheta}_{\dot{2}}
		  \mbox{\Large $($}
		     A_{(\dot{1}\dot{2})} B_{(0)} - A_{(\dot{1})}B_{(\dot{2})}
			   - A_{(\dot{2})}B_{(\dot{1})} + A_{(0)} B_{(\dot{1}\dot{2})}
          \mbox{\Large $)$}			
	   \\
   && \hspace{2em}
	 + \sum_{\dot{\beta}} \theta^1\theta^2 \bar{\theta}^{\dot{\beta}}
	      \left\{\rule{0ex}{1.2em}\right.
		   \sum_{\alpha, \mu}
		     \vartheta_\alpha \sigma^{\mu\alpha}_{\;\;\;\;\dot{\beta}}
			    \mbox{\Large $($}
			      A^\prime_{[\mu]}B_{(0)}  + A_{[\mu]}B_{(\alpha)} +A_{(\alpha)} B_{[\mu]}
				   + A_{(0)}B^\prime_{[\mu]}
			    \mbox{\Large $)$}	
	         \\
        && \hspace{6em}	  				
	       +\, \vartheta_1\vartheta_2 \bar{\vartheta}_{\dot{\beta}}
		          \mbox{\Large $($}
  		             A_{(12\dot{\beta})} B_{(0)} + A_{(12)}B_{(\dot{\beta})}
					  + A_{(1\dot{\beta})}B_{(2)}+ A_{(2\dot{\beta})} B_{(1)}
	                      \\[-1ex]
                         && \hspace{12em}	  					
					  +\, A_{(1)} B_{(2\dot{\beta})} + A_{(2)} B_{(1\dot{\beta})}
					  + A_{(\dot{\beta})} B_{(12)} + A_{(0)} B_{(12\dot{\beta})}
				  \mbox{\Large $)$}
		  \left.\rule{0ex}{1.2em}\right\}
	  \\
   && \hspace{2em}	  		
     + \sum_\alpha \theta^\alpha \bar{\theta}^{\dot{1}} \bar{\theta}^{\dot{2}}
	      \left\{\rule{0ex}{1.2em}\right.
		    \sum_{\dot{\beta}, \mu}
			   \bar{\vartheta}_{\dot{\beta}} \sigma^{\mu\dot{\beta}}_\alpha
			     \mbox{\Large $($}
			      A^{\prime\prime}_{[\mu]}B_{(0)} + A_{[\mu]} B_{(\dot{\beta})}
				     + A_{(\dot{\beta})} B_{[\mu]} + A_{(0)}B^{\prime\prime}_{[\mu]}
				 \mbox{\Large $)$}
           	  \\
              && \hspace{6em}	  			
		    +\, \vartheta_\alpha \vartheta_{\dot{1}} \vartheta_{\dot{2}}
			       \mbox{\Large $($}
			         A_{(\alpha\dot{1}\dot{2})} B_{(0)}  + A_{(\alpha\dot{1})}B_{(\dot{2})}
					  + A_{(\alpha\dot{2})}B_{(\dot{1})}  + A_{(\dot{1}\dot{2})}B_{(\alpha)}
					  	  \\[-1ex]
                     && \hspace{12em}	  					
					  +\, A_{(\alpha)}B_{(\dot{1}\dot{2})}+ A_{(\dot{1})}B_{(\alpha\dot{2})}
					  + A_{(\dot{2})} B_{(\alpha\dot{1})} + A_{(0)}B_{(\alpha\dot{1}\dot{2})}
				   \mbox{\Large $)$}
		  \left.\rule{0ex}{1.2em}\right\}
		\\
  && \hspace{2em}
      +\, \theta^1\theta^2\bar{\theta}^{\dot{1}} \bar{\theta}^{\dot{2}}
	       \left\{\rule{0ex}{1.2em}\right.
		     A^\sim_{(0)}B_{(0)}   + 2 \sum_{\mu, \nu}\eta^{\mu\nu} A_{[\mu]}B_{[\nu]}
			   + A_{(0)}B^\sim_{(0)}
        	  \\
           && \hspace{6em}	  				
			  + \sum_{\alpha, \dot{\beta}}
			        \vartheta_\alpha \bar{\vartheta}_{\dot{\beta}}\,
					   \bar{\sigma}^{\mu\dot{\beta}\alpha}
					 \mbox{\Large $($}
					   A^\sim_{[\mu]}B_{(0)} + A^\prime_{[\mu]} B_{(\dot{\beta})}
					   - A^{\prime\prime}_{[\mu]}B_{(\alpha)}- A_{(\alpha)}B^{\prime\prime}_{[\mu]}
					   + A_{(\dot{\beta})} B^\prime_{[\mu]} + A_{(0)}B^\sim_{[\mu]}
					 \mbox{\Large $)$}
        	  \\
            && \hspace{6em}	  						
			  +\, \vartheta_1\vartheta_2 \bar{\vartheta}_{\dot{1}} \bar{\vartheta}_{\dot{2}}
			        \mbox{\Large $($}
			          A_{(12\dot{1}\dot{2})}B_{(0)} - A_{(12\dot{1})}B_{(\dot{2})}
					  - A_{(12\dot{2})}B_{(\dot{1})} - A_{(1\dot{1}\dot{2})}B_{(2)}
					  - A_{(2\dot{1}\dot{2})}B_{(1)}
					  \\
                   && \hspace{8em}	  					
					  +\, A_{(12)}B_{(\dot{1}\dot{2})}
					  + A_{(1\dot{1})}B_{(2\dot{2})} + A_{(1\dot{2})}B_{(2\dot{1})}
					  + A_{(2\dot{1})}B_{(1\dot{2})} + A_{(2\dot{2})}B_{(1\dot{1})}
 	                   \\[-1ex]
                   && \hspace{8em}	  										
					  +\, A_{(\dot{1}\dot{2})}B_{(12)} - A_{(1)}B_{(2\dot{1}\dot{2})}
					  - A_{(2)}B_{(1\dot{1}\dot{2})} -A_{(\dot{1})}B_{(12\dot{2})}
					  - A_{(\dot{2})}B_{(12\dot{1})} + A_{(0)}B_{(12\dot{1}\dot{2})}
				   \mbox{\Large $)$}
		   \left.\rule{0ex}{1.2em}\right\}
	        \\
     && \hspace{1em}	  							
       \in\; C^\infty(\widehat{X}^{\widehat{\boxplus}})^\smallscriptsize\,.
\end{eqnarray*}
}

\bigskip

For a brief comparison:

\bigskip

\centerline{\small
\begin{tabular}{|l||c|c|c|c|}\hline
 class of superfield $\in C^\infty(\widehat{X}^{\widehat{\boxplus}})$ \rule{0ex}{1.2em}
     & general  & tame scalar  & medium scalar  & small scalar \\ \hline
 number of components in $C^\infty(X)^{\Bbb C}$\rule{0ex}{1.2em}
     & $2^8= 256$  & $41$   & $39$   & $33$ \\ \hline
\end{tabular}}

\bigskip

\subsection{The chiral and the antichiral condition on  $C^\infty(\widehat{X}^{\widehat{\boxplus}})$ and its subrings}

\begin{definition} {\bf [chiral and antichiral function on $\widehat{X}^{\widehat{\boxplus}}$ and its sectors]}\;
{\rm
 Recall the supersymmetrically invariant derivations on $\widehat{X}$
  $$
    e_{\alpha^\prime}\;
	  :=\; \mbox{\Large $\frac{\partial}{\partial \theta^\alpha}$}
	          + \sqrt{-1}\sum_{\dot{\beta},\mu}\sigma^\mu_{\alpha\dot{\beta}}
			       \bar{\theta}^{\dot{\beta}} \partial_\mu\,,\hspace{2em}
    e_{\beta^{\prime\prime}}\;
	  :=\; - \mbox{\Large $\frac{\partial}{\partial \bar{\theta}^{\dot{\beta}}}$}
	          - \sqrt{-1}\sum_{\alpha,\mu}
			      \theta^\alpha \sigma^\mu_{\alpha\dot{\beta}} \partial_\mu\,.
  $$
 An $\breve{f}\in C^\infty(\widehat{X}^{\widehat{\boxplus}})$ is called {\it chiral}
  (resp.\:{\it antichiral}) if
  $$
    e_{1^{\prime\prime}} \breve{f} \;=\; e_{2^{\prime\prime}}\breve{f}\;=\; 0
	\hspace{4em}
	(\,\mbox{resp.}\;\;
	     e_{1^\prime} \breve{f} \;=\; e_{2^\prime}\breve{f}\;=\; 0)\,.
  $$
 $\breve{f}$ is called a {\it tame chiral superfield}
    (resp.\:{\it medium chiral superfield}, {\it small chiral superfield}) on $X$
 if in addition
  $\breve{f}\in C^\infty(\widehat{X}^{\widehat{\boxplus}})^\tamescriptsize$
  (resp.\:$C^\infty(\widehat{X}^{\widehat{\boxplus}})^\mediumscriptsize$,
               $C^\infty(\widehat{X}^{\widehat{\boxplus}})^\smallscriptsize$).
}\end{definition}

\medskip

\begin{example} {\bf [basic chiral and antichiral function on $\widehat{X}$]}\; {\rm
 (1)
 Let
  $$
   x^{\prime\mu} \;
    := x^\mu
	    + \sqrt{-1}\sum_{\alpha,\dot{\beta}}
	         \theta^\alpha \sigma^\mu_{\alpha\dot{\beta}} \bar{\theta}^{\dot{\beta}}
	 \hspace{2em}\mbox{and}\hspace{2em}
   x^{\prime\prime\mu} \;
    := x^\mu
	    - \sqrt{-1}\sum_{\alpha,\dot{\beta}}
	         \theta^\alpha \sigma^\mu_{\alpha\dot{\beta}} \bar{\theta}^{\dot{\beta}}\;
	\in C^\infty(\widehat{X})\,,
  $$
   for $\mu=0,1,2,3$.
 (In collective short-hand,
     $x^\prime = x + \sqrt{-1}\theta\sigma\bar{\theta}$ and
     $x^{\prime\prime} = x - \sqrt{-1}\theta\sigma\bar{\theta}$.)
 Then $x^{\prime\mu}$'s are chiral and $x^{\prime\prime\mu}$   are antichiral.
 
 (2)
 $\theta^\alpha$, $\alpha=1,2$, are chiral   and
 $\bar{\theta}^{\dot{\beta}}$, $\dot{\beta}=\dot{1}, \dot{2}$, are antichiral.
}\end{example}

\bigskip

Note that
 since $x^\mu$, $x^{\prime\mu}$, and $x^{\prime\prime\mu}$ differ only by an even nilpotent element
  in $C^\infty(\widehat{X})$,
 any of the collection $(x, \theta, \bar{\theta})$, or $(x^\prime, \theta, \bar{\theta})$,
  $(x^{\prime\prime}, \theta, \bar{\theta})$
  generate $C^\infty(\widehat{X})$ as a complexified ${\Bbb Z}/2$-graded $C^\infty$-ring
 and, hence, can serve as coordinate functions on $\widehat{X}$.

\bigskip

\begin{definition} {\bf [standard chiral coordinate functions and
    standard antichiral coordinate functions on $\widehat{X}^{\widehat{\boxplus}}$ - with abuse]}\; {\rm
 For convenience but with slight abuse of the terminology,
   we shall call
     $(x^\prime, \theta, \bar{\theta})$
      the {\it standard chiral coordinate functions} on $\widehat{X}^{\widehat{\boxplus}}$
	  (despite that $\bar{\theta}_{\dot{\beta}}$ are not chiral)    and
     $(x^{\prime\prime}, \theta, \bar{\theta})$
      the {\it standard antichiral coordinate functions} on $\widehat{X}^{\widehat{\boxplus}}$
	  (despite that $\theta_\alpha$ are not antichiral).
}\end{definition}

\bigskip

The following lemma gives a characterization of chiral functions and antichiral functions
 on $\widehat{X}^{\widehat{\boxplus}}$:

\bigskip

\begin{lemma} {\bf [chiral function and antichiral function on $\widehat{X}^{\widehat{\boxplus}}$]}\;
 $(1)$
 $\breve{f}\in C^\infty(\widehat{X}^{\widehat{\boxplus}})$ is chiral if and only if,
   as an element of $C^\infty(X)^{\Bbb C}[\theta, \bar{\theta}, \vartheta, \bar{\theta}]^\anticommuting$,
  $\breve{f}$ is of the following form
  \begin{eqnarray*}
   \breve{f} & = &
     \breve{f}_{(0)}(x, \vartheta, \bar{\vartheta})
	 + \sum_\gamma \theta^\gamma \breve{f}_{(\gamma)}(x, \vartheta, \bar{\theta})
	    + \theta^1\theta^2 \breve{f}_{(12)}(x,\vartheta, \bar{\vartheta})		     	
	    + \sqrt{-1} \sum_{\gamma, \dot{\delta}, \nu}
		    \theta^\gamma \bar{\theta}^{\dot{\delta}} \sigma^\nu_{\gamma\dot{\delta}}
			  \partial_\nu \breve{f}_{(0)}(x, \vartheta, \bar{\vartheta})
 		\\
	&& 			
        +\, \sqrt{-1} \sum_{\dot{\delta}, \gamma, \nu}
		     \theta^1\theta^2 \bar{\theta}^{\dot{\delta}}
			  \sigma^{\nu\gamma}_{\;\;\;\;\dot{\delta}}\,\partial_\nu
			   \breve{f}_{(\gamma)}(x, \vartheta, \bar{\vartheta})
		- \theta^1\theta^2\bar{\theta}^{\dot{1}} \bar{\theta}^{\dot{2}}\,
		    \square \breve{f}_{(0)}(x,\vartheta, \bar{\vartheta})\,.
  \end{eqnarray*}
 In particular, a chiral $\breve{f}\in C^\infty(\widehat{X}^{\widehat{\boxplus}})$
  has four independent components in $C^\infty(X)^{\Bbb C}[\vartheta, \bar{\vartheta}]^\anticommuting$:
  $$
    \breve{f}_{(0)}\,,\;\;  \breve{f}_{(\gamma)}\,,\,{}_{\gamma\;=\;1, 2}\,,\;\;
	\breve{f}_{(12)}\,.
  $$
 In terms of the standard chiral coordinate functions $(x^\prime, \theta, \bar{\theta}, \vartheta, \bar{\vartheta})$
  on $\widehat{X}^{\widehat{\boxplus}}$,
 $$
  \breve{f}\;
   =\; \breve{f}_{(0)}(x^\prime, \vartheta, \bar{\vartheta})
         + \sum_\gamma \theta^\gamma
		      \breve{f}_{(\gamma)}(x^\prime, \vartheta, \bar{\vartheta})
		 + \theta^1\theta^2
		     \breve{f}_{(12)}(x^\prime, \vartheta, \bar{\vartheta})\,,		
 $$
 which is independent of $\bar{\theta}$.
 
 $(2)$
   $\breve{f}\in C^\infty(\widehat{X}^{\widehat{\boxplus}})$ is antichiral if and only if,
   as an element of $C^\infty(X)^{\Bbb C}[\theta, \bar{\theta}, \vartheta, \bar{\theta}]^\anticommuting$,
  $\breve{f}$ is of the following form
  \begin{eqnarray*}
   \breve{f} & = &
     \breve{f}_{(0)}(x, \vartheta, \bar{\vartheta})
	 + \sum_{\dot{\delta}} \bar{\theta}^{\dot{\delta}}
	      \breve{f}_{(\dot{\delta})}(x, \vartheta, \bar{\theta})
	    + \bar{\theta}^1\bar{\theta}^2 \breve{f}_{(\dot{1}\dot{2})}(x,\vartheta, \bar{\vartheta})		     	
	    - \sqrt{-1} \sum_{\gamma, \dot{\delta}, \nu}
		    \theta^\gamma \bar{\theta}^{\dot{\delta}} \sigma^\nu_{\gamma\dot{\delta}}
			  \partial_\nu \breve{f}_{(0)}(x, \vartheta, \bar{\vartheta})
 		\\
	&& 			
        +\, \sqrt{-1} \sum_{\gamma, \dot{\delta}, \nu}
		      \theta^\gamma \bar{\theta}^{\dot{1}} \bar{\theta}^{\dot{2}}
			  \sigma^{\nu\dot{\delta}}_\gamma\,
			  \partial_\nu \breve{f}_{(\dot{\delta})}(x, \vartheta, \bar{\vartheta})
		- \theta^1\theta^2\bar{\theta}^{\dot{1}} \bar{\theta}^{\dot{2}}\,
		    \square \breve{f}_{(0)}(x,\vartheta, \bar{\vartheta})\,.
  \end{eqnarray*}
 In particular, an antichiral $\breve{f}\in C^\infty(\widehat{X}^{\widehat{\boxplus}})$
  has four independent components in $C^\infty(X)^{\Bbb C}[\vartheta, \bar{\vartheta}]^\anticommuting$:
  $$
    \breve{f}_{(0)}\,,\;\;
	\breve{f}_{(\dot{\delta})}\,,\,{}_{\dot{\delta}\;=\;\dot{1}, \dot{2}}\,,\;\;
	\breve{f}_{(\dot{1}\dot{2})}\,.
  $$
 In terms of the standard antichiral coordinate functions
  $(x^{\prime\prime}, \theta, \bar{\theta}, \vartheta, \bar{\vartheta})$
  on $\widehat{X}^{\widehat{\boxplus}}$,
 $$
  \breve{f}\;
   =\; \breve{f}_{(0)}(x^{\prime\prime}, \vartheta, \bar{\vartheta})
         + \sum_{\dot{\delta}} \bar{\theta}^{\dot{\delta}}
		      \breve{f}_{(\dot{\delta})}(x^{\prime\prime}, \vartheta, \bar{\vartheta})
		 + \bar{\theta}^{\dot{1}}\bar{\theta}^{\dot{2}}
		     \breve{f}_{(\dot{1}\dot{2})}(x^{\prime\prime}, \vartheta, \bar{\vartheta})\,,		
 $$
 which is independent of $\theta$.
\end{lemma}

\begin{proof}
 Given
  \begin{eqnarray*}
   \breve{f}& = &
     \breve{f}_{(0)}
	  + \sum_\alpha \theta^\gamma \breve{f}_{(\gamma)}
	  + \sum_{\dot{\delta}} \bar{\theta}^{\dot{\delta}} \breve{f}_{(\dot{\delta})}
	  + \theta^1\theta^2 \breve{f}_{(12)}
	  + \sum_{\gamma, \dot{\delta}} \theta^\gamma \bar{\theta}^{\dot{\delta}}
	       \breve{f}_{(\gamma\dot{\delta})}
	  + \bar{\theta}^{\dot{1}} \bar{\theta}^{\dot{2}}
	        \breve{f}_{(\dot{1}\dot{2})}
		 \\
		 &&\hspace{2em}
	  +\, \sum_{\dot{\delta}}	  \theta^1\theta^2 \bar{\theta}^{\dot{\delta}}
	       \breve{f}_{(12\dot{\delta})}
	  + \sum_\gamma \theta^\gamma \bar{\theta}^{\dot{1}} \bar{\theta}^{\dot{2}}
	       \breve{f}_{(\gamma\dot{1}\dot{2})}
	  + \theta^1\theta^2 \bar{\theta}^{\dot{1}} \bar{\theta}^{\dot{2}}
	       \breve{f}_{(12\dot{1}\dot{2})}\;\;\;\;
     \in\; C^\infty(\widehat{X}^{\widehat{\boxplus}})\,,
  \end{eqnarray*}
 where
  $\breve{f}_{(\tinybullet)}\in C^\infty(X)^{\Bbb C}[\vartheta, \bar{\vartheta}]^\anticommuting$,
 one has
 \begin{eqnarray*}
  -\, e_{\beta^{\prime\prime}}\breve{f}
    & = &
    \breve{f}_{(\dot{\beta})}	
    + \sum_\gamma \theta^\gamma
	    \mbox{\Large $($}
		 - \breve{f}_{(\gamma\dot{\beta})}
		 + \sqrt{-1}\sum_\nu
		      \sigma^\nu_{\gamma\dot{\beta}} \partial_\nu \breve{f}_{(0)}
		\mbox{\Large $)$}
    - \sum_{\dot{\delta}}  \bar{\theta}^{\dot{\delta}}
	    \varepsilon_{\dot{\beta}\dot{\delta}} \breve{f}_{(\dot{1}\dot{2})}
		\\
		&&
    +\, \theta^1\theta^2
	    \mbox{\Large $($}
		 \breve{f}_{(12\dot{\beta})}
		  - \sqrt{-1} \sum_{\gamma, \nu} \sigma^{\nu\gamma}_{\;\;\;\;\dot{\beta}}
		      \partial_\nu \breve{f}_{(\gamma)}		
		\mbox{\Large $)$}
	+ \sum_{\gamma, \dot{\delta}}\theta^\gamma \bar{\theta}^{\dot{\delta}}
	     \mbox{\Large $($}
		  \varepsilon_{\dot{\beta}\dot{\delta}} \breve{f}_{(\gamma\dot{1}\dot{2})}
		  + \sqrt{-1}\,\sum_\nu \sigma^\nu_{\gamma\dot{\beta}}
		      \partial_\nu \breve{f}_{(\dot{\delta})}
		 \mbox{\Large $)$}
		\\
		&&
	 -\, \sum_{\dot{\delta}} \theta^1\theta^2 \bar{\theta}^{\dot{\delta}}
	    \mbox{\Large $($}
		 \varepsilon_{\dot{\beta}\dot{\delta}} \breve{f}_{(12\dot{1}\dot{2})}
		 + \sqrt{-1}\,\sum_{\gamma, \nu} \sigma^{\nu\gamma}_{\;\;\;\;\dot{\beta}}
		     \partial_\nu \breve{f}_{(\gamma\dot{\delta})}
		\mbox{\Large $)$}
	 + \sqrt{-1}\, \sum_{\gamma, \nu} \theta^\gamma \bar{\theta}^{\dot{1}} \bar{\theta}^{\dot{2}}
	     \sigma^\nu_{\gamma\dot{\beta}} \partial_\nu \breve{f}_{(\dot{1}\dot{2})}
	    \\
		&& 		
	 -\, \sqrt{-1}\, \theta^1\theta^2 \bar{\theta}^{\dot{1}} \bar{\theta}^{\dot{2}}
	     \sum_{\gamma, \nu}\sigma^{\nu\gamma}_{\;\;\;\;\dot{\beta}}
		  \partial_\nu \breve{f}_{(\gamma\dot{1}\dot{2})}\,.	
 \end{eqnarray*}
 Thus,
  $e_{1^{\prime\prime}}\breve{f}= e_{2^{\prime\prime}}\breve{f}=0$
  if and only if
  \begin{eqnarray*}
   &
    \breve{f}_{(\dot{\beta})}\;
	=\; \breve{f}_{(\dot{1}\dot{2})}\;
	=\; \breve{f}_{(\gamma\dot{1}\dot{2})}\; =\; 0\,,  &
	\\[.6ex]
	&
	\breve{f}_{(\gamma\dot{\beta})}\;
	 =\; \sqrt{-1}\, \sum_\nu \sigma^\nu_{\gamma\dot{\beta}}
	        \partial_\nu \breve{f}_{(0)}\,, \hspace{2em}
    \breve{f}_{(12\dot{\beta})}\;
	 =\; \sqrt{-1}\, \sum_{\gamma^\prime, \nu} \sigma^{\nu\gamma^\prime}_{\;\;\;\;\dot{\beta}}
	        \partial_\nu \breve{f}_{(\gamma^\prime)}\,, &
	 \\[.6ex]
	 &
	\breve{f}_{(12\dot{1}\dot{2})}\;
	 =\; -\, \mbox{\Large $\frac{\sqrt{-1}}{2}$}
	        \sum_{\gamma^\prime, \dot{\delta}, \mu} \bar{\sigma}^{\mu\dot{\delta}\gamma^\prime}
			  \partial_\mu \breve{f}_{(\gamma^\prime\dot{\delta})}\,, &
  \end{eqnarray*}
 for $\gamma=1, 2$, $\dot{\beta}=\dot{1}, \dot{2}$.
 The last equation simplifies to
  $\breve{f}_{(12\dot{1}\dot{2})}= -\,\square\,\breve{f}_{(0)}$ after plugging in
   the equation
    $\breve{f}_{(\gamma^\prime\dot{\delta})}
	 = \sqrt{-1}\, \sum_\nu \sigma^\nu_{\gamma^\prime\dot{\delta}} \partial_\nu \breve{f}_{(0)}$
   in the system.
 This proves the first part of Statement (1).
 The second part of Statement (1) follows from the observation that
     \marginpar{\vspace{1em}\raggedright\tiny
         \raisebox{-1ex}{\hspace{-2.4em}\LARGE $\cdot$}Cf.\,[\,\parbox[t]{20em}{Wess
		\& Bagger:\\ Eq.\:(5.4)].}}
  $$
   \mbox{\Large $($}
     e_{\alpha^\prime}\;
	   =\; \mbox{\Large $\frac{\partial}{\partial \theta^\alpha}$}
	          + 2\sqrt{-1}\sum_{\dot{\beta},\mu}\sigma^\mu_{\alpha\dot{\beta}}
			       \bar{\theta}^{\dot{\beta}} \partial_\mu	
	 \hspace{2em}\mbox{and}\,
   \mbox{\Large $)$}\hspace{2em}
     e_{\beta^{\prime\prime}}\; = - \mbox{\large $\frac{\partial}{\partial\bar{\theta}^{\dot{\beta}}}$}
  $$
   in the coordinate system $(x^\prime, \theta, \bar{\theta}, \vartheta, \bar{\vartheta})$
   for $\widehat{X}^{\widehat{\boxplus}}$.

 Similar argument proves Statement (2).
 In which case,
     \marginpar{\vspace{1em}\raggedright\tiny
         \raisebox{-1ex}{\hspace{-2.4em}\LARGE $\cdot$}Cf.\,[\,\parbox[t]{20em}{Wess
		\& Bagger:\\ Eq.\:(5.6)].}}
  $$
     e_{\alpha^\prime}\; = \mbox{\large $\frac{\partial}{\partial \theta^{\alpha}}$}
	  \hspace{2em}
      \mbox{\Large $($}
	   \,\mbox{and}\hspace{2em}
	      e_{\beta^{\prime\prime}}\;
	  :=\; - \mbox{\Large $\frac{\partial}{\partial \bar{\theta}^{\dot{\beta}}}$}
	          - 2\,\sqrt{-1}\sum_{\alpha,\mu}
			      \theta^\alpha \sigma^\mu_{\alpha\dot{\beta}} \partial_\mu
      \mbox{\Large $)$}	
  $$
   in the coordinate system $(x^{\prime\prime}, \theta, \bar{\theta}, \vartheta, \bar{\vartheta})$
   for $\widehat{X}^{\widehat{\boxplus}}$.
 
\end{proof}

\bigskip

\begin{corollary} {\bf [tame chiral superfield, tame antichiral superfield]}\;
 $(1)$
 A tame superfield $\breve{f}$ on $X$ is chiral if and only if,
   as an element of $C^\infty(X)^{\Bbb C}[\theta, \bar{\theta}, \vartheta, \bar{\theta}]^\anticommuting$,
  $\breve{f}$ is of the following form
  \begin{eqnarray*}
   \breve{f} & = &
     f_{(0)}(x)
	 + \sum_\gamma \theta^\gamma \vartheta_\gamma f_{(\gamma)}(x)
	    + \theta^1\theta^2
		    \mbox{\large $($}
			 f^\prime_{(0)}(x)  + \vartheta_1\vartheta_2 f_{(12)}(x)
			\mbox{\large $)$}
			\\
	&& \hspace{2em}
	    +\, \sqrt{-1} \sum_{\gamma, \dot{\delta}, \nu}
		    \theta^\gamma \bar{\theta}^{\dot{\delta}} \sigma^\nu_{\gamma\dot{\delta}}
			  \partial_\nu f_{(0)}(x)
        + \sqrt{-1} \sum_{\dot{\delta}, \gamma, \nu}
		     \theta^1\theta^2 \bar{\theta}^{\dot{\delta}} \vartheta_\gamma
			  \sigma^{\nu\gamma}_{\;\;\;\;\dot{\delta}}\,\partial_\nu f_{(\gamma)}(x)
		- \theta^1\theta^2\bar{\theta}^{\dot{1}} \bar{\theta}^{\dot{2}}\,
		    \square f_{(0)}(x)\,.               		
  \end{eqnarray*}
 In particular, a chiral $\breve{f}\in C^\infty(\widehat{X}^{\widehat{\boxplus}})^\tamescriptsize$
  has five independent components in $C^\infty(X)^{\Bbb C}$:
  $$
    f_{(0)}\,,\;\;  f_{(\gamma)}\,,\,{}_{\gamma\;=\;1, 2}\,,\;\; f^\prime_{(0)}\,,\;\; f_{(12)}\,.
  $$
 In terms of the standard chiral coordinate functions $(x^\prime, \theta, \bar{\theta}, \vartheta, \bar{\vartheta})$
  on $\widehat{X}^{\widehat{\boxplus}}$,
 $$
  \breve{f}\;
   =\; f_{(0)}(x^\prime)
         + \sum_\gamma \theta^\gamma \vartheta_\gamma f_{(\gamma)}(x^\prime)
		 + \theta^1\theta^2
		      \mbox{\large $($}
			   f^\prime_{(0)}(x^\prime)
			   + \vartheta_1\vartheta_2 f_{(12)}(x^\prime)
			  \mbox{\large $)$}\,,
 $$
 which is independent of $\bar{\theta}$ and $\bar{\vartheta}$.
 
 $(2)$
 A tame superfield $\breve{f}$ on $X$ is antichiral if and only if,
   as an element of $C^\infty(X)^{\Bbb C}[\theta, \bar{\theta}, \vartheta, \bar{\theta}]^\anticommuting$,
  $\breve{f}$ is of the following form
  \begin{eqnarray*}
   \breve{f} & = &
     f_{(0)}(x)
	 + \sum_{\dot{\delta}} \bar{\theta}^{\dot{\delta}} \bar{\vartheta}_{\dot{\delta}}
	       f_{(\dot{\delta})}(x)
	    + \bar{\theta}^{\dot{1}}\bar{\theta}^{\dot{2}}
		    \mbox{\large $($}
			 f^{\prime\prime}_{(0)}(x)
			 + \bar{\vartheta}_{\dot{1}}\bar{\vartheta}_{\dot{2}}
			       f_{(\dot{1}\dot{2})}(x)
			\mbox{\large $)$}
			\\
	&& \hspace{2em}
	    -\, \sqrt{-1} \sum_{\gamma, \dot{\delta}, \nu}
		    \theta^\gamma \bar{\theta}^{\dot{\delta}} \sigma^\nu_{\gamma\dot{\delta}}
			  \partial_\nu f_{(0)}(x)
        + \sqrt{-1} \sum_{\dot{\delta}, \gamma, \nu}
		     \theta^\gamma \bar{\theta}^{\dot{1}} \bar{\theta}^{\dot{2}}
			  \bar{\vartheta}_{\dot{\delta}}
			  \sigma^{\nu\dot{\delta}}_\gamma\,\partial_\nu f_{(\dot{\delta})}(x)
		- \theta^1\theta^2\bar{\theta}^{\dot{1}} \bar{\theta}^{\dot{2}}\,
		    \square f_{(0)}(x)\,.               		
  \end{eqnarray*}
 In particular, an antichiral $\breve{f}\in C^\infty(\widehat{X}^{\widehat{\boxplus}})^\tamescriptsize$
  has five independent components in $C^\infty(X)^{\Bbb C}$:
  $$
    f_{(0)}\,,\;\;  f_{(\dot{\delta})}\,,\,{}_{\dot{\delta}\;=\;\dot{1}, \dot{2}}\,,\;\;
	f^{\prime\prime}_{(0)}\,,\;\; f_{(\dot{1}\dot{2})}\,.
  $$
 In terms of the standard antichiral coordinate functions
   $(x^{\prime\prime}, \theta, \bar{\theta}, \vartheta, \bar{\vartheta})$
  on $\widehat{X}^{\widehat{\boxplus}}$,
 $$
  \breve{f}\;
   =\; f_{(0)}(x^{\prime\prime})
         + \sum_{\dot{\delta}} \bar{\theta}^{\dot{\delta}}
		      \bar{\vartheta}_{\dot{\delta}} f_{(\dot{\delta})}(x^{\prime\prime})
		 + \bar{\theta}^{\dot{1}} \bar{\theta}^{\dot{2}}
		      \mbox{\large $($}
			   f^{\prime\prime}_{(0)}(x^{\prime\prime})
			   + \bar{\vartheta}_{\dot{1}}\bar{\vartheta}_{\dot{2}}
			        f_{(\dot{1}\dot{2})}(x^{\prime\prime})
			  \mbox{\large $)$}\,,
 $$
 which is independent of $\theta$ and $\vartheta$.
\end{corollary}

\bigskip

\begin{corollary} {\bf [medium chiral = small chiral, medium antichiral = small antichiral]}\;
 $(1)$
 A medium chiral superfield coincides with a small chiral superfield.
 A medium or small superfield $\breve{f}$ on $X$ is chiral if and only if,
   as an element of $C^\infty(X)^{\Bbb C}[\theta, \bar{\theta}, \vartheta, \bar{\theta}]^\anticommuting$,
  $\breve{f}$ is of the following form
     \marginpar{\vspace{2em}\raggedright\tiny
         \raisebox{-1ex}{\hspace{-2.4em}\LARGE $\cdot$}Cf.\,[\,\parbox[t]{20em}{Wess
		\& Bagger:\\ Eq.\:(5.3)].}}
  \begin{eqnarray*}
   \breve{f} & = &
     f_{(0)}(x)
	 + \sum_\gamma \theta^\gamma \vartheta_\gamma f_{(\gamma)}(x)
	    + \theta^1\theta^2
		     \vartheta_1\vartheta_2 f_{(12)}(x)						
			\\
	&& \hspace{2em}
	    +\, \sqrt{-1} \sum_{\gamma, \dot{\delta}, \nu}
		    \theta^\gamma \bar{\theta}^{\dot{\delta}} \sigma^\nu_{\gamma\dot{\delta}}
			  \partial_\nu f_{(0)}(x)
        + \sqrt{-1} \sum_{\dot{\delta}, \gamma, \nu}
		     \theta^1\theta^2 \bar{\theta}^{\dot{\delta}} \vartheta_\gamma
			  \sigma^{\nu\gamma}_{\;\;\;\;\dot{\delta}}\,\partial_\nu f_{(\gamma)}(x)
		- \theta^1\theta^2\bar{\theta}^{\dot{1}} \bar{\theta}^{\dot{2}}\,
		    \square f_{(0)}(x)\,.               		
  \end{eqnarray*}
 In particular, a chiral $\breve{f}\in C^\infty(\widehat{X}^{\widehat{\boxplus}})^\mediumscriptsize$
   or\, $C^\infty(\widehat{X}^{\widehat{\boxplus}})^\smallscriptsize$
  has four independent components in $C^\infty(X)^{\Bbb C}$:
  $$
    f_{(0)}\,,\;\;  f_{(\gamma)}\,,\,{}_{\gamma\;=\;1, 2}\,,\;\; f_{(12)}\,.
  $$
 In terms of the standard chiral coordinate functions $(x^\prime, \theta, \bar{\theta}, \vartheta, \bar{\vartheta})$
  on $\widehat{X}^{\widehat{\boxplus}}$,
     \marginpar{\vspace{0em}\raggedright\tiny
         \raisebox{-1ex}{\hspace{-2.4em}\LARGE $\cdot$}Cf.\,[\,\parbox[t]{20em}{Wess
		\& Bagger:\\ Eq.\:(5.3)].}}
 $$
  \breve{f}\;
   =\; f_{(0)}(x^\prime)
         + \sum_\gamma \theta^\gamma \vartheta_\gamma f_{(\gamma)}(x^\prime)
		 + \theta^1\theta^2 \vartheta_1\vartheta_2 f_{(12)}(x^\prime)\,,			
 $$
 which is independent of $\bar{\theta}$ and $\bar{\vartheta}$.
 
 $(2)$
 A medium antichiral superfield coincides with a small antichiral superfield.
 A medium or small superfield $\breve{f}$ on $X$ is antichiral if and only if,
   as an element of $C^\infty(X)^{\Bbb C}[\theta, \bar{\theta}, \vartheta, \bar{\theta}]^\anticommuting$,
  $\breve{f}$ is of the following form
     \marginpar{\vspace{2em}\raggedright\tiny
         \raisebox{-1ex}{\hspace{-2.4em}\LARGE $\cdot$}Cf.\,[\,\parbox[t]{20em}{Wess
		\& Bagger:\\ Eq.\:(5.5)].}}
  \begin{eqnarray*}
   \breve{f} & = &
     f_{(0)}(x)
	 + \sum_{\dot{\delta}} \bar{\theta}^{\dot{\delta}} \bar{\vartheta}_{\dot{\delta}}
	       f_{(\dot{\delta})}(x)
	    + \bar{\theta}^{\dot{1}}\bar{\theta}^{\dot{2}} 		
			 \bar{\vartheta}_{\dot{1}}\bar{\vartheta}_{\dot{2}}  f_{(\dot{1}\dot{2})}(x)			
			\\
	&& \hspace{2em}
	    -\, \sqrt{-1} \sum_{\gamma, \dot{\delta}, \nu}
		    \theta^\gamma \bar{\theta}^{\dot{\delta}} \sigma^\nu_{\gamma\dot{\delta}}
			  \partial_\nu f_{(0)}(x)
        + \sqrt{-1} \sum_{\dot{\delta}, \gamma, \nu}
		     \theta^\gamma \bar{\theta}^{\dot{1}} \bar{\theta}^{\dot{2}}
			  \bar{\vartheta}_{\dot{\delta}}
			  \sigma^{\nu\dot{\delta}}_\gamma\,\partial_\nu f_{(\dot{\delta})}(x)
		- \theta^1\theta^2\bar{\theta}^{\dot{1}} \bar{\theta}^{\dot{2}}\,
		    \square f_{(0)}(x)\,.               		
  \end{eqnarray*}
 In particular, an antichiral $\breve{f}\in C^\infty(\widehat{X}^{\widehat{\boxplus}})^\mediumscriptsize$
  or\, $C^\infty(\widehat{X}^{\widehat{\boxplus}})^\smallscriptsize$
  has four independent components in $C^\infty(X)^{\Bbb C}$:
  $$
    f_{(0)}\,,\;\;  f_{(\dot{\delta})}\,,\,{}_{\dot{\delta}\;=\;\dot{1}, \dot{2}}\,,\;\;
	f_{(\dot{1}\dot{2})}\,.
  $$
 In terms of the standard antichiral coordinate functions
   $(x^{\prime\prime}, \theta, \bar{\theta}, \vartheta, \bar{\vartheta})$
  on $\widehat{X}^{\widehat{\boxplus}}$,
     \marginpar{\vspace{0em}\raggedright\tiny
         \raisebox{-1ex}{\hspace{-2.4em}\LARGE $\cdot$}Cf.\,[\,\parbox[t]{20em}{Wess
		\& Bagger:\\ Eq.\:(5.5)].}}
 $$
  \breve{f}\;
   =\; f_{(0)}(x^{\prime\prime})
         + \sum_{\dot{\delta}} \bar{\theta}^{\dot{\delta}}
		      \bar{\vartheta}_{\dot{\delta}} f_{(\dot{\delta})}(x^{\prime\prime})
		 + \bar{\theta}^{\dot{1}} \bar{\theta}^{\dot{2}} 		
			   \bar{\vartheta}_{\dot{1}}\bar{\vartheta}_{\dot{2}}
			        f_{(\dot{1}\dot{2})}(x^{\prime\prime})\,,
 $$
 which is independent of $\theta$ and $\vartheta$.
\end{corollary}
\bigskip

\medskip

\begin{corollary} {\bf [medium or small chiral = chiral multiplet]}
 A medium or small chiral superfield matches with the chiral multiplet from representations of $d=3+1$, $N=1$ supersymmetry algebra.
 Medium or small chiral superfields on $X$ form a complexified $C^\infty$-ring.
 Similar statements hold for medium or small antichiral superfields.
\end{corollary}

\medskip

\begin{corollary} {\bf [small ring generated by small chiral and small antichiral]}\;
 The small function ring $C^\infty(\widehat{X}^{\widehat{\boxplus}})^\smallscriptsize$
   of $\widehat{X}^{\widehat{\boxplus}}$
  is the subring of $C^\infty(\widehat{X}^{\widehat{\boxplus}})$
  generated by the small chiral superfields and the small antichiral superfields.\footnote{In
                                                                [L-Y5] (SUSY(1)), we take this as a starting point to understand
															     how physicists think of the function ring of a superspace
																   and its ``more physically relevant sector".
                                                                   }
\end{corollary}
  
\medskip

\begin{remark} $[$mathematical patch to {\sl [Wess \& Bagger]}: nilpotency of `$F$-component'\,$]$\; {\rm
 Caution that chiral functions on $\widehat{X}^{\widehat{\boxplus}}$ of the following form
  $$
  \breve{f}\;
   =\; f_{(0)}(x^\prime)
         + \sum_\gamma \theta^\gamma \vartheta_\gamma f_{(\gamma)}(x^\prime)
		 + \theta^1\theta^2 f^\prime_{(0)}(x^\prime)
 $$
  in coordinate functions $(x^\prime, \theta, \bar{\theta}, \vartheta, \bar{\vartheta})$
 do not form a ring as they are not closed under multiplication:
 The product $\breve{A}\breve{B}$ of general $\breve{A}$ and $\breve{B}$ of the above form
  will have an additional nilpotent summand
   $$
    -\,\theta^1\theta^2\vartheta_1\vartheta_2 (A_{(1)}B_{(2)} + A_{(2)}B_{(1)})
   $$
  that does not fit in.
 It follows that if one demands that chiral superfields that match with the chiral multiplets
   (which have only four independent components) to form a ring,
  then their `$F$-component' (cf.\:[Wess \& Bagger: Eq.\:(5.3)]) must be {\it nilpotent} Lorentz-scalar of the form
		$\vartheta_1\vartheta_2 f_{(12)}$,
   where $f_{(12)}\in C^\infty(X)^{\Bbb C}$.
 This is the first mathematical patch that one is forced to make if one wants to bring [Wess \& Bagger]
    in complexified ${\Bbb Z}/2$-graded Algebraic Geometry;
	([L-Y5: Sec.\:1.2] (SUSY(1))).\footnote{\makebox[9.2em][l]{\it Note for physicists}                              						
							   It turns that, as demonstrated in the current notes,
							    except the additional introduction of purge-evaluation maps
							    (which physicists already took implicitly whenever in need),
							   this is the only mathematical patch to make precise (of the non-quantum part of) [Wess \& Bagger].
                               }
							
 Similarly for antichiral functions on  $\widehat{X}^{\widehat{\boxplus}}$ of the following form
 $$
  \breve{f}\;
   =\; f_{(0)}(x^{\prime\prime})
         + \sum_{\dot{\delta}} \bar{\theta}^{\dot{\delta}}
		      \bar{\vartheta}_{\dot{\delta}} f_{(\dot{\delta})}(x^{\prime\prime})
		 + \bar{\theta}^{\dot{1}} \bar{\theta}^{\dot{2}}
			   f^{\prime\prime}_{(0)}(x^{\prime\prime}) 			   			
 $$
  in coordinate functions $(x^{\prime\prime}, \theta, \bar{\theta}, \vartheta, \bar{\vartheta})$.
}\end{remark}

\bigskip

\begin{flushleft}
{\bf Chirality on $C^\infty(X^{\widehat{\boxplus}})$ and its sectors
           under the twisted complex conjugation}
\end{flushleft}
\begin{lemma}  {\bf [$Q_\alpha$, $\bar{Q}_{\dot{\beta}}$, $e_{\alpha^\prime}$, $e_{\beta^{\prime\prime}}$
   under twisted complex conjugation]}\;
 Under the twisted complex conjugation ${}^\dag$ on
   $C^\infty(\widehat{X}^{\widehat{\boxplus}})
     = C^\infty(\widehat{X}^{\widehat{\boxplus}})_\even+ C^\infty(\widehat{X}^{\widehat{\boxplus}})_\odd$,
  $$
    Q_\alpha^\dag \; =\;  \bar{Q}_{\dot{\alpha}}\,,\hspace{2em}
    \bar{Q}_{\dot{\beta}}^\dag \; =\;   Q_\beta\,,\hspace{2em}
    e_{\alpha^\prime}^\dag \; =\;   e_{\alpha^{\prime\prime}}\,,\hspace{2em}
    e_{\beta^{\prime\prime}}^\dag \; =\;   e_{\beta^\prime}
  $$
 on $C^\infty(\widehat{X}^{\widehat{\boxplus}})_\even$
 in the sense that
 $$
   \mbox{\large $($}Q_\alpha \breve{f}\mbox{\large $)$}^\dag\;
     =\;  \bar{Q}_{\dot{\alpha}} \breve{f}^\dag\,,\;\;\;\;
   \mbox{\large $($}\bar{Q}_{\dot{\beta}} \breve{f}\mbox{\large $)$}^\dag
    =  Q_\beta \breve{f}^\dag\,,\;\;\;\; 	
  \mbox{\large $($}e_{\alpha^\prime} \breve{f}\mbox{\large $)$}^\dag
    = e_{\alpha^{\prime\prime}} \breve{f}^\dag\,,\;\;\;\;	
  \mbox{\large $($}e_{\beta^{\prime\prime}} \breve{f}\mbox{\large $)$}^\dag
    =  e_{\beta^\prime} \breve{f}^\dag
 $$
 for $\breve{f}\in C^\infty(\widehat{X}^{\widehat{\boxplus}})_\even$.
Similarly,
 $$
   Q_\alpha^\dag \; =\;  -\, \bar{Q}_{\dot{\alpha}}\,,\hspace{2em}
    \bar{Q}_{\dot{\beta}}^\dag \; =\;   -\, Q_\beta\,,\hspace{2em}
    e_{\alpha^\prime}^\dag \; =\;   -\, e_{\alpha^{\prime\prime}}\,,\hspace{2em}
    e_{\beta^{\prime\prime}}^\dag \; =\;  -\, e_{\beta^\prime}
 $$
 on $C^\infty(\widehat{X}^{\widehat{\boxplus}})_\odd$
 in the sense that
 $$
   \mbox{\large $($}Q_\alpha \breve{f}\mbox{\large $)$}^\dag\;
     =\;  -\, \bar{Q}_{\dot{\alpha}} \breve{f}^\dag\,,\;\;\;\;	
   \mbox{\large $($}\bar{Q}_{\dot{\beta}} \breve{f}\mbox{\large $)$}^\dag\;
    =\;   -\, Q_\beta \breve{f}^\dag\,,\;\;\;\; 	
  \mbox{\large $($}e_{\alpha^\prime} \breve{f}\mbox{\large $)$}^\dag\;
    =\;  -\, e_{\alpha^{\prime\prime}} \breve{f}^\dag\,,\;\;\;\;	
  \mbox{\large $($}e_{\beta^{\prime\prime}} \breve{f}\mbox{\large $)$}^\dag\;
    =\;  -\, e_{\beta^\prime} \breve{f}^\dag
 $$
 for $\breve{f}\in C^\infty(\widehat{X}^{\widehat{\boxplus}})_\odd$.
\end{lemma}

\medskip

\begin{proof}
 This follows from the following basic identities
  \begin{eqnarray*}
   &
      \mbox{\Large $($}
	   \mbox{\Large $\frac{\partial}{\partial\theta^\alpha}$}
	    \mbox{\large $($}
		 \theta^\alpha \breve{A}
		\mbox{\large $)$}
	  \mbox{\Large $)$}^\dag\;
	 =\; (-1)^{p(\breve{A})}\,\mbox{\Large $\frac{\partial}{\partial \bar{\theta}^{\dot{\alpha}}}$}	       
		   \mbox{\Large $($}
	         \mbox{\large $($}\theta^\alpha \breve{A}\mbox{\large $)$}^\dag
	  \mbox{\Large $)$}\,,\hspace{2em}
	   \mbox{\Large $($}
	   \mbox{\Large $\frac{\partial}{\partial\bar{\theta}^{\dot{\beta}}}$}
	    \mbox{\large $($}
		 \bar{\theta}^{\dot{\beta}} \breve{A}
		\mbox{\large $)$}
	  \mbox{\Large $)$}^\dag\;
	 =\; (-1)^{p(\breve{A})}\,\mbox{\Large $\frac{\partial}{\partial \theta^\beta}$}	
		   \mbox{\Large $($}
	         \mbox{\large $($}\bar{\theta}^{\dot{\beta}} \breve{A}\mbox{\large $)$}^\dag
	  \mbox{\Large $)$}\,,	
   & \\[.6ex]
   &
    \mbox{\Large $($}\partial_\mu \breve{B}\mbox{\Large $)$}^\dag \;=\; \partial_\mu\breve{B}^\dag\,,
	  \hspace{2em}\overline{\sigma^\mu_{\alpha\dot{\beta}}}\;=\; \sigma^\mu_{\beta\dot{\alpha}}\,,
   & \\[.6ex]
   &	
     \mbox{\Large $($}\theta^\alpha\breve{B} \mbox{\Large $)$}^\dag\;
	  =\; (-1)^{p(\breve{B})}\, \bar{\theta}^{\dot{\alpha}}\breve{B}^\dag\,,\hspace{2em}
     \mbox{\Large $($}\bar{\theta}^{\dot{\beta}}\breve{B} \mbox{\Large $)$}^\dag\;
	  =\; (-1)^{p(\breve{B})}\, \theta^\beta \breve{B}^\dag\,,
   &
  \end{eqnarray*}
  for $\breve{A}$, $\breve{B}$ parity-homogeneous elements of $C^\infty(\widehat{X}^{\widehat{\boxplus}})$.

\end{proof}

\bigskip

Since the twisted complex conjugation leaves
  each of the tame, the medium, and the small sectors of $\widehat{X}^{\widehat{\boxplus}}$
 invariant,
 one has the following consequence:

\bigskip
 
\begin{corollary} {\bf [swapping of `chiral' and `antichiral' under twisted complex conjugation]}\;
 {\rm (Cf.\:[L-Y5: Lemma 2.1.7] (SUSY(1)).)}			
 The twisted complex conjugation ${}^\dag$ on $C^\infty(\widehat{X}^{\widehat{\boxplus}})$
   takes chiral elements to antichiral elements, and vice versa.
 The same holds for the restriction of ${}^\dag$ on
     the tame subring $C^\infty(\widehat{X}^{\widehat{\boxplus}})^\tamescriptsize$,
     the medium subring $C^\infty(\widehat{X}^{\widehat{\boxplus}})^\mediumscriptsize$,   and
     the small subring $C^\infty(\widehat{X}^{\widehat{\boxplus}})^\smallscriptsize$
   of $C^\infty(\widehat{X}^{\widehat{\boxplus}})$.
\end{corollary}

\bigskip

\section{Purge-evaluation maps and
  the Fundamental Theorem on supersymmetric action functionals via a superspace formulation}

We highlight [L-Y5: Sec.\:1.5] (SUSY(1))
  on the purge-evaluation maps and
  the Fundamental Theorem on supersymmetric action functionals via a superspace formulation,
 with two additional remarks.

\bigskip

\begin{flushleft}
{\bf Purge-evaluation maps}
\end{flushleft}
The action functional of a supersymmetric quantum field theory model via a superspace formulation is
 itself a (Lorentz-)scalar superfield
 $\in C^\infty(\widehat{X}^{\widehat{\boxplus}})
    = C^\infty(X)^{\Bbb C}[\theta, \bar{\theta}, \vartheta, \bar{\vartheta}]^\anticommuting$.
After taking an appropriate fermionic integration:
 $$
  \int d\bar{\theta}^{\dot{2}} d\bar{\theta}^{\dot{1}}  d\theta^ 2 d\theta^1\,,
    \hspace{2em}\mbox{or}\hspace{2em}
  \int d\theta^ 2 d\theta^1\,,
    \hspace{2em}\mbox{or}\hspace{2em}
  \int d\bar{\theta}^{\dot{2}} d\bar{\theta}^{\dot{1}}\,,
 $$
it becomes an element in $C^\infty(X)^{\Bbb C}[\vartheta, \bar{\vartheta}]^\anticommuting$.
In general, a summand of which may still contain an even nilpotent factor of the form
  $\vartheta_1\vartheta_2$,
  $\vartheta_\alpha\bar{\vartheta}_{\dot{\beta}}$, ${}_{\alpha=1,2, \dot{\beta}=\dot{1}, \dot{2}}$,
  $\bar{\vartheta}_{\dot{1}}\bar{\vartheta}_{\dot{2}}$, or
  $\vartheta_1\vartheta_2\bar{\vartheta}_{\dot{1}}\bar{\vartheta}_{\dot{2}}$.
Such a nilpotent factor has to be ``removed" properly to obtain a final supersymmetric action functional of fields on $X$.
This removal is realized by a purge-evaluation:

\bigskip

\begin{sdefinition} {\bf [purge-evaluation map $\Pev$]}\: {\rm
 A {\it purge-evaluation map}
   $$
     \Pev\;:\; C^\infty(X)^{\Bbb C}[\vartheta, \bar{\vartheta}]^\anticommuting\;
       \longrightarrow\; C^\infty(X)^{\Bbb C}
   $$
   is a $C^\infty(X)^{\Bbb C}$-module homomorphism
   that sends monomials
   $\vartheta_1^{d_1}\vartheta_2^{d_2}
     \bar{\vartheta}_{\dot{1}}^{d_{\dot{1}}}\bar{\vartheta}_{\dot{2}}^{d_{\dot{2}}}$
   to a constant $c_{d_1d_2d_{\dot{1}}d_{\dot{2}}}\in {\Bbb C}$.
 This induces a map, also denoted by $\Pev$ and called {\it purge-evaluation map}
  $$
   \Pev\;:\; C^\infty(\widehat{X}^{\widehat{\boxplus}})\;
     \longrightarrow\; C^\infty(\widehat{X})\,.
  $$
}\end{sdefinition}

\bigskip

\noindent
(Cf.\: [L-Y5: Definition 1.5.1] (SUSY(1)) for a slightly different formulation.)

Any $\Pev$  thus defined has the following property:\;
 (Cf.\:[L-Y5: Lemma 1.5.2] (SUSY(1)).)

\bigskip

\begin{slemma} {\bf [property of $\Pev$]}\;
  $(1)$
  Let $\xi\in \Der_{\Bbb C}(\widehat{X})$ be a derivation on $\widehat{X}$.
  Then for $\breve{f}\in C^\infty(\widehat{X}^{\widehat{\boxplus}})$,
	$$
      {\cal P}(\xi\breve{f})= \xi {\cal P}(\breve{f})\,.
    $$
 $(2)$
  For $\breve{f}\in C^\infty(\widehat{X}^{\widehat{\boxplus}})$,
  $$
     \int\! d\bar{\theta}^{\dot{2}}d\bar{\theta}^{\dot{1}} d\theta^2d\theta^1 \Pev(\breve{f})\;		
	    =\;  \Pev(\int\! d\bar{\theta}^{\dot{2}}d\bar{\theta}^{\dot{1}} d\theta^2d\theta^1\breve{f}\,)\,.
  $$
 $(3)$
  For $\breve{f}\in C^\infty(\widehat{X}^{\widehat{\boxplus}})$ chiral (resp.\ antichiral),
    $$
	   \int\! d\theta^2d\theta^1  \Pev(\breve{f})\;	 =\; \Pev(\int\! d\theta^2d\theta^1 \breve{f}\,)
	  \hspace{2em}
      \mbox{(resp.}\;\;
	   \int\! d\bar{\theta}^{\dot{2}}d\bar{\theta}^{\dot{1}} \Pev(\breve{f})\;	
		=\; \Pev(\int\! d\bar{\theta}^{\dot{2}}d\bar{\theta}^{\dot{1}} \breve{f}\,)\;\;
     \mbox{)}\,.
   $$	
\end{slemma}

\medskip

\begin{proof}
 Statement (1) follows from
    the fact that $\xi\in \Der_{\Bbb C}(\widehat{X})$
    has no $(\vartheta, \bar{\vartheta})$-dependence
    while $\Pev$ applies to $\breve{f}$
	    $(\theta,\bar{\theta})$-degree-by-$(\theta,\bar{\theta})$-degree
      with $\Pev(\vartheta_1^{d_1}\vartheta_2^{d_2}
				      \bar{\vartheta}_{\dot{1}}^{d_{\dot{1}}}
					  \bar{\vartheta}_{\dot{2}}^{d_{\dot{2}}})$ constant,
   and hence $\xi(\Pev(\breve{f}_{(\tinybullet)}))
                          = \Pev(\xi\, \breve{f}_{(\tinybullet)})$,
	where $\breve{f}_{(\tinybullet)}\in C^\infty(X)^{\Bbb C}[\vartheta, \bar{\vartheta}]^\anticommuting$
	is a component of $\breve{f}$ as a polynomial in $(\theta, \bar{\theta})$.					
  Statement (2) and Statement (3) follow from the definition of fermionic integration on $\widehat{X}$.
  
\end{proof}

\bigskip

\begin{flushleft}
{\bf The Fundamental Theorem on supersymmetric action functionals}
\end{flushleft}
The following Fundamental Theorem is truly amazing and yet so elegantly simple!
It is {\it the} tool physicists use for the construction of  supersymmetric quantum field theory models via superspace.
Readers are referred to, e.g.\
  [Bi: Sec.\:4.3] and [B-T-T: Sec.\:2.9]
for the original account by physicists.

\bigskip

\begin{stheorem} {\bf [fundamental: supersymmetric functional]}\;
 Up to a boundary term\footnote{Though
                                                    ignored in the current work, it should be noted that
													such a boundary term becomes an important part of understanding
													 when one studies supersymmetric quantum field theory with boundary.
                                                       } 
  on $X$,\\
 $(1)$
    $$
	 S_1(\breve{f})\;
	  :=\;  \int_{\widehat{X}}d^4x\,
	             d\bar{\theta}^{\dot{2}}d\bar{\theta}^{\dot{1}} d\theta^2 d\theta^1
			   \breve{f}
	$$
      is a functional on $C^\infty(\widehat{X}^{\widehat{\boxplus}})$ that is invariant under supersymmetries;\\
 $(2)$
    $$
	   S_2(\breve{f})\;
	   :=\;  \int_{\widehat{X}}d^4x\, d\theta^2 d\theta^1   \breve{f}  \hspace{2em}	
	 \mbox{(resp.}\;\;
	    S_3(\breve{f})\;
		   :=\; \int_{\widehat{X}}d^4x\, d\bar{\theta}^{\dot{2}} d\bar{\theta}^{\dot{1}}
		            \breve{f} \;\;\mbox{)}
	$$
	 is a functional on   $C^\infty(\widehat{X}^{\widehat{\boxplus}})^{\scriptsizech}$
    (resp.\ $C^\infty(\widehat{X}^{\widehat{\boxplus}})^{\scriptsizeach}$)
	  that is invariant under supersymmetries.

 Let $\Pev$ be a purge-evaluation map.
 Then, up to a boundary term on $X$,\\
 $(1^\prime)$
    $$
	 S^\prime_1(\breve{f})\;
	  :=\;  \int_{\widehat{X}}d^4x\,
	             d\bar{\theta}^{\dot{2}}d\bar{\theta}^{\dot{1}} d\theta^2 d\theta^1
			   \Pev(\breve{f})
	$$
      is a functional on $C^\infty(\widehat{X}^{\widehat{\boxplus}})$ that is invariant under supersymmetries;\\
 $(2^\prime)$
    $$
	   S^\prime_2(\breve{f})\;
	   :=\;  \int_{\widehat{X}}d^4x\, d\theta^2 d\theta^1  \Pev(\breve{f})  \hspace{2em}	
	 \mbox{(resp.}\;\;
	    S^\prime_3(\breve{f})\;
		   :=\; \int_{\widehat{X}}d^4x\, d\bar{\theta}^{\dot{2}} d\bar{\theta}^{\dot{1}}
		            \Pev(\breve{f})    \;\;\mbox{)}
	$$
	 is a functional on   $C^\infty(\widehat{X}^{\widehat{\boxplus}})^{\scriptsizech}$
    (resp.\ $C^\infty(\widehat{X}^{\widehat{\boxplus}})^{\scriptsizeach}$)
	  that is invariant under supersymmetries.
\end{stheorem}		

\medskip

\begin{proof}
 We give the proof  of Statement $(1^\prime)$ and Statement $(2^\prime)$.
 (The proof is identical to that of [L-Y5: Theorem 1.5.3] (SUSY(1))
      and is repeated here due to that the Theorem is so fundamental.)
 A similar, simpler argument proves Statement (1) and Statement (2).	
	
 For Statement $(1^\prime)$,
   since $Q_{\alpha}, \bar{Q}_{\dot{\beta}}\in \Der_{\Bbb C}(\widehat{X})$,
   it follows from
     the invariance of
	     $d^4x\,d\bar{\theta}^{\dot{2}}d\bar{\theta}^{\dot{1}}d\theta^2 d\theta^1$,
         $d^4x d\theta^2 d\theta^1$,  and
		 $d^4x d\bar{\theta}^{\dot{2}} d\bar{\theta}^{\dot{1}}$
		 under the flow that generates supersymmetries,
     Lemma~2.2, and basic calculus
   that
   \begin{eqnarray*}
     \delta_{Q_\alpha}S^\prime_1(\breve{f})
	   & :=\: & \int_{\widehat{X}}d^4x\,
	                         d\bar{\theta}^{\dot{2}} d\bar{\theta}^{\dot{1}} d\theta^2 d\theta^1\,
	                       \Pev(Q_\alpha\breve{f})\;
			    =\;  \int_{\widehat{X}}d^4x\,
				        d\bar{\theta}^{\dot{2}} d\bar{\theta}^{\dot{1}}  d\theta^2 d\theta^1\,
	                         Q_{\alpha}\Pev(\breve{f})\\
       & = & 	-\sqrt{-1}\,
	               \int_{\widehat{X}}d^4x\,
				     d\bar{\theta}^{\dot{2}} d\bar{\theta}^{\dot{1}} d\theta^2 d\theta^1\,
	                \sum_{\dot{\beta},\,\mu}
					  \sigma^\mu_{\alpha\dot{\beta}}\bar{\theta}^{\dot{\beta}}
					  \partial_\mu (\Pev(\breve{f}))		\\
       & = & -\sqrt{-1}\,
                  \int_X	d^4x \sum_{\mu}\partial_\mu
                    \left(\rule{0ex}{1.2em}\right.				
					\int d\bar{\theta}^{\dot{2}} d\bar{\theta}^{\dot{1}}d\theta^2 d\theta^1\,
					   \sum_{\dot{\beta}}
					      \sigma^\mu_{\alpha\dot{\beta}}\bar{\theta}^{\dot{\beta}}
						   \Pev(\breve{f})
					\left.\rule{0ex}{1.2em}\right)\;\;
				=\;\; -\sqrt{-1}\int_X dB_{\alpha}\,,
   \end{eqnarray*}
   where
     $B_{\alpha}= B_\alpha^0 dx^1\wedge dx^2\wedge dx^3
	                                   - B_\alpha^1 dx^0\wedge dx^2\wedge dx^3
									  + B_\alpha^2 dx^0\wedge dx^1\wedge dx^3
									   - B_\alpha^3 dx^0\wedge dx^1\wedge dx^2$
      is a $3$-form on $X$ with
    $$
	  B_\alpha^\mu \;
	    =\;  \int d\bar{\theta}^{\dot{2}} d\bar{\theta}^{\dot{1}}d\theta^2 d\theta^1\,
					   \sum_{\dot{\beta}}
					      \sigma^\mu_{\alpha\dot{\beta}}\bar{\theta}^{\dot{\beta}}
						   \Pev(\breve{f})\,.
	$$
 The proof that  $\delta_{Q_{\alpha}}S_1^\prime(\breve{f})$ is also a boundary term is similar.

 For Statement $(2^\prime)$,
  note that  for $\breve{f}$ chiral,
  $\delta_{Q_\alpha}S_2^\prime(\breve{f})=0$ always, for $\alpha=1,2$, and,
  thus one only needs to check the variation
  $\delta_{\bar{Q}_{\dot{\beta}}}S^\prime_2(\breve{f})$:
   \begin{eqnarray*}
     \delta_{\bar{Q}_{\dot{\beta}}}S^\prime_2(\breve{f})
	   & :=\: & \int_{\widehat{X}}d^4x\, d\theta^2 d\theta^1\,
	                       {\cal P}(\bar{Q}_{\dot{\beta}} \breve{f})\;
                =\; \int_{\widehat{X}}d^4x\, d\theta^2 d\theta^1\,
	                       \Pev(
						       (e_{\beta^{\prime\prime}}\,
                                  + \, 2\sqrt{-1}\mbox{$\sum$}_{\alpha, \mu}
								          \theta^\alpha \sigma^\mu_{\alpha\dot{\beta}}\partial_\mu													   
							   )\breve{f})                                                   \\
       & = & 	2 \sqrt{-1}\,
	               \int_{\widehat{X}}d^4x\, d\theta^2 d\theta^1\,
	                \sum_{\alpha,\,\mu}
					   \theta^\alpha \sigma^\mu_{\alpha\dot{\beta}}
					   \partial_\mu (\Pev(\breve{f}))		\\
       & = & 2\sqrt{-1}\,
                  \int_X	d^4x \sum_{\mu}\partial_\mu
                    \left(\rule{0ex}{1.2em}\right.				
					\int d\theta^2 d\theta^1\,
					   \sum_\alpha
					      \theta^\alpha \sigma^\mu_{\alpha\dot{\beta}} \Pev(\breve{f})
					\left.\rule{0ex}{1.2em}\right)\;\;
				=\;\; 2\sqrt{-1}\int_X dC_{\dot{\beta}}\,,
   \end{eqnarray*}
     where
     $C_{\dot{\beta}}= C_{\dot{\beta}}^0 dx^1\wedge dx^2\wedge dx^3
	                                   - C_{\dot{\beta}}^1 dx^0\wedge dx^2\wedge dx^3
									  + C_{\dot{\beta}}^2 dx^0\wedge dx^1\wedge dx^3
									   - C_{\dot{\beta}}^3 dx^0\wedge dx^1\wedge dx^2$
      is a $3$-form on $X$ with
    $$
	  C_{\dot{\beta}}^\mu \;
	    =\;  \int d\theta^2 d\theta^1\,
					\sum_\alpha
					      \theta^\alpha \sigma^\mu_{\alpha\dot{\beta}}\Pev(\breve{f})\,.
	$$
  For $\breve{f}$ antichiral,
   $\delta_{\bar{Q}_{\dot{\beta}}}S_3^\prime(\breve{f})=0$ always,
       for $\dot{\beta}=\dot{1}, \dot{2}$,    and
   the variation $\delta_{Q_\alpha}S_3^\prime(\breve{f})$, $\alpha=1,2$,
     can be computed similarly to show that it is a boundary term on $X$.
   
  This completes the proof.
  
\end{proof}
  
\medskip

\begin{sremark}$[$application of Theorem 2.3\,$]$\;
{\rm
 In an application of Theorem~2.3,
   the $\breve{f}$ in Statements (1), (2), $(1^\prime)$, and $(2^\prime)$
   usually comes from a functional of elements in $C^\infty(\widehat{X}^{\widehat{\boxplus}})$.
}\end{sremark}

\bigskip

Having reviewed `purge-evaluation maps' and the `Fundamental Theorem',
 we give two remarks/impose two questions below on purge-evaluation maps that deserve to be understood better.

\bigskip

\begin{flushleft}
{\bf Remark: A purge-evaluation map merged into an exotic ring?}
\end{flushleft}
To begin, we impose a guiding question:
 \begin{itemize}
  \item[{\bf Q.}]
    \parbox[t]{42em}{\it
     	Is it possible to construct a new function ring on a towered superspace $\widehat{X}^{\widehat{\boxplus}}$
           that accommodate
		     both the anticommuting Grassmann coordinate functions $\theta$, $\bar{\theta}$, $\vartheta$, $\bar{\vartheta}$
			   on $\widehat{X}^{\widehat{\boxplus}}$
			and a purge-evaluation map 		
			  $C^\infty(X)^{\Bbb C}[\theta, \bar{\theta}, \vartheta, \bar{\vartheta}]^\anticommuting
			      \rightarrow   C^\infty(\widehat{X})$			
			 defined via
			  $\Pev: C^\infty(X)^{\Bbb C}[\vartheta, \bar{\vartheta}]^\anticommuting
			      \rightarrow C^\infty(X)^{\Bbb C}$?}		
 \end{itemize}
The design that
  $\Pev(\vartheta_1^{d_1}\vartheta_2^{d_2}
				      \bar{\vartheta}_{\dot{1}}^{d_{\dot{1}}}
					  \bar{\vartheta}_{\dot{2}}^{d_{\dot{2}}})$ are constants in ${\Bbb C}$
 motivates one to consider first a relevant ${\Bbb C}$-algebra generated by the variables
   $\vartheta_\alpha$, $\bar{\vartheta}_{\dot{\beta}}$.
Together with
 \begin{itemize}
  \item[(1)]
   the fact the variables $\vartheta$, $\bar{\vartheta}$ are in spinor representations $S^\prime$, $S^{\prime\prime}$
     of the Lorentz group and there are pairings $\varepsilon$ chosen on the spinor bundles,

  \item[(2)]
   the isomorphism of representations of the Lorentz group:
     $V^\vee_{\Bbb C}\simeq S^\prime\oplus S^{\prime\prime}$,
  
  \item[(3)]
   compatibility with the twisted complex conjugation ${}^\dag$,
   
  \item[(4)]
   a Fierz-type identity:
   $\sum_{\alpha, \dot{\beta}}
         \theta^\alpha \sigma^\mu_{\alpha\dot{\beta}}\bar{\theta}^{\dot{\beta}}
	 \cdot
	  \sum_{\gamma, \dot{\delta}}
	     \theta^\gamma \sigma^\nu_{\gamma\dot{\delta}} \bar{\theta}^{\dot{\delta}}
	= 2\, \theta^1\theta^2\bar{\theta}^{\dot{1}}\bar{\theta}^{\dot{2}}\eta^{\mu\nu}$,
 \end{itemize}					
 (and some sense of naturality),
one is led almost  uniquely to the following ${\Bbb C}$-algebra:

\bigskip

\begin{sdefinition}
{\bf [basic exotic ${\Bbb Z}/2$-graded ${\Bbb C}$-algebra $\widehat{R}^{\,\flat}$]}\;
{\rm
  Let
    $\widehat{R}^\flat$
    be the ${\Bbb C}$-algebra with the underline ${\Bbb C}$-vector space	
	  $({\Bbb C}\oplus (S^\prime\oplus S^{\prime\prime})\oplus V^{\vee}_{\Bbb C}\,,\,+)$
	and multiplication $\bullet$ defined as follows. \\
	(Recall that
          $\varepsilon_{12}=\varepsilon_{\dot{1}\dot{2}}=-1$ and
	      $\varepsilon_{21}=\varepsilon_{\dot{2}\dot{1}}=1$.)
    	
  \bigskip
  
  \noindent
  $(a)$\hspace{1em}
  In terms of the basis
  $(1\,,\,\vartheta_1\,,\, \vartheta_2\,,\,
         \bar{\vartheta}_{\dot{1}}\,,\,
         \bar{\vartheta}_{\dot{2}}\,,\,
         \boldsymbol{\sigma}_{1\dot{1}}\,,\, \boldsymbol{\sigma}_{1\dot{2}}\,,\,
         \boldsymbol{\sigma}_{2\dot{1}}\,,\, \boldsymbol{\sigma}_{2\dot{2}})$\,:
  \begin{eqnarray*}
   &
   \vartheta_\alpha\bullet \vartheta_{\beta}= -\varepsilon_{\alpha\beta}\,,\hspace{2em}
   \vartheta_\alpha\bullet \bar{\vartheta}_{\dot{\beta}}\;
    =\; -\bar{\vartheta}_{\dot{\beta}}\bullet \vartheta_\alpha\;
	=\; \boldsymbol{\sigma}_{\alpha\dot{\beta}}\,,\hspace{2em}
   \bar{\vartheta}_{\dot{\alpha}}\bullet \bar{\vartheta}_{\dot{\beta}}
     =  \varepsilon_{\dot{\alpha}\dot{\beta}}\,,
	  &           \\[1.2ex]
   &
    \vartheta_\alpha \bullet \boldsymbol{\sigma}_{\gamma\dot{\delta}}\;
      =\; \boldsymbol{\sigma}_{\gamma\dot{\delta}}\bullet \vartheta_\alpha\;
	  =\; \varepsilon_{\gamma\alpha}\bar{\vartheta}_{\dot{\delta}}\,,\hspace{2em}	
	\bar{\vartheta}_{\dot{\beta}}\bullet \boldsymbol{\sigma}_{\gamma\dot{\delta}}\;
	  =\; \boldsymbol{\sigma}_{\gamma\dot{\delta}}\bullet \bar{\vartheta}_{\dot{\beta}}\;
	  =\; \varepsilon_{\dot{\delta}\dot{\beta}}\vartheta_\gamma\,,
	  &        \\[1.2ex]
   &
	 \boldsymbol{\sigma}_{\alpha\dot{\beta}}
	   \bullet \boldsymbol{\sigma}_{\gamma\dot{\delta}}\;
	   =\; |\varepsilon_{\alpha\gamma}\,\varepsilon_{\dot{\beta}\dot{\delta}}|\,.
	  &
 \end{eqnarray*}
 Explicitly,

 \bigskip

 \centerline{\footnotesize
 \begin{tabular}{|c||c|cccc|rrrr|} \hline
  $(\mbox{\tiny $\downarrow$}) \bullet (\mbox{\tiny $\rightarrow$}\!)$
   & $\;1$\rule{0ex}{1.2em} \raisebox{-1.2ex}{\rule{0ex}{1ex}}
   & $\vartheta_1$   & $\vartheta_2$   & $\bar{\vartheta}_{\dot{1}}$   &  $\bar{\vartheta}_{\dot{2}}$
      &  $\boldsymbol{\sigma}_{1\dot{1}}$     & $\boldsymbol{\sigma}_{1\dot{2}}$
	  &  $\boldsymbol{\sigma}_{2\dot{1}}$     &  $\boldsymbol{\sigma}_{2\dot{2}}$   \\  \hline\hline
  $1$\rule{0ex}{1.2em} \raisebox{-1.2ex}{\rule{0ex}{1ex}} 	
      & $1$
      & $\vartheta_1$   & $\vartheta_2$
	  & $\bar{\vartheta}_{\dot{1}}$   &  $\bar{\vartheta}_{\dot{2}}$
      & $\boldsymbol{\sigma}_{1\dot{1}}$        &  $\boldsymbol{\sigma}_{1\dot{2}}$
	  & $\boldsymbol{\sigma}_{2\dot{1}}$        &  $\boldsymbol{\sigma}_{2\dot{2}}$        \\ \hline
  $\vartheta_1$\rule{0ex}{1.6em} \raisebox{-1.2ex}{\rule{0ex}{1ex}}
      & $\vartheta_1$       &  $0$     & $1$
	  & $\boldsymbol{\sigma}_{1\dot{1}}$
	  & $\boldsymbol{\sigma}_{1\dot{2}}$
	  & $0\:\:$   & $0\:\:$    & $\bar{\vartheta}_{\dot{1}}$    & $\bar{\vartheta}_{\dot{2}}$  \\[1.2ex]     	
 $\vartheta_2$\rule{0ex}{1.2em} \raisebox{-1.2ex}{\rule{0ex}{1ex}}
      & $\vartheta_2$  & $-1$  & $0$
	  & $\boldsymbol{\sigma}_{2\dot{1}}$
	  & $\boldsymbol{\sigma}_{2\dot{2}}$
	  & $-\bar{\vartheta}_{\dot{1}}$   & $-\bar{\vartheta}_{\dot{2}}$   & $0\:\:$   & $0\:\:$ \\[1.2ex]
  $\bar{\vartheta}_{\dot{1}}$\rule{0ex}{1.2em} \raisebox{-1.2ex}{\rule{0ex}{1ex}}
      & $\bar{\vartheta}_{\dot{1}}$
	  & $-\boldsymbol{\sigma}_{1\dot{1}}$  & $-\boldsymbol{\sigma}_{2\dot{1}}$	
	  & $0$  & $-1$
	  & $0\:\:$  & $\vartheta_1$  & $0\:\:$  & $\vartheta_2$      \\[1.2ex]
  $\bar{\vartheta}_{\dot{2}}$\rule{0ex}{1.2em} \raisebox{-1.2ex}{\rule{0ex}{1ex}}
      & $\bar{\vartheta}_{\dot{2}}$
	  & $-\boldsymbol{\sigma}_{1\dot{2}}$
	  & $-\boldsymbol{\sigma}_{2\dot{2}}$
	  & $1$  & $0$
	  & $-\vartheta_1$  & $0\:\:$  & $-\vartheta_2$  & $0\:\:$       \\[1.2ex] \hline
  $\boldsymbol{\sigma}_{1\dot{1}}$\rule{0ex}{1.2em} \raisebox{-1.6ex}{\rule{0ex}{1ex}}
      & $\boldsymbol{\sigma}_{1\dot{1}}$
	  & $0$  & $-\bar{\vartheta}_{\dot{1}}$  & $0$   & $-\vartheta_1$
	  & $0\:\:$  & $0\:\:$ & $0\:\:$ & $1\:\:$ \\[1.2ex]
  $\boldsymbol{\sigma}_{1\dot{2}}$\rule{0ex}{1.2em} \raisebox{-1.2ex}{\rule{0ex}{1ex}}
      & $\boldsymbol{\sigma}_{1\dot{2}}$
	  & $0$  & $-\bar{\vartheta}_{\dot{2}}$  & $\vartheta_1$    & $0$
	  & $0\:\:$   & $0\:\:$   & $1\:\:$    & $0\:\:$  \\[1.2ex]
  $\boldsymbol{\sigma}_{2\dot{1}}$\rule{0ex}{1.2em} \raisebox{-1.2ex}{\rule{0ex}{1ex}}
      & $\boldsymbol{\sigma}_{2\dot{1}}$
	  & $\bar{\vartheta}_{\dot{1}}$   & $0$   & $0$   & $-\vartheta_2$
	  & $0\:\:$    & $1\:\:$    & $0\:\:$    & $0\:\:$   \\[1.2ex]
  $\boldsymbol{\sigma}_{2\dot{2}}$\rule{0ex}{1.2em} \raisebox{-1.2ex}{\rule{0ex}{1ex}}
      & $\boldsymbol{\sigma}_{2\dot{2}}$
	  & $\bar{\vartheta}_{\dot{2}}$   & $0$   & $\vartheta_2$  & $0$	
	  & $1\:\:$     & $0\:\:$     & $0\:\:$     & $0\:\:$  \\  \hline  	
 \end{tabular}
 }

 \bigskip	
 \bigskip

 \noindent
 $(b)$\hspace{1em}
 Or, equivalently,
 in terms of the basis
  $(1\,,\, \vartheta_1\,, \vartheta_2\,,\, \bar{\vartheta}_{\dot{1}}\,,\, \bar{\vartheta}_{\dot{2}}\,,\,
       \boldsymbol{\sigma}^0\,,\, \boldsymbol{\sigma}^1\,,\,
       \boldsymbol{\sigma}^2\,,\, \boldsymbol{\sigma}^3)$\,:
  \begin{eqnarray*}
   & \vartheta_\alpha\bullet \vartheta_{\beta}= -\varepsilon_{\alpha\beta}\,,\;\;\;\;\;\;
       \vartheta_\alpha\bullet \bar{\vartheta}_{\dot{\beta}}\;
        =\; -\bar{\vartheta}_{\dot{\beta}}\bullet \vartheta_\alpha\;
		=\;  -\,\mbox{\Large $\frac{1}{2}$}\,
		             \sum_\mu \bdsigma^{\mu}\bar{\sigma}_\mu^{\dot{\beta}\alpha}\,,\;\;\;\;\;\;
   \bar{\vartheta}_{\dot{\alpha}}\bullet \bar{\vartheta}_{\dot{\beta}}
        = \varepsilon_{\dot{\alpha}\dot{\beta}}\,,
	  &           \\[1.2ex]
   & \vartheta_\alpha\bullet \boldsymbol{\sigma}^\mu\;
        =\; \boldsymbol{\sigma}^\mu\bullet \vartheta_\alpha\;
		=\; \sum_{\gamma, \dot{\delta}}\bar{\vartheta}_{\dot{\delta}}
		        \sigma^\mu_{\gamma\dot{\delta}}\varepsilon_{\gamma\alpha}\,,\;\;\;\;\;\;
	   \bar{\vartheta}_{\dot{\beta}}\bullet \boldsymbol{\sigma}^\mu\;
        =\; \boldsymbol{\sigma}^\mu\bullet \bar{\vartheta}_{\dot{\beta}}\;
		=\; \sum_{\gamma,\dot{\delta}}\vartheta_\gamma
                \sigma^\mu_{\gamma\dot{\delta}}\varepsilon_{\dot{\delta}\dot{\beta}}\,,
	  &           \\[1.2ex]	
	& \boldsymbol{\sigma}^\mu \bullet \boldsymbol{\sigma}^\nu\;
	   =\; \sum_{\alpha, \dot{\delta}, \gamma, \dot{\delta}}
	           |\varepsilon_{\alpha\gamma}\,\varepsilon_{\dot{\beta}\dot{\delta}}|\,
			       \sigma^\mu_{\alpha\dot{\beta}}\, \sigma^\nu_{\gamma\dot{\delta}}\,.
	  &
 \end{eqnarray*}
 Explicitly,

 \bigskip

 \centerline{\footnotesize
 \begin{tabular}{|c||c|cccc|rrcr|} \hline
  $(\mbox{\tiny $\downarrow$}) \bullet (\mbox{\tiny $\rightarrow$}\!)$
   & $\;1$\rule{0ex}{1.2em} \raisebox{-1.2ex}{\rule{0ex}{1ex}}
   & $\vartheta_1$   & $\vartheta_2$   & $\bar{\vartheta}_{\dot{1}}$   &  $\bar{\vartheta}_{\dot{2}}$
      &  $\boldsymbol{\sigma}^0$     & $\boldsymbol{\sigma}^1$           & $\boldsymbol{\sigma}^2$                                             
	  &  $\boldsymbol{\sigma}^3$   \\  \hline\hline
  $1$\rule{0ex}{1.2em} \raisebox{-1.2ex}{\rule{0ex}{1ex}} 	
      & $1$
      & $\vartheta_1$   & $\vartheta_2$
	  & $\bar{\vartheta}_{\dot{1}}$   &  $\bar{\vartheta}_{\dot{2}}$
      & $\boldsymbol{\sigma}^0$        & $\boldsymbol{\sigma}^1$
	  & $\boldsymbol{\sigma}^2$               &  $\boldsymbol{\sigma}^3$             \\ \hline
  $\vartheta_1$\rule{0ex}{1.6em} \raisebox{-1.2ex}{\rule{0ex}{1ex}}
      & $\vartheta_1$       &  $0$     & $1$
	  & $\;\;-\frac{1}{2}\boldsymbol{\sigma}^0 + \frac{1}{2}\boldsymbol{\sigma}^3\;\;$
	  & $\;\;\frac{1}{2}\boldsymbol{\sigma}^1 + \frac{\sqrt{-1}}{2}\boldsymbol{\sigma}^2\;\;$
	  & $-\bar{\vartheta}_{\dot{2}}$   &  $\bar{\vartheta}_{\dot{1}}$
	  & $\sqrt{-1}\bar{\vartheta}_{\dot{1}}$    & $-\bar{\vartheta}_{\dot{2}}$  \\[1.2ex]
  $\vartheta_2$\rule{0ex}{1.2em} \raisebox{-1.2ex}{\rule{0ex}{1ex}}
      & $\vartheta_2$  & $-1$  & $0$
	  & $\;\;\frac{1}{2}\boldsymbol{\sigma}^1-\frac{\sqrt{-1}}{2}\boldsymbol{\sigma}^2\;\;$
	  & $\;\;-\frac{1}{2}\boldsymbol{\sigma}^0-\frac{1}{2}\boldsymbol{\sigma}^3\;\;$
	  & $\bar{\vartheta}_{\dot{1}}$   & $-\bar{\vartheta}_{\dot{2}}$
	  & $\sqrt{-1}\bar{\vartheta}_{\dot{2}}$   & $-\bar{\vartheta}_{\dot{1}}$ \\[1.2ex]
  $\bar{\vartheta}_{\dot{1}}$\rule{0ex}{1.2em} \raisebox{-1.2ex}{\rule{0ex}{1ex}}
      & $\bar{\vartheta}_{\dot{1}}$
	  & $\;\;\frac{1}{2}\boldsymbol{\sigma}^0-\frac{1}{2}\boldsymbol{\sigma}^3\;\;$
      & $\;\;-\frac{1}{2}\boldsymbol{\sigma}^1+\frac{\sqrt{-1}}{2}\boldsymbol{\sigma}^2\;\;$	
	  & $0$  & $-1$
	  & $-\vartheta_2$  & $\vartheta_1$  & $-\sqrt{-1}\vartheta_1$  & $-\vartheta_2$ \\[1.2ex]
  $\bar{\vartheta}_{\dot{2}}$\rule{0ex}{1.2em} \raisebox{-1.2ex}{\rule{0ex}{1ex}}
      & $\bar{\vartheta}_{\dot{2}}$
	  & $\;\;-\frac{1}{2}\boldsymbol{\sigma}^1-\frac{\sqrt{-1}}{2}\boldsymbol{\sigma}^2\;\;$
	  & $\;\;\frac{1}{2}\boldsymbol{\sigma}^0+\frac{1}{2}\boldsymbol{\sigma}^3\;\;$
	  & $1$  & $0$
	  & $\vartheta_1$  & $-\vartheta_2$  & $-\sqrt{-1}\vartheta_2$  & $-\vartheta_1$  \\[1.2ex] \hline
  $\boldsymbol{\sigma}^0$\rule{0ex}{1.2em} \raisebox{-1.6ex}{\rule{0ex}{1ex}}
      & $\boldsymbol{\sigma}^0$
	  & $-\bar{\vartheta}_{\dot{2}}$  & $\bar{\vartheta}_{\dot{1}}$
	  & $-\vartheta_2$   & $\vartheta_1$
	  & $2\:\:$  & $0\:\:$ & $0$ & $0\:\:$ \\[1.2ex]
  $\boldsymbol{\sigma}^1$\rule{0ex}{1.2em} \raisebox{-1.2ex}{\rule{0ex}{1ex}}
      & $\boldsymbol{\sigma}^1$
	  & $\bar{\vartheta}_{\dot{1}}$  & $-\bar{\vartheta}_{\dot{2}}$
	  & $\vartheta_1$    & $-\vartheta_2$
	  & $0\:\:$   & $2\:\:$   & $0$    & $0\:\:$  \\[1.2ex]
  $\boldsymbol{\sigma}^2$\rule{0ex}{1.2em} \raisebox{-1.2ex}{\rule{0ex}{1ex}}
      & $\boldsymbol{\sigma}^2$
	  & $\sqrt{-1}\bar{\vartheta}_{\dot{1}}$
	  & $\sqrt{-1}\bar{\vartheta}_{\dot{2}}$
	  & $-\sqrt{-1}\vartheta_1$   & $-\sqrt{-1}\vartheta_2$
	  & $0\:\:$    & $0\:\:$    & $2$    & $0\:\:$ \\[1.2ex]
  $\boldsymbol{\sigma}^3$\rule{0ex}{1.2em} \raisebox{-1.2ex}{\rule{0ex}{1ex}}
      & $\boldsymbol{\sigma}^3$  & $-\bar{\vartheta}_{\dot{2}}$   & $-\bar{\vartheta}_{\dot{1}}$
      & $-\vartheta_2$  & $-\vartheta_1$	
	  & $0\:\:$     & $0\:\:$     & $0$     & $-2\:\:$  \\  \hline
 \end{tabular}
 }

 \bigskip
 \bigskip
 
 This gives a
   ${\Bbb Z}/2$-graded, ${\Bbb Z}/2$-commutative, unital, {\it non-associative} algebra over ${\Bbb C}$.
 We will call it the {\it basic exotic ${\Bbb Z}/2$-graded ${\Bbb C}$-algebra} and
   denote it by $\widehat{R}^{\,\flat}$.
}\end{sdefinition}

\bigskip

One can check directly that

\bigskip

\begin{slemma}
{\bf [$\widehat{R}^{\,\flat}$ under twisted complex conjugation ${}^{\dag}$]}\;
 For $r_1$, $r_2\in \widehat{R}^{\,\flat}$,
  $(r_1\bullet r_2)^{\dag}= r_2^\dag\bullet r_1^\dag$.
\end{slemma}

\medskip

\begin{sremark} $[$from Grassmann number to exotic number$]$\; {\rm
 Conceptually, one should think of
  the ${\Bbb Z}/2$-graded ${\Bbb C}$-algebra $\widehat{R}^{\,\flat}$ as {\it diminished/flattened} from
  the Grassmann algebra $\bigwedge^{\tinybullet}_{\Bbb C}(S^\prime\oplus S^{\prime\prime})$
   {\it via the purge-evaluation} defined by the ${\Bbb C}$-vector-space-homomorphism associated to the assignment
 $$
   \begin{array}{ccccrl}
    \widehat{\Pev} & : & \bigwedge^{\tinybullet}_{\Bbb C}(S^\prime\oplus S^{\prime\prime})
	  & \longrightarrow & \widehat{R}^{\,\flat}\;  &  \\[1.2ex]
     &&  1\;\;   & \longmapsto  & 1\;\;  &,\\[.8ex]
     &&  \vartheta_{\alpha}  & \longmapsto & \vartheta_\alpha   &, \\[.8ex]	
	 &&  \bar{\vartheta}_{\dot{\beta}}  & \longmapsto   & \bar{\vartheta}_{\dot{\beta}} &,\\[.8ex]
	 &&  \bdsigma_{\alpha\dot{\beta}}     & \longmapsto   & \bdsigma_{\alpha\dot{\beta}} &,\\[.8ex]
	 &&  \vartheta_1\vartheta_2  & \longmapsto   & 1\;\;   &,\\[.8ex]
	 &&  \bar{\vartheta}_{\dot{1}}\bar{\vartheta}_{\dot{2}}
	           & \longmapsto    &  -1\;\;  &,  \\[.8ex]
     && 	\vartheta_1\vartheta_2\bar{\vartheta}_{\dot{\beta}}
	           & \longmapsto    &  \bar{\vartheta}_{\dot{\beta}}   &, \\[.8ex]
   	 &&  \vartheta_\alpha\bar{\vartheta}_{\dot{1}}\bar{\vartheta}_{\dot{2}}
	           & \longmapsto    &  - \vartheta_\alpha   &, \\[.8ex]
     && \vartheta_1\vartheta_2\bar{\vartheta}_{\dot{1}}\bar{\vartheta}_{\dot{2}}			
	           & \longmapsto   & -1\;\;  &,
   \end{array}
 $$
 for $\alpha=1,2$ and $\dot{\beta}=\dot{1}, \dot{2}$.
 $\widehat{\Pev}$ is a ${\Bbb C}$-vector-space projection map
   from $\bigwedge_{\Bbb C}^{\tinybullet}(S^\prime\oplus S^{\prime\prime})$
     onto its sub-${\Bbb C}$-vector-space $\widehat{R}^{\,\flat}$.
 While $\widehat{\Pev}$ is not a ${\Bbb C}$-algebra-homomorphism, it is an $\SO^\uparrow(1,3)$-module-homomorphism.
}\end{sremark}

\medskip
  
\begin{sdefinition} {\bf [exotic function-ring of towered superspace]}\; {\rm
 Let
  $\widehat{X}^{\widehat{\squareflat}}$
  be a ringed-space with the underlying topology $X$ and the function ring
   \begin{eqnarray*}
    \lefteqn{C^\infty(\widehat{X}^{\widehat{\squareflat}})  \;\;	:=\;\;
	    C^\infty(X)^{\Bbb C}[\theta, \bar{\theta};  \vartheta, \bar{\vartheta}]^\exotic
	    }\\
	 && :=\;\;
	   \mbox{
	      the extension-over-${\Bbb C}$ of $C^\infty(X)^{\Bbb C}[\theta, \vartheta]^\anticommuting$
	       by $\widehat{R}^{\,\flat}$ with $\vartheta$, $\bar{\vartheta}$ anticumming with $\theta$, $\bar{\theta}$}\,.
   \end{eqnarray*}
 This is a ${\Bbb Z}/2$-graded, ${\Bbb Z}/2$-commutative nonassociative ring.
 We shall call $\widehat{X}^{\widehat{\squareflat}}$
   an {\it exotic towered superspace} and its function ring the {\it exotic function-ring} of a towered superspace.
}\end{sdefinition}

\bigskip

\noindent
The twisted complex conjugation $(\cdot)^\dag$ is naturally defined on $C^\infty(\widehat{X}^{\widehat{\squareflat}})$,
 with the property that
  $(\breve{f}\bullet \breve{g})^\dag\; =\;  \breve{g}^\dag \bullet \breve{f}^\dag$,
  for $\breve{f},\,\breve{g}\in C^\infty(\widehat{X}^{\widehat{\squareflat}})$,
 where $\bullet$ is the multiplication on $C^\infty(\widehat{X}^{\widehat{\squareflat}})$.

One can formulate
 (1) the notion of {\it chiral} and {\it antichiral superfields} on $X$,
 (2) the notion of the {\it small exotic function-ring} of $\widehat{X}^{\widehat{\squareflat}}$,
            which is a complexified $C^\infty$-ring --- in particular, commutative and associative ---
			  contained in $C^\infty(\widehat{X}^{\widehat{\squareflat}})$
			  as a subring,      and
 (3) Theorem~2.3
  all in terms of elements in $C^\infty(\widehat{X}^{\widehat{\squareflat}})$.
In particular, for Item (3),
 there is no need to introduce the additional purge-evaluation map now
 since that data is already merged into the construction of the exotic function-ring of the towered superspace.
This can be used to directly reproduce the Wess-Zumino model in [Wess \& Bagger: Chapter  V].
In this simple case,
  the nonassociativity of the $(\vartheta, \bar{\vartheta})$-part is completely
   overridden by the nilpotency of the $(\theta, \bar{\theta})$-part.

Unfortunately,
 when attempting to extend its application to the construction of supersymmetric gauge theories,
 one no longer has such a luck.
The nonassociativity of $C^\infty(\widehat{X}^{\widehat{\squareflat}}) $
 brings in new technical issues that do not look to have any simple cure for the time being.

\bigskip

\begin{flushleft}
{\bf Remark: A canonical/standard purge-evaluation map with respect to $(\vartheta, \bar{\vartheta})$?}
\end{flushleft}
So far in the consideration of purge-evaluation maps, we only take into account the requirement from physics
 that the density over $X$ of an action functional has to be real-valued
  (or complex-valued plus its complex conjugation ) without any nilpotent factors.
\begin{itemize}
  \item[{\bf Q.}]
   \parbox[t]{42em}{\it
     How additional physical considerations restrict the choice of the purge-evaluation map?}
\end{itemize}
For example, assume that
 $\Pev(\vartheta_1^{d_1}\vartheta_2^{d_2}
                \bar{\vartheta}_{\dot{1}}^{d_{\dot{1}}}\bar{\vartheta}_{\dot{2}}^{d_{\dot{2}}})
    \in {\Bbb R}$.
Then while the absolute-value of the value $\Pev$ takes may be absorbed into field-redefinitions,
 the sign of the value may not be.
In that case, {\it what physics consideration select  the sign of the values of the purge-evaluation map?}

To get a sense of this, let us consider the Wess-Zumino model quoted from [Wess \& Bagger: Eq.\:(5.13)]
 (in the case of one single chiral superfield ):
 $$
  {\cal L}\;=\;
    \sqrt{-1}\sum_\mu \partial_\mu \bar{\psi}\bar{\sigma}^\mu\psi
	  + A^\ast\square A
	  - \mbox{\large $\frac{1}{2}$} m\psi\psi
	  - \mbox{\large $\frac{1}{2}$} m^\ast \bar{\psi}\bar{\psi}
	  - g\psi\psi A - g^\ast \bar{\psi}\bar{\psi}A^\ast - F^\ast F\,.
 $$
Here, $(A, \psi, F)$  is a chiral multiplet and $m, g\in {\Bbb C}$.
To make the nilpotent feature of the independent components $A$, $\psi$, $F$ of the chiral superfield
  manifest,
 set
 \begin{eqnarray*}
  &
  A\;\rightsquigarrow\; f_{(0)}\,,\hspace{2em}
  A^\ast\;\rightsquigarrow\; \overline{f_{(0)}}\,,\hspace{2em}
  \psi_\alpha\;\rightsquigarrow\; \vartheta_\alpha f_{(\alpha)}\,,\hspace{2em}
  \bar{\psi}_{\dot{\alpha}}\;\rightsquigarrow\; \bar{\vartheta}_{\dot{\alpha}}
     \overline{f_{(\alpha)}}\,,
	  \\[.6ex]  &
  F\;\rightsquigarrow\; - \vartheta_1\vartheta_2 f_{(12)}\,,\hspace{2em}
  F^\ast\; \rightsquigarrow\; (-\vartheta_1\vartheta_2 f_{(12)})^\dag\;
     =\; \bar{\vartheta}_{\dot{1}} \bar{\vartheta}_{\dot{2}} \overline{f_{(12)}}
 \end{eqnarray*}
 and apply
   the rule of short-hand-to-long-hand notation change ([Wess \& Bagger: Eq.\:(A.21)]) and
   the rule of raising-or-lowering spinorial indices ([Wess \& Bagger: Eq.\:(A.9)]):
 \begin{eqnarray*}
  \psi\psi  &  :=  &
   \sum_\alpha \psi^\alpha\psi_\alpha\;
	    =\;  -2\, \psi_1\psi_2 \;
		=\; -2\,\vartheta_1\vartheta_2 f_{(1)}f_{(2)}\,,
	 \\		
   \bar{\psi}\bar{\psi} &  :=  &
   \sum_{\dot{\alpha}} \bar{\psi}_{\dot{\alpha}}\bar{\psi}^{\dot{\alpha}}\;
	    =\;\;\;\;   2\, \bar{\psi}_{\dot{1}}\bar{\psi}_{\dot{2}} \;
		=\;\;\;\;   2\, \bar{\vartheta}_{\dot{1}}\bar{\vartheta}_{\dot{2}}
		               \overline{f_{(1)}}\, \overline{f_{(2)}}\,,  	
 \end{eqnarray*}
 the density ${\cal L}$ of the action functional of the Wess-Zumino model is then converted to
 \begin{eqnarray*}
  \lefteqn{
  {\cal L}\;
     =\; \sqrt{-1} \sum_{\alpha, \dot{\beta}, \mu}
	         \bar{\vartheta}_{\dot{\beta}}\vartheta_\alpha
			     \partial_\mu \overline{f_{(\beta)}}
				    \bar{\sigma}^{\mu\dot{\beta}\alpha} f_{(\alpha)}
			 + \overline{f_{(0)}}\,\square f_{(0)}
			 + m\, \vartheta_1\vartheta_2 f_{(1)}f_{(2)}
			 - \bar{m}\,\bar{\vartheta}_{\dot{1}} \bar{\vartheta}_{\dot{2}}\,
	   		       \overline{f_{(1)}}\, \overline{f_{(2)}}
				 } \\
     &&	\hspace{2em}			
       +\, 2g\, \vartheta_1\vartheta_2 f_{(1)}f_{(2)} f_{(0)} 		
		- 2 \bar{g}\,\bar{\vartheta}_{\dot{1}} \bar{\vartheta}_{\dot{2}}\,
		          \overline{f_{(1)}}\, \overline{f_{(2)}}\, \overline{f_{(0)}}
		+ \vartheta_1\vartheta_2 \bar{\vartheta}_{\dot{1}} \bar{\vartheta}_{\dot{2}}\,
		     \overline{f_{(12)}} f_{(12)}\,.
 \end{eqnarray*}
Which can be converted further to a manifestly real expression
 \begin{eqnarray*}
  \lefteqn{
  {\cal L}\;
     =\; \mbox{\large $\frac{\sqrt{-1}}{2}$}\,
	         \sum_{\alpha, \dot{\beta}, \mu}
	          \bar{\vartheta}_{\dot{\beta}}\vartheta_\alpha
			   \mbox{\Large $($}
			      \partial_\mu \overline{f_{(\beta)}}
				    \bar{\sigma}^{\mu\dot{\beta}\alpha} f_{(\alpha)}
				  -   \overline{f_{(\beta)}}
				       \bar{\sigma}^{\mu\dot{\beta}\alpha} \partial_\mu f_{(\alpha)}	
			    \mbox{\Large $)$}		
			 - \sum_\mu \partial_\mu \overline{f_{(0)}}\,\partial^\mu f_{(0)}
    			 } \\
     && \hspace{2em}
	   +\, m\, \vartheta_1\vartheta_2 f_{(1)}f_{(2)}
			 - \bar{m}\,\bar{\vartheta}_{\dot{1}} \bar{\vartheta}_{\dot{2}}\,
	   		       \overline{f_{(1)}}\, \overline{f_{(2)}}		
       + 2g\, \vartheta_1\vartheta_2 f_{(1)}f_{(2)} f_{(0)} 		
		- 2 \bar{g}\,\bar{\vartheta}_{\dot{1}} \bar{\vartheta}_{\dot{2}}\,
		          \overline{f_{(1)}}\, \overline{f_{(2)}}\, \overline{f_{(0)}}
	     \\[.6ex]
     && \hspace{2em}
		+\, \vartheta_1\vartheta_2 \bar{\vartheta}_{\dot{1}} \bar{\vartheta}_{\dot{2}}\,
		        \overline{f_{(12)}} f_{(12)}\,,
 \end{eqnarray*}
up to boundary terms on $X$.
The Hamiltonian density ${\cal H}$ over the spatial $3$-space, identified as, say, $\{0\}\times {\Bbb R}^3\subset X$,
 on the configuration space of the system
 is given by
 \begin{eqnarray*}
  {\cal H} & = &
       \frac{\delta {\cal L}}{\delta (\partial_0 f_{(0)})} \cdot  \partial_0 f_{(0)}
	  +    \frac{\delta {\cal L}}{\delta (\partial_0 \overline{f_{(0)}})}
	            \cdot  \partial_0 \overline{f_{(0)}}
      + \sum_\alpha	
	         \frac{\delta {\cal L}}{\delta (\partial_0 f_{(\alpha)})} \cdot  \partial_0 f_{(\alpha)}
	  + \sum_\beta
	         \frac{\delta {\cal L}}{\delta (\partial_0 \overline{f_{(\beta)}})}
			    \cdot  \partial_0 \overline{f_{(\beta)}}	
	  - {\cal L}
	  \\
	&= &
	 - \mbox{\large $\frac{\sqrt{-1}}{2}$}\,
	         \sum_{\mu=1}^3 \sum_{\alpha, \dot{\beta}}
	          \bar{\vartheta}_{\dot{\beta}}\vartheta_\alpha
			   \mbox{\Large $($}
			      \partial_\mu \overline{f_{(\beta)}}
				    \bar{\sigma}^{\mu\dot{\beta}\alpha} f_{(\alpha)}
				  -   \overline{f_{(\beta)}}
				       \bar{\sigma}^{\mu\dot{\beta}\alpha} \partial_\mu f_{(\alpha)}	
			    \mbox{\Large $)$}		
	   +  \sum_{\mu=1}^3 \partial_\mu \overline{f_{(0)}}\,\partial^\mu f_{(0)}
    			\\
     && \hspace{2em}
	   -\, m\, \vartheta_1\vartheta_2 f_{(1)}f_{(2)}
	    + \bar{m}\,\bar{\vartheta}_{\dot{1}} \bar{\vartheta}_{\dot{2}}\,
	   		       \overline{f_{(1)}}\, \overline{f_{(2)}}		
       - 2g\, \vartheta_1\vartheta_2 f_{(1)}f_{(2)} f_{(0)} 		
		+ 2 \bar{g}\,\bar{\vartheta}_{\dot{1}} \bar{\vartheta}_{\dot{2}}\,
		          \overline{f_{(1)}}\, \overline{f_{(2)}}\, \overline{f_{(0)}}
	     \\[.6ex]
     && \hspace{2em}
	 -\, \vartheta_1\vartheta_2 \bar{\vartheta}_{\dot{1}} \bar{\vartheta}_{\dot{2}}\,
		        \overline{f_{(12)}} f_{(12)}\,,
 \end{eqnarray*}
 with the final expression restricted to the spatial slice $\{0\}\times {\Bbb R}^3\subset X$.
The Hamiltonian density over the spatial $3$-space is meant to be the energy density
   along a trajectory in the phase space following the equation of motion and hence better be real positive-or-bounded-below.
Thus,
 \begin{itemize}
  \item[(1)]
   If setting
      $\vartheta_1\vartheta_2 \rightsquigarrow \pm 1$ by convention,
    then it is required that
	  $\bar{\vartheta}_{\dot{1}}\bar{\vartheta}_{\dot{2}}\rightsquigarrow \mp 1$
	  in order to be compatible with the twisted complex conjugation.
	This also helps keep ${\cal H}$ real.

  \item[(2)]
   The summand
    $\,-\,\vartheta_1\vartheta_2 \bar{\vartheta}_{\dot{1}}\bar{\vartheta}_{\dot{2}}
	         \overline{f_{(12)}} f_{(12)}\,$ in ${\cal H}$
	is a potential energy term.
   Since $\overline{f_{(12)}} f_{(12)}\ge 0$,
    it is required that
     $\,-\,\vartheta_1\vartheta_2 \bar{\vartheta}_{\dot{1}}\bar{\vartheta}_{\dot{2}}
      \rightsquigarrow \mbox{positie real number}$.
   One may thus set
    $\vartheta_1\vartheta_2 \bar{\vartheta}_{\dot{1}}\bar{\vartheta}_{\dot{2}}
	   \rightsquigarrow -1$ to meet this requirement.

  \item[(3)]
   For the nilpotent factor $\bar{\vartheta}_{\dot{\beta}}\vartheta_\alpha$ in the kinetic terms of the spinor component fields
     $f_{(\alpha)}$, $\alpha=1,2$,
	exchanging the chirality/handedness of spinor fields $\psi \leftrightarrow \bar{\psi}$ changes the sign in front of the kinetic term.
  Thus, the sign of $\Pev(\bar{\vartheta}_{\dot{\beta}}\vartheta_\alpha)$ is only a matter of convention
   and we may set
     $\bar{\vartheta}_{\dot{\beta}}\vartheta_\alpha = -\vartheta_\alpha\bar{\vartheta}_{\dot{\beta}}
	    \rightsquigarrow \pm 1$
	either choice by convention.	
 \end{itemize}
Thus overall up to positive constant factors, which can be absorbed into a redefinition of fields or coupling constants,
  the following four choices are physically acceptable:
 \begin{itemize}
  \item[\LARGE $\cdot$]
   $\vartheta_1\vartheta_2 \rightsquigarrow \;\;\;1\,,\;\;
     \bar{\vartheta}_{\dot{1}}\bar{\vartheta}_{\dot{2}}\rightsquigarrow -1\,,\;\;
	 \vartheta_1\vartheta_2 \bar{\vartheta}_{\dot{1}}\bar{\vartheta}_{\dot{2}}
	   \rightsquigarrow -1\,,\;\;
	 \bar{\vartheta}_{\dot{\beta}}\vartheta_\alpha\rightsquigarrow \pm 1$\,;
 
  \item[\LARGE $\cdot$]
   $\vartheta_1\vartheta_2 \rightsquigarrow -1\,,\;\;
     \bar{\vartheta}_{\dot{1}}\bar{\vartheta}_{\dot{2}}\rightsquigarrow\;\;\;  1\,,\;\;
	 \vartheta_1\vartheta_2 \bar{\vartheta}_{\dot{1}}\bar{\vartheta}_{\dot{2}}
	   \rightsquigarrow -1\,,\;\;
	 \vartheta_\alpha \bar{\vartheta}_{\dot{\beta}}\rightsquigarrow \pm 1$\,.
 \end{itemize}

\bigskip

\section{The small function-ring of $\widehat{X}^{\widehat{\boxplus}}$ and the Wess-Zumino model\\
                {\large (cf.\:[Wess \& Bagger: Chapter V])}}

We reconstruct in this section
  \begin{itemize}
   \item[\LARGE $\cdot$] [Wess \& Bagger: Chapter V.\:{\it Chiral superfields}\,]
  \end{itemize}
 in the complexified ${\Bbb Z}/2$-graded $C^\infty$-Algebraic Geometry setting of Sec.\:1.
The same construction was made in [L-Y5: Sec.\:2] (SUSY(1)) with slightly different terminology;
cf.\:footnote~6 and footnote~10.

\bigskip

\begin{flushleft}
{\bf The small function-ring $C^\infty(\widehat{X}^{\widehat{\boxplus}})^\smallscriptsize$
           and its chiral and antichiral sectors}
\end{flushleft}
Recall from Sec.\:1.2
 the following complexified $C^\infty$-subrings of $C^\infty(\widehat{X}^{\widehat{\boxplus}})$.
Their elements give the superfields involved in the superspace formulation of the Wess-Zumino model.
 \begin{itemize}
  \item[\Large $\cdot$]
  The {\it small function-ring} $C^\infty(\widehat{X}^{\widehat{\boxplus}})^\smallscriptsize$
   of $\widehat{X}^{\widehat{\boxplus}}$,
  which consists of elements of $C^\infty(\widehat{X}^{\widehat{\boxplus}})$
   of the following form
 {\small
 \begin{eqnarray*}
  \breve{f}
 &= &
    f_{(0)}
	+ \sum_{\alpha}\theta^\alpha\vartheta_\alpha f_{(\alpha)}
	+ \sum_{\dot{\beta}}
	       \bar{\theta}^{\dot{\beta}}\bar{\vartheta}_{\dot{\beta}} f_{(\dot{\beta})}  \\
  && \hspace{1em}			
    +\; \theta^1\theta^2\vartheta_1\vartheta_2 f_{(12)}
	+ \sum_{\alpha,\dot{\beta}}\theta^\alpha \bar{\theta}^{\dot{\beta}}
	      \left(\rule{0ex}{1.2em}\right.\!
		    \sum_\mu \sigma^\mu_{\alpha\dot{\beta}} f_{[\mu]}\,
			 +\, \vartheta_\alpha\bar{\vartheta}_{\dot{\beta}} f_{(\alpha\dot{\beta})}
		  \!\left.\rule{0ex}{1.2em}\right)
    + \bar{\theta}^{\dot{1}}\bar{\theta}^{\dot{2}}
	   \bar{\vartheta}_{\dot{1}}\bar{\vartheta}_{\dot{2}} f_{(\dot{1}\dot{2})}  \\
  && \hspace{1em}		
	+ \sum_{\dot{\beta}}\theta^1\theta^2\bar{\theta}^{\dot{\beta}}
	     \left(\rule{0ex}{1.2em}\right.\!
		   \sum_{\alpha,\mu }
		          \vartheta_\alpha\,\sigma^{\mu \alpha}_{\;\;\;\;\dot{\beta}} f^\prime_{[\mu]}\,
		    +\, \vartheta_1\vartheta_2\bar{\vartheta}_{\dot{\beta}} f_{(12\dot{\beta})}
		 \!\left.\rule{0ex}{1.2em}\right)
    + \sum_\alpha \theta^\alpha\bar{\theta}^{\dot{1}}\bar{\theta}^{\dot{2}}
	   \left(\rule{0ex}{1.2em}\right.\!
	     \sum_{\dot{\beta}, \mu}
		  \sigma^{\mu\dot{\beta}}_\alpha
		     \bar{\vartheta}_{\dot{\beta}} f^{\prime\prime}_{[\mu]}\,
         +\, \vartheta_\alpha \bar{\vartheta}_{\dot{1}}\bar{\vartheta}_{\dot{2}}	
		           f_{(\alpha\dot{1}\dot{2})}
	   \!\left.\rule{0ex}{1.2em}\right)\\
  && \hspace{1em}
	+\; \theta^1\theta^2\bar{\theta}^{\dot{1}}\bar{\theta}^{\dot{2}}
	     \left(\rule{0ex}{1.2em}\right.\!
		  f^\sim_{(0)}
		  + \sum_{\alpha,\dot{\beta}, \mu} \vartheta_\alpha\bar{\vartheta}_{\dot{\beta}}
		         \bar{\sigma}^{\mu\dot{\beta}\alpha} f^{\sim}_{[\mu]}
		  + \vartheta_1\vartheta_2\bar{\vartheta}_{\dot{1}}\bar{\vartheta}_{\dot{2}}
		        f_{(12\dot{1}\dot{2})}
		 \!\left.\rule{0ex}{1.2em}\right)
	 \\
   & \in &	C^\infty(X)^{\Bbb C}[\theta, \bar{\theta}, \vartheta, \bar{\vartheta}]^\anticommuting\,.		
 \end{eqnarray*}
  }
    
  \item[\Large $\cdot$]
  The {\it small chiral function-ring} $C^\infty(\widehat{X}^{\widehat{\boxplus}})^{\smallscriptsize, \scriptsizech}$
   of $\widehat{X}^{\widehat{\boxplus}}$,
  which consists of elements of $C^\infty(\widehat{X}^{\widehat{\boxplus}})^\smallscriptsize$
   of the following form
     \marginpar{\vspace{2em}\raggedright\tiny
         \raisebox{-1ex}{\hspace{-2.4em}\LARGE $\cdot$}Cf.\,[\,\parbox[t]{20em}{Wess
		\& Bagger:\\ Eq.\:(5.3)].}}
  \begin{eqnarray*}
   \breve{f} & = &
     f_{(0)}(x)
	 + \sum_\gamma \theta^\gamma \vartheta_\gamma f_{(\gamma)}(x)
	    + \theta^1\theta^2
		     \vartheta_1\vartheta_2 f_{(12)}(x)						
			\\
	&& \hspace{2em}
	    +\, \sqrt{-1} \sum_{\gamma, \dot{\delta}, \nu}
		    \theta^\gamma \bar{\theta}^{\dot{\delta}} \sigma^\nu_{\gamma\dot{\delta}}
			  \partial_\nu f_{(0)}(x)
        + \sqrt{-1} \sum_{\dot{\delta}, \gamma, \nu}
		     \theta^1\theta^2 \bar{\theta}^{\dot{\delta}} \vartheta_\gamma
			  \sigma^{\nu\gamma}_{\;\;\;\;\dot{\delta}}\,\partial_\nu f_{(\gamma)}(x)
		- \theta^1\theta^2\bar{\theta}^{\dot{1}} \bar{\theta}^{\dot{2}}\,
		    \square f_{(0)}(x)\,.
  \end{eqnarray*}
  
   \item[\Large $\cdot$]
  The {\it small antichiral function-ring}
    $C^\infty(\widehat{X}^{\widehat{\boxplus}})^{\smallscriptsize, \scriptsizeach}$
   of $\widehat{X}^{\widehat{\boxplus}}$,
  which consists of elements of\\   $C^\infty(\widehat{X}^{\widehat{\boxplus}})^\smallscriptsize$
   of the following form
     \marginpar{\vspace{2em}\raggedright\tiny
         \raisebox{-1ex}{\hspace{-2.4em}\LARGE $\cdot$}Cf.\,[\,\parbox[t]{20em}{Wess
		\& Bagger:\\ Eq.\:(5.5)].}}
  \begin{eqnarray*}
   \breve{f} & = &
     f_{(0)}(x)
	 + \sum_{\dot{\delta}} \bar{\theta}^{\dot{\delta}} \bar{\vartheta}_{\dot{\delta}}
	       f_{(\dot{\delta})}(x)
	    + \bar{\theta}^{\dot{1}}\bar{\theta}^{\dot{2}} 		
			 \bar{\vartheta}_{\dot{1}}\bar{\vartheta}_{\dot{2}}  f_{(\dot{1}\dot{2})}(x)			
			\\
	&& \hspace{2em}
	    -\, \sqrt{-1} \sum_{\gamma, \dot{\delta}, \nu}
		    \theta^\gamma \bar{\theta}^{\dot{\delta}} \sigma^\nu_{\gamma\dot{\delta}}
			  \partial_\nu f_{(0)}(x)
        + \sqrt{-1} \sum_{\dot{\delta}, \gamma, \nu}
		     \theta^\gamma \bar{\theta}^{\dot{1}} \bar{\theta}^{\dot{2}}
			  \bar{\vartheta}_{\dot{\delta}}
			  \sigma^{\nu\dot{\delta}}_\gamma\,\partial_\nu f_{(\dot{\delta})}(x)
		- \theta^1\theta^2\bar{\theta}^{\dot{1}} \bar{\theta}^{\dot{2}}\,
		    \square f_{(0)}(x)\,.               		
  \end{eqnarray*}
 \end{itemize}
Recall also
 that, as a ring,
   $C^\infty(\widehat{X}^{\widehat{\boxplus}})^\smallscriptsize$
   is generated by
    $C^\infty(\widehat{X}^{\widehat{\boxplus}})^{\smallscriptsize, \scriptsizech}
      \cup C^\infty(\widehat{X}^{\widehat{\boxplus}})^{\smallscriptsize, \scriptsizeach}$
   and
 that the twisted complex conjugation ${}^\dag$ takes
 $C^\infty(\widehat{X}^{\widehat{\boxplus}})^{\smallscriptsize, \scriptsizech}$  and
 $C^\infty(\widehat{X}^{\widehat{\boxplus}})^{\smallscriptsize, \scriptsizeach}$
 to each other.
 
We now proceed to construct the Wess-Zumino model on $X$ in terms of
 $C^\infty(\widehat{X}^{\widehat{\boxplus}})^\smallscriptsize$.

\bigskip

\begin{flushleft}
{\bf Relevant basic computations/formulae}
\end{flushleft}
Let
   \begin{eqnarray*}
     \breve{f} & = &
	   f_{(0)}(x)
	   + \sum_\alpha \theta^\alpha\vartheta_\alpha f_{(\alpha)}(x)
	   + \theta^1\theta^2\vartheta_1\vartheta_2 f_{(12)}(x)
       + \sqrt{-1} \sum_{\alpha,\dot{\beta};\,\mu}
	          \theta^\alpha\bar{\theta}^{\dot{\beta}}	
			     \sigma^\mu_{\alpha\dot{\beta}}\, \partial_\mu f_{(0)}(x) \\
      && \hspace{1em}				
       + \sqrt{-1}\sum_{\dot{\beta};\, \alpha, \mu}				
	        \theta^1\theta^2\bar{\theta}^{\dot{\beta}}			
			   \vartheta_\alpha\sigma^{\mu\alpha}_{\;\;\;\;\dot{\beta}}\, \partial_\mu f_{(\alpha)}(x)
       - \theta^1\theta^2\bar{\theta}^{\dot{1}}\bar{\theta}^{\dot{2}}\,
	        \square f_{(0)}(x)\,,
   \end{eqnarray*}
  be a chiral function on $\widehat{X}^{\widehat{\boxplus}, \smallscriptsize}$, determined by
   $(f_{(0)},  f_{(\alpha)}, f_{(12)} )_{\alpha}$.
It follows from Corollary~1.3.11
 that its twisted complex conjugate $\breve{f}^\dag$ is
  the antichiral function on $X^{\physics}$
     determined by
	   $(f^{\dag }_{(0)},
	         f^{\dag }_{(\dot{\beta})},
			 f^{\dag }_{(\dot{1}\dot{2})})_{\dot{\beta}}
		 = (\overline{f_{(0)}}, -\,\overline{f_{(\beta)}},
		       \overline{f_{(12)}})_\beta$,
		where
		   $\overline{f_{(\tinybullet)}}$
     	  	is the complex conjugate of $f_{(\tinybullet)}\in C^\infty(X)^{\,\Bbb C}$.
 Explicitly,
   \begin{eqnarray*}
     \breve{f}^\dag & = &
	   \overline{f_{(0)}(x)}
	   - \sum_{\dot{\beta}} \bar{\theta}^{\dot{\beta}}\bar{\vartheta}_{\dot{\beta}}
	           \overline{f_{(\beta)}(x)}
	   + \bar{\theta}^{\dot{1}}\bar{\theta}^{\dot{2}}
	      \bar{\vartheta}_{\dot{1}}\bar{\vartheta}_{\dot{2}}
		               \overline{f_{(12)}(x)}
       - \sqrt{-1} \sum_{\alpha,\dot{\beta};\,\mu}
	          \theta^\alpha\bar{\theta}^{\dot{\beta}}	
			     \sigma^\mu_{\alpha\dot{\beta}}\,
				       \partial_\mu \overline{f_{(0)}(x)} \\
      && \hspace{1em}				
       -\, \sqrt{-1}\sum_{\alpha;\, \dot{\beta}, \mu}				
	        \theta^\alpha\bar{\theta}^{\dot{1}} \bar{\theta}^{\dot{2}}
			    \bar{\vartheta}_{\dot{\beta}}
			      \sigma^{\mu\dot{\beta}}_\alpha\, \partial_\mu \overline{f_{(\beta)}(x)}
       - \theta^1\theta^2\bar{\theta}^{\dot{1}}\bar{\theta}^{\dot{2}}\,
	        \square \overline{f_{(0)}(x)}\,.
   \end{eqnarray*}

Consequently, (recall that $\square:= -\partial_0^2+\partial_1^2+\partial_2^2+\partial_3^2$\,)
 {\small
 \begin{eqnarray*}
   \breve{f}^\dag \breve{f} & =  & \breve{f} \breve{f}^\dag \\
    & = & \overline{f_{(0)}}(x)f_{(0)}(x)
	            + \sum_{\alpha}\theta^\alpha\vartheta_\alpha
                      \overline{f_{(0)}(x)}\, f_{(\alpha)}(x)
			    - \sum_{\dot{\beta}}\bar{\theta}^{\dot{\beta}}\bar{\vartheta}_{\dot{\beta}}
				      \overline{f_{(\beta)}(x)}f_{(0)}(x)
                + \theta^1\theta^2\vartheta_1\vartheta_2
				      \overline{f_{(0)}(x)}\, f_{(12)}(x)\\
    &&	+ \sum_{\alpha,\dot{\beta}}
	             \theta^\alpha\bar{\theta}^{\dot{\beta}}
                \left\{\rule{0ex}{1.2em}\right.\!
                  \sqrt{-1}\sum_\mu \sigma^\mu_{\alpha\dot{\beta}}
				   \mbox{\Large $($}
 				    \overline{f_{(0)}(x)}\, \partial_\mu f_{(0)}(x)
				       - \partial_\mu \overline{f_{(0)}(x)}\, f_{(0)}(x)
				   \mbox{\Large $)$}
					+ \vartheta_\alpha \bar{\vartheta}_{\dot{\beta}}
					    \overline{f_{(\beta)}(x)}\, f_{(\alpha)}(x)
                \!\left.\rule{0ex}{1.2em}\right\}	  \\				
    && +\; \bar{\theta}^{\dot{1}}\bar{\theta}^{\dot{2}}
	            \bar{\vartheta}^{\dot{1}}\bar{\vartheta}^{\dot{2}}\,
                  \overline{f_{(12)}(x)}\, f_{(0)}(x)				\\
	&&   + \sum_{\dot{\beta}}\theta^1\theta^2\bar{\theta}^{\dot{\beta}}
	             \left\{\rule{0ex}{1.2em}\right.\!	
				  \sqrt{-1}\sum_{\mu, \alpha}	
						  \vartheta_\alpha \sigma^{\mu\alpha}_{\;\;\;\;\dot{\beta}}	
                             \mbox{\Large $($}	
						     	 \overline{f_{(0)}(x)}\, \partial_\mu f_{(\alpha)}(x)
							      - \partial_\mu\overline{f_{(0)}(x)}\, f_{(\alpha)}(x)
						     \mbox{\Large $)$}		
			   -  \vartheta_1\vartheta_2\bar{\vartheta}_{\dot{\beta}}
			           \overline{f_{(\beta)}(x)}\, f_{(12)}(x)
				 \!\left.\rule{0ex}{1.2em}\right\}             \\[1.2ex]
	&&   + \sum_{\alpha}\theta^\alpha\bar{\theta}^{\dot{1}}\bar{\theta}^{\dot{2}}
	             \left\{\rule{0ex}{1.2em}\right.\!				
				  \sqrt{-1}\sum_{\mu, \dot{\beta}}			
						  \bar{\vartheta}_{\dot{\beta}}\sigma^{\mu\dot{\beta}}_{\alpha}
						      \mbox{\Large $($}		
							    - \partial_\mu \overline{f_{(\beta)}(x)}\, f_{(0)}(x)
							      + \overline{f_{(\beta)}(x)}\, \partial_\mu f_{(0)}(x)
							 \mbox{\Large $)$}
			   +\, \vartheta_\alpha \bar{\vartheta}_{\dot{1}}\bar{\vartheta}_{\dot{2}}
			           \overline{f_{(12)}(x)}\, f_{(\alpha)}(x)
				 \!\left.\rule{0ex}{1.2em}\right\}             \\[1.2ex]	
    && +\; \theta^1\theta^2\bar{\theta}^{\dot{1}}\bar{\theta}^{\dot{2}}
	             \left\{\rule{0ex}{1.2em}\right.\!
				   -\,\square\overline{f_{(0)}(x)}\cdot f_{(0)}(x)
				    -\, \overline{f_{(0)}(x)}\cdot\square f_{(0)}(x)
					+\,2 \sum_\mu
				             \partial_\mu \overline{f_{(0)}(x)}\,
							      \partial^\mu f_{(0)}(x)  \\
       && \hspace{6em}				
               +\, \sqrt{-1}\,
			        \sum_{\alpha,\dot{\beta}, \mu}	
					  \vartheta_\alpha\bar
					    {\vartheta}_{\dot{\beta}}\, \bar{\sigma}^{\mu,\dot{\beta}\alpha}\,
			        \mbox{\Large $($}
			          -\, \overline{f_{(\beta)}(x)}\, \partial_\mu f_{(\alpha)}(x)
                     + f_{(\alpha)}(x)\,\partial_\mu \overline{f_{(\beta)}(x)}
					\mbox{\Large $)$}					\\
    && 	\hspace{6em}
	         +\, \vartheta_1\vartheta_2\bar{\vartheta}_{\dot{1}}\bar{\vartheta}_{\dot{2}}
			        \overline{f_{(12)}(x)}\, f_{(12)}(x)
           	    \!\left.\rule{0ex}{1.2em}\right\}   \,;
 \end{eqnarray*}}  
     \marginpar{\vspace{-24.4em}\raggedright\tiny
         \raisebox{-1ex}{\hspace{-2.4em}\LARGE $\cdot$}Cf.\,[\,\parbox[t]{20em}{Wess
		\& Bagger:\\ Eq.\:(5.9)].}}

   \marginpar{\vspace{1em}\raggedright\tiny
         \raisebox{-1ex}{\hspace{-2.4em}\LARGE $\cdot$}Cf.\,[\,\parbox[t]{20em}{Wess
		\& Bagger:\\ Eq.\:(5.7)].}}
{\small
 \begin{eqnarray*}
  \breve{f}^2 & =
    & f_{(0)}(x)^2
	   + 2 \sum_\alpha\theta^\alpha\vartheta_\alpha
	           f_{(0)}(x) f_{(\alpha)}(x)
	   + 2\, \theta^1\theta^2\vartheta_1\vartheta_2
		      \mbox{\Large $($}
			   f_{(0)}(x) f_{(12)}(x)
			   - f_{(1)}(x) f_{(2)}(x)
			  \mbox{\Large $)$}\\	
    &&  +\, (\mbox{terms of $\bar{\theta}$-degree $\ge 1$})\, ;
 \end{eqnarray*}}     

  \marginpar{\vspace{1em}\raggedright\tiny
         \raisebox{-1ex}{\hspace{-2.4em}\LARGE $\cdot$}Cf.\,[\,\parbox[t]{20em}{Wess
		\& Bagger:\\ Eq.\:(5.8)].}}
{\small
 \begin{eqnarray*}
  \breve{f}^3 & =
    & f_{(0)}(x)^3
	   + 3 \sum_\alpha\theta^\alpha\vartheta_\alpha
	           f_{(0)}(x)^2 f_{(\alpha)}(x)
	   + 3\, \theta^1\theta^2\vartheta_1\vartheta_2
		      \mbox{\Large $($}
			   f_{(0)}(x)^2 f_{(12)}(x)
			   - 2\, f_{(0)}(x)f_{(1)}(x)f_{(2)}(x)
			  \mbox{\Large $)$}\\	
    &&  +\, (\mbox{terms of $\bar{\theta}$-degree $\ge 1$})\, ;
 \end{eqnarray*}}     
{\small
 \begin{eqnarray*}
  (\breve{f}^\dag)^2 & =
    & \overline{f_{(0)}}(x)^2
	   - 2 \sum_{\dot{\beta}}\bar{\theta}^{\dot{\beta}}\bar{\vartheta}_{\dot{\beta}}
	           \overline{f_{(0)}}(x)\overline{f_{(\beta)}}(x)
	   + 2\, \bar{\theta}^{\dot{1}}\bar{\theta}^{\dot{2}}
	            \bar{\vartheta}_{\dot{1}}\bar{\vartheta}_{\dot{2}}
		      \mbox{\Large $($}
			   \overline{f_{(0)}}(x) \overline{f_{(12)}}(x)
			   - \overline{f_{(1)}}(x) \overline{f_{(2)}}(x)
			  \mbox{\Large $)$}\\	
    &&  +\, (\mbox{terms of $\theta$-degree $\ge 1$})\, ;
 \end{eqnarray*}    } 
{\small
 \begin{eqnarray*}
  (\breve{f}^\dag)^3 & =
    & \overline{f_{(0)}}(x)^3
	   - 3 \sum_{\dot{\beta}}\bar{\theta}^{\dot{\beta}}\bar{\vartheta}_{\dot{\beta}}
	          \overline{f_{(0)}}(x)^2
			  \overline{f_{(\beta)}}(x)
	   + 3\, \bar{\theta}^{\dot{1}}\bar{\theta}^{\dot{2}}
	            \bar{\vartheta}_{\dot{1}}\bar{\vartheta}_{\dot{2}}
		      \mbox{\Large $($}
			   \overline{f_{(0)}}(x)^2    \overline{f_{(12)}}(x)
			   - 2\, \overline{f_{(0)}}(x)\overline{f_{(1)}}(x)
			           \overline{f_{(2)}}(x)
			  \mbox{\Large $)$}\\	
    &&  +\, (\mbox{terms of $\theta$-degree $\ge 1$})\,.
 \end{eqnarray*}(Cf.}     
 \![Wess \& Bagger: Eqs.\:(5.9), (5.7), (5.8)].)

\bigskip

\begin{flushleft}
{\bf The action functional of the Wess-Zumino model}
\end{flushleft}
The action functional of the {\it Wess-Zumino model}
 is given by:\\
 (caution that
     $(d\theta^2 d\theta^1)^\dag= d\bar{\theta}^{\dot{1}}d\bar{\theta}^{\dot{2}}
	    = - d\bar{\theta}^{\dot{2}}d\bar{\theta}^{\dot{1}}$)
\bigskip

 {\small
 \begin{eqnarray*}
  \lefteqn{
     S(\breve{f})\;
    :=\;  \int_Xd^4x\,
	         \left\{\rule{0ex}{1.2em}\right.
	           -\, \mbox{\large $\frac{1}{4}$}\,
			         \int  d\bar{\theta}^{\dot{2}} d\bar{\theta}^{\dot{1}} d\theta^2 d\theta^1\,
				         \mbox{\Large $($}  \breve{f}^\dag \breve{f} \mbox{\Large $)$}
				}\\
          && \hspace{7em}				
			+\, \int d\theta^2 d\theta^1\,
			        \mbox{\Large $($}
				            \lambda \breve{f}+ \mbox{\large $\frac{1}{2}$}m \breve{f}^2
                               + \mbox{\large $\frac{1}{3}$}g \breve{f}^3
					   \mbox{\Large $)$}
            - \int d\bar{\theta}^{\dot{2}} d\bar{\theta}^{\dot{1}}\,
			        \mbox{\Large $($}
				                     \bar{\lambda} \breve{f}^\dag
				                      + \mbox{\large $\frac{1}{2}$}\bar{m} (\breve{f}^\dag)^2
                                      + \mbox{\large $\frac{1}{3}$}\bar{g} (\breve{f}^\dag)^3
								   \mbox{\Large $)$}
		     \left.\rule{0ex}{1.2em}\right\}
	 \\
 && =\;
   \int_X d^4x
     \left\{\rule{0ex}{1.2em}\right.   	
	   \mbox{\large $\frac{1}{4}$}\,\square\overline{f_{(0)}(x)}\cdot f_{(0)}(x)
		+ \mbox{\large $\frac{1}{4}$}\,  \overline{f_{(0)}(x)}\cdot\square f_{(0)}(x)
		- \mbox{\large $\frac{1}{2}$} \sum_\mu
				  \partial_\mu \overline{f_{(0)}(x)}\, \partial^\mu f_{(0)}(x)
		 \\
       && \hspace{6em}				
       +\, \mbox{\large $\frac{\sqrt{-1}}{4}$}\,\sum_{\alpha,\dot{\beta}, \mu}	
	            \vartheta_\alpha\bar{\vartheta}_{\dot{\beta}}\cdot
				  \bar{\sigma}^{\mu\bar{\beta}\alpha}\,
			        \mbox{\Large $($}
			         \overline{f_{(\beta)}(x)}\,\partial_\mu f_{(\alpha)}(x)
                     - f_{(\alpha)}(x)\, \partial_\mu \overline{f_{(\beta)}(x)}
					\mbox{\Large $)$}					
	     \\
        && 	\hspace{6em}
	      -\, \mbox{\large $\frac{1}{4}$}
	         \vartheta_1\vartheta_2\bar{\vartheta}_{\dot{1}}\bar{\vartheta}_{\dot{2}}\cdot
			 \overline{f_{(12)}(x)}\, f_{(12)}(x)	 	 	 	 	 	
    \\					
  && \hspace{6em}
     +\, \vartheta_1\vartheta_2\cdot
	        \left(\rule{0ex}{1.2em}\right.
               m\,
			      \mbox{\Large $($}
			        f_{(0)}(x) f_{(12)}(x) - f_{(1)}(x) f_{(2)}(x)
			      \mbox{\Large $)$}    	
				 \\[-1ex]
			 && \hspace{12em}
             +\, g\,
		         \mbox{\Large $($}
			      f_{(0)}(x)^2 f_{(12)}(x)- 2\, f_{(0)}(x)f_{(1)}(x)f_{(2)}(x)
			     \mbox{\Large $)$}
              + \lambda f_{(12)}(x)   	
			\left.\rule{0ex}{1.2em}\right)	
	   \\
	&& \hspace{6em}
	 -\,  \bar{\vartheta}_{\dot{1}}\bar{\vartheta}_{\dot{2}} \cdot
	         \left(\rule{0ex}{1.2em}\right.
		       \bar{m}\,
		         \mbox{\Large $($}
			      \overline{f_{(0)}}(x)\, \overline{f_{(12)}}(x)
			        - \overline{f_{(1)}}(x)\, \overline{f_{(2)}}(x)
			     \mbox{\Large $)$}
				  \\[-1ex]
			 && \hspace{12em}
             +\, \bar{g}\,
		          \mbox{\Large $($}
			        \overline{f_{(0)}}(x)^2\, \overline{f_{(12)}}(x)
			          - 2\, \overline{f_{(0)}}(x)\,\overline{f_{(1)}}(x)\,\overline{f_{(2)}}(x)
			  \mbox{\Large $)$}		
              +   \bar{\lambda}\, \overline{f_{(12)}(x)}			
			\left.\rule{0ex}{1.2em}\right)	    						
	  \left.\rule{0ex}{1.2em}\right\}\,.
\end{eqnarray*}
}
It follows from Theorem~2.3
 that up to a space-time boundary term, this functional is supersymmetric.
After imposing the purge-evaluation map
$$
  \Pev\;:\;
    \vartheta_1\vartheta_2\;\rightsquigarrow\; 1\,,\;\;\;\;
	\vartheta_\alpha\bar{\vartheta}_{\dot{\beta}}\; \rightsquigarrow\; -1\,,\;\;\;\;
    \bar{\vartheta}_{\dot{1}}	\bar{\vartheta}_{\dot{2}}\; \rightsquigarrow\; -1\,,\;\;\;\;
	 \vartheta_1\vartheta_2 \bar{\vartheta}_{\dot{1}}	\bar{\vartheta}_{\dot{2}}\;
	      \rightsquigarrow\; -1\,.
 $$
 to remove the even nilpotent factors
  $\vartheta_1\vartheta_2,\,
	\vartheta_\alpha\bar{\vartheta}_{\dot{\beta}},\,
    \bar{\vartheta}_{\dot{1}}	\bar{\vartheta}_{\dot{2}},\,
	 \vartheta_1\vartheta_2 \bar{\vartheta}_{\dot{1}}	\bar{\vartheta}_{\dot{2}}$	
 in the expression and integration by parts to fix the kinetic terms,
 $S(\breve{f})$ becomes
%
 %
  \marginpar{\vspace{7em}\raggedright\tiny
         \raisebox{-1ex}{\hspace{-2.4em}\LARGE $\cdot$}Cf.\,[\,\parbox[t]{20em}{Wess
		\& Bagger:\\ Eq.\:(5.11)].}}
 {\small
 \begin{eqnarray*}
  \lefteqn{
     S(\breve{f})\;
    :=\;  \int_Xd^4x\,
	         \Pev \left\{\rule{0ex}{1.2em}\right.
	           -\, \mbox{\large $\frac{1}{4}$}\,
			         \int  d\bar{\theta}^{\dot{2}} d\bar{\theta}^{\dot{1}} d\theta^2 d\theta^1\,
				         \mbox{\Large $($}  \breve{f}^\dag \breve{f} \mbox{\Large $)$}
				}\\
          && \hspace{7em}				
			+\, \int d\theta^2 d\theta^1\,
			        \mbox{\Large $($}
				            \lambda \breve{f}+ \mbox{\large $\frac{1}{2}$}m \breve{f}^2
                               + \mbox{\large $\frac{1}{3}$}g \breve{f}^3
					   \mbox{\Large $)$}
            - \int d\bar{\theta}^{\dot{2}} d\bar{\theta}^{\dot{1}}\,
			        \mbox{\Large $($}
				                     \bar{\lambda} \breve{f}^\dag
				                      + \mbox{\large $\frac{1}{2}$}\bar{m} (\breve{f}^\dag)^2
                                      + \mbox{\large $\frac{1}{3}$}\bar{g} (\breve{f}^\dag)^3
								   \mbox{\Large $)$}
		     \left.\rule{0ex}{1.2em}\right\}
	 \\
 && =\;
   \int_X d^4x
     \left\{\rule{0ex}{1.2em}\right.   	
	   \overline{f_{(0)}(x)}\cdot\square f_{(0)}(x)		
       +  \mbox{\large $\frac{\sqrt{-1}}{2}$}\,\sum_{\alpha,\dot{\beta}, \mu}\,	
				  \bar{\sigma}^{\mu\bar{\beta}\alpha}\,
			        f_{(\alpha)}(x)\, \partial_\mu \overline{f_{(\beta)}(x)}			
	   + \mbox{\large $\frac{1}{4}$}\,
			 \overline{f_{(12)}(x)}\, f_{(12)}(x)	 	 	 	 	 	
    \\					
  && \hspace{6em}
     +\, \left(\rule{0ex}{1.2em}\right.
               m\,
			      \mbox{\Large $($}
			        f_{(0)}(x) f_{(12)}(x) - f_{(1)}(x) f_{(2)}(x)
			      \mbox{\Large $)$}    	
				 \\[-1ex]
			 && \hspace{12em}
             +\, g\,
		         \mbox{\Large $($}
			      f_{(0)}(x)^2 f_{(12)}(x)- 2\, f_{(0)}(x)f_{(1)}(x)f_{(2)}(x)
			     \mbox{\Large $)$}
              + \lambda f_{(12)}(x)   	
			\left.\rule{0ex}{1.2em}\right)	
	   \\
	&& \hspace{6em}
	 +\,  \left(\rule{0ex}{1.2em}\right.
		       \bar{m}\,
		         \mbox{\Large $($}
			      \overline{f_{(0)}}(x)\, \overline{f_{(12)}}(x)
			        - \overline{f_{(1)}}(x)\, \overline{f_{(2)}}(x)
			     \mbox{\Large $)$}
				  \\[-1ex]
			 && \hspace{12em}
             +\, \bar{g}\,
		          \mbox{\Large $($}
			        \overline{f_{(0)}}(x)^2\, \overline{f_{(12)}}(x)
			          - 2\, \overline{f_{(0)}}(x)\,\overline{f_{(1)}}(x)\,\overline{f_{(2)}}(x)
			  \mbox{\Large $)$}		
              +   \bar{\lambda}\, \overline{f_{(12)}(x)}			
			\left.\rule{0ex}{1.2em}\right)	
           \\
	&& \hspace{6em}         			
	  +\; \mbox{(space-time boundary terms)}
	  \left.\rule{0ex}{1.2em}\right\}\,.
\end{eqnarray*}
}This 
is [Wess \& Bagger: Eq.\:(5.11)] in the setting of Sec.\:1 and Sec.\:2.

The component field $f_{(12)}$ (and hence $\overline{f_{(12)}}$)
  has no kinetic term and thus is non-dynamical.
It can be removed from the action functional
 by solving its equations of motion from $S(\breve{f})$
 $$
  \begin{array}{rcl}
   \mbox{\large $\frac{1}{4}$}\, f_{(12)}(x)
    + \bar{m}\,\overline{f_{(0)}(x)}+ \bar{g}\,\overline{f_{(0)}(x)}^{\,2}
	+ \bar{\lambda}  & = & 0\,,
	\\[2ex]
  \mbox{\large $\frac{1}{4}$}\, \overline{f_{(12)}(x)}
    + m\,f_{(0)}(x) +  g\,f_{(0)}(x)^2
	+ \lambda    & = & 0
  \end{array}
 $$
 and plugging back into $S(\breve{f})$.
This gives another form of the action functional that involves only dynamical component fields:

{\small
 \begin{eqnarray*}
  \lefteqn{
    S(f_{(0)}, (f_{\alpha})_{\alpha})
      }\\
  && :=\;
   \int_X d^4x
     \left\{\rule{0ex}{1.2em}\right.   	
	   \overline{f_{(0)}(x)}\cdot\square f_{(0)}(x)		
       +  \mbox{\large $\frac{\sqrt{-1}}{2}$}\,\sum_{\alpha,\dot{\beta}, \mu}\,	
				  \bar{\sigma}^{\mu\bar{\beta}\alpha}\,
			        f_{(\alpha)}(x)\, \partial_\mu \overline{f_{(\beta)}(x)}				   	 	 	
      \\					
  && \hspace{6em}
        -\, \mbox{\Large $($}
              m\, f_{(1)}(x) f_{(2)}(x)			     		
              + 2g\, f_{(0)}(x)f_{(1)}(x)f_{(2)}(x)			
	   \\[.6ex]
 	&& \hspace{12em}
	   +\, \bar{m}\, \overline{f_{(1)}}(x)\, \overline{f_{(2)}}(x)			  			
               + 2\,\bar{g}\, 		
			        \overline{f_{(0)}}(x)\,\overline{f_{(1)}}(x)\,\overline{f_{(2)}}(x)
			\mbox{\Large $)$}
           \\[1ex]
    && \hspace{6em}
     -\,4\,
	    \mbox{\Large $($}
	      m\,f_{(0)}(x) +  g\,f_{(0)}(x)^2 + \lambda
		\mbox{\Large $)$}
		 \cdot
		\mbox{\Large $($}
		 \bar{m}\,\overline{f_{(0)}(x)}+ \bar{g}\,\overline{f_{(0)}(x)}^{\,2}
	       + \bar{\lambda}
        \mbox{\Large $)$}		
	 \\[1ex]
	&& \hspace{6em}         			
	  +\, \mbox{(space-time boundary terms)}
	  \left.\rule{0ex}{1.2em}\right\}\,.
\end{eqnarray*}}(Cf.\:[Wess \& Bagger: Eq.\:(5.13)].) 
  \marginpar{\vspace{-13em}\raggedright\tiny
         \raisebox{-1ex}{\hspace{-2.4em}\LARGE $\cdot$}Cf.\,[\,\parbox[t]{20em}{Wess
		\& Bagger:\\ Eq.\:(5.13)].}}

\bigskip
 
\section{Supersymmetric $U(1)$ gauge theory with matter on $X$ in terms of $\widehat{X}^{\widehat{\boxplus}}$\;
 {\large (cf.\:[Wess \& Bagger: Chapter VI and Chapter VII, $U(1)$ part])}}

 We reconstruct in this section
  \begin{itemize}
   \item[\LARGE $\cdot$] [Wess \& Bagger:
               \parbox[t]{30em}{Chapter VI.\:{\it Vector superfields} \; and\\
                                                 Chapter VII.\:{\it Gauge invariant interactions}, $U(1)$ part\,]}
  \end{itemize}
 in the complexified ${\Bbb Z}/2$-graded $C^\infty$-Algebraic Geometry setting of Sec.\:1.

\bigskip
 
\subsection{Vector superfields and their associated (even left) connection}
		
Since all the bundles and sheaves involved in the construction of supersymmetric gauge theories in the current notes
 are trivialized, we will directly take a connection as a differential operator acting on
 $C^\infty(\widehat{X}^{\widehat{\boxplus}})$ to keep our focus on [Wess \& Bagger].
Readers are referred to [L-Y5: Sec.\:3.1] (SUSY(1)=D(14.1.Supp.1)) and references ibidem
 for more words on connections and gauge theories in the superworld.

\bigskip

\begin{definition} {\bf [even left connection on
                                             $C^\infty(\widehat{X}^{\widehat{\boxplus}})$-module]}\;
{\rm (Cf.\;[L-Y4:Definition 2.1.2] (D(14.1)).)
 Let $\widehat{M}$ be an $C^\infty(\widehat{X}^{\widehat{\boxplus}})$-module.
 An {\it even left connection} $\widehat{\nabla}$ on $\widehat{M}$
    is a ${\Bbb C}$-bilinear pairing	
  $$
    \begin{array}{ccccc}
	 \widehat{\nabla} & : & \Der_{\Bbb C}(\widehat{X}^{\widehat{\boxplus}}) \times \widehat{M}
	     & \longrightarrow   & \widehat{M}  \\[1.2ex]
    && (\xi, s)              &  \longmapsto    &  \widehat{\nabla}\!_{\xi}s		 	
	\end{array}
  $$
  such that
	\begin{itemize}
	 \item[(1)]  [{\it $C^\infty(\widehat{X}^{\widehat{\boxplus}})$-linearity
	                                   in the $\Der_{\Bbb C}(\widehat{X}^{\widehat{\boxplus}})$-argument}]\\[.6ex]	
	  $\mbox{\hspace{1em}}$
	  $\widehat{\nabla}\!_{f_1\xi_1 + f_2\xi_2}s\;
	     =\; f_1 \widehat{\nabla}\!_{\xi_1}s + f_2 \widehat{\nabla}\!_{\xi_2}s$, \hspace{1em}
      for $f_1, f_2 \in C^\infty(\widehat{X}^{\widehat{\boxplus}})$,
	       $\xi_1, \xi_2 \in \Der_{\Bbb C}(\widehat{X}^{\widehat{\boxplus}})$,  and
	       $s\in \widehat{M}$;

     \item[(2)]  [{\it ${\Bbb C}$-linearity in the $\widehat{M}$-argument}]\\[.6ex]
	 $\mbox{\hspace{1em}}$
	 $\widehat{\nabla}\!_\xi(c_1s_1+c_2s_2)\;
	     =\;  c_1 \widehat{\nabla}\!_\xi s_1 + c_2 \widehat{\nabla}\!_\xi s_2$, \hspace{1em}
	  for $c_1, c_2\in {\Bbb C}$,
	       $\xi\in \Der_{\Bbb C}(\widehat{X}^{\widehat{\boxplus}})$, and
	       $s_1, s_2\in \widehat{M}$;
	
	 \item[(3)] [{\it ${\Bbb Z}/2$-graded Leibniz rule in the $\widehat{M}$-argument}]\footnote{In
	                                       [L-Y4: Definition 2.1.2 ] (D(14.1)), a left connection on $\widehat{\cal E}$ is required to satisfy
	                                       the  generalized ${\Bbb Z}/2$-graded Leibniz rule in the $\widehat{\cal E}$-argument:\;
	                                         $\widehat{\nabla}_\xi(fs)\;
	                                            =\; (\xi f)s
	                                              + (-1)^{p(f)p(\xi)}\,f\cdot\,\!^{\varsigma_{\!f}}\!
											                    (\widehat{\nabla})_\xi s$,
                                                for $f\in \widehat{\cal O}_X^{\,\widehat{\boxplus}}$,
												      $\xi\in {\cal T}_{\widehat{X}^{\widehat{\boxplus}}}$ parity homogeneous
	                                                  and $s\in\widehat{\cal E}$,
                                             where
	                                               $^{\varsigma_{\!f}}\!(\widehat{\nabla})$
                                                      is the parity-conjugation of $\widehat{\nabla}$ induced by $f$;
                                                  i.e., 	
	                                               $^{\varsigma_{\!f}}\!(\widehat{\nabla})
	                                                      = \widehat{\nabla}$,  if $f$ is even, or
		                                           $\,\!^{\varsigma}\widehat{\nabla}
			                                                := \mbox{(even part of $\widehat{\nabla}$)}\,
					                                          -\, \mbox{(odd part of $\widehat{\nabla}$)}$  if $f$ is odd;
                                               (cf.\ [L-Y4: Definition~1.3.1] (D(14.1))).
											When $\widehat{\nabla}$ is even,
											    $^{\varsigma_{\!f}}\!(\widehat{\nabla})=\widehat{\nabla}$ always
											 and the general ${\Bbb Z}/2$-graded Leibniz rule reduces to
 											  the ${\Bbb Z}$-graded Leibniz rule.											  
												                                                                     }\\[.6ex]   
	 $\mbox{\hspace{1em}}$
	 $\widehat{\nabla}\!_\xi(fs)\;
	   =\; (\xi f)s
	           + (-1)^{p(f)p(\xi)}\,f\cdot  \widehat{\nabla}\!_\xi s$,\\[.6ex]
      for $f\in C^\infty(\widehat{X}^{\widehat{\boxplus}})$,
	       $\xi\in \Der_{\Bbb C}(\widehat{X}^{\widehat{\boxplus}})$ parity homogeneous
	       and $s\in\widehat{M}$.
	\end{itemize}
  As an operation on the pairs $(\xi, s)$,
   a connection $\nabla$ on $\widehat{M}$ is applied to $\xi$ from the right
   while applied to $s$ from the left;\footnote{In the ${\Bbb Z}/2$-graded world,
                                                                     it is instructive to denote $\widehat{\nabla}_{\!\xi}s$ as
																	  $\xi \widehat{\nabla} s$ or $_{\xi}\!\widehat{\nabla} s$
																	 (though we do not adopt it as a regularly used notation in this work).
																	In particular,
  																	  from $_{f\xi}\!\widehat{\nabla} s$
																	  to $f (\,\!_{\xi}\!\widehat{\nabla} s)$,
                                                                     $f$ and $\widehat{\nabla}$ do {\it not} pass each other.
                                                                      }  
  cf.\ [L-Y4: Lemma~1.3.7 \& Remark~1.3.8] (D(14.1)).
}\end{definition}

\bigskip

Note that since $\widehat{\nabla}$ is even, the parity of $\widehat{\nabla}\!_\xi$ is the same as that of $\xi$.

\bigskip

\begin{lemma-definition} {\bf [curvature tensor of (even left) connection]}\;\\ {\rm
 (Cf.\ [L-Y4: Lemma/Definition~2.1.9] (D(14.1)).)
 Continuing Definition~4.1.1.
 {\it Let $\widehat{\nabla}$ be an even left connection on $\widehat{M}$.
 Then the correspondence
  $$
   F^{\widehat{\nabla}}\;:\;
   (\xi_1,\xi_2;  s)\;\longmapsto\;
        \mbox{\Large $($}
		  [\widehat{\nabla}_{\!\xi_1}, \widehat{\nabla}_{\!\xi_2}\}\,
	         -\,\widehat{\nabla}_{[\xi_1,\xi_2\}}
		\mbox{\Large $)$}\, s\,,
  $$
  for $\xi_1,\xi_2\in \Der_{\Bbb C}(\widehat{X}^{\widehat{\boxplus}})$ parity-homogeneous and
       $s\in \widehat{M}$,
 satisfies the following tensorial property on $\widehat{X}^{\widehat{\boxplus}}$:
  \begin{eqnarray*}
    \mbox{\Large $($}
     [\widehat{\nabla}\!_{f\xi_1}, \widehat{\nabla}\!_{\xi_2}   \}
        -\widehat{\nabla}\!_{[f\xi_1,\xi_2\}}	
    \mbox{\Large $)$}s
     & = &  f\cdot
       \mbox{\Large $($}
         [\widehat{\nabla}\!_{\xi_1}, \widehat{\nabla}\!_{\xi_2}   \}
        -\widehat{\nabla}\!_{[\xi_1,\xi_2\}}	
	   \mbox{\Large $)$}s\,,                  \\[1.2ex]
    \mbox{\Large $($}
     [\widehat{\nabla}\!_{\xi_1}, \widehat{\nabla}\!_{f\xi_2}   \}
        -\widehat{\nabla}\!_{[\xi_1,f\xi_2\}}	
    \mbox{\Large $)$}s
     & = &  (-1)^{p(f)p(\xi_1)}\,f\cdot
       \mbox{\Large $($}
         [\widehat{\nabla}\!_{\xi_1}, \widehat{\nabla}\!_{\xi_2}   \}
        -\widehat{\nabla}\!_{[\xi_1,\xi_2\}}	
	   \mbox{\Large $)$}s\,,                  \\[1.2ex]
    \mbox{\Large $($}
     [\widehat{\nabla}\!_{\xi_1}, \widehat{\nabla}\!_{\xi_2}   \}
        -\widehat{\nabla}\!_{[\xi_1,\xi_2\}}	
    \mbox{\Large $)$}(fs)
     & = & (-1)^{p(f)(p(\xi_1)+ p(\xi_2))} f\cdot
       \mbox{\Large $($}
         [\widehat{\nabla}\!_{\xi_1}, \widehat{\nabla}\!_{\xi_2}   \}
        -\widehat{\nabla}\!_{[\xi_1,\xi_2\}}	
	   \mbox{\Large $)$}s\,,
  \end{eqnarray*}
  for $f\in C^\infty(\widehat{X}^{\widehat{\boxplus}})$,
       $\xi_1$, $\xi_2\in \Der_{\Bbb C}(\widehat{X}^{\widehat{\boxplus}})$,
	   $s\in \widehat{M}$
        all parity homogeneous.
   }
 {\rm $F^{\widehat{\nabla}}$ is called
     the {\it curvature tensor} on $\widehat{X}^{\widehat{\boxplus}}$
	   associated to the even left connection $\widehat{\nabla}$ on $\widehat{M}$.}			
}\end{lemma-definition}

\medskip

\begin{proof}
 This is a special case of [L-Y4: Lemma/Definition~2.1.9] (D(14.1)) with the odd part of $\widehat{\nabla}$ vanishes.

\end{proof}

\bigskip

Since in this work, we only address even left connections, we will simply call them {\it connections}.

\bigskip

\begin{definition} {\bf [vector superfield]}\; {\rm
 A {\it vector superfield} $\breve{V}$ is an element in $C^\infty(\widehat{X}^{\widehat{\boxplus}})$
  of the following form
 {\small
 \begin{eqnarray*}
  \breve{V}
  &= &
    V_{(0)}
	+ \sum_{\alpha}\theta^\alpha\vartheta_\alpha V_{(\alpha)}
	+ \sum_{\dot{\beta}}
	       \bar{\theta}^{\dot{\beta}}\bar{\vartheta}_{\dot{\beta}}V_{(\dot{\beta})}  		
    + \theta^1\theta^2\vartheta_1\vartheta_2 V_{(12)}
	+ \sum_{\alpha,\dot{\beta}, \mu}\theta^\alpha \bar{\theta}^{\dot{\beta}}
		    \sigma^\mu_{\alpha\dot{\beta}} V_{[\mu]}		
    + \bar{\theta}^{\dot{1}}\bar{\theta}^{\dot{2}}
	   \bar{\vartheta}_{\dot{1}}\bar{\vartheta}_{\dot{2}} V_{(\dot{1}\dot{2})}  \\
  && \hspace{1em}		
	+ \sum_{\dot{\beta}}\theta^1\theta^2\bar{\theta}^{\dot{\beta}}
	     \mbox{\Large $($}
		   \sqrt{-1}
		    \sum_{\alpha,\mu }
		          \vartheta_\alpha\,\sigma^{\mu \alpha}_{\;\;\;\;\dot{\beta}}\,
				    \partial_\mu V_{(\alpha)}
          + \bar{\vartheta}_{\dot{\beta}}\,
		       V^\prime_{(\dot{\beta})}					
         \mbox{\Large $)$}					
     \\
    && \hspace{1em}	
    +\, \sum_\alpha \theta^\alpha\bar{\theta}^{\dot{1}}\bar{\theta}^{\dot{2}}
	     \mbox{\Large $($}
		  \vartheta_\alpha\, V^{\prime\prime}_{(\alpha)}		
		  + \sqrt{-1}
	          \sum_{\dot{\beta}, \mu}
			      \bar{\vartheta}_{\dot{\beta}}\,
		          \sigma^{\mu\dot{\beta}}_\alpha 		
			      \partial_\mu V_{(\dot{\beta})}
	     \mbox{\Large $)$}
	+ \theta^1\theta^2\bar{\theta}^{\dot{1}}\bar{\theta}^{\dot{2}}\,
		  V^\sim_{(0)}		
 \end{eqnarray*}
  }that 
 satisfies the realness condition
     \marginpar{\vspace{0em}\raggedright\tiny
         \raisebox{-1ex}{\hspace{-2.4em}\LARGE $\cdot$}Cf.\,[\,\parbox[t]{20em}{Wess
		\& Bagger:\\ Eq.\:(6.1)].}}
  $$
    \breve{V}^{\,\dag} \;=\; \breve{V}\,.
  $$
 Here, $\breve{V} $ is expressed in coordinate functions
   $(x, \theta, \bar{\theta}, \vartheta, \bar{\vartheta})$ of $\widehat{X}^{\widehat{\boxplus}}$
  with the components\\ $V^{\tinybullet}_{\tinybullet}= V^{\tinybullet}_{\tinybullet}(x)$.
 The set of vector superfields in $C^\infty(\widehat{X}^{\widehat{\boxplus}})$
   form a $C^\infty(X)$-module.\footnote{In
                                            [L-Y5: Sec.\:3] (SUSY (1)),  we took the attitude that all physics-related superfields should
											 be connected to elements in the subring of $C^\infty(\widehat{X}^{\widehat{\boxplus}})$
											  generated by the small chiral elements and the small antichiral elements in
											  $C^\infty(\widehat{X}^{\widehat{\boxplus}})$.
										    Under such a restriction,
											   the most natural definition for the candidate for physicists' vector superfields
											   has more degrees of freedom
											  than that defined by physicists, particularly in [Wess \& Bagger].
											To distinguish them, we call it  {\it pre-vector superfield} in [L-Y5] (SUSY (1))
											  and had to introduce linear constraints to reduce a pre-vector superfield to a vector superfield.
											Surprisingly, that still defines a supersymmetric gauge theory mimicking [Wess \& Bagger].
                                           For the current work (SUSY (2.1)), we find a better picture.
										   In [L-Y5] (SUSY(1)), we look for vector superfields in the ring generated by chiral superfields
										       and antichiral superfields.
											Under this constraints there is no other choice to make than what is done in [L-Y5] (SUSY (1)).
                                            Why do we need to require vector superfields lie in the this ring?
                                           After all, unlike chiral superfields or antichiral superfields that form a ring,
										    vector superfields do not form a ring:
											 the multiplication of two vector superfields in general is no longer a vector superfield.
											 If we choose them still in the grand function-ring of the towered superspace but
											 not require them to lie in the above subring, then what shall we get?
											It is to answer this question that leads us in the end to the setting presented here.
											It matches now with [Wess \& Bagger].
                                            }     
}\end{definition}
	
\bigskip

\noindent
Explicitly, a vector superfield can be expressed as
{\small
 \begin{eqnarray*}
  \breve{V}
  &= &
    V_{(0)}
	+ \sum_{\alpha}\theta^\alpha\vartheta_\alpha V_{(\alpha)}
	- \sum_{\dot{\beta}}
	       \bar{\theta}^{\dot{\beta}}\bar{\vartheta}_{\dot{\beta}}\,  \overline{V_{(\beta)}}
    + \theta^1\theta^2\vartheta_1\vartheta_2 V_{(12)}
	+ \sum_{\alpha,\dot{\beta}, \mu}\theta^\alpha \bar{\theta}^{\dot{\beta}}
		  \sigma^\mu_{\alpha\dot{\beta}} V_{[\mu]}		
    + \bar{\theta}^{\dot{1}}\bar{\theta}^{\dot{2}}
	   \bar{\vartheta}_{\dot{1}}\bar{\vartheta}_{\dot{2}} \, \overline{V_{(12)}}  \\
  && \hspace{1em}		
	+ \sum_{\dot{\beta}}\theta^1\theta^2\bar{\theta}^{\dot{\beta}}
	     \mbox{\Large $($}
		   \sqrt{-1}
		    \sum_{\alpha,\mu }
		          \vartheta_\alpha\,\sigma^{\mu \alpha}_{\;\;\;\;\dot{\beta}}\,
				    \partial_\mu V_{(\alpha)}
          + \bar{\vartheta}_{\dot{\beta}}\,
		        \overline{V^{\prime\prime}_{(\beta)}}					
         \mbox{\Large $)$}					
     \\
    && \hspace{1em}	
    +\, \sum_\alpha \theta^\alpha\bar{\theta}^{\dot{1}}\bar{\theta}^{\dot{2}}
	     \mbox{\Large $($}
		  \vartheta_\alpha\, V^{\prime\prime}_{(\alpha)}
		  - \sqrt{-1}
	         \sum_{\dot{\beta}, \mu}
			  \bar{\vartheta}_{\dot{\beta}}\,
		      \sigma^{\mu\dot{\beta}}_\alpha 		
			  \partial_\mu  \overline{V_{(\beta)}}
	     \mbox{\Large $)$}
	+ \theta^1\theta^2\bar{\theta}^{\dot{1}}\bar{\theta}^{\dot{2}}\,
		  V^\sim_{(0)}		
 \end{eqnarray*}
}
     \marginpar{\vspace{-7em}\raggedright\tiny
         \raisebox{-1ex}{\hspace{-2.4em}\LARGE $\cdot$}Cf.\,[\,\parbox[t]{20em}{Wess
		\& Bagger:\\ Eq.\:(6.2)].}}
with
$$
  V_{(0)}\,,\;\; V_{[\mu]}\,,\;\; V^\sim_{(0)}\;\in\; C^\infty(X)\,,\;
  \mbox{i.e.\: real-valued}\,.
$$
Thus, while a small (resp.\ chiral/antichiral) superfield has
  $33$ (resp.\ $4$) independent complex components in $C^\infty(X)^{\Bbb C}$,
 a vector superfield $\breve{V}$ has $16$ independent real-valued components from
  $V_{(0)}$, $V_{[\mu]}$, $V^\sim_{(0)}\in C^\infty(X)$, $\mu=0,1,2,3$, and
  $V_{(\alpha)}$, $V_{(12)}$, $V^{\prime\prime}_{(\gamma)}\in C^\infty(X)^{\Bbb C}$,
    $\alpha, \gamma=1, 2$.
Note that by definition
 a vector superfield is contained in $C^\infty(\widehat{X}^{\widehat{\boxplus}})^\mediumscriptsize$,
  but in general not in $C^\infty(\widehat{X}^{\widehat{\boxplus}})^\smallscriptsize$.
		
\bigskip

For physicists working on supersymmetric gauge theories,
 the following class of even left connections (adapted to the current $U(1)$ case) is the major concern.

\bigskip

\begin{definition} {\bf [(left, even) connection associated to vector superfield]}\; {\rm
With the above setting,
   let $\breve{V}\in C^\infty(\widehat{X}^{\widehat{\boxplus}})$ be a vector superfield on $X$.
 Then, one can define
   an (even left) connection $\widehat{\nabla}^{\breve{V}}$
   on $C^\infty(\widehat{X}^{\widehat{\boxplus}})$ (as a left $C^\infty(\widehat{X}^{\widehat{\boxplus}})$-module)
   associated to $\breve{V}$ as follows.
 \begin{itemize}
  \item[(1)]
  Firstly, we acquire the compatibility with the chiral structure on $C^\infty(\widehat{X}^{\widehat{\boxplus}})$
   by setting
     $$
	    \widehat{\nabla}^{\breve{V}}_{e_{\beta^{\prime\prime}}}\;
		 :=\;  e_{\beta^{\prime\prime}}\,.
	 $$
 
  \item[(2)]
  Secondly, we set\footnote{The
                                             	 choice of using whether
												   $ e^{-\breve{V}}\circ e_{\alpha^\prime}\circ e^{\breve{V}}$ or
												   $e^{\breve{V}}\circ e_{\alpha^\prime}\circ e^{-\breve{V}}$
												   as the definition of $\widehat{\nabla}^{\breve{V}}_{e_{\alpha^\prime}}$
												   is dictated by how one would construct the action functional
												   for the gauge-invariant kinetic term for the chiral superfield
												   in the supersymmetric $U(1)$-gauge theory with matter.
												 The former is consistent with the setting in Sec.\,4.4
												   while the latter isn't.
												 Cf.\:Lemma~4.2.6 vs.\;Sec.\,4.4, theme
												           '{\it Explicit computations/formulae for $\breve{V}$ in Wess-Zumino gauge}'.
											                                                       } 
     $$
	   \widehat{\nabla}^{\breve{V}}_{e_{\alpha^\prime}}\;
	    :=\;  e^{-\breve{V}}\circ e_{\alpha^\prime}\circ e^{\breve{V}}\;
		  =\; e_{\alpha^\prime}\,+\, e^{-\breve{V}}(e_{\alpha^\prime}e^{\breve{V}})\,.
     $$
  Thus, in a way $\breve{V}$ is an indication of the twisting of the original antichiral structure of
    $C^\infty(\widehat{X}^{\widehat{\boxplus}})$
    to the one selected by $\widehat{\nabla}_{e_{\alpha^\prime}}$.
   
  \item[(3)]
  Finally, we set
   $$
      \widehat{\nabla}^{\breve{V}}_{e_{\mu}} \; =\;
	   -\, \mbox{\Large $\frac{\sqrt{-1}}{4}$}
	    \sum_{\dot{\beta}, \alpha} \bar{\sigma}_\mu^{\dot{\beta}\alpha}
		   \{\widehat{\nabla}^{\breve{V}}_{e_{\alpha^\prime}},
		        \widehat{\nabla}^{\breve{V}}_{e_{\beta^{\prime\prime}}}\}\,,
   $$
   where
	  $\bar{\sigma}_\mu
	     = (\bar{\sigma}_\mu^{\dot{\beta}\alpha})_{\dot{\beta}\alpha}
	     = \sum_\nu \eta_{\mu\nu}\bar{\sigma}^\nu$
	  satisfies
	 $$
	   \sum_{\alpha, \dot{\beta}} \bar{\sigma}_\mu^{\dot{\beta}\alpha} \sigma^\nu_{\alpha\dot{\beta}}\;
	    =\; -\,2 \delta_\mu^\nu\,,\;\;\;\;
		\mbox{for $\mu, \nu\,= 0,\,1,\,2,\,3$}\,.
	 $$		
     Explicitly,
	 {\footnotesize
     $$
      \bar{\sigma}_0\;:=\;
        \left[\!\begin{array}{rr}  1 & 0 \\ 0 & 1\end{array}\!\right]\!,\;\;
	  \bar{\sigma}_1\;:=\;
        \left[\!\begin{array}{rr} 0 & -1 \\  -1 & 0\end{array}\!\right]\!,\;\;
      \bar{\sigma}_2\;:=\;
	    \left[\!\begin{array}{rr} 0 & \sqrt{-1} \\ -\sqrt{-1} & 0\end{array}\!\right]\!,\;\;	
      \bar{\sigma}_3\;:=\;
        \left[\!\begin{array}{rr} -1 & 0 \\ 0 & 1\end{array}\!\right]\,.
     $$}
  This is indeed a flatness condition on the curvature of $\widehat{\nabla}^{\breve{V}}$
	 in the fermionic directions $(e_{\alpha^\prime}, e_{\beta^{\prime\prime}})$.
   (Cf.\;Lemma~4.1.5 for the precise statement.)
 \end{itemize}
 Since $\breve{V}$ is even, $\widehat{\nabla}^{\breve{V}}$ as defined is even as well.
 In this way
   a vector superfield $\breve{V}\in C^\infty(\widehat{X}^{\widehat{\boxplus}})$ determines an even left connection
   $\widehat{\nabla}^{\breve{V}}$ on $C^\infty(\widehat{X}^{\widehat{\boxplus}})$.
 $\widehat{\nabla}^{\breve{V}}$ is called
  the {\it connection} on $C^\infty(\widehat{X}^{\widehat{\boxplus}})$ {\it associated to $\breve{V}$};  					
 For simplicity of notations, we often denote $\widehat{\nabla}^{\breve{V}}$ by $\widehat{\nabla}$,
  keeping $\breve{V}$ implicit.
}\end{definition}

\medskip

\begin{lemma} {\bf [flatness of $\nabla^{\breve{V}}$ along fermionic directions]}\;
 Let $\widehat{\nabla}=\widehat{\nabla}^{\breve{V}}$ be the connection
   on $C^\infty(\widehat{X}^{\widehat{\boxplus}})$ associated to a vector superfield $\breve{V}$.
 Let $F^{\widehat{\nabla}}$ be the curvature $2$-tensor of $\widehat{\nabla}$ and
 denote
  $F^{\widehat{\nabla}}(e_{\alpha^\prime}, e_{\beta^\prime})$
  (resp.\
      $F^{\widehat{\nabla}}(e_{\alpha^{\prime\prime}}, e_{\beta^{\prime\prime}})$,
	  $F^{\widehat{\nabla}}(e_{\alpha^\prime}, e_{\beta^{\prime\prime}})$)
 by
  $F^{\widehat{\nabla}}_{\alpha^\prime\beta^\prime}$
   (resp.\  $F^{\widehat{\nabla}}_{\alpha^{\prime\prime}\beta^{\prime\prime}}$,
                  $F^{\widehat{\nabla}}_{\alpha^\prime\beta^{\prime\prime}}$).		
 Then
   with respect to the supersymmetrically invariant coframe $(e^I)_I$ on $\widehat{X}$,
  the components of the curvature tensor $F^{\widehat{\nabla}}$ of $\widehat{\nabla}$
  in purely fermionic directions all vanish:
  For $\alpha^\prime, \beta^\prime=1^\prime, 2^\prime$ and
       $\alpha^{\prime\prime}, \beta^{\prime\prime}=1^{\prime\prime}, 2^{\prime\prime}$,
   $$
     F^{\widehat{\nabla}}_{\alpha^{\prime}\beta^{\prime}}\;
      =\; F^{\widehat{\nabla}}_{\alpha^{\prime\prime}\beta^{\prime\prime}}\;
      =\; F^{\widehat{\nabla}}_{\alpha^{\prime}\beta^{\prime\prime}}\;=\; 0\,.
   $$
\end{lemma}

\medskip

\begin{proof}
 See [L-Y5: proof of Lemma 3.1.9] (SUSY(1)).

\end{proof}

\bigskip

From the perspective of a physicist, one indeed defines the connection $\widehat{\nabla}^{\breve{V}}$
  as in Definition~4.1.4 {\it so that} Lemma~4.1.5 holds.

\bigskip
	
\subsection{Vector superfields in Wess-Zumino gauge}

We prove the existence of Wess-Zumino gauge for a vector superfield $\breve{V}$
  in the sense of Definition~4.1.3
 and
 work out the connection $\widehat{\nabla}^{\breve{V}}$
   on $C^\infty(\widehat{X}^{\widehat{\boxplus}})$,
      as a (left) $C^\infty(\widehat{X}^{\widehat{\boxplus}})$-module,
  for $\breve{V}$ in Wess-Zumino gauge.

\bigskip

\begin{flushleft}
{\bf Gauge transformations of a vector superfield}
\end{flushleft}
Under a gauge transformation specified by a small chiral superfield
  $\breve{\Lambda}\in C^\infty(\widehat{X}^{\widehat{\boxplus}})^\smallscriptsize$,
  a vector superfield $\breve{V}$ transforms as
     \marginpar{\vspace{0em}\raggedright\tiny
         \raisebox{-1ex}{\hspace{-2.4em}\LARGE $\cdot$}Cf.\,[\,\parbox[t]{20em}{Wess
		\& Bagger:\\ Eq.\:(6.4)].}}
  $$
    \breve{V}\;\longrightarrow\;
	   \breve{V}+\delta_{\breve{\Lambda}}\breve{V}\,
	     :=\,\breve{V}- \sqrt{-1}(\breve{\Lambda}-\breve{\Lambda}^\dag)\,.
  $$
Explicitly in terms of the standard coordinate functions $(x, \theta, \bar{\theta}, \vartheta, \bar{\vartheta})$,
let
 {\small
 \begin{eqnarray*}
  \breve{V} & =
    &  V_{(0)}
       + \sum_\alpha \theta^\alpha\vartheta_\alpha  V_{(\alpha)}
	    -  \sum_{\dot{\beta}} \bar{\theta}^{\dot{\beta}}\bar{\vartheta}_{\dot{\beta}}
		       \overline{V_{(\beta)}}     		
	   +\, \theta^1\theta^2\vartheta_1\vartheta_2 V_{(12)}
	   + \sum_{\alpha,\dot{\beta}, \mu} \theta^\alpha \bar{\theta}^{\dot{\beta}}
               \sigma^\mu_{\alpha\dot{\beta}} V_{[\mu]}
       + \bar{\theta}^{\dot{1}}\bar{\theta}^{\dot{2}}
              \bar{\vartheta}_{\dot{1}}\bar{\vartheta}_{\dot{2}}\overline{V_{(12)}}			
	   \\
    &&
	   +\, \sum_{\dot{\beta}} \theta^1\theta^2 \bar{\theta}^{\dot{\beta}}
	          \mbox{\Large $($}
			   \sqrt{-1}
                 \sum_{\alpha, \mu} \vartheta_\alpha
				   \sigma^{\mu\alpha}_{\;\;\;\;\dot{\beta}}\partial_\mu	V_{(\alpha)}			
			    + \bar{\vartheta}_{\dot{\beta}}\,\overline{V^{\prime\prime}_{(\beta)}}
			  \mbox{\Large $)$}			  			
          \\
         && 				
	   +\,   \sum_\alpha \theta^\alpha \bar{\theta}^{\dot{1}}\bar{\theta}^{\dot{2}}
	          \mbox{\Large $($}			
	            \vartheta_\alpha V^{\prime\prime}_{(\alpha)}
				-\sqrt{-1}
       	           \sum_{\dot{\beta}, \mu}\bar{\vartheta}_{\dot{\beta}}
				     \sigma^{\mu\dot{\beta}}_\alpha \partial_\mu\,\overline{V_{(\beta)}}					   					  
			  \mbox{\Large $)$}	   				          			
       +  \theta^1\theta^2\bar{\theta}^{\dot{1}}\bar{\theta}^{\dot{2}}
	          V^\sim_{(0)}
			\hspace{2em}\in\; C^\infty(\widehat{X}^{\widehat{\boxplus}})
 \end{eqnarray*}}
 be a vector superfield and
{\small
   \begin{eqnarray*}
     \breve{\Lambda} & = &
	   \Lambda_{(0)}
	   + \sum_\alpha \theta^\alpha\vartheta_\alpha  \Lambda_{(\alpha)}
	   + \theta^1\theta^2\vartheta_1\vartheta_2 \Lambda_{(12)}
       + \sqrt{-1} \sum_{\alpha,\dot{\beta}, \mu}
	          \theta^\alpha\bar{\theta}^{\dot{\beta}}	
			     \sigma^\mu_{\alpha\dot{\beta}}\partial_\mu \Lambda_{(0)}\\
      && \hspace{1em}				
       +\, \sqrt{-1}\sum_{\dot{\beta}, \alpha, \mu}				
	        \theta^1\theta^2\bar{\theta}^{\dot{\beta}} \vartheta_\alpha
		       \sigma^{\mu\alpha}_{\;\;\;\;\dot{\beta}} \partial_\mu \Lambda_{(\alpha)}
       -\, \theta^1\theta^2\bar{\theta}^{\dot{1}}\bar{\theta}^{\dot{2}}\, \square \Lambda_{(0)}
	   \hspace{2em} \in\; C^\infty(\widehat{X}^{\widehat{\boxplus}})^{\smallscriptsize, \scriptsizech}\,,
   \end{eqnarray*}}
    where $\square := -\partial_0^2+\partial_1^2+\partial_2^2+\partial_3^2$,
  be a small chiral superfield.
The twisted complex conjugate $\breve{\Lambda}^\dag$ of $\breve{\Lambda}$ is given by
 {\small
   \begin{eqnarray*}
     \breve{\Lambda}^\dag & = &
	   \overline{\Lambda_{(0)}}
	   - \sum_{\dot{\beta}} \bar{\theta}^{\dot{\beta}}
	          \bar{\vartheta}_{\dot{\beta}}\overline{\Lambda_{(\beta)}}
	   + \bar{\theta}^{\dot{1}}\bar{\theta}^{\dot{2}}
	         \bar{\vartheta}_{\dot{1}}\bar{\vartheta}_{\dot{2}} \overline{\Lambda_{(12)}}
       - \sqrt{-1} \sum_{\alpha,\dot{\beta}, \mu}\theta^\alpha\bar{\theta}^{\dot{\beta}}	
			  \sigma^\mu_{\alpha\dot{\beta}}\partial_\mu \overline{\Lambda_{(0)}}    \\
      && \hspace{1em}				
       -\, \sqrt{-1}\sum_{\alpha, \dot{\beta}, \mu}
	        \theta^\alpha\bar{\theta}^{\dot{1}} \bar{\theta}^{\dot{2}} \bar{\vartheta}_{\dot{\beta}}
			   \sigma^{\mu\dot{\beta}}_\alpha \partial_\mu \overline{\Lambda_{(\beta)}}			
       -\, \theta^1\theta^2\bar{\theta}^{\dot{1}}\bar{\theta}^{\dot{2}}\, \square\, \overline{\Lambda_{(0)}}
          \hspace{2em}\in\;  C^\infty(\widehat{X}^{\widehat{\boxplus}})^{\smallscriptsize, \scriptsizeach}\,.
   \end{eqnarray*}}  
Then,
     \marginpar{\vspace{1em}\raggedright\tiny
         \raisebox{-1ex}{\hspace{-2.4em}\LARGE $\cdot$}Cf.\,[\,\parbox[t]{20em}{Wess
		\& Bagger:\\ Eq.\:(6.3)].}}
 {\small
 \begin{eqnarray*}
  \lefteqn{\breve{V}+\delta_{\breve{\Lambda}}\breve{V}
    := \breve{V} -\sqrt{-1} (\breve{\Lambda}-\breve{\Lambda}^\dag)   }
     \\
   &&  =\;
     \mbox{\Large $($}
      V_{(0)}
	   - \sqrt{-1}(\Lambda_{(0)}-\overline{\Lambda_{(0)}})
	 \mbox{\Large $)$}
       + \sum_\alpha \theta^\alpha\vartheta_\alpha
	        \mbox{\Large $($}
	           V_{(\alpha)}
			   - \sqrt{-1} \Lambda_{(\alpha)}
			\mbox{\Large $)$}
	   - \sum_{\dot{\beta}} \bar{\theta}^{\dot{\beta}}\bar{\vartheta}_{\dot{\beta}}
	        \mbox{\Large $($}
		       \overline{V_{(\beta)}}
			  + \sqrt{-1}\,\overline{\Lambda_{(\beta)}}
			\mbox{\Large $)$}                \\
    && \hspace{2em}			
	   +\, \theta^1\theta^2\vartheta_1\vartheta_2
	          \mbox{\Large $($}
	           V_{(12)} - \sqrt{-1} \Lambda_{(12)}
			  \mbox{\Large $)$}
	   + \sum_{\alpha,\dot{\beta}, \mu} \theta^\alpha \bar{\theta}^{\dot{\beta}}
             \sigma^\mu_{\alpha\dot{\beta}}
			     \mbox{\Large $($}
				  V_{[\mu]}
				  + \partial_\mu
				       \mbox{\large $($}
					     \Lambda_{(0)} + \overline{\Lambda_{(0)}}
					   \mbox{\large $)$}
				 \mbox{\Large $)$}       \\
    && \hspace{2em} 				
       +\, \bar{\theta}^{\dot{1}}\bar{\theta}^{\dot{2}}
              \bar{\vartheta}_{\dot{1}}\bar{\vartheta}_{\dot{2}}		
			   \mbox{\Large $($}
                \overline{V_{(12)}} + \sqrt{-1}\,\overline{\Lambda_{(12)}}
			   \mbox{\Large $)$}	
		\\[1ex]    	
	  && \hspace{2em}
	   +\, \sum_{\dot{\beta}} \theta^1\theta^2 \bar{\theta}^{\dot{\beta}}	
	          \left\{\rule{0ex}{1.2em}\right.
			   \sqrt{-1}
                 \sum_{\alpha, \mu} \vartheta_\alpha
				   \sigma^{\mu\alpha}_{\;\;\;\;\dot{\beta}}\, \partial_\mu	
				      \mbox{\Large $($}
				       V_{(\alpha)}	 - \sqrt{-1}\Lambda_{(\alpha)}
					  \mbox{\Large $)$}
			    + \bar{\vartheta}_{\dot{\beta}}\,				
				        \overline{V^{\prime\prime}_{(\beta)}} 					
			  \left.\rule{0ex}{1.2em}\right\}
          \\
      && \hspace{2em}				
	   +\,   \sum_\alpha \theta^\alpha \bar{\theta}^{\dot{1}}\bar{\theta}^{\dot{2}}
	          \left\{\rule{0ex}{1.2em}\right.
	            \vartheta_\alpha 		
				   V^{\prime\prime}_{(\alpha)}			
				-\sqrt{-1}
       	           \sum_{\dot{\beta}, \mu}\bar{\vartheta}_{\dot{\beta}}
				     \sigma^{\mu\dot{\beta}}_\alpha\, \partial_\mu\,
					   \mbox{\Large $($}
					     \overline{V_{(\beta)}} + \sqrt{-1}\,\overline{\Lambda_{(\beta)}}
					   \mbox{\Large $)$}
			  \left.\rule{0ex}{1.2em}\right\}				
		   \\		
     && \hspace{2em}				
       +\, \theta^1\theta^2\bar{\theta}^{\dot{1}}\bar{\theta}^{\dot{2}}
	         \mbox{\Large $($}
	          V^\sim_{(0)}
			  +\sqrt{-1}\,\square(\Lambda_{(0)}- \overline{\Lambda_{(0)}})
             \mbox{\Large $)$}			  	
      \\ [1.2ex]		
	 && =:\;
     \mbox{\Large $($}
      V_{(0)} +\delta_{\breve{\Lambda}}V_{(0)}
	 \mbox{\Large $)$}
       + \sum_\alpha \theta^\alpha\vartheta_\alpha
	        \mbox{\Large $($}
	           V_{(\alpha)}+ \delta_{\breve{\Lambda}}V_{(\alpha)}
			\mbox{\Large $)$}
	   - \sum_{\dot{\beta}} \bar{\theta}^{\dot{\beta}}\bar{\vartheta}_{\dot{\beta}}
	        \mbox{\Large $($}
		       \overline{V_{(\beta)}}+\delta_{\breve{\Lambda}}\overline{V_{(\beta)}}
			\mbox{\Large $)$}                \\
    && \hspace{2em}			
	   +\, \theta^1\theta^2\vartheta_1\vartheta_2
	          \mbox{\Large $($}
	           V_{(12)} + \delta_{\breve{\Lambda}}V_{(12)}
			  \mbox{\Large $)$}
	   + \sum_{\alpha,\dot{\beta}, \mu} \theta^\alpha \bar{\theta}^{\dot{\beta}}
             \sigma^\mu_{\alpha\dot{\beta}}
			     \mbox{\Large $($}
				   V_{[\mu]} + \delta_{\breve{\Lambda}}V_{([\mu])} 				 					
				 \mbox{\Large $)$}       				
       +  \bar{\theta}^{\dot{1}}\bar{\theta}^{\dot{2}}
              \bar{\vartheta}_{\dot{1}}\bar{\vartheta}_{\dot{2}}		
			   \mbox{\Large $($}
                \overline{V_{(12)}}+ \delta_{\breve{\Lambda}}\overline{V_{(12)}}
			   \mbox{\Large $)$}	     \\[1ex]			
     && \hspace{2em}
	   +\, \sum_{\dot{\beta}} \theta^1\theta^2 \bar{\theta}^{\dot{\beta}}	
	          \left\{\rule{0ex}{1.2em}\right.
			   \sqrt{-1}
                 \sum_{\alpha, \mu} \vartheta_\alpha
				   \sigma^{\mu\alpha}_{\;\;\;\;\dot{\beta}}\, \partial_\mu	
				      \mbox{\Large $($}
				       V_{(\alpha)}	 + \delta_{\breve{\Lambda}} V_{(\alpha)}		
					  \mbox{\Large $)$}
			    + \bar{\vartheta}_{\dot{\beta}}\,
				       \mbox{\Large $($}
				        \overline{V^{\prime\prime}_{(\beta)}}
						  + \delta_{\breve{\Lambda}}\overline{V^{\prime\prime}_{(\beta)}}
					   \mbox{\Large $)$}
			  \left.\rule{0ex}{1.2em}\right\}
          \\
      && \hspace{2em}				
	   +\,   \sum_\alpha \theta^\alpha \bar{\theta}^{\dot{1}}\bar{\theta}^{\dot{2}}
	          \left\{\rule{0ex}{1.2em}\right.
	            \vartheta_\alpha
				  \mbox{\Large $($}
				   V^{\prime\prime}_{(\alpha)} + \delta_{\breve{\Lambda}} V^{\prime\prime}_{(\alpha)}
				  \mbox{\Large $)$}
				-\sqrt{-1}
       	           \sum_{\dot{\beta}, \mu}\bar{\vartheta}_{\dot{\beta}}
				     \sigma^{\mu\dot{\beta}}_\alpha\, \partial_\mu\,
					   \mbox{\Large $($}
					     \overline{V_{(\beta)}} + \delta_{\breve{\Lambda}}	\overline{V_{(\beta)}}
					   \mbox{\Large $)$}
			  \left.\rule{0ex}{1.2em}\right\}				
		   \\
      && \hspace{2em}			
       +\,  \theta^1\theta^2\bar{\theta}^{\dot{1}}\bar{\theta}^{\dot{2}}
	            \mbox{\Large $($}
	              V^\sim_{(0)} + \delta_{\breve{\Lambda}} V^\sim_{(0)}
			    \mbox{\Large $)$}
     		\hspace{2em}\in\; C^\infty(\widehat{X}^{\widehat{\boxplus}})			
 \end{eqnarray*}
  }
 is another vector superfield in the sense of Definition~4.1.3.

A comparison of
    $\delta_{\breve{\Lambda}}V^\sim_{(0)}$
      against
    $\delta_{\breve{\Lambda}}V_{(0)}$
 implies that
if one expresses a vector superfield $\breve{V}\in C^\infty(\widehat{X}^{\widehat{\boxplus}})$
 in the following shifted form, which can always be made:

\bigskip
 
\begin{definition} {\bf [vector superfield in shifted expression]}\; {\rm
 By re-defining
    the $\theta^1\theta^2\bar{\theta}^{\dot{1}}\bar{\theta}^{\dot{2}}$-component $V^\sim_{(0)}$,
  any vector superfield in $C^\infty(\widehat{X}^{\widehat{\boxplus}})$ can be expressed
  in the following form
     \marginpar{\vspace{4em}\raggedright\tiny
         \raisebox{-1ex}{\hspace{-2.4em}\LARGE $\cdot$}Cf.\,[\,\parbox[t]{20em}{Wess
		\& Bagger:\\ Eq.\:(6.2)].}}
{\small
 \begin{eqnarray*}
  \breve{V}
  &= &
    V_{(0)}
	+ \sum_{\alpha}\theta^\alpha\vartheta_\alpha V_{(\alpha)}
	- \sum_{\dot{\beta}}
	       \bar{\theta}^{\dot{\beta}}\bar{\vartheta}_{\dot{\beta}}\,  \overline{V_{(\beta)}}
    + \theta^1\theta^2\vartheta_1\vartheta_2 V_{(12)}
	+ \sum_{\alpha,\dot{\beta}, \mu}\theta^\alpha \bar{\theta}^{\dot{\beta}}
		  \sigma^\mu_{\alpha\dot{\beta}} V_{[\mu]}		
    + \bar{\theta}^{\dot{1}}\bar{\theta}^{\dot{2}}
	   \bar{\vartheta}_{\dot{1}}\bar{\vartheta}_{\dot{2}} \, \overline{V_{(12)}}  \\
  && \hspace{1em}		
	+ \sum_{\dot{\beta}}\theta^1\theta^2\bar{\theta}^{\dot{\beta}}
	     \mbox{\Large $($}
		   \sqrt{-1}
		    \sum_{\alpha,\mu }
		          \vartheta_\alpha\,\sigma^{\mu \alpha}_{\;\;\;\;\dot{\beta}}\,
				    \partial_\mu V_{(\alpha)}
          + \bar{\vartheta}_{\dot{\beta}}\,
		        \overline{V^{\prime\prime}_{(\beta)}}					
         \mbox{\Large $)$}					
     \\
    && \hspace{1em}	
    +\, \sum_\alpha \theta^\alpha\bar{\theta}^{\dot{1}}\bar{\theta}^{\dot{2}}
	     \mbox{\Large $($}
		  \vartheta_\alpha\, V^{\prime\prime}_{(\alpha)}
		  - \sqrt{-1}
	         \sum_{\dot{\beta}, \mu}
		      \sigma^{\mu\dot{\beta}}_\alpha
		        \bar{\vartheta}_{\dot{\beta}}\,
			      \partial_\mu V_{(\dot{\beta})}
	     \mbox{\Large $)$}
	+ \theta^1\theta^2\bar{\theta}^{\dot{1}}\bar{\theta}^{\dot{2}}\,
		  \mbox{\large $($}V^\sim_{(0)} - \square\, V_{(0)}		  \mbox{\large $)$}
 \end{eqnarray*}
 }We 
 call a vector superfield of such form a {\it vector superfield in the shifted expression}.
}\end{definition}

\bigskip

\noindent
Then
 the components
   $V^{\prime\prime}_{(\alpha)}$, ${}_{\alpha=1,2}$,
   $V^\sim_{(0)} \in C^\infty(X)^{\Bbb C}$
  are invariant under gauge transformations:
     \marginpar{\vspace{3em}\raggedright\tiny
         \raisebox{-1ex}{\hspace{-2.4em}\LARGE $\cdot$}Cf.\,[\,\parbox[t]{20em}{Wess
		\& Bagger:\\ Eq.\:(6.5)].}}
 $$
  \delta_{\breve{\Lambda}}\;:\;\;\;
    \begin{array}{lcl}
     V_{(0)}             & \longrightarrow
       & V_{(0)}- \sqrt{-1}(\Lambda_{(0)}-\overline{\Lambda_{(0)}})\,, \\[1.2ex]
	 V_{(\alpha)}   & \longrightarrow
       & V_{(\alpha)}- \sqrt{-1} \Lambda_{(\alpha)}\,, \\[1.2ex]
			
     V_{(12)}           & \longrightarrow			
	   & V_{(12)} - \sqrt{-1} \Lambda_{(12)}\,,   \\[1.2ex]
     V_{[\mu]}             & \longrightarrow 			
	   & V_{[\mu]}
   		   + \partial_\mu \mbox{\large $($} \Lambda_{(0)} + \overline{\Lambda_{(0)}} \mbox{\large $)$}\,,\\[1.2ex]
	 V^{\prime\prime}_{(\alpha)}   & \longrightarrow
	   & V^{\prime\prime}_{(\alpha)}\,, \\[1.2ex]			
	 V^\sim_{(0)}     & \longrightarrow    & V^\sim_{(0)}\,.			  	
    \end{array}
  $$
  
\bigskip

\begin{remark} $[\Bbb R$-linearity\,$]$\;  {\rm
 An ${\Bbb R}$-linear combination\footnote{However,
                                                               caution that a $C^\infty(X)$-linear combination of vector superfields in the shifted expression
															    in general is not directly a vector superfield in the shifted expression.
															   One has to convert it accordingly.
															   }  
  of vector superfields in the shifted expression is also a vector superfield in the shifted expression.
}\end{remark}

\bigskip

It follows that, for a given vector superfield $\breve{V}$ in the shifted expression,
 if one chooses a small chiral superfield
 $\breve{\Lambda}\in C^\infty(\widehat{X}^{\widehat{\boxplus}})^{\smallscriptsize, \scriptsizech}$ with
  $$
    \Imaginary \Lambda_{(0)}\;
	  =\; -\, \mbox{\large $\frac{1}{2}$}\,V_{(0)}\,,\hspace{2em}
	\Lambda_{(\alpha)}\;
	  =\; -\,\sqrt{-1}\,V_{(\alpha)}\,,\hspace{2em}
    \Lambda_{(12)}\;
	  =\; -\,\sqrt{-1}\,V_{(12)}\,,
  $$
 which always exists,
 then after the gauge transformation specified by $\breve{\Lambda}$,
 $\breve{V}$ becomes
 {\small
 $$
   \breve{V}^\prime\;  =\;
	  \sum_{\alpha,\dot{\beta}, \mu} \theta^\alpha \bar{\theta}^{\dot{\beta}}
         \sigma^\mu_{\alpha\dot{\beta}}
			     \mbox{\Large $($}
				   V_{[\mu]}+ 2\,\partial_\mu \Real \Lambda_{(0)}	
				 \mbox{\Large $)$}
	   +\, \sum_{\dot{\beta}} \theta^1\theta^2 \bar{\theta}^{\dot{\beta}}			
	           \bar{\vartheta}_{\dot{\beta}}\, \overline{V^{\prime\prime}_{(\beta)}}
	   +\, \sum_\alpha \theta^\alpha \bar{\theta}^{\dot{1}}\bar{\theta}^{\dot{2}}	   			
       	         \vartheta_\alpha V^{\prime\prime}_{(\alpha)}
       +\, \theta^1\theta^2\bar{\theta}^{\dot{1}}\bar{\theta}^{\dot{2}}\, 	
	          V^\sim_{(0)}\,.		
 $$
  }
Here,
 $\Real \Lambda_{(0)}$ and $\Imaginary \Lambda_{(0)}$
  are the real part and the imaginary part of $\Lambda_{(0)}\in C^\infty(X)^{\Bbb C}$ respectively.
	
We summarize the above discussion into the following definition and lemma:

\bigskip

\begin{definition} {\bf [vector superfield in Wess-Zumino gauge]}\; {\rm
 A vector superfield $\breve{V}\in C^\infty(\widehat{X}^{\widehat{\boxplus}})$ that is of the following form
  in the standard coordinate functions $(x, \theta, \bar{\theta}, \vartheta, \bar{\vartheta})$
  on $\widehat{X}^{\widehat{\boxplus}}$
     \marginpar{\vspace{0em}\raggedright\tiny
         \raisebox{-1ex}{\hspace{-2.4em}\LARGE $\cdot$}Cf.\,[\,\parbox[t]{20em}{Wess
		\& Bagger:\\ above Eq.\:(6.6);\\   footnote, p.37].}}
 {\small
 $$
   \breve{V}\; =\;   			
	  \sum_{\alpha,\dot{\beta}, \mu} \theta^\alpha \bar{\theta}^{\dot{\beta}}
        \sigma^\mu_{\alpha\dot{\beta}} V_{[\mu]}		
	   +\, \sum_{\dot{\beta}} \theta^1\theta^2 \bar{\theta}^{\dot{\beta}}			
	           \bar{\vartheta}_{\dot{\beta}}\, \overline{V^{\prime\prime}_{(\beta)}}
	   +\, \sum_\alpha \theta^\alpha \bar{\theta}^{\dot{1}}\bar{\theta}^{\dot{2}}	   			
       	         \vartheta_\alpha V^{\prime\prime}_{(\alpha)}
       +\, \theta^1\theta^2\bar{\theta}^{\dot{1}}\bar{\theta}^{\dot{2}} 	
	          V^\sim_{(0)}
 $$}
  is called a {\it vector superfield in Wess-Zumino gauge}.
}\end{definition}

\bigskip

\medskip

\begin{lemma} {\bf [vector superfield representative in Wess-Zumino gauge]}\;
 Given any vector superfield $\breve{V}\in C^\infty(\widehat{X}^{\widehat{\boxplus}})$,
   there exists a unique small chiral superfield
    $\breve{\Lambda}\in C^\infty(\widehat{X}^{\widehat{\boxplus}})^{\smallscriptsize, \scriptsizech}$
          depending on $\breve{V}$ with $\Real \Lambda^{(0)}_{(0)}=0$
	 such that the gauge transformation specified by $\breve{\Lambda}$ takes $\breve{V}$ to a vector superfield in Wess-Zumino gauge.
 In particular, any vector superfield can be transformed to a vector superfield in Wess-Zumino gauge by a gauge transformation.
\end{lemma}

\medskip

\begin{lemma}  {\bf [naturality]}\;
 $(1)$
 The set of vector superfields in Wess-Zumino gauge is a $C^\infty(X)$-submodule of
     $C^\infty(\widehat{X}^{\widehat{\boxplus}})$.
 $(2)$
 If a vector superfield $\breve{V}$ expressed in terms of the standard coordinate functions
     $(x,\theta,\bar{\theta},\vartheta, \bar{\vartheta})$
	 on $\widehat{X}^{\widehat{\boxplus}}$ is in Wess-Zumino gauge,
 then it remains in Wess-Zumino gauge when re-expressed in terms of
    the chiral coordinate functions $(x^\prime,\theta,\bar{\theta},\vartheta,\bar{\vartheta})$ or
	the antichiral coordinate functions $(x^{\prime\prime},\theta, \bar{\theta},\vartheta,\bar{\vartheta})$
     on $\widehat{X}^{\widehat{\boxplus}}$.
\end{lemma}

\medskip

\begin{proof}
 Statement (1) is clear. We focus on Statement (2).

 Recall that, in shorthand,  $x^\prime = x+\sqrt{-1}\theta\sigma\bar{\theta}$.
 When in Wess-Zumino gauge,
  a vector superfield $\breve{V}$ in terms of the standard coordinate functions
    $(x,\theta,\bar{\theta},\vartheta, \bar{\vartheta})$
  is written as
   {\small
 $$
   \breve{V}\; =\;   			
	  \sum_{\alpha,\dot{\beta}, \mu} \theta^\alpha \bar{\theta}^{\dot{\beta}}
        \sigma^\mu_{\alpha\dot{\beta}} V_{[\mu]}(x)		
	   +\, \sum_{\dot{\beta}} \theta^1\theta^2 \bar{\theta}^{\dot{\beta}}			
	           \bar{\vartheta}_{\dot{\beta}}\, \overline{V^{\prime\prime}_{(\beta)}(x)}
	   +\, \sum_\alpha \theta^\alpha \bar{\theta}^{\dot{1}}\bar{\theta}^{\dot{2}}	   			
       	         \vartheta_\alpha V^{\prime\prime}_{(\alpha)}(x)
       +\, \theta^1\theta^2\bar{\theta}^{\dot{1}}\bar{\theta}^{\dot{2}} 	
	          V^\sim_{(0)}(x)
 $$}
 To re-express $\breve{V}$ in terms of
    the chiral coordinate functions $(x^\prime,\theta,\bar{\theta},\vartheta,\bar{\vartheta})$
	on $\widehat{X}^{\widehat{\boxplus}}$,
 one substitutes $x$ in $V^{\tinybullet}_{\tinybullet}(x)$ by $x^\prime-\sqrt{-1}\theta\sigma\bar{\theta}$   and
   use the $C^\infty$-hull structure of $C^\infty(\widehat{X})$ to expand it in $x^\prime$.
 Due to the product structure of $\theta^\alpha$, $\bar{\theta}^{\dot{\beta}}$,
  this will only influence the coefficient of $\theta^1\theta^2\bar{\theta}^{\dot{1}}\bar{\theta}^{\dot{2}}$
   and, hence, keep the vector superfield in Wess-Zumino gauge.
 Explicitly, the result is
 {\small
  \begin{eqnarray*}
   \lefteqn{
    \breve{V}\; =\;   			
	  \sum_{\alpha,\dot{\beta}, \mu} \theta^\alpha \bar{\theta}^{\dot{\beta}}
        \sigma^\mu_{\alpha\dot{\beta}} V_{[\mu]}(x^\prime)		
	   +\, \sum_{\dot{\beta}} \theta^1\theta^2 \bar{\theta}^{\dot{\beta}}			
	           \bar{\vartheta}_{\dot{\beta}}\, \overline{V^{\prime\prime}_{(\beta)}(x^\prime)}
	   +\, \sum_\alpha \theta^\alpha \bar{\theta}^{\dot{1}}\bar{\theta}^{\dot{2}}	   			
       	         \vartheta_\alpha V^{\prime\prime}_{(\alpha)}(x^\prime)			
			  }\\
     && \hspace{2em}				
       +\;  \theta^1\theta^2\bar{\theta}^{\dot{1}}\bar{\theta}^{\dot{2}}
	         \mbox{\Large $($}
	          V^\sim_{(0)}(x^\prime)- 2\sqrt{-1}\,\sum_\mu (\partial^\mu V_{[\mu]})(x^\prime)
			 \mbox{\Large $)$}\,. \hspace{18em}
  \end{eqnarray*}}

 Similar argument goes
 when re-expressing $\breve{V}$ in terms of the antichiral coordinate functions\\
    $(x^{\prime\prime},\theta,\bar{\theta},\vartheta,\bar{\vartheta})$ on $\widehat{X}^{\widehat{\boxplus}}$,
      with $x^{\prime\prime} = x-\sqrt{-1}\theta\sigma\bar{\theta}$.
 The explicit expression is given by
 {\small
  \begin{eqnarray*}
   \lefteqn{
    \breve{V}\; =\;   			
	  \sum_{\alpha,\dot{\beta}, \mu} \theta^\alpha \bar{\theta}^{\dot{\beta}}
        \sigma^\mu_{\alpha\dot{\beta}} V_{[\mu]}(x^{\prime\prime})		
	   +\, \sum_{\dot{\beta}} \theta^1\theta^2 \bar{\theta}^{\dot{\beta}}			
	           \bar{\vartheta}_{\dot{\beta}}\, \overline{V^{\prime\prime}_{(\beta)}(x^{\prime\prime})}
	   +\, \sum_\alpha \theta^\alpha \bar{\theta}^{\dot{1}}\bar{\theta}^{\dot{2}}	   			
       	         \vartheta_\alpha V^{\prime\prime}_{(\alpha)}(x^{\prime\prime})			
			  }\\
     && \hspace{2em}				
       +\;  \theta^1\theta^2\bar{\theta}^{\dot{1}}\bar{\theta}^{\dot{2}}
	         \mbox{\Large $($}
	          V^\sim_{(0)}(x^{\prime\prime})+ 2\sqrt{-1}\,\sum_\mu (\partial^\mu V_{[\mu]})(x^{\prime\prime})
			 \mbox{\Large $)$}\,. \hspace{18em}
  \end{eqnarray*}}

 This completes the proof.
	
\end{proof}

\bigskip

In the shifted expression, once a vector superfield is rendered a vector superfield $\breve{V}$ in Wess-Zumino gauge,
 a gauge transformation specified by $\breve{\Lambda}$ with
  $$
    \Imaginary \Lambda_{(0)}\;
	=\; \Lambda_{(\alpha)}\;
    =\; \Lambda_{(12)}\; =\; 	0\,,
  $$
 (i.e.\
   $$
     \breve{\Lambda}\;=\;	
	   \Lambda_{(0)}
	   + \sqrt{-1} \sum_{\alpha,\dot{\beta}, \mu}
	          \theta^\alpha\bar{\theta}^{\dot{\beta}}	
			     \sigma^\mu_{\alpha\dot{\beta}}\partial_\mu \Lambda_{(0)}
       -\, \theta^1\theta^2\bar{\theta}^{\dot{1}}\bar{\theta}^{\dot{2}}\,
	        \square \Lambda_{(0)}
   $$
   specified by a real-valued component $\Lambda_{(0)}$)
 will send $\breve{V}$ to another vector superfield $\breve{V}^\prime$ still in Wess-Zumino gauge
  with only the components $V_{[\mu]}$ of $\breve{V}$ transformed,
     \marginpar{\vspace{0em}\raggedright\tiny
         \raisebox{-1ex}{\hspace{-2.4em}\LARGE $\cdot$}Cf.\,[\,\parbox[t]{20em}{Wess
		\& Bagger:\\ above Eq.\:(6.6)].}}
  by
   $$
     V_{[\mu]}\;\longrightarrow\;
	 V_{[\mu]} \,+\, 2\,\partial_\mu \Lambda_{(0)}\,.
   $$
Thus, the residual gauge symmetries on vector superfields in Wess-Zumino gauge
 are the usual $U(1)$ gauge symmetries.

\bigskip

\begin{lemma}
{\bf [restriction $\nabla^{\breve{V}}$ of $\widehat{\nabla}^{\breve{V}}$ to $X^{\Bbb C}$]}\;
 Let $\breve{V}$ be a vector superfield in Wess-Zumino gauge.
 Then the restriction $\nabla^{\breve{V}}$ of $\widehat{\nabla}^{\breve{V}}_\mu$ to $X^{\Bbb C}$
  is given by $\partial_\mu -\frac{\sqrt{-1}}{2}\,V_{[\mu]}$, $\mu=0,1,2,3$\,.
\end{lemma}
 
\medskip

\begin{proof}
 This follows from the same argument of [L-Y5: proof of Lemma~3.2.6] (SUSY(1)).
 See also the full expression of $\widehat{\nabla}^{\breve{V}}_\mu$ in the next theme.

\end{proof}

\bigskip

\begin{flushleft}
{\bf Explicit formula of $\widehat{\nabla}^{\breve{V}}_{\tinybullet}$ for $\breve{V}$ in Wess-Zumino gauge}
\end{flushleft}
Note that the powers of a vector superfield $\breve{V}\in C^\infty(\widehat{X}^{\widehat{\boxplus}})$ in Wess-Zumino gauge
 can be computed easily:
     \marginpar{\vspace{2em}\raggedright\tiny
         \raisebox{-1ex}{\hspace{-2.4em}\LARGE $\cdot$}Cf.\,[\,\parbox[t]{20em}{Wess
		\& Bagger:\\ Eq.\:(6.6)].}}
{\small
\begin{eqnarray*}
 \breve{V} & = &
	\sum_{\alpha,\dot{\beta}, \mu} \theta^\alpha \bar{\theta}^{\dot{\beta}}
        \sigma^\mu_{\alpha\dot{\beta}} V_{[\mu]}		
	   +\, \sum_{\dot{\beta}} \theta^1\theta^2 \bar{\theta}^{\dot{\beta}}			
	           \bar{\vartheta}_{\dot{\beta}}\, \overline{V^{\prime\prime}_{(\beta)}}
	   +\, \sum_\alpha \theta^\alpha \bar{\theta}^{\dot{1}}\bar{\theta}^{\dot{2}}	   			
       	         \vartheta_\alpha V^{\prime\prime}_{(\alpha)}
       +\, \theta^1\theta^2\bar{\theta}^{\dot{1}}\bar{\theta}^{\dot{2}} 	V^\sim_{(0)}\,,
     \\
 \breve{V}^2 & = &
   2\, \theta^1\theta^2\bar{\theta}^{\dot{1}}\bar{\theta}^{\dot{2}}
      \sum_{\mu, \nu} \eta^{\mu\nu} V_{[\mu]} V_{[\nu]}\,,
    \\
   \breve{V}^3 & = & 0\,.
\end{eqnarray*}
}
It follows that
{\small
\begin{eqnarray*}
 e^{\breve{V}} &  = &   1 + \breve{V} + \mbox{\large $\frac{1}{2}$}\,\breve{V}^2
     \\
  & = &
	1 + \sum_{\alpha,\dot{\beta}, \mu} \theta^\alpha \bar{\theta}^{\dot{\beta}}
        \sigma^\mu_{\alpha\dot{\beta}} V_{[\mu]}		
	   +\, \sum_{\dot{\beta}} \theta^1\theta^2 \bar{\theta}^{\dot{\beta}}			
	           \bar{\vartheta}_{\dot{\beta}}\, \overline{V^{\prime\prime}_{(\beta)}}
	   +\, \sum_\alpha \theta^\alpha \bar{\theta}^{\dot{1}}\bar{\theta}^{\dot{2}}	   			
       	         \vartheta_\alpha V^{\prime\prime}_{(\alpha)}
	  \\
     && \hspace{1em}
       +\; \theta^1\theta^2\bar{\theta}^{\dot{1}}\bar{\theta}^{\dot{2}} 	
	          \mbox{\Large $($}
	            V^\sim_{(0)}+  \sum_{\mu, \nu} \eta^{\mu\nu} V_{[\mu]} V_{[\nu]}
	          \mbox{\Large $)$}
\end{eqnarray*}
}
and $\widehat{\nabla}^{\breve{V}}_{e_I}$, $I=\mu, \alpha^\prime, \beta^{\prime\prime}$,  can be computed
 less tediously for $\breve{V}$ in Wess-Zumino gauge.
The result is given below in the standard coordinate functions $(x, \theta, \bar{\theta}, \vartheta, \bar{\vartheta})$
  on $\widehat{X}^{\widehat{\boxplus}}$.

First,
$$
  \widehat{\nabla}^{\breve{V}}_{e_{\beta^{\prime\prime}}}\;
    :=\; e_{\beta^{\prime\prime}}
$$
by definition.

Next, $\widehat{\nabla}^{\breve{V}}_{e_{\alpha^\prime}}$
 can be computed slightly easier in the antichiral coordinate functions
 $(x^{\prime\prime}, \theta, \bar{\theta}, \vartheta, \bar{\vartheta})$ first and then converted back to
 $(x, \theta, \bar{\theta}, \vartheta, \bar{\vartheta})$:
{\small
\begin{eqnarray*}
 \lefteqn{
   \widehat{\nabla}^{\breve{V}}_{e_{\alpha^\prime}}\;\;
   :=\;\;  e^{-\breve{V}} \circ  e_{\alpha^\prime} \circ e^{\breve{V}}\;\;
     =\;\;  e_{\alpha^\prime}
	            + e^{-\breve{V}} \mbox{\large $($}e_{\alpha^\prime} e^{\breve{V}}\mbox{\large $)$}
		} \\
  && =\;\;
   e_{\alpha^\prime}
   + \sum_{\dot{\beta}, \mu}
        \bar{\theta}^{\dot{\beta}}
	      \sigma^\mu_{\alpha\dot{\beta}} V_{[\mu]}
	       - \sum_{\gamma, \dot{\beta}}
	           \theta^\gamma \bar{\theta}^{\dot{\beta}} \bar{\vartheta}_{\dot{\beta}}\,
		         \varepsilon_{\alpha\gamma}  \overline{V^{\prime\prime}_{(\beta)}}
           + \bar{\theta}^{\dot{1}} \bar{\theta}^{\dot{2}}	
	             \vartheta_\alpha V^{\prime\prime}_{(\alpha)}	
	  \\
	 && \hspace{2em}
	  +\, \sum_\gamma
	            \theta^\gamma \bar{\theta}^{\dot{1}}\bar{\theta}^{\dot{2}}
				 \mbox{\Large $($}
				  -\varepsilon_{\alpha\gamma}V^\sim_{(0)}
				  + \sqrt{-1}\sum_{\dot{\beta}, \mu, \nu}
				         \sigma^{\mu\dot{\beta}}_\alpha \sigma^\nu_{\gamma\dot{\beta}}  \partial_\mu V_{[\nu]}
				 \mbox{\Large $)$}
	   - \sqrt{-1}\theta^1\theta^2 \bar{\theta}^{\dot{1}} \bar{\theta}^{\dot{2}}
	        \sum_{\dot{\beta}}\bar{\vartheta}_{\dot{\beta}}
			     \sigma^{\mu\dot{\beta}}_\alpha \partial_\mu\, \overline{V^{\prime\prime}_{(\beta)}}\,.
\end{eqnarray*}
}
The identity
 $$
   (\sigma^\mu\bar{\sigma}^\nu + \sigma^\nu \bar{\sigma}^\mu )_{\alpha\gamma}\;
    :=\; \sum_{\dot{\beta}, \dot{\delta}}
	         \varepsilon^{\dot{\beta}\dot{\delta}}
	      (\sigma^\mu_{\alpha\dot{\beta}}\sigma^\nu_{\gamma\dot{\delta}}
		      + \sigma^\nu_{\alpha\dot{\beta}} \sigma^\mu_{\gamma\dot{\delta}})\;
	  =\; 2\,\eta^{\mu\nu}\varepsilon_{\alpha\gamma}
 $$
 is used in the computation to simplify
 the coefficient function of $\theta^\gamma\bar{\theta}^{\dot{1}}\bar{\theta}^{\dot{2}}$.
Note that terms quadratic in components of $\breve{V}$ cancel each other   and
 only terms linear in components of $\breve{V}$ remain in the end.
 
Finally,
 $$
    \widehat{\nabla}^{\breve{V}}_{e_{\mu}} \;
	   =\;
	   -\, \mbox{\Large $\frac{\sqrt{-1}}{4}$}
	    \sum_{\dot{\beta}, \alpha} \bar{\sigma}_\mu^{\dot{\beta}\alpha}
		   \{\widehat{\nabla}^{\breve{V}}_{e_{\alpha^\prime}},
		        \widehat{\nabla}^{\breve{V}}_{e_{\beta^{\prime\prime}}}\}\;
	   =
	     \partial_\mu\,
		  -\, \mbox{\Large $\frac{\sqrt{-1}}{4}$}
	            \sum_{\dot{\beta}, \alpha} \bar{\sigma}_\mu^{\dot{\beta}\alpha}
		            \Theta_{\alpha\dot{\beta}}\,,
 $$
 where
 $$
  \Theta_{\alpha\dot{\beta}}\;
   :=\; \{e^{-\breve{V}}(e_{\alpha^\prime}e^{\breve{V}}), e_{\beta^{\prime\prime}}\}\;
	=\; (e_{\beta^{\prime\prime}} e^{-\breve{V}})
	        (e_{\alpha^\prime} e^{\breve{V}})
			 + e^{-\breve{V}} (e_{\beta^{\prime\prime}}e_{\alpha^\prime}e^{\breve{V}})\,.				   		
 $$
After some (pages of) straightforward computations, one obtains the explicit expression for $\Theta_{\alpha\dot{\beta}}$
 in the standard coordinate functions $(x,\theta, \bar{\theta}, \vartheta, \bar{\vartheta})$:
 
{\small
 \begin{eqnarray*}
   \Theta_{\alpha\dot{\beta}}
    & = &
	 -\,\sum_\nu \sigma^\nu_{\alpha\dot{\beta}} V_{[\nu]}
	 - \sum_\gamma
	     \theta^\gamma \bar{\vartheta}_{\dot{\beta}}
	        \varepsilon_{\alpha\gamma}   \overline{V^{\prime\prime}_{(\beta)}}
	 + \sum_{\dot{\delta}}
	       \bar{\theta}^{\dot{\delta}} \vartheta_\alpha
		    \varepsilon_{\dot{\beta}\dot{\delta}} V^{\prime\prime}_{(\alpha)}
		 \\
		&&
     +\, \sum_{\gamma, \dot{\delta}}				
	      \theta^\gamma \bar{\theta}^{\dot{\delta}}
	       \left\{\rule{0ex}{1.2em}\right.
		     \sum_{\mu, \nu}
			  \mbox{\Large $($}
			    \sigma^\mu_{\alpha\dot{\beta}} \sigma^\nu_{\gamma\dot{\delta}}
				 - \sigma^\mu_{\alpha\dot{\delta}} \sigma^\nu_{\gamma\dot{\beta}}
				 +\varepsilon_{\alpha\gamma} \varepsilon_{\dot{\beta}\dot{\delta}} \eta^{\mu\nu}
			  \mbox{\Large $)$}\,
			  V_{[\mu]} V_{[\nu]}
			    \\[-.6ex]
				&& \hspace{7em}
		 +\, \sqrt{-1}\sum_{\mu, \nu}
		      \mbox{\Large $($}
			   \sigma^\mu_{\alpha\dot{\delta}} \sigma^\nu_{\gamma\dot{\beta}}
			   - \sigma^\nu_{\alpha\dot{\delta}} \sigma^\mu_{\gamma\dot{\beta}}
			   - \sigma^\mu_{\alpha\dot{\beta}} \sigma^\nu_{\gamma\dot{\delta}}
			  \mbox{\Large $)$}\,
			  \partial_\mu V_{[\nu]}
          + \varepsilon_{\alpha\gamma} \varepsilon_{\dot{\beta}\dot{\delta}} V^\sim_{(0)}
		   \left.\rule{0ex}{1.2em}\right\}		
		\\
		&&
	 +\, \sqrt{-1} \sum_{\dot{\delta}, \nu}
	      \theta^1\theta^2\bar{\theta}^{\dot{\delta}}
		    \mbox{\Large $($}
			 -2 \bar{\vartheta}_{\dot{\delta}}
			     \sigma^\nu_{\alpha\dot{\beta}}\,\partial_\nu\,\overline{V^{\prime\prime}_{(\delta)}}
			  + \bar{\vartheta}_{\dot{\beta}}  \sigma^\nu_{\alpha\dot{\delta}}\,
			      \partial_\nu\,\overline{V^{\prime\prime}_{(\beta)}}\,
			\mbox{\Large $)$}
	  -  \sqrt{-1}\sum_{\gamma, \nu}
	      \theta^\gamma\bar{\theta}^{\dot{1}}\bar{\theta}^{\dot{2}}
		    \vartheta_\alpha \sigma^\nu_{\gamma\dot{\beta}}\partial_\nu V^{\prime\prime}_{(\alpha)}
        \\
       && 		
	 +\,\theta^1\theta^2 \bar{\theta}^{\dot{1}}\bar{\theta}^{\dot{2}}
	        \left\{\rule{0ex}{1.2em}\right.
			 -\sqrt{-1}\sum_\nu \sigma^\nu_{\alpha\dot{\beta}} \partial_\nu V^\sim_{(0)}			
			 - \sum_{\gamma, \dot{\delta}, \mu, \mu^\prime, \nu}
			     \sigma^{\mu\dot{\delta}}_\alpha \sigma^{\mu^\prime\gamma}_{\;\;\;\;\dot{\beta}}
				   \sigma^\nu_{\gamma\dot{\delta}}
				  \partial_\mu\partial_{\mu^\prime} V_{[\nu]}
				  \\[-.6ex]
				  && \hspace{7em}
			 +\, 2\,\sqrt{-1}\,\sum_{\gamma, \dot{\delta}, \nu^\prime, \mu, \nu}
			      \sigma^{\nu^\prime}_{\gamma\dot{\delta}}\, \sigma^{\mu\dot{\delta}}_\alpha\,
				    \sigma^{\nu\gamma}_{\;\;\;\;\dot{\beta}}
					 V_{[\nu^\prime]}
					 \mbox{\large $($}
					  \partial_\mu V_{[\nu]} - \partial_\nu V_{[\mu]}
					 \mbox{\large $)$}
			\left.\rule{0ex}{1.2em}\right\}\,.	
 \end{eqnarray*}
}

\noindent
The identity
$$
   (\sigma^\mu\bar{\sigma}^\nu + \sigma^\nu \bar{\sigma}^\mu )_\gamma^{\;\gamma^\prime}\;
    :=\; \sum_{\dot{\delta}}
	      (\sigma^\mu_{\gamma\dot{\delta}}\bar{\sigma}^{\nu\dot{\delta}\gamma^\prime}
		      + \sigma^\nu_{\gamma\dot{\delta}} \bar{\sigma}^{\mu\dot{\delta}\gamma^\prime})\;
	  =\; -\, 2\,\eta^{\mu\nu}\delta_\gamma^{\;\gamma^\prime}
 $$
 is used in the computation to simplify the coefficient functions.
Terms cubic in components of $\breve{V}$ do appear in the intermediate steps but they cancel each other in the end.
 It is a curious feature that the curvature tensor $\partial_\mu V_{[\nu]}-\partial_\nu V_{[\mu]} $ on $X$
  appear in the coefficient function of $\theta^1\theta^2\bar{\theta}^{\dot{1}}\bar{\theta}^{\dot{2}}$.
From the explicit expression of $\Theta_{\alpha\dot{\beta}}$, one concludes also that
 $$
  \widehat{\nabla}^{\breve{V}}_\mu\;
   =\; \partial_\mu
         -\, \mbox{\Large $\frac{\sqrt{-1}}{2}$}\,V_{[\mu]}\,
		 +\, \mbox{(terms of total $(\theta, \bar{\theta})$-degree $\ge 1$)}
 $$
 in Lemma~4.2.6.

\bigskip

\subsection{Supersymmetry transformations of a vector superfield in Wess-Zumino gauge}	

In this subsection we repeat the discussion of [L-Y5: Sec.\:3.3] (SUSY(1))
 for the upgraded notion of vector superfield.
Readers are recommend to read this subsection along with
   [Argu: Sec.\,4.3, pp.\:76, 77] of {\sl Argurio} and
   [G-G-R-S: Sec.\,4.2.1, from before Eq.\,(4.2.7) to after  Eq.\,(4.2.13)]
      of {\it Gates, Grosaru, Ro\u{c}ek, \& Siegel}
 for comparison.
 		
Recall the Grassmann parameter level of $\widehat{X}^{\widehat{\boxplus}}$.
The setup, discussions, and results in Sec.\,1.3, Sec.\,4.1 and Sec.\,4.2 can be generalized straightforwardly
 to the case where the Grassmann parameter level is turned on.
In particular,
 recall the additional coordinate functions
    $(\eta,\bar{\eta}):= (\eta^1,\eta^2; \bar{\eta}^{\dot{1}},\bar{\eta}^{\dot{2}})$
    on $\widehat{X}^{\widehat{\boxplus}}$ from the parameter level,
 then: (cf.\:Lemma~1.3.4)
 
\bigskip

\begin{lemma} {\bf [chiral function and antichiral function on $\widehat{X}^{\widehat{\boxplus}}$
              with parameter level activated]}\\
 $(1)$
 $\breve{f}\in C^\infty(\widehat{X}^{\widehat{\boxplus}})$ is chiral if and only if,
   as an element of
    $C^\infty(X)^{\Bbb C}[\theta, \bar{\theta}, \eta, \bar{\eta}, \vartheta, \bar{\theta}]^\anticommuting$,
  $\breve{f}$ is of the following form
  \begin{eqnarray*}
   \breve{f} & = &
     \breve{f}_{(0)}(x, \eta, \bar{\eta}, \vartheta, \bar{\vartheta})
	 + \sum_\gamma \theta^\gamma \breve{f}_{(\gamma)}(x, \eta, \bar{\eta}, \vartheta, \bar{\theta})
	     \\
	&&
	    +\, \theta^1\theta^2 \breve{f}_{(12)}(x,\eta, \bar{\eta}, \vartheta, \bar{\vartheta})		     	
	    + \sqrt{-1} \sum_{\gamma, \dot{\delta}, \nu}
		    \theta^\gamma \bar{\theta}^{\dot{\delta}} \sigma^\nu_{\gamma\dot{\delta}}
			  \partial_\nu \breve{f}_{(0)}(x, \eta, \bar{\eta}, \vartheta, \bar{\vartheta})
 		\\
	&& 			
        +\, \sqrt{-1} \sum_{\dot{\delta}, \gamma, \nu}
		     \theta^1\theta^2 \bar{\theta}^{\dot{\delta}}
			  \sigma^{\nu\gamma}_{\;\;\;\;\dot{\delta}}\,\partial_\nu
			   \breve{f}_{(\gamma)}(x, \eta, \bar{\eta}, \vartheta, \bar{\vartheta})
		- \theta^1\theta^2\bar{\theta}^{\dot{1}} \bar{\theta}^{\dot{2}}\,
		    \square \breve{f}_{(0)}(x,\eta, \bar{\eta}, \vartheta, \bar{\vartheta})\,.
  \end{eqnarray*}
 In particular, a chiral $\breve{f}\in C^\infty(\widehat{X}^{\widehat{\boxplus}})$
  has four independent components in $C^\infty(X)^{\Bbb C}[\eta, \bar{\eta}, \vartheta, \bar{\vartheta}]^\anticommuting$:
  $$
    \breve{f}_{(0)}\,,\;\;  \breve{f}_{(\gamma)}\,,\,{}_{\gamma\;=\;1, 2}\,,\;\;
	\breve{f}_{(12)}\,.
  $$
 In terms of the standard chiral coordinate functions $(x^\prime, \theta, \bar{\theta}, \eta, \bar{\eta}, \vartheta, \bar{\vartheta})$
  on $\widehat{X}^{\widehat{\boxplus}}$,
 $$
  \breve{f}\;
   =\; \breve{f}_{(0)}(x^\prime, \eta, \bar{\eta}, \vartheta, \bar{\vartheta})
         + \sum_\gamma \theta^\gamma
		      \breve{f}_{(\gamma)}(x^\prime, \eta, \bar{\eta}, \vartheta, \bar{\vartheta})
		 + \theta^1\theta^2
		     \breve{f}_{(12)}(x^\prime, \eta, \bar{\eta}, \vartheta, \bar{\vartheta})\,,		
 $$
 which is independent of $\bar{\theta}$.
 
 $(2)$
   $\breve{f}\in C^\infty(\widehat{X}^{\widehat{\boxplus}})$ is antichiral if and only if,
   as an element of $C^\infty(X)^{\Bbb C}[\theta, \bar{\theta}, \eta, \bar{\eta}, \vartheta, \bar{\theta}]^\anticommuting$,
  $\breve{f}$ is of the following form
  \begin{eqnarray*}
   \breve{f} & = &
     \breve{f}_{(0)}(x, \eta, \bar{\eta}, \vartheta, \bar{\vartheta})
	 + \sum_{\dot{\delta}} \bar{\theta}^{\dot{\delta}}
	      \breve{f}_{(\dot{\delta})}(x, \eta, \bar{\eta}, \vartheta, \bar{\theta})
		\\
    &&
	    +\, \bar{\theta}^1\bar{\theta}^2 \breve{f}_{(\dot{1}\dot{2})}(x, \eta, \bar{\eta}, \vartheta, \bar{\vartheta})		     	
	    - \sqrt{-1} \sum_{\gamma, \dot{\delta}, \nu}
		    \theta^\gamma \bar{\theta}^{\dot{\delta}} \sigma^\nu_{\gamma\dot{\delta}}
			  \partial_\nu \breve{f}_{(0)}(x, \eta,  \bar{\eta}, \vartheta, \bar{\vartheta})
 		\\
	&& 			
        +\, \sqrt{-1} \sum_{\gamma, \dot{\delta}, \nu}
		      \theta^\gamma \bar{\theta}^{\dot{1}} \bar{\theta}^{\dot{2}}
			  \sigma^{\nu\dot{\delta}}_\gamma\,
			  \partial_\nu \breve{f}_{(\dot{\delta})}(x, \eta, \bar{\eta}, \vartheta, \bar{\vartheta})
		- \theta^1\theta^2\bar{\theta}^{\dot{1}} \bar{\theta}^{\dot{2}}\,
		    \square \breve{f}_{(0)}(x,\eta, \bar{\eta}, \vartheta, \bar{\vartheta})\,.
  \end{eqnarray*}
 In particular, an antichiral $\breve{f}\in C^\infty(\widehat{X}^{\widehat{\boxplus}})$
  has four independent components in $C^\infty(X)^{\Bbb C}[\eta, \bar{\eta}, \vartheta, \bar{\vartheta}]^\anticommuting$:
  $$
    \breve{f}_{(0)}\,,\;\;
	\breve{f}_{(\dot{\delta})}\,,\,{}_{\dot{\delta}\;=\;\dot{1}, \dot{2}}\,,\;\;
	\breve{f}_{(\dot{1}\dot{2})}\,.
  $$
 In terms of the standard antichiral coordinate functions
  $(x^{\prime\prime}, \theta, \bar{\theta}, \eta, \bar{\eta}, \vartheta, \bar{\vartheta})$
  on $\widehat{X}^{\widehat{\boxplus}}$,
 $$
  \breve{f}\;
   =\; \breve{f}_{(0)}(x^{\prime\prime}, \eta, \bar{\eta}, \vartheta, \bar{\vartheta})
         + \sum_{\dot{\delta}} \bar{\theta}^{\dot{\delta}}
		      \breve{f}_{(\dot{\delta})}(x^{\prime\prime}, \eta, \bar{\eta}, \vartheta, \bar{\vartheta})
		 + \bar{\theta}^{\dot{1}}\bar{\theta}^{\dot{2}}
		     \breve{f}_{(\dot{1}\dot{2})}(x^{\prime\prime}, \eta, \bar{\eta}, \vartheta, \bar{\vartheta})\,,		
 $$
 which is independent of $\theta$.
\end{lemma}
 
\bigskip

\noindent
We shall use this to understand how supersymmetries act on vector superfields in Wess-Zumino gauge.
 
Consider the infinitesimal supersymmetry transformation
  $$
    \delta_{(\eta,\bar{\eta})}\breve{V}\;
     :=\; (\eta Q+ \bar{\eta}\bar{Q})  \breve{V}\;
	 :=\;  \mbox{\Large $($}
	            \mbox{$\sum$}_\alpha \eta^\alpha Q_\alpha
				- \mbox{$\sum$}_{\dot{\beta}}\bar{\eta}^{\dot{\beta}} \bar{Q}_{\dot{\beta}}
		 	 \mbox{\Large $)$}\, \breve{V}			
  $$
  of $\breve{V}$.
Then
  $(\delta_{(\eta,\bar{\eta})}\breve{V})^\dag
     = \delta_{(\eta,\bar{\eta})}\breve{V}$.
However, for $\breve{V}$ in Wess-Zumino gauge,
  $(\eta Q+ \bar{\eta}\bar{Q})  \breve{V}$ remains a vector superfield
    (in the sense of reality condition and
	    matching of the independent components with the vector multiplet from representations of the $d=3+1$, $N=1$ supersymmetry algebra)
     but in general no longer in Wess-Zumino gauge (in the sense of the pattern as a $(\theta, \bar{\theta})$-polynomial).
This can be remedied by a gauge transformation:
(e.g., [Argu: Sec.\,4.3.1], [G-G-R-S: Sec.\,4.2.a.1], [W-B: Chap.\,VII, Exercise (8)], and
                    [We: Sec.\,15.3, Eq.\,(15.78)])

\bigskip

\begin{lemma} {\bf [existence and uniqueness of correcting gauge transformation]}\;
 Let $\breve{V}$ be a vector superfield in Wess-Zumino gauge.
 Then there exists a chiral function
   $\breve{\Lambda}_{(\eta,\bar{\eta}; \breve{V})}\in C^\infty(\widehat{X}^{\widehat{\boxplus}})^{\scriptsizech}$
    (now with the parameter level activated)
   depending ${\Bbb C}$-multilinearly on $(\eta,\bar{\eta})$ and $\breve{V}$
   such that
   the gauge transformation
   $$
     (\eta Q+ \bar{\eta}\bar{Q})  \breve{V}  \,
       -\, \sqrt{-1}
	      (\breve{\Lambda}_{(\eta,\bar{\eta}; \breve{V})}- \breve{\Lambda}_{(\eta,\bar{\eta}; \breve{V})}^\dag)	
   $$
   of $(\eta Q+ \bar{\eta}\bar{Q})  \breve{V}$
   is in Wess-Zumino gauge.
 Furthermore, one may require that the $(\theta,\bar{\theta})$-degree-$0$ component
   $\breve{\Lambda}_{(\eta,\bar{\eta}; \breve{V}), (0)}$
     of $\breve{\Lambda}_{(\eta,\bar{\eta}; \breve{V})}$ vanish,
  in which case $\breve{\Lambda}_{(\eta,\bar{\eta}; \breve{V})}$ is unique.
\end{lemma}

\medskip

\begin{proof}
 When $\breve{V}$ is in Wess-Zumino gauge,
   $(\theta, \bar{\theta})$-degree-$0$ component $((\eta Q+ \bar{\eta}\bar{Q})  \breve{V})_{(0)}$
   of $(\eta Q+ \bar{\eta}\bar{Q})  \breve{V}$ is always zero.
 It follows from Lemma~4.3.1
  and the same reasoning as the explicit computation in Sec.\:4.2 that leads to Lemma~4.2.4
  that there is a unique chiral function $\breve{\Lambda}\in C^\infty(\widehat{X}^{\widehat{\boxplus}})^\scriptsizech$
   associated to $(\eta Q+ \bar{\eta}\bar{Q})  \breve{V}$
    with the $(\theta, \bar{\theta})$-degree-$0$ component $\breve{\Lambda}_{(0)}=0$
  such that
   $(\eta Q+ \bar{\eta}\bar{Q})  \breve{V}  \,-\, \sqrt{-1}(\breve{\Lambda}- \breve{\Lambda}^\dag)$
   is in Wess-Zumino gauge.
 The same explicit computation implies also that this unique $\Lambda$ depends ${\Bbb C}$-multilinearly
   on $(\eta,\bar{\eta})$ and $\breve{V}$.
 This proves the lemma.
 
\end{proof}					

\medskip

\begin{definition} {\bf [supersymmetry in Wess-Zumino gauge]}\; {\rm
 Continuing Lemma~4.3.2.
 Set
  $$
   (\eta Q+ \bar{\eta}\bar{Q})  \breve{V}  \,
       -\, \sqrt{-1}
	        (\breve{\Lambda}_{(\eta,\bar{\eta}; \breve{V})}
		                 - \breve{\Lambda}_{(\eta,\bar{\eta}; \breve{V})}^\dag)\;	
     =\; \sum_\alpha \eta^\alpha Q_\alpha^{\!\tinyWZ} \breve{V}
	       - \sum_{\dot{\beta}}
		        \bar{\eta}^{\dot{\beta}}\bar{Q}_{\dot{\beta}}^{\!\tinyWZ}\breve{V}\,.
   $$
 This defines {\it (infinitesimal) supersymmetry transformations in Wess-Zumino gauge}
    $Q_\alpha^{\!\tinyWZ}$, $\bar{Q}_{\dot{\beta}}^{\!\tinyWZ}$
  that take a vector superfield in Wess-Zumino gauge to another in Wess-Zumino gauge.  	
}\end{definition}

\bigskip

Explicitly, let
$$
   \breve{V}\;=\;
    \sum_{\gamma,\dot{\delta},\nu}\theta^\gamma \bar{\theta}^{\dot{\delta}}
	   \sigma^\nu_{\gamma\dot{\delta}} V_{[\nu]}
	 + \sum_{\dot{\delta}}\theta^1\theta^2\bar{\theta}^{\dot{\delta}} \bar{\vartheta}_{\dot{\delta}}\,
	       \overline{V^{\prime\prime}_{(\delta)}}
	 + \sum_\gamma \theta^\gamma \bar{\theta}^{\dot{1}} \bar{\theta}^{\dot{2}}\vartheta_\gamma	
	        V^{\prime\prime}_{(\gamma)}
     + \theta^1\theta^2\bar{\theta}^{\dot{1}}\bar{\theta}^{\dot{2}}
	       V^\sim_{(0)}
 $$
 in $C^\infty(\widehat{X}^{\widehat{\boxplus}})$ be a vector superfield in Wess-Zumino gauge.
Then,
 {\small
 \begin{eqnarray*}
   \lefteqn{
    \delta_{\eta Q+ \bar{\eta}\bar{Q}}\breve{V}\;
      :=\; (\eta Q+ \bar{\eta}\bar{Q})  \breve{V} }
	\\
   && =\; \mbox{\Large $($}
	          \sum_\alpha \eta^\alpha \frac{\partial}{\partial\theta^\alpha}
			    -\sqrt{-1}\sum_{\alpha,\dot{\beta},\mu}
				    \eta^\alpha \sigma^\mu_{\alpha\dot{\beta}}\bar{\theta}^{\dot{\beta}}\partial_\mu
	         \mbox{\Large $)$}	    \breve{V}
			+  \mbox{\Large $($}
	          \sum_{\dot{\beta}}
			     \bar{\eta}^{\dot{\beta}} \frac{\partial}{\partial\bar{\theta}^{\dot{\beta}}}
			    + \sqrt{-1}\sum_{\alpha,\dot{\beta},\mu}
				    \theta^\alpha \sigma^\mu_{\alpha\dot{\beta}}\bar{\eta}^{\dot{\beta}}\partial_\mu
	            \mbox{\Large $)$}	    \breve{V}
    \\
   && =\; \sum_{\gamma, \dot{\beta}, \nu}
                 \theta^\gamma
                    \bar{\eta}^{\dot{\beta}}	\sigma^\nu_{\gamma\dot{\beta}} V_{[\nu]}
			 - \sum_{\dot{\delta}, \alpha, \nu}
			      \bar{\theta}^{\dot{\delta}}
				    \eta^\alpha \sigma^\nu_{\alpha\dot{\delta}}V_{[\nu]}
			 + \theta^1\theta^2
			      \sum_{\dot{\beta}}\bar{\eta}^{\dot{\beta}} \bar{\vartheta}_{\dot{\beta}}\,
					\overline{V^{\prime\prime}_{(\beta)})}
	 \\
    &&  \hspace{2em}
	         + \sum_{\gamma, \dot{\delta}} \theta^\gamma \bar{\theta}^{\dot{\delta}}
			      \mbox{\Large $($}
				    - \sum_{\alpha}\eta^\alpha \bar{\vartheta}_{\dot{\delta}}					
					    \varepsilon_{\alpha\gamma}\, \overline{V^{\prime\prime}_{(\delta)}}																								
                    + \sum_{\dot{\beta}}\bar{\eta}^{\dot{\beta}} \vartheta_\gamma									
					     \varepsilon_{\dot{\beta}\dot{\delta}}	V^{\prime\prime}_{(\gamma)}				
				  \mbox{\Large $)$}
             + \bar{\theta}^{\dot{1}} \bar{\theta}^{\dot{2}}
			      \sum_\alpha \eta^\alpha \vartheta_\alpha	V^{\prime\prime}_{(\alpha)}
				\\
	&& \hspace{2em}
	         + \sum_{\dot{\delta}, \dot{\beta}}
			      \theta^1\theta^2 \bar{\theta}^{\dot{\delta}} \bar{\eta}^{\dot{\beta}}
			      \mbox{\Large $($}
			         \varepsilon_{\dot{\beta}\dot{\delta}} V^\sim_{(0)}
				     - \sqrt{-1} \sum_{\gamma, \mu, \nu}
					      \sigma^{\mu\gamma}_{\;\;\;\;\dot{\beta}} \sigma^\nu_{\gamma\dot{\delta}}\,
					       \partial_\mu V_{[\nu]}
			      \mbox{\Large $)$}    \\
    && \hspace{2em}	
             + \sum_{\gamma, \alpha}
			       \theta^\gamma \bar{\theta}^{\dot{1}} \bar{\theta}^{\dot{2}} \eta^\alpha
                   \mbox{\Large $($}
                     \varepsilon_{\alpha\gamma} V^\sim_{(0)}
					  + \sqrt{-1} \sum_{\dot{\delta}, \mu, \nu}
						   \sigma^{\mu\dot{\delta}}_\alpha \sigma^\nu_{\gamma\dot{\delta}}\,
						    \partial_\mu V_{[\nu]}
                   \mbox{\Large $)$}				   \\
    && \hspace{2em}
             +\, \sqrt{-1}\,	\theta^1\theta^2\bar{\theta}^{\dot{1}}\bar{\theta}^{\dot{2}}			       
			        \mbox{\Large $($}					
					  \sum_{\alpha,\dot{\delta}, \mu}
					       \eta^\alpha \bar{\vartheta}_{\dot{\delta}}
						      \sigma^{\mu\dot{\delta}}_\alpha\,
						        \partial_\mu\, \overline{V^{\prime\prime}_{(\delta)}}
					  + \sum_{\dot{\beta},\gamma, \mu}
					       \bar{\eta}^{\dot{\beta}} \vartheta_\gamma
						     \sigma^{\mu\gamma}_{\dot{\beta}}\,
						      \partial_\mu V^{\prime\prime}_{(\gamma)}
					\mbox{\Large $)$}\,.
 \end{eqnarray*}}
Recall Lemma~4.3.1
  and
let $\breve{\Lambda}= \breve{\Lambda}_{(\eta, \bar{\eta}; \breve{V})}$
 be the unique chiral function in $C^\infty(\widehat{X}^{\widehat{\boxplus}})$ with
 $$
    \breve{\Lambda}_{(0)}\;=\; 0\,,\hspace{2em}
	\breve{\Lambda}_{(\gamma)}\;
	  =\; -\sqrt{-1} \sum_{\dot{\beta}, \nu} \bar{\eta}^{\dot{\beta}}
	              \sigma^\nu_{\gamma\dot{\beta}} V_{[\nu]}\,,\hspace{2em}				
    \breve{\Lambda}_{(12)}\;
	  =\; -\sqrt{-1} \sum_{\dot{\beta}} \bar{\eta}^{\dot{\beta}} \bar{\vartheta}_{\dot{\beta}}\,
				                   \overline{V^{\prime\prime}_{(\beta)}}\,.
 $$
I.e.
 $$
  \breve{\Lambda}_{(\eta, \bar{\eta}; \breve{V})} \; =\;
    -\,\sqrt{-1}\sum_{\gamma, \dot{\beta}, \nu}
	      \theta^\gamma \bar{\eta}^{\dot{\beta}} \sigma^\nu_{\gamma\dot{\beta}}V_{[\nu]}
	-\sqrt{-1}\,\theta^1\theta^2 \sum_{\dot{\beta}}	
	        \bar{\eta}^{\dot{\beta}}\bar{\vartheta}_{\dot{\beta}}\overline{V^{\prime\prime}_{(\beta)}} 			
	+ \sum_{\dot{\delta}, \dot{\beta}, \gamma, \mu, \nu}
	     \theta^1\theta^2 \bar{\theta}^{\dot{\delta}} \bar{\eta}^{\dot{\beta}}
			\sigma^{\mu\gamma}_{\;\;\;\;\dot{\delta}} \sigma^\nu_{\gamma\dot{\beta}} \partial_\mu V_{[\nu]}\,.
 $$
 Then,
 $$
  \breve{\Lambda}_{(\eta, \bar{\eta}; \breve{V})}^\dag \; =\;
    -\,\sqrt{-1}\sum_{\dot{\delta}, \alpha, \nu}
	      \bar{\theta}^{\dot{\delta}} \eta^\alpha \sigma^\nu_{\alpha\dot{\delta}}V_{[\nu]}
	+ \sqrt{-1}\,\bar{\theta}^{\dot{1}}\bar{\theta}^{\dot{2}}
	       \sum_{\alpha}	
	         \eta^\alpha \vartheta_\alpha V^{\prime\prime}_{(\alpha)}
	+ \sum_{\gamma, \alpha, \dot{\delta}, \mu, \nu}
	     \theta^\gamma \bar{\theta}^{\dot{1}} \bar{\theta}^{\dot{2}} \eta^\alpha
			\sigma^{\mu\dot{\delta}}_\gamma \sigma^\nu_{\alpha\dot{\delta}} \partial_\mu V_{[\nu]}
 $$
 and
 {\small
 \begin{eqnarray*}
   \lefteqn{
     \delta_{\eta Q+ \bar{\eta}\bar{Q}}\breve{V}\,
	    +\, \delta_{\breve{\Lambda}_{(\eta,\bar{\eta}; \breve{V})}}\breve{V}\;
      =\; (\eta Q+\bar{\eta}\bar{Q})\breve{V}
	           -\sqrt{-1}\,(\breve{\Lambda}_{(\eta, \bar{\eta}; \breve{V})}
			                               - \breve{\Lambda}_{(\eta, \bar{\eta}; \breve{V})}^\dag) 	 }
	 \\[1ex]
	  && =\;
	    \sum_{\gamma, \dot{\delta}} \theta^\gamma \bar{\theta}^{\dot{\delta}}
			      \mbox{\Large $($}
				    - \sum_{\alpha}\eta^\alpha \bar{\vartheta}_{\dot{\delta}}					
					    \varepsilon_{\alpha\gamma}\, \overline{V^{\prime\prime}_{(\delta)}}																								
                    + \sum_{\dot{\beta}}\bar{\eta}^{\dot{\beta}} \vartheta_\gamma									
					     \varepsilon_{\dot{\beta}\dot{\delta}}	V^{\prime\prime}_{(\gamma)}				
				  \mbox{\Large $)$}				
				\\
	&& \hspace{2em}
	         + \sum_{\dot{\delta}, \dot{\beta}}
			      \theta^1\theta^2 \bar{\theta}^{\dot{\delta}} \bar{\eta}^{\dot{\beta}}
			      \mbox{\Large $($}
			         \varepsilon_{\dot{\beta}\dot{\delta}} V^\sim_{(0)}
				     - \sqrt{-1} \sum_{\gamma, \mu, \nu}
					      \sigma^{\mu\gamma}_{\;\;\;\;\dot{\beta}} \sigma^\nu_{\gamma\dot{\delta}}\, F_{\mu\nu}
			      \mbox{\Large $)$}    \\
    && \hspace{2em}	
             + \sum_{\gamma, \alpha}
			       \theta^\gamma \bar{\theta}^{\dot{1}} \bar{\theta}^{\dot{2}} \eta^\alpha
                   \mbox{\Large $($}
                     \varepsilon_{\alpha\gamma} V^\sim_{(0)}
					  + \sqrt{-1} \sum_{\dot{\delta}, \mu, \nu}
						   \sigma^{\mu\dot{\delta}}_\alpha \sigma^\nu_{\gamma\dot{\delta}}\, F_{\mu\nu}
                   \mbox{\Large $)$}				   \\
    && \hspace{2em}
             +\, \sqrt{-1}\,	\theta^1\theta^2\bar{\theta}^{\dot{1}}\bar{\theta}^{\dot{2}}			       
			        \mbox{\Large $($}					
					  \sum_{\alpha,\dot{\delta}, \mu}
					       \eta^\alpha \bar{\vartheta}_{\dot{\delta}}
						      \sigma^{\mu\dot{\delta}}_\alpha\,
						        \partial_\mu\, \overline{V^{\prime\prime}_{(\delta)}}
					  + \sum_{\dot{\beta},\gamma, \mu}
					       \bar{\eta}^{\dot{\beta}} \vartheta_\gamma
						     \sigma^{\mu\gamma}_{\dot{\beta}}\,
						      \partial_\mu V^{\prime\prime}_{(\gamma)}
					\mbox{\Large $)$}\,,
 \end{eqnarray*}}
   where $F_{\mu\nu}:= \partial_\mu V^{(0)}_{[\nu]}- \partial_\nu V^{(0)}_{[\mu]}$,
 now resumes in Wess-Zumino gauge.

From the last expression and Definition~4.3.3,
one reads off
{\small
 \begin{eqnarray*}
   Q_\alpha^{\!\tinyWZ} \breve{V}
    & = &
     - \sum_{\gamma, \dot{\delta}} \theta^\gamma \bar{\theta}^{\dot{\delta}}
			\bar{\vartheta}_{\dot{\delta}}					
					    \varepsilon_{\alpha\gamma}\, \overline{V^{\prime\prime}_{(\delta)}}	
			- \sum_\gamma
   			\theta^\gamma \bar{\theta}^{\dot{1}} \bar{\theta}^{\dot{2}}
                   \mbox{\Large $($}
                     \varepsilon_{\alpha\gamma} V^\sim_{(0)}
					  + \sqrt{-1} \sum_{\dot{\delta}, \mu, \nu}
						   \sigma^{\mu\dot{\delta}}_\alpha \sigma^\nu_{\gamma\dot{\delta}}\, F_{\mu\nu}
                   \mbox{\Large $)$}				   \\
    && \hspace{2em}
             +\, \sqrt{-1}\,	\theta^1\theta^2\bar{\theta}^{\dot{1}}\bar{\theta}^{\dot{2}}			       			
					  \sum_{\dot{\delta}, \mu}
					      \bar{\vartheta}_{\dot{\delta}}
						      \sigma^{\mu\dot{\delta}}_\alpha\,
						        \partial_\mu\, \overline{V^{\prime\prime}_{(\delta)}}\,, 							
 \end{eqnarray*}}		
{\small
 \begin{eqnarray*}
   \bar{Q}_{\dot{\beta}}^{\!\tinyWZ}\breve{V}
    & = &
	    -\, \sum_{\gamma, \dot{\delta}} \theta^\gamma \bar{\theta}^{\dot{\delta}}
	           \vartheta_\gamma	\varepsilon_{\dot{\beta}\dot{\delta}} V^{\prime\prime}_{(\gamma)}			
	         + \sum_{\dot{\delta}}
			      \theta^1\theta^2 \bar{\theta}^{\dot{\delta}} \
			      \mbox{\Large $($}
			         \varepsilon_{\dot{\beta}\dot{\delta}} V^\sim_{(0)}
				     - \sqrt{-1} \sum_{\gamma, \mu, \nu}
					      \sigma^{\mu\gamma}_{\;\;\;\;\dot{\beta}} \sigma^\nu_{\gamma\dot{\delta}}\, F_{\mu\nu}
			      \mbox{\Large $)$}
				  \\	
    && \hspace{2em}
             -\, \sqrt{-1}\,	\theta^1\theta^2\bar{\theta}^{\dot{1}}\bar{\theta}^{\dot{2}}			       			 
					\sum_{\gamma, \mu}  \vartheta_\gamma \sigma^{\mu\gamma}_{\;\;\;\;\dot{\beta}}\,
					  \partial_\mu V^{\prime\prime}_{(\gamma)}\,.
 \end{eqnarray*}}
The supersymmetry algebra generated by
   $Q_\alpha^{\!\tinyWZ}$'s, $\bar{Q}_{\dot{\beta}}^{\!\tinyWZ}$'s, and $\partial_\mu$'s
  is now closed only up to a gauge transformation.

\bigskip

\subsection{Supersymmetric $U(1)$ gauge theory with matter on $X$
          in terms of $\widehat{X}^{\widehat{\boxplus}}$}

With the preparations in Sec.\,4.1 and Sec.\,4.2, we are now ready to construct\footnote{{\it Note
                                                                                                     for mathematicians.}\hspace{1em}
															                                        See [L-Y5: Sec.\:3.5, footnote 28] (SUSY(1)).
	                                                                                                  } 
 a supersymmetric $U(1)$ gauge theory with matter on $X$ in terms of functions on
 $\widehat{X}^{\widehat{\boxplus}}$.

\bigskip

\begin{flushleft}
{\bf Two basic derived\footnote{Here,
                                              we are not using the term `{\it derived}' in any deeper sense.
											  We only mean that such superfields arise from the combination of more basic superfields
											    such as small chiral superfields and vector superfields.
											  For example, the superpotential is a polynomial (or more generally holomorphic function) of
											    small chiral superfields and thus can be regarded as a "derived" superfield.
											  Caution that these derived superfields may go beyond the small function-ring
											    $C^\infty(\widehat{X}^{\widehat{\boxplus}})^\smalltiny$
											     and lie only in $C^\infty(\widehat{X}^{\widehat{\boxplus}})$.
												 } 
			superfields: gaugino superfield and kinetic-term superfield}
\end{flushleft}
Unlike chiral or antichiral superfields, a vector superfield $\breve{V}$
 contains no components that involve space-time derivatives.
For that reason, to construct a supersymmetric action functional for components of $\breve{V}$,
one needs to work out appropriate derived superfields from $\breve{V}$  first.

\bigskip

\begin{lemma-definition} {\bf [gaugino superfield]}\;
 {\rm (Cf.\:[Wess \& Bagger: Eq.\:(6.7)].)}\\
 Let $\breve{V}\in C^\infty(\widehat{X}^{\widehat{\boxplus}})$ be a vector superfield.
 Define\footnote{The
                                design here  is made so that\;
								  $W_\alpha
								    = \vartheta_\alpha V^{\prime\prime}_{(\alpha)}\,
								        +\,(\mbox{terms of $(\theta,\bar{\theta})$-degree $\ge 1$})$\;  and\\
								  $\bar{W}_{\dot{\beta}}
								    =  \bar{\vartheta}_{\dot{\beta}}\,
									       \overline{V^{\prime\prime}_{(\beta)}}\,
								        +\, (\mbox{terms of $(\theta,\bar{\theta})$-degree $\ge 1$})$.			
                                Caution that, while $e_{\alpha^\prime}= \partial/\partial \theta^\alpha\,+\,\cdots$,\;
                                    $e_{\beta^{\prime\prime}}
									   = -\,\partial/\partial\bar{\theta}^{\dot{\beta}}\,+\,\cdots\,.$
								          } 
     \marginpar{\vspace{0em}\raggedright\tiny
         \raisebox{-1ex}{\hspace{-2.4em}\LARGE $\cdot$}Cf.\,[\,\parbox[t]{20em}{Wess
		\& Bagger:\\ Eq.\:(6.7)].}}
   $$
    W_\alpha\;
	  :=\; e_{2^{\prime\prime}}e_{1^{\prime\prime}} e_{\alpha^\prime} \breve{V}
	\hspace{2em}
	(\,\mbox{resp.}\;\;
	\bar{W}_{\dot{\beta}}\;
	  :=\; e_{1^\prime}e_{2^\prime} e_{\beta^{\prime\prime}} \breve{V}
	   \,)\,
   $$
   $\alpha=1,2$, $\dot{\beta}=\dot{1}, \dot{2}$.
 Then
  $(1)$
   $W_\alpha$ (resp.\ $\bar{W}_{\dot{\beta}}$) is chiral (resp.\ antichiral).
  $(2)$
   $W_\alpha$ and $\bar{W}_{\dot{\beta}}$ are invariant under gauge transformations on $\breve{V}$.

 {\rm
     $W_\alpha$, $\bar{W}_{\dot{\beta}}$
	   are called the {\it gaugino superfields} associated to the vector superfield $\breve{V}$.
      }
\end{lemma-definition}

\medskip

\begin{proof}
 The same as [L-Y5: Proof of Lemma/Definition 3.5.1] (SUSY(1))
  but now for vector superfield in the sense of Definition~4.1.3.
 Details are repeated below due to the importance of these quantities.
 
 For Statement (1),
 $$
  \begin{array}{llrll}
   e_{1^{\prime\prime}}W_\alpha
    & = & -\,  e_{2^{\prime\prime}}(e_{1^{\prime\prime}})^2 e_{\alpha^\prime}\breve{V}
	       & = & 0\,,          \\
  e_{2^{\prime\prime}}W_\alpha
    & = &   (e_{2^{\prime\prime}})^2 e_{1^{\prime\prime}} e_{\alpha^\prime}\breve{V}
	       &  = & 0
  \end{array}
 $$
  since $(e_{1^{\prime\prime}})^2=(e_{2^{\prime\prime}})^2=0$.
 Similarly for the antichirality of  $\bar{W}_{\dot{\beta}}$.
  
 For Statement (2),
   under a gauge transformation
     $\breve{V}\rightarrow \breve{V} -\sqrt{-1}(\breve{\Lambda}-\breve{\Lambda}^\dag)$
	 on $\breve{V}$ specified by a small chiral superfield
	  $\breve{\Lambda}\in C^\infty(\widehat{X}^{\widehat{\boxplus}})^{\smallscriptsize, \scriptsizech}$,
   \begin{eqnarray*}
     W_\alpha  & \rightarrow
	  & e_{2^{\prime\prime}}e_{1^{\prime\prime}}e_{\alpha^\prime}
            \mbox{\large $($}
			  \breve{V} -\sqrt{-1}(\breve{\Lambda}-\breve{\Lambda}^\dag)
			\mbox{\large $)$}\;\; 		
         =\;\; W_\alpha\,
	              -\,\sqrt{-1}\, e_{2^{\prime\prime}}e_{1^{\prime\prime}}e_{\alpha^\prime}
                                              \breve{\Lambda}  \\				
	 && =\;  W_\alpha\,
	              -\,\sqrt{-1}\,
				      \mbox{\large $($}
					   \{e_{1^{\prime\prime}}, e_{\alpha^\prime}\} e_{2^{\prime\prime}}
					   -  e_{2^{\prime\prime}}e_{\alpha^\prime}e_{1^{\prime\prime}}
					  \mbox{\large $ )$}
                                      \breve{\Lambda}\;\;\;=\;\;\;  W_\alpha
   \end{eqnarray*}
 since
   $\Lambda^\dag$ is antichiral (thus, $e_{\alpha^\prime}\breve{\Lambda}^\dag=0$)  and
   $\Lambda$ is chiral
     (thus, $e_{1^{\prime\prime}}\breve{\Lambda}
	                = e_{2^{\prime\prime}}\breve{\Lambda}=0$).
 Similarly for $\bar{W}_{\dot{\beta}}$.
  
\end{proof}

\bigskip

It follows that in the construction of a supersymmetric $U(1)$-gauge theory with matter,
 one may assume that the vector superfield $\breve{V}$ is in Wess-Zumino gauge,
  which encodes the component fields
    $$
	  V_{[\mu]}\,,\hspace{1em}
	  V^{\prime\prime}_{(\alpha)}\,,  \hspace{1em}
	  V^\sim_{(0)}
	$$
 on $X$.
Here, $\mu=0,1,2,3$, and $\alpha=1,2$.
For $\breve{V}$  in Wess-Zumino gauge,
 $\breve{V}^3=0$ and
 its exponential $e^{\breve{V}}$
   is simply the polynomial $1+\breve{V}+\frac{1}{2}\breve{V}^2$ in $\breve{V}$.

Under a gauge transformation specified by a small chiral superfield
  $\breve{\Lambda}$,
  a small chiral superfield
   $\breve{\Phi}\in C^\infty(\widehat{X}^{\widehat{\boxplus}})^{\smallscriptsize, \scriptsizech}$
   transforms as
     \marginpar{\vspace{0em}\raggedright\tiny
         \raisebox{-1ex}{\hspace{-2.4em}\LARGE $\cdot$}Cf.\,[\,\parbox[t]{20em}{Wess
		\& Bagger:\\ Eq.\:(7.1)].}}
  $$
    \breve{\Phi}\; \longrightarrow\; e^{\sqrt{-1}\,\breve{\Lambda}}\breve{\Phi}
  $$
  while
  $\breve{\Phi}^\dag \in C^\infty(\widehat{X}^{\widehat{\boxplus}})^{\smallscriptsize, \scriptsizeach}$
   transforms as
     \marginpar{\vspace{0em}\raggedright\tiny
         \raisebox{-1ex}{\hspace{-2.4em}\LARGE $\cdot$}Cf.\,[\,\parbox[t]{20em}{Wess
		\& Bagger:\\ Eq.\:(7.3)].}}
  $$
    \breve{\Phi}^\dag\; \longrightarrow\;  \breve{\Phi}^\dag e^{-\sqrt{-1}\,\breve{\Lambda}^\dag}\,.
  $$
It follows that
  
\bigskip

\begin{lemma-definition} {\bf [gauge-invariant kinetic term for small chiral superfield]}\; Let\\
 $\breve{V}\in C^\infty(\widehat{X}^{\widehat{\boxplus}})$ be a vector superfield and
 $\breve{\Phi}\in C^\infty(\widehat{X}^{\widehat{\boxplus}})^{\smallscriptsize, \scriptsizech}$
  be a small chiral superfield on $X$.
 Then the product
     \marginpar{\vspace{0em}\raggedright\tiny
         \raisebox{-1ex}{\hspace{-2.4em}\LARGE $\cdot$}Cf.\,[\,\parbox[t]{20em}{Wess
		\& Bagger:\\ Eq.\:(7.6)].}}
   $$
     \breve{\Phi}^\dag  e^{\breve{V}}\breve{\Phi}
   $$
  is gauge-invariant.
 {\rm
  Since the expression of the product in $(x,\theta,\bar{\theta},\vartheta,\bar{\vartheta})$
   involves space-time derivatives ($\partial_\mu$, $\mu=0,1,2,3$) of components of $\breve{\Phi}$,
   this product is called the {\it gauge-invariant kinetic term for the small chiral superfield} $\breve{\Phi}$.
   }
\end{lemma-definition}

\medskip

\begin{proof}
 By construction,
 under the gauge transformation specified by a small chiral superfield $\breve{\Lambda}$,
    \begin{eqnarray*}
       \breve{\Phi}^\dag  e^{\breve{V}}\breve{\Phi}
	     & \longrightarrow
		 & \mbox{\Large $($}
		      \Phi^\dag  e^{-\,\sqrt{-1}\,\breve{\Lambda}^\dag}
		    \mbox{\Large $)$}\,
		      e^{\breve{V}-\sqrt{-1}(\breve{\Lambda}-\breve{\Lambda}^\dag)}\,
		    \mbox{\Large $($}
		      e^{\sqrt{-1}\,\breve{\Lambda}}\breve{\Phi}
		    \mbox{\Large $)$}                                 \\
    && \hspace{1em}
	        =\; \breve{\Phi}^\dag\,
		        e^{-\,\sqrt{-1}\breve{\Lambda}^\dag
				        + \breve{V}-\sqrt{-1}(\breve{\Lambda}-\breve{\Lambda}^\dag)
						+\sqrt{-1}\,\breve{\Lambda}}\,
			   \breve{\Phi}\;\;
            =\;\;    \breve{\Phi}^\dag  e^{\breve{V}}\breve{\Phi}\,.			
   \end{eqnarray*}
 		
\end{proof}

\bigskip

Note that, for $\Phi$ a small chiral superfield and $V$ a vector superfield on $X$,
  $W_\alpha$ and $\bar{W}_{\dot{\beta}}$ in general are no longer tame
  while
    $\breve{\Phi}^\dag e^{\breve{V}}\breve{\Phi}$
	always lies in $C^\infty(\widehat{X}^{\widehat{\boxplus}})^\mediumscriptsize$.

\bigskip

\begin{flushleft}
{\bf Explicit computations/formulae for $\breve{V}$ in Wess-Zumino gauge}
\end{flushleft}
Let
{\small
 \begin{eqnarray*}
  \breve{V}
  &= &		
	\sum_{\alpha,\dot{\beta}, \mu}\theta^\alpha \bar{\theta}^{\dot{\beta}}
	  \sigma^\mu_{\alpha\dot{\beta}} V_{[\mu]}(x)	
	+ \sum_{\dot{\beta}}\theta^1\theta^2\bar{\theta}^{\dot{\beta}}	
          \bar{\vartheta}_{\dot{\beta}}\,
		        \overline{V^{\prime\prime}_{(\beta)}}(x)					
    +\, \sum_\alpha \theta^\alpha\bar{\theta}^{\dot{1}}\bar{\theta}^{\dot{2}}
		  \vartheta_\alpha\, V^{\prime\prime}_{(\alpha)} (x)
	+ \theta^1\theta^2\bar{\theta}^{\dot{1}}\bar{\theta}^{\dot{2}}\,
		  V^\sim_{(0)}(x)\,,		
 \end{eqnarray*}
}be 
 a vector superfield in Wess-Zumino gauge,
  in the standard coordinate functions $(x, \theta, \bar{\theta}, \vartheta, \bar{\vartheta})$
  on $\widehat{X}^{\widehat{\boxplus}}$.
Recall
   the chiral coordinate functions $(x^\prime, \theta, \bar{\theta}, \vartheta, \bar{\vartheta})$ and
   the antichiral coordinate functions $(x^{\prime\prime}, \theta, \bar{\theta}, \vartheta, \bar{\vartheta})$
 on $\widehat{X}^{\widehat{\boxplus}}$, where
 $$
   x^{\prime\,\mu} \;
    :=\;  x^\mu
	       + \sqrt{-1}\sum_{\alpha, \dot{\beta}}
		        \theta^\alpha \sigma^\mu_{\alpha\dot{\beta}}   \bar{\theta}^{\dot{\beta}}
			\hspace{2em}\mbox{and}\hspace{2em}
   x^{\prime\prime\,\mu} \;
    :=\;  x^\mu
	       - \sqrt{-1}\sum_{\alpha, \dot{\beta}}
		        \theta^\alpha \sigma^\mu_{\alpha\dot{\beta}}   \bar{\theta}^{\dot{\beta}}\,.			
 $$
It is convenient to compute
  $W_\alpha$ and $W_1W_2$
     in the chiral coordinates $(x^\prime, \theta, \bar{\theta}, \vartheta, \bar{\vartheta})$ and
 $\bar{W}_{\dot{\beta}}$ and $\bar{W}_{\dot{2}}\bar{W}_{\dot{1}}$
     in the antichiral coordinates $(x^\prime, \theta, \bar{\theta}, \vartheta, \bar{\vartheta})$.

\bigskip

\noindent $\tinybullet$
$W_\alpha$:
(in chiral coordinate functions $(x^\prime, \theta,\bar{\theta},\vartheta,\bar{\vartheta})$
    on $\widehat{X}^{\widehat{\boxplus}}$)
 
 \medskip

 \noindent
Through the $C^\infty$-hull structure of $C^\infty(\widehat{X})$,
 one can express $\breve{V}$
 in terms of the chiral coordinate functions $(x^\prime, \theta, \bar{\theta}, \vartheta, \bar{\vartheta})$
 on $\widehat{X}^{\widehat{\boxplus}}$ as
 {\small
 \begin{eqnarray*}
  \breve{V}
  &= &		
	\sum_{\alpha,\dot{\beta}, \mu}\theta^\alpha \bar{\theta}^{\dot{\beta}}
	   \sigma^\mu_{\alpha\dot{\beta}} V_{[\mu]}(x^\prime)	
	+ \sum_{\dot{\beta}}\theta^1\theta^2\bar{\theta}^{\dot{\beta}}	
          \bar{\vartheta}_{\dot{\beta}}\,
		        \overline{V^{\prime\prime}_{(\beta)}}(x^\prime)					
    +\, \sum_\alpha \theta^\alpha\bar{\theta}^{\dot{1}}\bar{\theta}^{\dot{2}}
		  \vartheta_\alpha\, V^{\prime\prime}_{(\alpha)} (x^\prime)
     \\
	&& \hspace{3em}	
	+\, \theta^1\theta^2\bar{\theta}^{\dot{1}}\bar{\theta}^{\dot{2}}\,
	    \mbox{\Large $($}
		  V^\sim_{(0)}(x^\prime)
		  - 2 \sqrt{-1}\,\sum_{\mu, \nu} \eta^{\mu\nu}\,\partial_{\mu} V_{[\nu]}(x^\prime)
		\mbox{\Large $)$}  \,.
 \end{eqnarray*}
  }
 
 \noindent
 Recall that
   $e_{\alpha^\prime}x^{\prime\,\mu}
      =  2\,\sqrt{-1}\,\sum_{\dot{\beta}}
	           \sigma^\mu_{\alpha\dot{\beta}}\,\bar{\theta}^{\dot{\beta}}$.
 Then, a straightforward computation gives
 {\small
 \begin{eqnarray*}
  \lefteqn{
   e_{\alpha^\prime}\breve{V} \;\;
     =\;\; \sum_{\dot{\delta}, \nu}
	       \bar{\theta}^{\dot{\delta}}\sigma^\nu_{\alpha\dot{\delta}}\,V_{[\nu]}(x^\prime)
		  - \sum_{\gamma, \dot{\delta}}
		      \theta^\gamma \bar{\theta}^{\dot{\delta}}\bar{\vartheta}_{\dot{\delta}}\,
			     \varepsilon_{\alpha\gamma}\overline{V^{\prime\prime}_{(\delta)}}(x^\prime)
		 + \bar{\theta}^{\dot{`}}\bar{\theta}^{\dot{2}} \vartheta_\alpha
		       V^{\prime\prime}_{(\alpha)} (x^\prime)	
       }\\
	 && \hspace{0em}
	 +\,\sum_{\gamma}\theta^\gamma \bar{\theta}^{\dot{1}} \bar{\theta}^{\dot{2}}
	     \left\{\rule{0ex}{1.2em}\right.
		  2\,\sqrt{-1}\,
		    \sum_{\dot{\delta}, \dot{\beta}, \mu, \nu}
			  \varepsilon^{\dot{\delta}\dot{\beta}} \sigma^\nu_{\gamma\dot{\delta}}
			    \sigma^\mu_{\alpha\dot{\beta}} \, \partial_\mu V_{[\nu]}(x^\prime)
			 + \varepsilon_{\alpha\gamma}
			     \mbox{\Large $($}
			     - V^\sim_{(0)}(x^\prime)
				 + 2\sqrt{-1} \sum_{\mu, \nu}\eta^{\mu\nu}
				     \partial_\mu V_{[\nu]}(x^\prime)
			    \mbox{\Large $)$}		
		 \left.\rule{0ex}{1.2em}\right\}
	 \\
    && \hspace{0em}
	  -\,2\,\sqrt{-1}\, \theta^1\theta^2\bar{\theta}^{\dot{1}} \bar{\theta}^{\dot{2}}
	      \sum_{\dot{\delta}, \dot{\beta}. \mu}
		  \bar{\vartheta}_{\dot{\delta}}  \varepsilon^{\dot{\delta}\dot{\beta}}
		     \sigma^\mu_{\alpha\dot{\beta}}\,
			 \partial_\mu \overline{V^{\prime\prime}_{\delta}}(x^\prime)     \,.	
 \end{eqnarray*}
 }Since 
 $e_{1^{\prime\prime}}x^{\prime}= e_{2^{\prime\prime}}x^\prime =0$,
 {\small
 \begin{eqnarray*}
  \lefteqn{
   W_\alpha\;\; :=\;\;
      e_{2^{\prime\prime}}e_{1^{\prime\prime}}(e_{\alpha^\prime}\breve{V})\;\;
     =\;\;
	  (-1)^2\,
	 \mbox{\Large $\frac{\partial}{\partial \bar{\theta}^{\dot{2}}}$}\,
 	   \mbox{\Large $\frac{\partial}{\partial \bar{\theta}^{\dot{1}}}$}\,
	    \mbox{\large $($}
		   (e_\alpha \breve{V})(x^\prime, \theta, \bar{\theta}, \vartheta, \bar{\vartheta})
		\mbox{\large $)$}
	  }\\
  && =\;\;
    \vartheta_\alpha V^{\prime\prime}_{(\alpha)}(x^\prime)
  			\\
  && \hspace{2em}		
	+ \sum_\gamma
	    \theta^\gamma
		 \left\{\rule{0ex}{1.2em}\right.
		 2\sqrt{-1} \,
		  \sum_{\dot{\delta}, \dot{\beta}, \mu, \nu}
		   \varepsilon^{\dot{\delta}\dot{\beta}}\,\sigma^\nu_{\gamma\dot{\delta}}\,
		    \sigma^\mu_{\alpha\dot{\beta}}\,
			\partial_\mu V_{[\nu]}(x^\prime)
	    + \varepsilon_{\alpha\gamma}
		     \mbox{\Large $($}
		      -V^\sim_{(0)}(x^\prime)
			   + 2\sqrt{-1} \sum_{\mu, \nu}\eta^{\mu\nu}
				     \partial_\mu V_{[\nu]}(x^\prime)
		     \mbox{\Large $)$}
		 \left.\rule{0ex}{1.2em}\right\}
	   \\
   && \hspace{2em}
	  -\,2\,\sqrt{-1}\, \theta^1\theta^2
	      \sum_{\dot{\delta}, \dot{\beta}. \mu}
		  \bar{\vartheta}_{\dot{\delta}}  \varepsilon^{\dot{\delta}\dot{\beta}}
		     \sigma^\mu_{\alpha\dot{\beta}}\,
			 \partial_\mu \overline{V^{\prime\prime}_{\delta}}(x^\prime)     \,.			
 \end{eqnarray*}	
 } 

\noindent
Applying the family of identities from raising or lowering the spinor index in the defining identity  of the Dirac/Pauli matrices
   $$
  (\sigma^\mu\,\bar{\sigma}^\nu + \sigma^\nu\,\bar{\sigma}^\mu)_\alpha^{\:\:\gamma}\;
     =\; -2\eta^{\mu\nu}\,\delta_\alpha^{\:\:\gamma}
   $$
    and
   a relabelling of the $\mu$, $\nu$ indices
 to the summation of terms of involving $\partial_\mu V_{[\nu]}(x^\prime)$,
one has the simplification in the end
 {\small
 \begin{eqnarray*}
  W_\alpha & = &
     \vartheta_\alpha V^{\prime\prime}_{(\alpha)}(x^\prime)
	- \sum_\gamma
	    \theta^\gamma
		 \left\{\rule{0ex}{1.2em}\right.
		  \sqrt{-1} \,
		  \sum_{\mu, \nu, \dot{\beta}, \dot{\delta}}
		   \varepsilon^{\dot{\beta}\dot{\delta}}\,
		     \sigma^\mu_{\alpha\dot{\beta}}\,\sigma^\nu_{\gamma\dot{\delta}}\,
			 F_{\mu\nu}(x^\prime)
	    + \varepsilon_{\alpha\gamma}\,V^\sim_{(0)}(x^\prime)			   			 			
		 \left.\rule{0ex}{1.2em}\right\}
	   \\
   &&
	  -\,2\,\sqrt{-1}\, \theta^1\theta^2
	      \sum_{\dot{\delta}, \dot{\beta}. \mu}
		  \bar{\vartheta}_{\dot{\delta}}  \varepsilon^{\dot{\delta}\dot{\beta}}
		     \sigma^\mu_{\alpha\dot{\beta}}\,
			 \partial_\mu \overline{V^{\prime\prime}_{\delta}}(x^\prime)     \,,
 \end{eqnarray*}
 }where 
 $F_{\mu\nu}:= \partial_\mu V_{[\nu]}- \partial_\nu V_{[\nu]}$.

\bigskip

\noindent $\tinybullet$
$W_1W_2$:
(in chiral coordinate functions $(x^\prime, \theta,\bar{\theta},\vartheta,\bar{\vartheta})$
    on $\widehat{X}^{\widehat{\boxplus}}$)

\medskip

\noindent
From the expression of $W_\alpha$ above,
{\small
\begin{eqnarray*}
 \lefteqn{W_1W_2\;\;
  =\;\;
  \vartheta_1\vartheta_2\, V^{\prime}_{(1)}(x^\prime) V^{\prime\prime}_{(2)}(x^\prime)
     }\\
  && \hspace{2em}	
  +\,\sum_\gamma
     \theta^\gamma
      \left\{\rule{0ex}{1.2em}\right.
	   \vartheta_1 V^{\prime\prime}_{(1)}(x^\prime)\,
	      \mbox{\Large $($}
		   \sqrt{-1}
		    \sum_{\mu, \nu, \dot{\beta}, \dot{\delta}}
			  \varepsilon^{\dot{\beta}\dot{\delta}}\,
			    \sigma^\mu_{2\dot{\beta}}\,\sigma^\nu_{\gamma\dot{\delta}}\,
				F_{\mu\nu}(x^\prime)
		    + \varepsilon_{2\gamma} V^\sim_{(0)}(x^\prime)
		  \mbox{\Large $)$}
		  \\
		&& \hspace{8em}
	   -\, \vartheta_2 V^{\prime\prime}_{(2)}(x^\prime)\,
	      \mbox{\Large $($}
		   \sqrt{-1}
		    \sum_{\mu, \nu, \dot{\beta}, \dot{\delta}}
			  \varepsilon^{\dot{\beta}\dot{\delta}}\,
			    \sigma^\mu_{1\dot{\beta}}\,\sigma^\nu_{\gamma\dot{\delta}}\,
				F_{\mu\nu}(x^\prime)
		   + \varepsilon_{1\gamma} V^\sim_{(0)}(x^\prime)
		  \mbox{\Large $)$}
	  \left.\rule{0ex}{1.2em}\right\}
	  \\
  && \hspace{2em}
   +\, \theta^1\theta^2
          \left\{\rule{0ex}{1.2em}\right.
		   2\,\sqrt{-1}\,
		    \mbox{\Large $($}
		     \vartheta_1
			  \sum_{\dot{\beta}, \dot{\delta}, \mu}
			    \bar{\vartheta}_{\dot{\delta}}\,
				 \varepsilon^{\dot{\beta}\dot{\delta}}\, \sigma^\mu_{2\dot{\beta}}
				 , \partial_\mu\,\overline{V^{\prime\prime}_{(\delta)}}(x^\prime)\,
				   V^{\prime\prime}_{(1)}(x^\prime)
			    \\[-1ex]
			&& \hspace{14em}	
             -\,  \vartheta_2
			  \sum_{\dot{\beta}, \dot{\delta}, \mu}
			    \bar{\vartheta}_{\dot{\delta}}\,
				 \varepsilon^{\dot{\beta}\dot{\delta}}\, \sigma^\mu_{1\dot{\beta}}
				 , \partial_\mu\,\overline{V^{\prime\prime}_{(\delta)}}(x^\prime)\,
				   V^{\prime\prime}_{(2)}(x^\prime)	
            \mbox{\Large $)$}				
                \\
			&& \hspace{8em}			
       +\,\sum_{\gamma\gamma^\prime}
	         \varepsilon^{\gamma\gamma^\prime}
			 \mbox{\Large $($}
			  \sqrt{-1} \sum_{\mu, \nu, \dot{\beta}, \dot{\delta}}
			   \varepsilon^{\dot{\beta}\dot{\delta}}\,
			    \sigma^\mu_{1\dot{\beta}}\, \sigma^\nu_{\gamma\dot{\delta}}\,
				  F_{\mu\nu}(x^\prime)
			  + \varepsilon_{1\gamma}\, V^\sim_{(0)}(x^\prime)	
			 \mbox{\Large $)$}
			      \\
			  && \hspace{14em}		
             \cdot			
			 \mbox{\Large $($}
              \sqrt{-1} \sum_{\mu^\prime, \nu^\prime, \dot{\beta}^\prime, \dot{\delta}^\prime}
			   \varepsilon^{\dot{\beta}^\prime \dot{\delta}^\prime}\,
			    \sigma^{\mu^\prime}_{2\dot{\beta}^\prime}\,
				\sigma^{\nu^\prime}_{\gamma^\prime \dot{\delta}^\prime}\,
				  F_{\mu^\prime\nu^\prime}(x^\prime)
			  + \varepsilon_{2\gamma^\prime}\, V^\sim_{(0)}(x^\prime)	
            \mbox{\Large $)$}			 	
		  \left.\rule{0ex}{1.2em}\right\}\,.
\end{eqnarray*}
} 

\noindent
The $\theta^1\theta^2$-component can be simplified as follows:
 \begin{itemize}
  \item[\LARGE $\cdot$]
   Terms involving $\vartheta_\alpha \bar{\vartheta}_{\dot{\delta}}$
    can be combined to\\
	$
	 -\,2\,\sqrt{-1}\sum_{\alpha, \dot{\beta}, \mu}
      \vartheta_\alpha \bar{\vartheta}_{\dot{\beta}}	 \,
	   \bar{\sigma}^{\mu\dot{\beta}\alpha}\,
	     V^{\prime\prime}_{(\alpha)}(x^\prime)\,
		    \partial_\mu\overline{V^{\prime\prime}_{(\beta)}}(x^\prime)
	$
	after some relabelling.

  \item[\LARGE $\cdot$]	
   For the remaining terms, terms involving a
     $F_{\mu\nu}(x^\prime)F_{\mu^\prime\nu^\prime}(x^\prime)$ factor
	can be combined to\\
    $-2\sum_{\mu, \nu, \mu^\prime, \nu^\prime}
	    \eta^{\mu\mu^\prime} \eta^{\nu\nu^\prime}
		    F_{\mu\nu}(x^\prime)F_{\mu^\prime\nu^\prime}(x^\prime)
	  + \sqrt{-1}
	       \sum_{\mu, \nu, \mu^\prime, \nu^\prime}
	        \varepsilon^{\,\mu\nu\mu^\prime\nu^\prime} F_{\mu\nu}(x^\prime)F_{\mu^\prime\nu^\prime}$
	after first converting them to the sum
	 $$
	   -\,\mbox{\Large $\frac{1}{4}$}
	     \sum_{\mbox{\tiny
		                    $\begin{array}{c}
		                        \mu, \nu, \mu^\prime, \nu^\prime; \\
		                        \gamma, \dot{\beta}, \dot{\delta},\\
						        \gamma^\prime, \dot{\beta}^\prime, \dot{\delta}^\prime
						      \end{array}$
						   }}							
		  \varepsilon^{\gamma\gamma^\prime}\, \varepsilon^{\dot{\beta}\dot{\delta}}\,
            \varepsilon^{\dot{\beta}^\prime \dot{\delta}^\prime}\,
           \mbox{\large $($}
		     \sigma^\mu_{1\dot{\beta}}\,\sigma^\nu_{\gamma\dot{\delta}}
			    -    \sigma^\nu_{1\dot{\beta}}\,\sigma^\mu_{\gamma\dot{\delta}}
           \mbox{\large $)$}\,
		   \mbox{\large $($}
		     \sigma^{\mu^\prime}_{2\dot{\beta}^\prime}\,
			      \sigma^{\nu^\prime}_{\gamma^\prime\dot{\delta}^\prime}
			 - \sigma^{\nu^\prime}_{2\dot{\beta}^\prime}\,
			      \sigma^{\mu^\prime}_{\gamma^\prime\dot{\delta}^\prime}				
           \mbox{\large $)$}\,
         F_{\mu\nu}(x^\prime)F_{\mu^\prime \nu^\prime}(x^\prime)\,,
	 $$
	via relabelling of indices and using the property that $F_{\mu\nu}=-F_{\nu\mu}$,
	and then employing the following identity
	 $$
	   \Tr \mbox{\large $($}\sigma^{\mu\nu}\,\sigma^{\mu^\prime\nu^\prime}  \mbox{\large $)$}\;
	    =\; -\,\mbox{\Large $\frac{1}{2}$}\,
		           \mbox{\large $($}
			         \eta^{\mu\mu^\prime}\,\eta^{\nu\nu^\prime}
				     - \eta^{\mu\nu^\prime}\,\eta^{\nu\mu^\prime}\,			
			      \mbox{\large $)$}
			 -\,\mbox{\Large $\frac{\sqrt{-1}}{2}$}\,
			    \varepsilon^{\,\mu\nu\mu^\prime\nu^\prime}\,,
	 $$
	 where
	    $$		  		
          {\sigma^{\mu\nu}}_\alpha^{\:\:\gamma}\;
		    :=\;  \mbox{\Large $\frac{1}{4}$}
			       (\sigma^\mu\,\bar{\sigma}^\nu - \sigma^\nu\,\bar{\sigma}^\mu)_\alpha^{\:\:\gamma}\;
        $$
		   are the generators of the Lorentz group in the spinor representation        and
		$\varepsilon^{\,\mu\nu\mu^\prime\nu^\prime}$ is the space-time volume-element tensor
	       with respect to the frame $(\partial_\mu)_\mu$ with $\varepsilon^{0123}=1$.
	
  \item[\LARGE $\cdot$]	
   Terms involving a simple $F_{\mu\nu}$ or $F_{\mu^\prime\nu^\prime}$ factor
	cancel each other after relabelling and using the property that $F_{\mu\nu}= - F_{\nu\mu}$;
	and, hence, do not contribute in the end.
	
  \item[\LARGE $\cdot$]	
   Terms without a $F_{\mu\nu}$ or $F_{\mu^\prime\nu^\prime}$ factor can be combined to
   $V^\sim_{(0)}(x^\prime)^{\,2}$.
 \end{itemize}

 In summary,
{\small
\begin{eqnarray*}
 \lefteqn{W_1W_2\;\;
  =\;\;
  \vartheta_1\vartheta_2\, V^{\prime}_{(1)}(x^\prime) V^{\prime\prime}_{(2)}(x^\prime)
     }\\
  && \hspace{2em}	
  +\,\sum_\gamma
     \theta^\gamma
      \left\{\rule{0ex}{1.2em}\right.
	   \vartheta_1 V^{\prime\prime}_{(1)}(x^\prime)\,
	      \mbox{\Large $($}
		   \sqrt{-1}
		    \sum_{\mu, \nu, \dot{\beta}, \dot{\delta}}
			  \varepsilon^{\dot{\beta}\dot{\delta}}\,
			    \sigma^\mu_{2\dot{\beta}}\,\sigma^\nu_{\gamma\dot{\delta}}\,
				F_{\mu\nu}(x^\prime)
		    + \varepsilon_{2\gamma} V^\sim_{(0)}(x^\prime)
		  \mbox{\Large $)$}
		  \\
		&& \hspace{8em}
	   -\, \vartheta_2 V^{\prime\prime}_{(2)}(x^\prime)\,
	      \mbox{\Large $($}
		   \sqrt{-1}
		    \sum_{\mu, \nu, \dot{\beta}, \dot{\delta}}
			  \varepsilon^{\dot{\beta}\dot{\delta}}\,
			    \sigma^\mu_{1\dot{\beta}}\,\sigma^\nu_{\gamma\dot{\delta}}\,
				F_{\mu\nu}(x^\prime)
		   + \varepsilon_{1\gamma} V^\sim_{(0)}(x^\prime)
		  \mbox{\Large $)$}
	  \left.\rule{0ex}{1.2em}\right\}
	  \\
  && \hspace{2em}
   +\, \theta^1\theta^2
          \left\{\rule{0ex}{1.2em}\right.		
		  V^\sim_{(0)}(x^\prime)^{\,2}\,	
		   -\,2\,\sqrt{-1}\,		
			\sum_{\alpha, \dot{\beta}, \mu}
			   \vartheta_\alpha \bar{\vartheta}_{\dot{\beta}}\,
			      \bar{\sigma}^{\mu \dot{\beta}\alpha}
				   	V^{\prime\prime}_{(\alpha)}(x^\prime)\,
				    \partial_\mu\,\overline{V^{\prime\prime}_{(\beta)}}(x^\prime)\; 			
                \\
			&& \hspace{7em}			
          -\,2\,\sum_{\mu, \nu, \mu^\prime, \nu^\prime}
	              \eta^{\mu\mu^\prime} \eta^{\nu\nu^\prime}
		             F_{\mu\nu}(x^\prime)F_{\mu^\prime\nu^\prime}(x^\prime)
	      +\,  \sqrt{-1}\,
	            \sum_{\mu, \nu, \mu^\prime, \nu^\prime}
	             \varepsilon^{\,\mu\nu\mu^\prime\nu^\prime} F_{\mu\nu}(x^\prime)F_{\mu^\prime\nu^\prime}			
		  \left.\rule{0ex}{1.2em}\right\}\,.
\end{eqnarray*}
} 
  \marginpar{\vspace{-12em}\raggedright\tiny
          \raisebox{-1ex}{\hspace{-2.4em}\LARGE $\cdot$}Cf.\,[\,\parbox[t]{20em}{Wess
		\& Bagger:\\ Eq.\:(6.13)].}}

The substitutions
{\small
 $$
   \begin{array}{c}
     \theta\theta
	    \rightsquigarrow  -2\theta^1\theta^2\,,\hspace{2em}
	  m, n, k, l \rightsquigarrow \mu, \nu, \mu^\prime, \nu^\prime\, \hspace{2em}
     v_n \rightsquigarrow  - V_{[\mu]}\,, \hspace{2em}
     \lambda_\alpha  \rightsquigarrow
	    \mbox{\large $\frac{\sqrt{-1}}{2}$}
		   \vartheta_\alpha V^{\prime\prime}_{(\alpha)}\,,
		 \\[1.2ex]	
     D  \rightsquigarrow  - \mbox{\large $\frac{1}{2}$}V^\sim_{(0)}\,, \hspace{2em}        					
	 \bar{\lambda}_{\dot{\beta}} \rightsquigarrow
	   - \mbox{\large $\frac{\sqrt{-1}}{2}$}
	       \bar{\vartheta}_{\dot{\beta}} \overline{V^{\prime\prime}_{(\beta)}}\,, \hspace{2em}
	v_{mn}\rightsquigarrow - F_{\mu\nu}
   \end{array}
 $$}plus 
 an appropriate raising-and-lowering of the spinor and the space-time indices turn [Wess \& Bagger: Eq.\:(6.13)]
  into the expression for the $\theta^1\theta^2$-component above up to a factor of $\,-\frac{1}{2}\,$:
{\small
 \begin{eqnarray*}
  &&
   W^\alpha W_\alpha\;\;\mbox{in [Wess \& Bagger: Eq.\:(6.13) via Eq.\:(6.7)] }
      \\[1.2ex]
	&& \hspace{2em}
    =\;\; (\,\cdots\,)
	   + \theta\theta
	        \mbox{\Large $($}
			  -2\sqrt{-1} \lambda\sigma^m\partial_m \bar{\lambda}
			  - \mbox{\large $\frac{1}{2}$}v^{mn}v_{mn}
			  + D^2
			  + \mbox{\large $\frac{\sqrt{-1}}{2}$}
			       v^{mn}v^{kl}\varepsilon_{mnkl}
			\mbox{\Large $)$}
	  \\[1.2ex]
    && \hspace{2em}
	\rightsquigarrow\;
	  (\,\cdots\,)
	    - \mbox{\large $\frac{1}{2}$}\cdot
		   \theta^1\theta^2\,
		    \mbox{(the $\theta^1\theta^2$-component computed above)}\,.
 \end{eqnarray*}
} 

\bigskip

\noindent $\tinybullet$
$\bar{W}_{\dot{\beta}}$ and $\bar{W}_{\dot{2}}\bar{W}_{\dot{1}}$:
 (in antichiral coordinate functions $(x^{\prime\prime}, \theta,\bar{\theta}, \vartheta, \bar{\vartheta})$
  on $\widehat{X}^{\widehat{\boxplus}}$)

\medskip

\noindent
The formulae for
  $\bar{W}_{\dot{\beta}}$ and $\bar{W}_{\dot{2}}\bar{W}_{\dot{1}}$
 follow either by similar computations
             or by taking the twisted complex conjugate on $W_\alpha$ and $W_1W_2$ respectively.
The results are listed below.

In terms of the antichiral coordinate functions
 $(x^{\prime\prime}, \theta, \bar{\theta}, \vartheta, \bar{\vartheta})$
 on $\widehat{X}^{\widehat{\boxplus}}$,
 {\small
 \begin{eqnarray*}
  \breve{V}
  &= &		
	\sum_{\alpha,\dot{\beta}}\theta^\alpha \bar{\theta}^{\dot{\beta}}
		    \sum_\mu \sigma^\mu_{\alpha\dot{\beta}} V_{[\mu]}(x^{\prime\prime})	
	+ \sum_{\dot{\beta}}\theta^1\theta^2\bar{\theta}^{\dot{\beta}}	
          \bar{\vartheta}_{\dot{\beta}}\,
		        \overline{V^{\prime\prime}_{(\beta)}}(x^{\prime\prime})					
    +\, \sum_\alpha \theta^\alpha\bar{\theta}^{\dot{1}}\bar{\theta}^{\dot{2}}
		  \vartheta_\alpha\, V^{\prime\prime}_{(\alpha)} (x^{\prime\prime})
     \\
	&& \hspace{3em}	
	+\, \theta^1\theta^2\bar{\theta}^{\dot{1}}\bar{\theta}^{\dot{2}}\,
	    \mbox{\Large $($}
		  V^\sim_{(0)}(x^{\prime\prime})
		  +2 \sqrt{-1}\,\sum_{\mu, \nu} \eta^{\mu\nu}\,\partial_{\mu} V_{[\nu]}(x^{\prime\prime})
		\mbox{\Large $)$}  \,;
 \end{eqnarray*}
  }
 
{\small
\begin{eqnarray*}
 \lefteqn{
   e_{\beta^{\prime\prime}}\breve{V}\;\;
     =\;\; \sum_{\gamma, \nu}
	          \theta^\gamma
	            \sigma^{\nu}_{\gamma\dot{\beta}} V_{[\nu]}(x^{\prime\prime})
		     - \theta^1\theta^2 \bar{\vartheta}_{\dot{\beta}}\,
			     \overline{V^{\prime\prime}_{(\beta)}}(x^{\prime\prime})				
		    + \sum_{\gamma, \dot{\delta}}
			     \theta^\gamma \bar{\theta}^{\dot{\delta}}
			        \vartheta_\gamma \varepsilon_{\dot{\delta}\dot{\beta}}
				     V^{\prime\prime}_{(\gamma)}(x^{\prime\prime})
    }\\
  && \hspace{2em}
  -\, \sum_{\dot{\delta}}
     \theta^1\theta^2 \bar{\theta}^{\dot{\delta}}
	   \left\{\rule{0ex}{1.2em}\right.	
	     2\sqrt{-1}\sum_{\alpha, \gamma, \mu, \nu}
		   \varepsilon^{\alpha\gamma}
		     \sigma^\mu_{\alpha\dot{\beta}} \sigma^\nu_{\gamma\dot{\delta}}\,
			  \partial_\mu V_{[\nu]}(x^{\prime\prime})
		 + \varepsilon_{\dot{\delta}\dot{\beta}}
		       \mbox{\Large $($}
			     V^\sim_{(0)}(x^{\prime\prime})
				+ 2\sqrt{-1}\sum_{\mu, \nu}
				     \eta^{\mu\nu} \partial_\mu V_{[\nu]}(x^{\prime\prime})
			   \mbox{\Large $)$}
	   \left.\rule{0ex}{1.2em}\right\}
	   \\
	 && \hspace{2em}
	   +\,2\sqrt{-1}\,\theta^1\theta^2\bar{\theta}^{\dot{1}} \bar{\theta}^{\dot{2}}
	        \sum_{\gamma, \nu}\vartheta_\gamma \sigma^{\nu\gamma}_{\;\;\;\;\dot{\beta}}
			   \partial_\nu V^{\prime\prime}_{(\gamma)}(x^{\prime\prime})\,;
\end{eqnarray*}
}

{\small
\begin{eqnarray*}
 \bar{W}_{\dot{\beta}}
   & :=\: & e_{1^\prime}e_{2^\prime} e_{\beta^{\prime\prime}} \breve{V}
       \\
  &  =  &
    \bar{\vartheta}_{\dot{\beta}}\,\overline{V^{\prime\prime}_{(\beta)}}(x^{\prime\prime})
	+ \sum_{\dot{\delta}}
	     \bar{\theta}^{\dot{\delta}}
		   \left\{\rule{0ex}{1.2em}\right.
	        \varepsilon_{\dot{\delta}\dot{\beta}} V^\sim_{(0)}(x^{\prime\prime})
			+\sqrt{-1}\sum_{\alpha, \gamma, \mu, \nu}
			    \varepsilon^{\alpha\gamma}
				  \sigma^\mu_{\alpha\dot{\beta}}  \sigma^\nu_{\gamma\dot{\delta}}
				F_{\mu\nu}(x^{\prime\prime})
	       \left.\rule{0ex}{1.2em}\right\}
   \\
  &&
    -\,2\sqrt{-1}\,\bar{\theta}^{\dot{1}}\bar{\theta}^{\dot{2}}
	  \sum_{\gamma, \nu}
	    \vartheta_\gamma \sigma^{\nu\gamma}_{\;\;\;\;\dot{\beta}}\,
		  \partial_\nu V^{\prime\prime}_{(\gamma)}(x^{\prime\prime})\,,
\end{eqnarray*}
}where  
$F_{\mu\nu}:= \partial_\mu V_{[\nu]}- \partial_\nu V_{[\mu]}$; and

{\small
\begin{eqnarray*}
 \bar{W}_{\dot{2}}\bar{W}_{\dot{1}}
   &= &
   -\,\bar{\vartheta}_{\dot{1}}  \bar{\vartheta}_{\dot{2}}\,
     \overline{V^{\prime\prime}_{(1)}}(x^{\prime\prime})\,
	   \overline{V^{\prime\prime}_{(2)}}(x^{\prime\prime})
   \\
  &&
   +\, \sum_{\dot{\beta}}
          \bar{\theta}^{\dot{\beta}}
		   \left\{\rule{0ex}{1.2em}\right.
		     \bar{\vartheta}_{\dot{1}}\,\overline{V^{\prime\prime}_{(1)}}(x^{\prime\prime})
			  \mbox{\Large $($}
		        \varepsilon_{\dot{\beta}\dot{2}} V^\sim_{(0)}(x^{\prime\prime})
				+ \sqrt{-1}\sum_{\alpha, \gamma, \mu, \nu}
				     \varepsilon^{\alpha\gamma}
					   \sigma^\mu_{\alpha\dot{2}} \sigma^\nu_{\gamma\dot{\beta}}
					   F_{\mu\nu}(x^{\prime\prime})
			  \mbox{\Large $)$}			
			   \\
		  && \hspace{6em}			  			
			-\, \bar{\vartheta}_{\dot{2}}\,\overline{V^{\prime\prime}_{(2)}}(x^{\prime\prime})
			  \mbox{\Large $($}
		        \varepsilon_{\dot{\beta}\dot{1}} V^\sim_{(0)}(x^{\prime\prime})
				+ \sqrt{-1}\sum_{\alpha, \gamma, \mu, \nu}
				     \varepsilon^{\alpha\gamma}
					   \sigma^\mu_{\alpha\dot{1}} \sigma^\nu_{\gamma\dot{\beta}}
					   F_{\mu\nu}(x^{\prime\prime})
			  \mbox{\Large $)$}			  			
		   \left.\rule{0ex}{1.2em}\right\}	
     \\		
  &&
    +\, \bar{\theta}^{\dot{1}}  \bar{\theta}^{\dot{2}}
	       \left\{\rule{0ex}{1.2em}\right.
		    - V^\sim_{(0)}(x^{\prime\prime})^{\,2}
		    -2\sqrt{-1}\sum_{\alpha,\dot{\beta}, \mu}
			   \vartheta_\alpha \bar{\vartheta}_{\dot{\beta}}\,
			     \bar{\sigma}^{\mu\dot{\beta}\alpha}\,
				  \overline{V^{\prime\prime}_{(\beta)}}(x^{\prime\prime})\,
				   \partial_\mu V^{\prime\prime}_{(\alpha)}(x^{\prime\prime})
			 \\
			 && \hspace{5em}			
		    +\,2\sum_{\mu,\nu, \mu^\prime, \nu^\prime}
			     \eta^{\mu\mu^\prime}\eta^{\nu\nu^\prime}
			      F_{\mu\nu}(x^{\prime\prime}) F_{\mu^\prime\nu^\prime}(x^{\prime\prime})
		     + \sqrt{-1}\sum_{\mu, \nu, \mu^\prime, \nu^\prime}
			       \varepsilon^{\mu\nu\mu^\prime\nu^\prime}
				    F_{\mu\nu}(x^{\prime\prime}) F_{\mu^\prime\nu^\prime}(x^{\prime\prime})
	       \left.\rule{0ex}{1.2em}\right\},
\end{eqnarray*}
}where 
$\varepsilon^{\,\mu\nu\mu^\prime\nu^\prime}$ is the space-time volume-element tensor
	       with respect to the frame $(\partial_\mu)_\mu$ with $\varepsilon^{0123}=1$.

\bigskip

\noindent $\tinybullet$
$\breve{\Phi}^\dag e^{\breve{V}}\breve{\Phi}$:
 (in standard coordinate functions $(x, \theta,\bar{\theta}, \vartheta, \bar{\vartheta})$
  on $\widehat{X}^{\widehat{\boxplus}}$)		

\medskip

\noindent
Let
   \begin{eqnarray*}
     \breve{\Phi} & = &
	   \Phi_{(0)}
	   + \sum_\alpha \theta^\alpha\vartheta_\alpha  \Phi_{(\alpha)}
	   + \theta^1\theta^2\vartheta_1\vartheta_2 \Phi_{(12)}
       + \sqrt{-1} \sum_{\alpha,\dot{\beta};\,\mu}
	          \theta^\alpha\bar{\theta}^{\dot{\beta}}	
			     \sigma^\mu_{\alpha\dot{\beta}}\, \partial_\mu \Phi_{(0)}\\
      && \hspace{1em}				
       + \sqrt{-1}\sum_{\dot{\beta};\, \alpha, \mu}				
	        \theta^1\theta^2\bar{\theta}^{\dot{\beta}}			
			   \vartheta_\alpha\sigma^{\mu\alpha}_{\;\;\;\;\dot{\beta}}\, \partial_\mu \Phi_{(\alpha)}
       - \theta^1\theta^2\bar{\theta}^{\dot{1}}\bar{\theta}^{\dot{2}}\,
	        \square \Phi_{(0)}\,,
   \end{eqnarray*}
  be a chiral superfield on $X^{\physics}$, determined by
   $(\Phi_{(0)},  (\Phi_{(\alpha)})_\alpha, \Phi_{(12)} )$
   and
   \begin{eqnarray*}
     \breve{\Phi}^\dag & = &
	   \overline{\Phi_{(0)}}
	   - \sum_{\dot{\beta}} \bar{\theta}^{\dot{\beta}}\bar{\vartheta}_{\dot{\beta}}
	           \overline{\Phi_{(\beta)}}
	   + \bar{\theta}^{\dot{1}}\bar{\theta}^{\dot{2}}
	      \bar{\vartheta}_{\dot{1}}\bar{\vartheta}_{\dot{2}}
		               \overline{\Phi_{(12)}}
       - \sqrt{-1} \sum_{\alpha,\dot{\beta};\,\mu}
	          \theta^\alpha\bar{\theta}^{\dot{\beta}}	
			     \sigma^\mu_{\alpha\dot{\beta}}\,
				       \partial_\mu \overline{\Phi_{(0)}} \\
      && \hspace{1em}				
       -\, \sqrt{-1}\sum_{\alpha;\, \dot{\beta}, \mu}				
	        \theta^\alpha\bar{\theta}^{\dot{1}} \bar{\theta}^{\dot{2}}
			    \bar{\vartheta}_{\dot{\beta}}
			      \sigma^{\mu\dot{\beta}}_\alpha\, \partial_\mu \overline{\Phi_{(\beta)}}
       - \theta^1\theta^2\bar{\theta}^{\dot{1}}\bar{\theta}^{\dot{2}}\,
	        \square \overline{\Phi_{(0)}}\,.
   \end{eqnarray*}
 be its twisted complex conjugate, which is antichiral.
Note also that, for the vector superfield $\breve{V}$  in the Wess-Zumino gauge,
 $$
   \breve{V}^2\;
     =\;  2\,\theta^1\theta^2 \bar{\theta}^{\dot{1}} \bar{\theta}^{\dot{2}}
	        \sum_{\mu, \nu}\eta^{\mu\nu} V_{[\mu]} V_{[\nu]}\,,
     \hspace{2em} \breve{V}^3\;=\; 0\,,
 $$
  and hence
 the exponential $e^{\breve{V}}$ is given by
 {\small
 \begin{eqnarray*}
    e^{\breve{V}}& =  &
	 1 + \breve{V} + \mbox{\large $\frac{1}{2}$}\,\breve{V}^2	
      \\
  &=&
     1\,+\,
	 \sum_{\alpha,\dot{\beta}}\theta^\alpha \bar{\theta}^{\dot{\beta}}
		    \sum_\mu \sigma^\mu_{\alpha\dot{\beta}} V_{[\mu]}(x)	
	 + \sum_{\dot{\beta}}\theta^1\theta^2\bar{\theta}^{\dot{\beta}}	
          \bar{\vartheta}_{\dot{\beta}}\,
		        \overline{V^{\prime\prime}_{(\beta)}}(x)					
     +\, \sum_\alpha \theta^\alpha\bar{\theta}^{\dot{1}}\bar{\theta}^{\dot{2}}
		  \vartheta_\alpha\, V^{\prime\prime}_{(\alpha)} (x)
		 \\
  &&
	 +\, \theta^1\theta^2\bar{\theta}^{\dot{1}}\bar{\theta}^{\dot{2}}
	     \mbox{\large $($}		
		  V^\sim_{(0)}(x)
		    + \sum_{\mu, \nu}\eta^{\mu\nu} V_{[\mu]} V_{[\nu]}
         \mbox{\large $)$}\,.		
 \end{eqnarray*}
}
		
\noindent
A straightforward computation gives
{\small
 \begin{eqnarray*}
    e^{\breve{V}}\,\breve{\Phi}& = &
	  \Phi_{(0)}
	  + \sum_\gamma
	       \theta^\gamma \vartheta_\gamma \Phi_{(\gamma)}
	  + \theta^1\theta^2 \vartheta_1\vartheta_2 \Phi_{(12)}
	  + \sum_{\gamma, \dot{\delta}, \nu}
	      \theta^\gamma \bar{\theta}^{\dot{\delta}}
		   \sigma^\nu_{\gamma \dot{\delta}}
		     \mbox{\large $($}
			  V_{[\nu]}\Phi_{(0)}  + \sqrt{-1}\, \partial_\nu \Phi_{(0)}			
			 \mbox{\large $)$}
     \\
	&&
	+\, \sum_{\dot{\delta}}
	       \theta^1 \theta^2 \bar{\theta}^{\dot{\delta}}
		    \left\{\rule{0ex}{1.2em}\right.
			   \sum_{\gamma, \nu}
				  \vartheta_\alpha \sigma^{\nu\gamma}_{\;\;\;\;\dot{\delta}}
				   \mbox{\large $($}
				    V_{[\nu]}\Phi_{(\gamma)}
					 + \sqrt{-1}\,\partial_\nu \Phi_{(\gamma)}
				   \mbox{\large $)$}				
              + \bar{\vartheta}_{\dot{\delta}}\,
			        \overline{V^{\prime\prime}_{(\delta)}}\,\Phi_{(0)}				   				
			\left.\rule{0ex}{1.2em}\right\}\,
	  +\,  \sum_\gamma
	           \theta^\gamma \bar{\theta}^{\dot{1}} \bar{\theta}^{\dot{2}}\,
			     \vartheta_\gamma V^{\prime\prime}_{(\gamma)} \Phi_{(0)}
     \\
	 &&
	 +\,\theta^1\theta^2 \bar{\theta}^{\dot{1}} \bar{\theta}^{\dot{2}}
	       \left\{\rule{0ex}{1.2em}\right.
		       \mbox{\large $($}
			     V^\sim_{(0)}
				  + \sum_{\mu, \nu}\eta^{\mu\nu} V_{[\mu]} V_{[\nu]}
			   \mbox{\large $)$}\,
			   \Phi_{(0)}
             + 2\sqrt{-1}\sum_{\mu, \nu}
			      \eta^{\mu\nu} V_{[\mu]}\,\partial_\nu \Phi_{(0)} 				
			 - \square \Phi_{(0)}
				\\
	           && \hspace{24em}		
			 -\, \vartheta_1\vartheta_2
            	\mbox{\large $($}
				 V^{\prime\prime}_{(1)}\Phi_{(2)}
				 + V^{\prime\prime}_{(2)}\Phi_{(1)}
				\mbox{\large $)$}
		   \left.\rule{0ex}{1.2em}\right\}
 \end{eqnarray*}
} 

\noindent
and

{\small
 \begin{eqnarray*}
  \lefteqn{
    \breve{\Phi}^\dag\, e^{\breve{V}}\,\breve{\Phi}\;\;=\;\;
	 \breve{\Phi}^\dag\,\mbox{\large $($}e^{\breve{V}}\,\breve{\Phi} \mbox{\large $)$}
      }
	 \\
 && =\;\;
   \overline{\Phi_{(0)}} \,\Phi_{(0)}
	+ \sum_\alpha
	     \theta^\alpha \vartheta_\alpha\,
		   \overline{\Phi_{(0)}}\,\Phi_{(\alpha)}
	- \sum_{\dot{\beta}}
	    \bar{\theta}^{\dot{\beta}}\bar{\vartheta}_{\dot{\beta}}\,
		   \overline{\Phi_{(\beta)}}\,\Phi_{(0)}
	+ \theta^1\theta^2 \vartheta_1\vartheta_2\,
	     \overline{\Phi_{(0)}}\,\Phi_{(12)}
   \\
   && \hspace{2em}
   \sum_{\alpha, \dot{\beta}}
     \theta^\alpha \bar{\theta}^{\dot{\beta}}
	  \left\{\rule{0ex}{1.2em}\right.
	   \sum_\mu  \sigma^\mu_{\alpha\dot{\beta}}
	     \mbox{\Large $($}
	      -\sqrt{-1}\,\partial_\mu \overline{\Phi_{(0)}}\,\Phi_{(0)}
		  + \sqrt{-1}\,\overline{\Phi_{(0)}}\,\partial_\mu \Phi_{(0)}
		  + \overline{\Phi_{(0)}}\, V_{[\mu]}\Phi_{(0)}
	     \mbox{\Large $)$}
	   + \vartheta_\alpha \bar{\vartheta}_{\dot{\beta}}\,
	        \overline{\Phi_{(\beta)}}\,\Phi_{(\alpha)}	
	  \left.\rule{0ex}{1.2em}\right\}
	 \\
   && \hspace{2em}
	+\,\bar{\theta}^{\dot{1}}\bar{\theta}^{\dot{2}}
	       \bar{\vartheta}_{\dot{1}} \bar{\vartheta}_{\dot{2}}\,
		     \overline{\Phi_{(12)}}\,\Phi_{(0)}
     \\
   && \hspace{2em}			
    +\,\sum_{\dot{\beta}}
         \theta^1\theta^2 \bar{\theta}^{\dot{\beta}}		
		   \left\{\rule{0ex}{1.2em}\right.
		    \sum_{\alpha, \mu}
			   \vartheta_{\alpha}  \sigma^{\mu\alpha}_{\;\;\;\;\dot{\beta}}
			  \mbox{\Large $($}
			    -\sqrt{-1}\,\partial_\mu \overline{\Phi_{(0)}}\,\Phi_{(\alpha)}
				+ \sqrt{-1} \,\overline{\Phi_{(0)}}\,\partial_\mu\Phi_{(\alpha)}
				+\,\overline{\Phi_{(0)}}\,V_{[\mu]}\Phi_{(\alpha)}
			  \mbox{\Large $)$}			
     		 \\[-1.6ex]
            && \hspace{10em}	
           +\, \bar{\vartheta}_{\dot{\beta}}\,
		         \overline{\Phi_{(0)}}\,\overline{V^{\prime\prime}_{(\beta)}}\,\Phi_{(0)}	
		   - \vartheta_1\vartheta_2 \bar{\vartheta}_{\dot{\beta}}\,
		         \overline{\Phi_{(\beta)}}\,\Phi_{(12)}				
		  \left.\rule{0ex}{1.2em}\right\}		
     \\
   && \hspace{2em}
    +\,\sum_\alpha
	     \theta^\alpha \bar{\theta}^{\dot{1}} \bar{\theta}^{\dot{2}}
	      \left\{\rule{0ex}{1.2em}\right.
		    \sum_{\dot{\beta}, \mu}
			   \bar{\vartheta}_{\dot{\beta}}  \sigma^{\mu\dot{\beta}}_\alpha
			  \mbox{\Large $($}
			    -\sqrt{-1}\,\partial_\mu \overline{\Phi_{(\beta)}}\,\Phi_{(0)}
				+\sqrt{-1} \,\overline{\Phi_{(\beta)}}\,\partial_\mu\Phi_{(0)}
				+\,\overline{\Phi_{(\beta)}}\,V_{[\mu]}\Phi_{(0)}
			  \mbox{\Large $)$}			
     		 \\[-1.6ex]
            && \hspace{10em}	
           +\, \vartheta_\alpha\,\overline{\Phi_{(0)}}\,V^{\prime\prime}_{(\alpha)}\Phi_{(0)}			  
		   + \vartheta_\alpha \bar{\vartheta}_{\dot{1}} \bar{\vartheta}_{\dot{2}}\,
		         \overline{\Phi_{(12)}}\,\Phi_{(\alpha)}				
		  \left.\rule{0ex}{1.2em}\right\}
     \\
   && \hspace{2em}		
    +\,\theta^1\theta^2 \bar{\theta}^{\dot{1}} \bar{\theta}^{\dot{2}}
	       \left\{\rule{0ex}{1.2em}\right.
		   - \square\,\overline{\Phi_{(0)}}\,\Phi_{(0)}
		   + 2 \sum_{\mu,\nu}  \eta^{\mu\nu}\,
		           \partial_\mu\,\overline{\Phi_{(0)}}\,\partial_\nu \Phi_{(0)}
		    - \overline{\Phi_{(0)}}\,\square\Phi_{(0)}
			- 2\sqrt{-1}\sum_{\mu, \nu} \eta^{\mu\nu}
			      \partial_\mu \overline{\Phi_{(0)}}\,V_{[\nu]} \Phi_{(0)}
               \\[-.6ex]
              && \hspace{10em}				
			+\, 2\sqrt{-1}\, \overline{\Phi_{(0)}}
                   \sum_{\mu, \nu}\eta^{\mu\nu}
				     V_{[\mu]} \partial_\nu\Phi_{(0)}
			 + \overline{\Phi_{(0)}}
			      \mbox{\large $($}
				    V^\sim_{(0)}
					+ \sum_{\mu, \nu}\eta^{\mu\nu}V_{[\mu]}V_{[\nu]}
				  \mbox{\large $)$}
				  \Phi_{(0)}
       \\
      && \hspace{4em}				
	    -\, \vartheta_1\vartheta_2\,
		      \overline{\Phi_{(0)}}\,
			    \mbox{\large $($}
				  V^{\prime\prime}_{(1)}\Phi_{(2)}
				  + V^{\prime\prime}_{(2)} \Phi_{(1)}
				\mbox{\large $)$}
	      \\
      && \hspace{4em}		
	    +\, \sum_{\alpha, \dot{\beta}}
		       \vartheta_\alpha \bar{\vartheta}_{\dot{\beta}}
			    \mbox{\Large $($}
				 \sqrt{-1}\sum_\mu
				     \bar{\sigma}^{\mu\dot{\beta}\alpha}
					   \mbox{\large $($}
					     \partial_\mu  \overline{\Phi_{(\beta)}}\,\Phi_{(\alpha)}
						 - \overline{\Phi_{(\beta)}}\,\partial_\mu\Phi_{(\alpha)}
					   \mbox{\large $)$}
				 - \sum_\mu \bar{\sigma}^{\mu\dot{\beta}\alpha}
				     \overline{\Phi_{(\beta)}}\,V_{[\mu]}\Phi_{(\alpha)}
				\mbox{\Large $)$}
	      \\
      && \hspace{4em}						
	    +\,\bar{\vartheta}_{\dot{1}}\bar{\vartheta}_{\dot{2}}\,
		       \mbox{\large $($}
			    \overline{\Phi_{(1)}}\,\overline{V^{\prime\prime}_{(2)}}
				+ \overline{\Phi_{(2)}}\,\overline{V^{\prime\prime}_{(1)}}
			   \mbox{\large $)$}\,
			   \Phi_{(0)}
         + \vartheta_1 \vartheta_2 \bar{\vartheta}_{\dot{1}}\bar{\vartheta}_{\dot{2}}\,
               \overline{\Phi_{(12)}}\,\Phi_{(12)}		
     \left.\rule{0ex}{1.2em}\right\}\,.
 \end{eqnarray*}
} 
  \marginpar{\vspace{-9em}\raggedright\tiny
          \raisebox{-1ex}{\hspace{-2.4em}\LARGE $\cdot$}Cf.\,[\,\parbox[t]{20em}{Wess
		\& Bagger:\\ Eq.\:(7.7)].}}

For comparison with [Wess \& Bagger: Eq.\:(7.7)],
replace $\breve{V}$ by $t\breve{V}$ for a real parameter $t$,
 focus on the $\theta^1\theta^2\bar{\theta}^{\dot{1}}\bar{\theta}^{\dot{2}}$-component,
 and extract some boundary terms:
 
{\small
\begin{eqnarray*}
 \lefteqn{
   \breve{\Phi}^\dag e^{t \breve{V}}\breve{\Phi}\;\;
      =\;\;   \mbox{(terms of total $(\theta, \bar{\theta})$-degree $\le 3$)}
    }\\
   && \hspace{2em}		
    +\,\theta^1\theta^2 \bar{\theta}^{\dot{1}} \bar{\theta}^{\dot{2}}
	       \left\{\rule{0ex}{1.2em}\right.
		   \sum_\mu
		     \partial_\mu
			   \mbox{\Large $($}
	             3\,\partial^\mu \overline{\Phi_{(0)}}\,\Phi_{(0)}
                 - \overline{\Phi_{(0)}}\,\partial^\mu \Phi_{(0)}				
                - \sqrt{-1}\sum_{\alpha, \dot{\beta}}		
				     \vartheta_\alpha \bar{\vartheta}_{\dot{\beta}}
					   \bar{\sigma}^{\mu\dot{\beta}\alpha}\,
					     \overline{\Phi_{(\beta)}}\,\Phi_{(\alpha)}
		       \mbox{\Large $)$}
   \\
   && \hspace{4em}			
		   - 4\, \square\,\overline{\Phi_{(0)}}\,\Phi_{(0)}		   						
			- 2\sqrt{-1}\,t\, \sum_{\mu, \nu} \eta^{\mu\nu}
			      \partial_\mu \overline{\Phi_{(0)}}\,V_{[\nu]} \Phi_{(0)}
               \\[-.6ex]
              && \hspace{10em}				
			+\, 2\sqrt{-1} \,t\, \overline{\Phi_{(0)}}
                   \sum_{\mu, \nu}\eta^{\mu\nu}
				     V_{[\mu]} \partial_\nu\Phi_{(0)}
			 + \overline{\Phi_{(0)}}
			      \mbox{\large $($}
				    t\,V^\sim_{(0)}
					+ t^2\, \sum_{\mu, \nu}\eta^{\mu\nu}V_{[\mu]}V_{[\nu]}
				  \mbox{\large $)$}
				  \Phi_{(0)}
       \\
      && \hspace{4em}				
	    -\, \vartheta_1\vartheta_2\,t\,
		      \overline{\Phi_{(0)}}\,
			    \mbox{\large $($}
				  V^{\prime\prime}_{(1)}\Phi_{(2)}
				  + V^{\prime\prime}_{(2)} \Phi_{(1)}
				\mbox{\large $)$}
	      \\
      && \hspace{4em}		
	    +\, 2\sqrt{-1}\sum_{\alpha, \dot{\beta}}
		       \vartheta_\alpha \bar{\vartheta}_{\dot{\beta}}
			    \mbox{\Large $($}
				  \sum_\mu
				     \bar{\sigma}^{\mu\dot{\beta}\alpha}					 					
					     \partial_\mu  \overline{\Phi_{(\beta)}}\,\Phi_{(\alpha)}						 					   
				 -\,t\, \sum_\mu \bar{\sigma}^{\mu\dot{\beta}\alpha}
				     \overline{\Phi_{(\beta)}}\,V_{[\mu]}\Phi_{(\alpha)}
				\mbox{\Large $)$}
	      \\
      && \hspace{4em}						
	    +\,\bar{\vartheta}_{\dot{1}}\bar{\vartheta}_{\dot{2}}\,t\,
		       \mbox{\large $($}
			    \overline{\Phi_{(1)}}\,\overline{V^{\prime\prime}_{(2)}}
				+ \overline{\Phi_{(2)}}\,\overline{V^{\prime\prime}_{(1)}}
			   \mbox{\large $)$}\,
			   \Phi_{(0)}
         + \vartheta_1 \vartheta_2 \bar{\vartheta}_{\dot{1}}\bar{\vartheta}_{\dot{2}}\,
               \overline{\Phi_{(12)}}\,\Phi_{(12)}		
     \left.\rule{0ex}{1.2em}\right\}\,.
\end{eqnarray*}
} 
 
\noindent
Then the substitutions
{\small
 $$
   \begin{array}{c}
     \theta\theta\bar{\theta}\bar{\theta}
	    \rightsquigarrow  -4\theta^1\theta^2 \bar{\theta}^{\dot{1}} \bar{\theta}^{\dot{2}}\,,\hspace{2em}
	  n \rightsquigarrow \mu\,, \hspace{2em}
     A \rightsquigarrow \Phi_{(0)}\,,\hspace{2em}
     \psi_\alpha   \rightsquigarrow
	    \mbox{\large $\frac{1}{\sqrt{2}}$} \vartheta_\alpha\Phi_{(\alpha)}\,, \hspace{2em} 		
     F  \rightsquigarrow
	    - \mbox{\large $\frac{1}{2}$}\vartheta_1\vartheta_2 \Phi_{(12)}\,,
		\\[1.2ex]
     v_n \rightsquigarrow  - V_{[\mu]}\,, \hspace{2em}
     \lambda_\alpha  \rightsquigarrow
	    \mbox{\large $\frac{\sqrt{-1}}{2}$}
		   \vartheta_\alpha V^{\prime\prime}_{(\alpha)}\,, \hspace{2em}
     D  \rightsquigarrow  - \mbox{\large $\frac{1}{2}$}V^\sim_{(0)}\,, \hspace{2em}
     F^\ast  \rightsquigarrow
	     \mbox{\large $\frac{1}{2}$}
		    \bar{\vartheta}_{\dot{1}} \bar{\vartheta}_{\dot{2}} \overline{\Phi_{(12)}}\,,
			\hspace{2em}
	 \bar{\lambda}_{\dot{\beta}} \rightsquigarrow
	   - \mbox{\large $\frac{\sqrt{-1}}{2}$}
	       \bar{\vartheta}_{\dot{\beta}} \overline{V^{\prime\prime}_{(\beta)}}
   \end{array}
 $$}turn 
  [Wess \& Bagger: Eq.\:(22.7)]  into the expression
  for the $\theta^1\theta^2\bar{\theta}^{\dot{1}}\bar{\theta}^{\dot{2}}$-component above:

 {\center\footnotesize
 $$
 \hspace{-1em}\begin{array}{ccl}
  \underline{\raisebox{-1ex}{\hspace{2em}}
     \mbox{[Wess \& Bagger:\:Eq.\,(7.7), in $\theta\theta\bar{\theta}\bar{\theta}$]}\hspace{2em}}
    & \underline{\raisebox{-1ex}{\hspace{2em}}
	        \mbox{explicit formula for $\breve{\Phi}^\dag e^{\breve{V}}\breve{\Phi}$
			               in $\theta^1\theta^2\bar{\theta}^{\dot{1}}\bar{\theta}^{\dot{2}}$}\hspace{2em}}
						   \\[2ex]
   F F^{\ast } + A\square A^\ast
    &
	  \vartheta_1 \vartheta_2 \bar{\vartheta}_{\dot{1}}\bar{\vartheta}_{\dot{2}}\,
               \overline{\Phi_{(12)}}\,\Phi_{(12)}	
            - 4\, \square\,\overline{\Phi_{(0)}}\,\Phi_{(0)}		   			   	
		 \\[2ex]
    + \sqrt{-1}	\partial_n\bar{\psi}\bar{\sigma}^n\psi
     & +\, 2\sqrt{-1}\sum_{\alpha, \dot{\beta}, \mu}
		              \vartheta_\alpha \bar{\vartheta}_{\dot{\beta}}
				        \bar{\sigma}^{\mu\dot{\beta}\alpha}					 					
					     \partial_\mu  \overline{\Phi_{(\beta)}}\,\Phi_{(\alpha)}		
	     \\[2ex]
    + \mbox{\large $\frac{1}{2}$}\,v_n\bar{\psi}\bar{\sigma}^n\psi
    &
	   -\, 2\sqrt{-1}\, t\, \sum_{\alpha, \dot{\beta}, \mu}
		       \vartheta_\alpha \bar{\vartheta}_{\dot{\beta}}			   			
				 \bar{\sigma}^{\mu\dot{\beta}\alpha}
				     \overline{\Phi_{(\beta)}}\,V_{[\mu]}\Phi_{(\alpha)}			
		  \\[2ex]
   +\,\mbox{\large $\frac{\sqrt{-1}}{2}$}\,t v^n
         (A^\ast \partial_n A - \partial_nA^\ast A )
	&
      -\, 2\,\sqrt{-1}\,t\, \sum_{\mu, \nu} \eta^{\mu\nu}
			      \partial_\mu \overline{\Phi_{(0)}}\,V_{[\nu]} \Phi_{(0)}
			+\, 2\sqrt{-1} \,t\, \overline{\Phi_{(0)}}
                   \sum_{\mu, \nu}\eta^{\mu\nu}
				     V_{[\mu]} \partial_\nu\Phi_{(0)}				
	   \\[1.2ex]	
   -\,\mbox{\large $\frac{\sqrt{-1}}{\sqrt{2}}$}\,t\,
    (A\bar{\lambda}\bar{\psi} - A^\ast\lambda\psi)
   &	
	  -\, \vartheta_1\vartheta_2\,t\,
		      \overline{\Phi_{(0)}}\,
			    \mbox{\large $($}
				  V^{\prime\prime}_{(1)}\Phi_{(2)}
				  + V^{\prime\prime}_{(2)} \Phi_{(1)}
				\mbox{\large $)$}	
	    + \bar{\vartheta}_{\dot{1}}\bar{\vartheta}_{\dot{2}}\,t\,
		       \mbox{\large $($}
			    \overline{\Phi_{(1)}}\,\overline{V^{\prime\prime}_{(2)}}
				+ \overline{\Phi_{(2)}}\,\overline{V^{\prime\prime}_{(1)}}
			   \mbox{\large $)$}\,
			   \Phi_{(0)}
       \\[2ex]
  +\, \mbox{\large $\frac{1}{2}$}\,
         \mbox{\large $($}
		  tD- \mbox{\large $\frac{1}{2}$}\,t^2 v_n v^n
		 \mbox{\large $)$}
		A^\ast A
	&
	  +\, \overline{\Phi_{(0)}}
			      \mbox{\large $($}
				    t\,V^\sim_{(0)}
					+ t^2\, \sum_{\mu, \nu}\eta^{\mu\nu}V_{[\mu]}V_{[\nu]}
				  \mbox{\large $)$}
				  \Phi_{(0)}
 \end{array}
 $$
 } 

\bigskip
 
Finally, up to boundary terms,
 there is another expression
 for the $\theta^1\theta^2\bar{\theta}^{\dot{1}}\bar{\theta}^{\dot{2}}$-component
 of $\breve{\Phi}^\dag e^{t \breve{V}} \breve{\Phi}$
 that is mathematically more appealing:
 
{\small
\begin{eqnarray*}
 \lefteqn{
   \breve{\Phi}^\dag e^{t \breve{V}}\breve{\Phi}\;\;
      =\;\;   \mbox{(terms of total $(\theta, \bar{\theta})$-degree $\le 3$)
	                +  (space-time boundary terms)}
    }\\
   && \hspace{2em}		
    +\,\theta^1\theta^2 \bar{\theta}^{\dot{1}} \bar{\theta}^{\dot{2}}
	       \left\{\rule{0ex}{1.2em}\right.
		  - \sum_\mu
		     \partial_\mu
			   \mbox{\large $($}
	             \partial^\mu \overline{\Phi_{(0)}}\,\Phi_{(0)}
                 + \overline{\Phi_{(0)}}\,\partial^\mu \Phi_{(0)}				
			   \mbox{\large $)$}
           + \sqrt{-1}				
		        \sum_\mu
				 \partial_\mu
				 \mbox{\large $($}
                  \sum_{\alpha, \dot{\beta}}		
				     \vartheta_\alpha \bar{\vartheta}_{\dot{\beta}}
					   \bar{\sigma}^{\mu\dot{\beta}\alpha}\,
					     \overline{\Phi_{(\beta)}}\,\Phi_{(\alpha)}
		       \mbox{\large $)$}
	   \\
     && \hspace{4em}
	 +\, 4\, \sum_{\mu,\nu} \eta^{\mu\nu}
	       \mbox{\large $($}
		    \partial_\mu + \mbox{\large $\frac{\sqrt{-1}}{2}$}\,tV_{[\mu]}
		   \mbox{\large $)$}
		   \overline{\Phi_{(0)}}
		   \cdot
		    \mbox{\large $($}
		    \partial_\mu - \mbox{\large $\frac{\sqrt{-1}}{2}$} t V_{[\nu]}
		   \mbox{\large $)$}
		   \Phi_{(0)}	
       \\
      && \hspace{4em}		
	    -\, 2\sqrt{-1}\sum_{\alpha, \dot{\beta}, \mu}
		       \vartheta_\alpha \bar{\vartheta}_{\dot{\beta}}
				  \bar{\sigma}^{\mu\dot{\beta}\alpha}\,					 					
 				  \overline{\Phi_{(\beta)}}\cdot
                     \mbox{\large $($}				
				   	    \partial_\mu  - \mbox{\large $\frac{\sqrt{-1}}{2}$}t V_{[\mu]}  																
                     \mbox{\large $)$}
				     \Phi_{(\alpha)}						 					   					 									
	      \\
     && \hspace{4em}				
	    -\, \vartheta_1\vartheta_2\,t\,
		      \overline{\Phi_{(0)}}\,
			    \mbox{\large $($}
				  V^{\prime\prime}_{(1)}\Phi_{(2)}
				  + V^{\prime\prime}_{(2)} \Phi_{(1)}
				\mbox{\large $)$}	     	
	    + \bar{\vartheta}_{\dot{1}}\bar{\vartheta}_{\dot{2}}\,t\,
		       \mbox{\large $($}
			    \overline{\Phi_{(1)}}\,\overline{V^{\prime\prime}_{(2)}}
				+ \overline{\Phi_{(2)}}\,\overline{V^{\prime\prime}_{(1)}}
			   \mbox{\large $)$}\,
			   \Phi_{(0)}
        \\
     && \hspace{4em}					
         +\, \vartheta_1 \vartheta_2 \bar{\vartheta}_{\dot{1}}\bar{\vartheta}_{\dot{2}}\,
               \overline{\Phi_{(12)}}\,\Phi_{(12)}		 	
		  +\, t\, \overline{\Phi_{(0)}}\, V^\sim_{(0)}\Phi_{(0)}			
     \left.\rule{0ex}{1.2em}\right\}
	 \\[2ex]
  &&
      =\;\;   \mbox{(terms of total $(\theta, \bar{\theta})$-degree $\le 3$)
	               				     +  (space-time boundary terms)}
     \\
   && \hspace{2em}		
    +\,\theta^1\theta^2 \bar{\theta}^{\dot{1}} \bar{\theta}^{\dot{2}}
	       \left\{\rule{0ex}{1.2em}\right.
		  - \sum_\mu
		     \partial_\mu
			   \mbox{\large $($}
	             \partial^\mu \overline{\Phi_{(0)}}\,\Phi_{(0)}
                 + \overline{\Phi_{(0)}}\,\partial^\mu \Phi_{(0)}				
			   \mbox{\large $)$}
           + \sqrt{-1}				
		        \sum_\mu
				 \partial_\mu
				 \mbox{\large $($}
                  \sum_{\alpha, \dot{\beta}}		
				     \vartheta_\alpha \bar{\vartheta}_{\dot{\beta}}
					   \bar{\sigma}^{\mu\dot{\beta}\alpha}\,
					     \overline{\Phi_{(\beta)}}\,\Phi_{(\alpha)}
		       \mbox{\large $)$}
	   \\
     && \hspace{4em}
	 +\, 4\, \sum_{\mu,\nu} \eta^{\mu\nu}
	        \nabla^{t \breve{V}}_\mu\,\overline{\Phi_{(0)}}\, \nabla^{t \breve{V}}_\nu\Phi_{(0)}
	    - 2\sqrt{-1}\sum_{\alpha, \dot{\beta}, \mu}
		       \vartheta_\alpha \bar{\vartheta}_{\dot{\beta}}
				  \bar{\sigma}^{\mu\dot{\beta}\alpha}\,					 					
 				  \overline{\Phi_{(\beta)}}\,\nabla^{t \breve{V}}_\mu\Phi_{(\alpha)}						 					   					 									
	      \\
     && \hspace{4em}				
	    -\, \vartheta_1\vartheta_2\,t\,
		      \overline{\Phi_{(0)}}\,
			    \mbox{\large $($}
				  V^{\prime\prime}_{(1)}\Phi_{(2)}
				  + V^{\prime\prime}_{(2)} \Phi_{(1)}
				\mbox{\large $)$}	     	
	    + \bar{\vartheta}_{\dot{1}}\bar{\vartheta}_{\dot{2}}\,t\,
		       \mbox{\large $($}
			    \overline{\Phi_{(1)}}\,\overline{V^{\prime\prime}_{(2)}}
				+ \overline{\Phi_{(2)}}\,\overline{V^{\prime\prime}_{(1)}}
			   \mbox{\large $)$}\,
			   \Phi_{(0)}
        \\
     && \hspace{4em}					
         +\, \vartheta_1 \vartheta_2 \bar{\vartheta}_{\dot{1}}\bar{\vartheta}_{\dot{2}}\,
               \overline{\Phi_{(12)}}\,\Phi_{(12)}		 	
		  +\, t\, \overline{\Phi_{(0)}}\, V^\sim_{(0)}\Phi_{(0)}			
     \left.\rule{0ex}{1.2em}\right\}\,,
\end{eqnarray*}
}
where
 \begin{eqnarray*}
    \nabla^{t\breve{V}}_\mu \Phi_{(0)}\;
     :=\; (\partial_\mu - \mbox{\large $\frac{\sqrt{-1}}{2}$}t V_{[\mu]})\,
            	 \Phi_{(0)}\,,
	 &&
    \nabla^{t\breve{V}}_\mu \overline{\Phi_{(0)}}\;
     :=\; (\partial_\mu + \mbox{\large $\frac{\sqrt{-1}}{2}$}t V_{[\mu]})\,
            	 \overline{\Phi_{(0)}}\,,
		\\[1.2ex]
   \nabla^{t\breve{V}}_\mu \Phi_{(\alpha)}\;
     :=\; (\partial_\mu - \mbox{\large $\frac{\sqrt{-1}}{2}$}t V_{[\mu]})\,
            	 \Phi_{(\alpha)}\,,
	 &&
    \nabla^{t\breve{V}}_\mu \overline{\Phi_{(\beta)}}\;
     :=\; (\partial_\mu + \mbox{\large $\frac{\sqrt{-1}}{2}$}t V_{[\mu]})\,
            	 \overline{\Phi_{(\beta)}}
 \end{eqnarray*}
 are the covariant derivative of the component fields along $\partial/\partial x^\mu$
 associated to the connection $\widehat{\nabla}^{t\breve{V}}$ associated to $t\breve{V}$.
Note that this is consistent with Lemma~4.2.6;
cf.\;footnote~17.									

\bigskip

\begin{flushleft}
{\bf A supersymmetric action functional for $U(1)$ gauge theory with matter on $X$}
\end{flushleft}
Now restore the electric charge $e_m$ in the discussion.
Then the gauge-invariant kinetic term for the matter chiral superfield $\breve{\Phi}$ becomes
 $$
   \breve{\Phi}^\dag e^{e_m \breve{V}}\breve{\Phi}\,.
 $$
Thus, replacing $\breve{\Lambda}$  with $e_m\breve{\Lambda}$ and
                          $\breve{V}$ with $e_m\breve{V}$ in the above discussion and computations,
 we recover the charge $e_m$ case we well.
It follows now from Theorem~2.3
 that\footnote{\makebox[11.6em][l]{\it Note for mathematicians}
                                               The coefficients are chosen to make the kinetic term
											     of the complex scalar  field $\Phi_{(0)}$ in the standard/normalized form:
											    $\sum_\mu\partial_\mu\overline{\Phi_{(0)}}
												                       \partial^\mu\Phi_{(0)}$
												and the kinetic term of the gauge field $V_{[\mu]}$
												 in the standard form
												  $-\frac{1}{2g_{\mbox{\tiny\it gauge}}^2}
												      \mbox{\it Tr} \sum_{\mu,\nu}F_{\mu\nu}F^{\mu\nu}$
												  of the Yang-Mills theory with gauge coupling $g_{\mbox{\tiny\it gauge}}$.
									}
     \marginpar{\vspace{2em}\raggedright\tiny
         \raisebox{-1ex}{\hspace{-2.4em}\LARGE $\cdot$}Cf.\,[\,\parbox[t]{20em}{Wess
		\& Bagger:\\ Eq.\:(6.15), \\  Eq.\:(7,2), and \\  Eq.\:(7.6)].}}
 {\small
 \begin{eqnarray*}
  \lefteqn{
   S(\breve{V}, \breve{\Phi})
      }\\[1ex]
   && :=\;\;
    \mbox{\Large $\frac{\tau}{8}$}\,
            \int_{\widehat{X}}d^4x\, d\theta^2d\theta^1\,	 W_1W_2\;
     -\; \mbox{\Large $\frac{\bar{\tau}}{8}$}
            \int_{\widehat{X}}d^4x\, d\bar{\theta}^{\dot{2}}d\bar{\theta}^{\dot{1}}\,	
			      \bar{W}_{\dot{2}}\bar{W}_{\dot{1}}
		\\[1ex]
   && \hspace{2em}
     -\; \mbox{\Large $\frac{1}{4}$}
        \int_{\widehat{X}}
	     d^4x\,    d\bar{\theta}^{\dot{2}}d\bar{\theta}^{\dot{1}} d\theta^2d\theta^1\,	
		       \breve{\Phi}^\dag e^{e_m\breve{V}} \breve{\Phi}
			                \\[1ex]								
	&& \hspace{2em}
	  +\; \int_{\widehat{X}}d^4x\,d\theta^2 d\theta^1\,
			                       \mbox{\Large $($}
				                     \lambda \breve{\Phi}
				                      + \mbox{\large $\frac{1}{2}$}m \breve{\Phi}^2
                                      + \mbox{\large $\frac{1}{3}$}g \breve{\Phi}^3
								   \mbox{\Large $)$}\,
            +\, \int_{\widehat{X}}d^4x\,d\bar{\theta}^{\dot{2}} d\bar{\theta}^{\dot{1}}\,
			                       \mbox{\Large $($}
				                     \bar{\lambda} \breve{\Phi}^\dag
				                      + \mbox{\large $\frac{1}{2}$}\bar{m} (\breve{\Phi}^\dag)^2
                                      + \mbox{\large $\frac{1}{3}$}\bar{g} (\breve{\Phi}^\dag)^3
								   \mbox{\Large $)$}
 \end{eqnarray*}
 }
 gives a gauge-invariant action functional for the component fields
  $(\Phi_{(0)}; \Phi_{(\alpha)}; \Phi_{(12)})_{\alpha=1,2}$
    of $\breve{\Phi}$ (cf.\;small chiral matter) and
  $(V_{[\mu]}; V^{\prime\prime}_{(\alpha)}; V^\sim_{(0)} )_{\mu=0,1,2,3; \alpha=1,2}$
    of $\breve{V}$ (cf.\;gauge field and gaugino field)   on $X$
 that is invariant under supersymmetries, up to boundary terms on $X$.
Here,
  $$
      \tau\; :=\;  \frac{1}{g_\gaugescriptsize^2} \,-\,\sqrt{-1}\,\frac{\theta_\gaugescriptsize}{8\pi^2}\;
	  \in\; {\Bbb C}
  $$
  is the complexified gauge coupling constant, $e_m\in {\Bbb R}$ is the matter charge, and
  $\lambda$, $m$, $g\in {\Bbb C}$ are the complex coupling constant as in the Wess-Zumino model for a small chiral superfield $\Phi$,
  cf.\:Sec.\,3.
 
Explicitly, up to boundary terms on $X$,
{\small
\begin{eqnarray*}
 \lefteqn{
   S(\breve{V}, \breve{\Phi})\;\;=\;\;
     S(V_{[\mu],\, \mu=0,1,2,3},  V^{\prime\prime}_{(\alpha),\, \alpha=1,2},  V^\sim_{(0)}\,;\,
	        \Phi_{(0)}, \Phi_{(\alpha),\, \alpha=1,2}, \Phi_{(12)})
     }\\
  && =\;
   \int_X d^4x\,
     \left\{\rule{0ex}{1.4em}\right.
	  \mbox{\Large $\frac{\tau+ \bar{\tau}}{8}$}\, {V^\sim_{(0)}}^2
	 + \mbox{\Large $\frac{\sqrt{-1}}{4}$}
	      \sum_{\alpha, \dot{\beta}, \mu}
	      \vartheta_\alpha \bar{\vartheta}_{\dot{\beta}}\, \bar{\sigma}^{\mu\dot{\beta}\alpha}
		    \mbox{\Large $($}
			 -\tau V^{\prime\prime}_{(\alpha)}\partial_\mu\,\overline{V^{\prime\prime}_{(\beta)}}
			 +\bar{\tau}\,\overline{V^{\prime\prime}_{(\beta)}}\,\partial_\mu V^{\prime\prime}_{(\alpha)}
			\mbox{\Large $)$}
			\\
	 && \hspace{8em}
	  -\, \mbox{\Large $\frac{\tau+\bar{\tau}}{4}$}\sum_{\mu, \nu} F^{\mu\nu} F_{\mu\nu}
	  + \mbox{\Large $\frac{\sqrt{-1}\,(\tau-\bar{\tau})}{8}$}\sum_{\mu,\nu,\mu^\prime, \nu^\prime}
	       \varepsilon^{\mu\nu\mu^\prime\nu^\prime} F_{\mu\nu} F_{\mu^\prime\nu^\prime}
    \\
   && \hspace{4em}
    -\, \sum_{\mu, \nu}
	       \eta^{\mu\nu}
             \mbox{\large $($}\partial_\mu
			     + \mbox{\Large $\frac{\sqrt{-1}}{2}$} e_m V_{[\mu]} \mbox{\large $)$}
			   \overline{\Phi_{(0)}}\,
             \mbox{\large $($}\partial_\mu
			      - \mbox{\Large $\frac{\sqrt{-1}}{2}$} e_m V_{[\mu]} \mbox{\large $)$}
			  \Phi_{(0)}
			\\
	  && \hspace{6em}
	  +\, \mbox{\Large $\frac{\sqrt{-1}}{2}$}
	      \sum_{\alpha, \dot{\beta}, \mu}
	       \vartheta_\alpha \bar{\vartheta}_{\dot{\beta}}\, \bar{\sigma}^{\mu\dot{\beta}\alpha}\,
		    \overline{\Phi_{(\beta)}}\,
			  \mbox{\large $($}\partial_\mu
			             - \mbox{\Large $\frac{\sqrt{-1}}{2}$} e_m V_{[\mu]} \mbox{\large $)$}
			 \Phi_{(\alpha)}
			 \\
	&& \hspace{6em}
     +\, \mbox{\Large $\frac{1}{4}$}
	      \vartheta_1\vartheta_2\,e_m\,
		      \overline{\Phi_{(0)}}\,
			    \mbox{\large $($}
				  V^{\prime\prime}_{(1)}\Phi_{(2)}
				  + V^{\prime\prime}_{(2)} \Phi_{(1)}
				\mbox{\large $)$}	     	
     - 	\mbox{\Large $\frac{1}{4}$}
		\bar{\vartheta}_{\dot{1}}\bar{\vartheta}_{\dot{2}}\,e_m\,
		       \mbox{\large $($}
			    \overline{\Phi_{(1)}}\,\overline{V^{\prime\prime}_{(2)}}
				+ \overline{\Phi_{(2)}}\,\overline{V^{\prime\prime}_{(1)}}
			   \mbox{\large $)$}\,
			   \Phi_{(0)}
        \\
     && \hspace{6em}					
	    -\, \mbox{\Large $\frac{1}{4}$}
         \vartheta_1 \vartheta_2 \bar{\vartheta}_{\dot{1}}\bar{\vartheta}_{\dot{2}}\,
               \overline{\Phi_{(12)}}\,\Phi_{(12)}		 	
       - \mbox{\Large $\frac{1}{4}$}			
		   e_m\, \overline{\Phi_{(0)}}\, V^\sim_{(0)}\Phi_{(0)}			
		  \\
 && \hspace{6em}
     +\, \vartheta_1\vartheta_2
	        \left(\rule{0ex}{1.2em}\right.
               m\,
			      \mbox{\Large $($}
			        \Phi_{(0)} \Phi_{(12)} - \Phi_{(1)} \Phi_{(2)}
			      \mbox{\Large $)$}    	
				 \\[-1ex]
			 && \hspace{12em}
             +\, g\,
		         \mbox{\Large $($}
			      \Phi_{(0)}^2 \Phi_{(12)}- 2\, \Phi_{(0)}\Phi_{(1)}\Phi_{(2)}
			     \mbox{\Large $)$}
              + \lambda \Phi_{(12)}
			\left.\rule{0ex}{1.2em}\right)	
	   \\
	&& \hspace{6em}
	 -\,  \bar{\vartheta}_{\dot{1}}\bar{\vartheta}_{\dot{2}}
	         \left(\rule{0ex}{1.2em}\right.
		       \bar{m}\,
		         \mbox{\Large $($}
			      \overline{\Phi_{(0)}}\, \overline{\Phi_{(12)}}
			        - \overline{\Phi_{(1)}}\, \overline{\Phi_{(2)}}
			     \mbox{\Large $)$}
				  \\[-1ex]
			 && \hspace{12em}
             +\, \bar{g}\,
		          \mbox{\Large $($}
			        \overline{\Phi_{(0)}}^2\, \overline{\Phi_{(12)}}
			          - 2\, \overline{\Phi_{(0)}}\,\overline{\Phi_{(1)}}\,\overline{\Phi_{(2)}}
			  \mbox{\Large $)$}		
              +   \bar{\lambda}\, \overline{\Phi_{(12)}}			
			\left.\rule{0ex}{1.2em}\right)	    								   		   		
     \left.\rule{0ex}{1.4em}\right\}\,.
\end{eqnarray*}
}
 
After imposing the purge-evaluation map
$$
  \Pev\;:\;
    \vartheta_1\vartheta_2\;\rightsquigarrow\; 1\,,\;\;\;\;
	\vartheta_\alpha\bar{\vartheta}_{\dot{\beta}}\; \rightsquigarrow\; -1\,,\;\;\;\;
    \bar{\vartheta}_{\dot{1}}	\bar{\vartheta}_{\dot{2}}\; \rightsquigarrow\; -1\,,\;\;\;\;
	 \vartheta_1\vartheta_2 \bar{\vartheta}_{\dot{1}}	\bar{\vartheta}_{\dot{2}}\;
	      \rightsquigarrow\; -1\,.
 $$
 to remove the even nilpotent factors
  $\vartheta_1\vartheta_2,\,
	\vartheta_\alpha\bar{\vartheta}_{\dot{\beta}},\,
    \bar{\vartheta}_{\dot{1}}	\bar{\vartheta}_{\dot{2}},\,
	 \vartheta_1\vartheta_2 \bar{\vartheta}_{\dot{1}}	\bar{\vartheta}_{\dot{2}}$	
 in the expression,
the action functional becomes
{\small
\begin{eqnarray*}
 \lefteqn{
   S(\breve{V}, \breve{\Phi})\;\;=\;\;
     S(V_{[\mu],\, \mu=0,1,2,3},  V^{\prime\prime}_{(\alpha),\, \alpha=1,2},  V^\sim_{(0)}\,;\,
	        \Phi_{(0)}, \Phi_{(\alpha),\, \alpha=1,2}, \Phi_{(12)})
     }\\
  && =\;
   \int_X d^4x\,
     \left\{\rule{0ex}{1.4em}\right.
	  \mbox{\Large $\frac{\tau+ \bar{\tau}}{8}$}\, {V^\sim_{(0)}}^2
	 + \mbox{\Large $\frac{\sqrt{-1}}{4}$}
	      \sum_{\alpha, \dot{\beta}, \mu}
	       \bar{\sigma}^{\mu\dot{\beta}\alpha}
		    \mbox{\Large $($}
			 \tau V^{\prime\prime}_{(\alpha)}\partial_\mu\,\overline{V^{\prime\prime}_{(\beta)}}
			 - \bar{\tau}\,\overline{V^{\prime\prime}_{(\beta)}}\,\partial_\mu V^{\prime\prime}_{(\alpha)}
			\mbox{\Large $)$}
			\\
	 && \hspace{8em}
	  -\, \mbox{\Large $\frac{\tau+\bar{\tau}}{4}$}\sum_{\mu, \nu} F^{\mu\nu} F_{\mu\nu}
	  + \mbox{\Large $\frac{\sqrt{-1}\,(\tau-\bar{\tau})}{8}$}\sum_{\mu,\nu,\mu^\prime, \nu^\prime}
	       \varepsilon^{\mu\nu\mu^\prime\nu^\prime} F_{\mu\nu} F_{\mu^\prime\nu^\prime}
    \\
   && \hspace{4em}
    -\, \sum_{\mu, \nu}
	       \eta^{\mu\nu}
             \mbox{\large $($}\partial_\mu
			     + \mbox{\Large $\frac{\sqrt{-1}}{2}$} e_m V_{[\mu]} \mbox{\large $)$}
			   \overline{\Phi_{(0)}}\,
             \mbox{\large $($}\partial_\mu
			      - \mbox{\Large $\frac{\sqrt{-1}}{2}$} e_m V_{[\mu]} \mbox{\large $)$}
			  \Phi_{(0)}
			\\
	  && \hspace{6em}
	  -\, \mbox{\Large $\frac{\sqrt{-1}}{2}$}
	      \sum_{\alpha, \dot{\beta}, \mu}
	       \bar{\sigma}^{\mu\dot{\beta}\alpha}\,
		    \overline{\Phi_{(\beta)}}\,
			  \mbox{\large $($}\partial_\mu
			             - \mbox{\Large $\frac{\sqrt{-1}}{2}$} e_m V_{[\mu]} \mbox{\large $)$}
			 \Phi_{(\alpha)}
			 \\
	&& \hspace{6em}
     +\, \mbox{\Large $\frac{1}{4}$}\,e_m\,
		      \overline{\Phi_{(0)}}\,
			    \mbox{\large $($}
				  V^{\prime\prime}_{(1)}\Phi_{(2)}
				  + V^{\prime\prime}_{(2)} \Phi_{(1)}
				\mbox{\large $)$}	     	
     + 	\mbox{\Large $\frac{1}{4}$}\,e_m\,
		       \mbox{\large $($}
			    \overline{\Phi_{(1)}}\,\overline{V^{\prime\prime}_{(2)}}
				+ \overline{\Phi_{(2)}}\,\overline{V^{\prime\prime}_{(1)}}
			   \mbox{\large $)$}\,
			   \Phi_{(0)}
        \\
     && \hspace{6em}					
	    +\, \mbox{\Large $\frac{1}{4}$}\,
               \overline{\Phi_{(12)}}\,\Phi_{(12)}		 	
       - \mbox{\Large $\frac{1}{4}$}			
		   e_m\, \overline{\Phi_{(0)}}\, V^\sim_{(0)}\Phi_{(0)}			
		  \\
 && \hspace{6em}
     +\,  m\,
			      \mbox{\Large $($}
			        \Phi_{(0)} \Phi_{(12)} - \Phi_{(1)} \Phi_{(2)}
			      \mbox{\Large $)$}    	
				 \\[-1ex]
			 && \hspace{12em}
             +\, g\,
		         \mbox{\Large $($}
			      \Phi_{(0)}^2 \Phi_{(12)}- 2\, \Phi_{(0)}\Phi_{(1)}\Phi_{(2)}
			     \mbox{\Large $)$}
              + \lambda \Phi_{(12)}
	   \\
	&& \hspace{6em}
	 +\,  \bar{m}\,
		         \mbox{\Large $($}
			      \overline{\Phi_{(0)}}\, \overline{\Phi_{(12)}}
			        - \overline{\Phi_{(1)}}\, \overline{\Phi_{(2)}}
			     \mbox{\Large $)$}
				  \\[-1ex]
			 && \hspace{12em}
             +\, \bar{g}\,
		          \mbox{\Large $($}
			        \overline{\Phi_{(0)}}^2\, \overline{\Phi_{(12)}}
			          - 2\, \overline{\Phi_{(0)}}\,\overline{\Phi_{(1)}}\,\overline{\Phi_{(2)}}
			  \mbox{\Large $)$}		
              +   \bar{\lambda}\, \overline{\Phi_{(12)}}			
     \left.\rule{0ex}{1.4em}\right\}\,.
\end{eqnarray*}
}

\bigskip

\begin{remark}$[$supersymmetric non-Abelian gauge theory\,$]$\; {\rm
 (Cf.\:[Wess \& Bagger: Ch.\:VII, pp.\,45, 46, 47].)
 The above construction can be generalized to the non-Abelian case.
 In particular, [Wess \& Bagger: Eqs.\:(7.22), (7.24)] can be computed explicitly under the setting of Sec.\:1.
 However, the proof of the existence of Wess-Zumino gauge is more technical. 
 Thus, the discussion of supersymmetric non-Abelian gauge theories in the current setting deserves a separate work in its own right.
}\end{remark}

\bigskip

\section{$d=3+1$, $N=1$ nonlinear sigma models\;\\
 {\large (cf.\:[Wess \& Bagger: Chapter XXII])}}	

We reconstruct in this section
  \begin{itemize}
   \item[\LARGE $\cdot$] [Wess \& Bagger: Chapter XXII.\:{\it Chiral models and K\"{a}hler geometry}\,]
  \end{itemize}
 in the complexified ${\Bbb Z}/2$-graded $C^\infty$-Algebraic Geometry setting of Sec.\:1.

\bigskip
 
\subsection{Smooth maps from $\widehat{X}^{\widehat{\boxplus}, \smallscriptsize}$ to a smooth manifold $Y$}

Basics of smooth maps
  from the small towered superspace $\widehat{X}^{\widehat{\boxplus}, \smallscriptsize}$ to a smooth manifold $Y$
 from the aspect of complexified $C^\infty$-Algebraic Geometry are given in this subsection.

\bigskip
 
\begin{flushleft}
{\bf Smooth maps from the small towered superspace $\widehat{X}^{\widehat{\boxplus}, \smallscriptsize}$
          to a smooth manifold $Y$}
\end{flushleft}
Let $Y$ be a smooth manifold,   with
  its function ring $C^\infty(Y)$ of smooth functions,
  its structure sheaf ${\cal O}_Y$ of smooth functions,   and
  their complexification $C^\infty(Y)^{\Bbb C}$ and ${\cal O}_Y^{\,\Bbb C}$ respectively.

\bigskip
 
\begin{definition} {\bf [smooth map from small towered superspace]}\; {\rm
 A {\it smooth map $\breve{f}$} from the small towered superspace
   $\widehat{X}^{\widehat{\boxplus}, \smallscriptsize}$ to $Y$, in notation
  $$
     \breve{f}\; :\; \widehat{X}^{\widehat{\boxplus}, \smallscriptsize}\; \longrightarrow\; Y\,,
  $$
  is a pair $(f, \breve{f}^\sharp)$,
  where $f:X\rightarrow Y$ is a smooth map of smooth manifolds and
  $$
      \breve{f}^\sharp\; :\;  C^\infty(Y)\; \longrightarrow\;  C^\infty(\widehat{X}^{\widehat{\boxplus}})^\smallscriptsize
  $$
  is a ring-homomorphism
  with its $(\theta, \bar{\theta})$-degree-zero component $f_{(0)}^\sharp$ identical to
  the $C^\infty$-ring-homomorphism $f^\sharp: C^\infty(Y)\rightarrow C^\infty(X)$ associated to $f$.
 As locally-ringed spaces,
  we may regard $\widehat{X}^{\widehat{\boxplus}, \smallscriptsize}$ and $Y$
    as equivalence classes of gluing systems of rings and
  $\breve{f}^\sharp$ as from an equivalence class of gluing systems of ring-homomorphisms and write
  $\breve{f}^\sharp:{\cal O}_Y\rightarrow  \widehat{\cal O}_X^{\,\widehat{\boxplus}, \smallscriptsize}$;
 (cf.\ the similar setting in [L-Y1: Sec.\,1] (D(1))).
}\end{definition}

\bigskip

The following lemma can be checked directly:

\bigskip

\begin{lemma}
{\bf [$\breve{f}^\sharp$ as $C^\infty$-ring-homomorphism and
            natural extension to $C^\infty(Y)^{\Bbb C}$]}\;
 The ring-homomorphism
    $\breve{f}^\sharp: C^\infty(Y)\rightarrow C^\infty(\widehat{X}^{\widehat{\boxplus}})^\smallscriptsize$
   is a $C^\infty$-ring-homomorphism from $C^\infty(Y)$ to the $C^\infty$-hull of
    $C^\infty(\widehat{X}^{\widehat{\boxplus}})^\smallscriptsize$.
 $\breve{f}^\sharp$ extends naturally to a ring-homomorphism
  $C^\infty(Y)^{\Bbb C}\rightarrow C^\infty(\widehat{X}^{\widehat{\boxplus}})^\smallscriptsize$
  by the correspondence
  $$
    h_1+\sqrt{-1}\,h_2\;
	  \longrightarrow\; \breve{f}^\sharp(h_1)+\sqrt{-1}\,\breve{f}^\sharp(h_2)\,,
  $$
   for $h_1, h_2\in C^\infty(Y)$.
 We will denote this extension of $\breve{f}^\sharp$ still  by $\breve{f}^\sharp$.
\end{lemma}

\bigskip

\noindent
Due to the above lemma and with slight abuse of terminology,
  we will call\\
  $\breve{f}^\sharp:~C^\infty(Y)\rightarrow~C^\infty(\widehat{X}^{\widehat{\boxplus}})^\smallscriptsize$ 
  directly a $C^\infty$-ring-homomorphism.
Also, depending on context, we will write
 $\breve{f}^\sharp:~C^\infty(Y)^{\Bbb C}
    \rightarrow~C^\infty(\widehat{X}^{\widehat{\boxplus}})^\smallscriptsize$
   as a ring-homomorphism
 or $\breve{f}^\sharp: {\cal O}_Y^{\,\Bbb C}
        \rightarrow  \widehat{\cal O}_X^{\,\widehat{\boxplus}, \smallscriptsize}$
   as an equivalence of gluing systems of ring-homomorphisms.

The components of $\breve{f}^\sharp$ in the $(\theta, \bar{\theta}, \vartheta, \bar{\vartheta})$-expansion
   are $C^\infty(X)^{\Bbb C}$-valued operation on $C^\infty(X)$.
Their basic properties are summarized in the following proposition.
   
\bigskip

\begin{proposition} {\bf [components of $\breve{f}^\sharp$]}\;
 Given a smooth map $\breve{f}: \widehat{X}^{\widehat{\boxplus}, \smallscriptsize} \rightarrow Y$,
  let
 {\small
  \begin{eqnarray*}
   \lefteqn{
    \breve{f}^\sharp\; =\;
	  f^\sharp_{(0)}
      + \sum_\alpha \theta^\alpha\vartheta_\alpha f^\sharp_{(\alpha)}
      + \sum_{\dot{\beta}} \bar{\theta}^{\dot{\beta}} \bar{\vartheta}_{\dot{\beta}}	
	        f^\sharp_{(\dot{\beta})}  }\\
   && \hspace{2em}						
	  +\, \theta^1\theta^2 \vartheta_1\vartheta_2 f^\sharp_{(12)}
	  + \sum_{\alpha, \dot{\beta}} \theta^\alpha \bar{\theta}^{\dot{\beta}}
	       \mbox{\Large $($}
		    \sum_\mu
			   \sigma^\mu_{\alpha\dot{\beta}} f^\sharp_{[\mu]}
			 + \vartheta_\alpha \bar{\vartheta}_{\dot{\beta}}   f^\sharp_{(\alpha\dot{\beta})}
		   \mbox{\Large $)$}	
      +\, \bar{\theta}^{\dot{1}} \bar{\theta}^{\dot{2}}
	          \bar{\vartheta}_{\dot{1}} \bar{\vartheta}_{\dot{2}} f^\sharp_{(\dot{1}\dot{2})}
			   \\
   && \hspace{2em}			
	  +\, \sum_{\dot{\beta}} \theta^1\theta^2\bar{\theta}^{\dot{\beta}}
	       \mbox{\Large $($}
		     \sum_{\alpha, \mu}
			      \vartheta_\alpha \sigma^{\mu\alpha}_{\;\;\;\;\dot{\beta}} f^{\sharp\prime}_{[\mu]}
	         + \vartheta_1\vartheta_2\bar{\vartheta}_{\dot{\beta}} f^\sharp_{(12\dot{\beta})}
		   \mbox{\Large $)$}
	  +\, \sum_\alpha \theta^\alpha \bar{\theta}^{\dot{1}} \bar{\theta}^{\dot{2}}
	      \mbox{\Large $($}
		    \sum_{\dot{\beta},\mu} \bar{\vartheta}_{\dot{\beta}}
			      \sigma^{\mu\dot{\beta}}_\alpha f^{\sharp\prime\prime}_{[\mu]}
			+ \vartheta_\alpha \bar{\vartheta}_{\dot{1}} \bar{\vartheta}_{\dot{2}}
			        f^\sharp_{(\alpha\dot{1}\dot{2})}
		  \mbox{\Large $)$}		   		    \\
   && \hspace{2em}		
	  +\, \theta^1\theta^2 \bar{\theta}^{\dot{1}}\bar{\theta}^{\dot{2}}
	       \mbox{\Large $($}
		    f^{\sharp\sim}_{(0)}
			+ \sum_{\alpha, \dot{\beta}, \mu}
			     \vartheta_\alpha \bar{\vartheta}_{\dot{\beta}}
			       \bar{\sigma}^{\mu\dot{\beta}\alpha}  f^{\sharp\sim}_{[\mu]}		
	        + \vartheta_1\vartheta_2 \bar{\vartheta}_{\dot{1}} \bar{\vartheta}_{\dot{2}}
           			f^\sharp_{(12\dot{1}\dot{2})}	
		   \mbox{\Large $)$}
  \end{eqnarray*}}be 
   the presentation of the associated $\breve{f}^\sharp$ in terms of components
     in the $(\theta, \bar{\theta}, \vartheta, \bar{\vartheta})$-expansion  and
   ${\cal D}_Y^{[1, i]}$, $i\in {\Bbb Z}_{\ge 1}$
    be the sheaf of differential operators on $Y$ of order $i$ with smooth coefficients and without the zero-th order term.
 (Note that ${\cal D}_Y^{[1,1]}$ is the tangent sheaf ${\cal T}_Y$ of $Y$.)
 Then
  $f_{(0)}^\sharp = f^\sharp: {\cal O}_Y\rightarrow {\cal O}_X$
  is the equivalence class of gluing systems of $C^\infty$-ring-homomorphisms associated to
    $f:X\rightarrow Y$ underlying $\breve{f}$;
   and  $\breve{f}^\sharp-f_{(0)}^\sharp$  is a smooth section of
  {\small
  \begin{eqnarray*}
   \lefteqn{
     \sum_\alpha
	      {\cal O}_X^{\,\Bbb C}\cdot \theta^\alpha\vartheta_\alpha
	       \otimes f_{(0)}^\ast {\cal T}_Y
      + \sum_{\dot{\beta}}
	      {\cal O}_X^{\,\Bbb C}\cdot  \bar{\theta}^{\dot{\beta}} \bar{\vartheta}_{\dot{\beta}}
			\otimes f_{(0)}^\ast {\cal T}_Y
	  + {\cal O}_X^{\,\Bbb C}\cdot
	       \theta^1\theta^2 \vartheta_1\vartheta_2
	        \otimes f_{(0)}^\ast {\cal D}_Y^{[1,2]}    }\\
   &&			
   	  +\, \sum_{\alpha, \dot{\beta}}
	       {\cal O}_X^{\,\Bbb C}\cdot  \theta^\alpha \bar{\theta}^{\dot{\beta}}		
		    \mbox{\Large $($}
			  \sum_\mu \sigma^\mu_{\alpha\dot{\beta}}\otimes f_{(0)}^\ast{\cal T}_Y
			+ \vartheta_\alpha\vartheta_{\dot{\beta}}
                  \otimes f_{(0)}^\ast {\cal D}_Y^{[1,2]}		  	
		    \mbox{\Large $)$}				
      + {\cal O}_X^{\,\Bbb C}\cdot
	       \bar{\theta}^{\dot{1}} \bar{\theta}^{\dot{2}}
		       \bar{\vartheta}_{\dot{1}} \bar{\vartheta}_{\dot{2}}
	          \otimes f_{(0)}^\ast {\cal D}_Y^{[1,2]}    \\
   &&
	  +\,  \sum_{\dot{\beta}}
	        {\cal O}_X^{\,\Bbb C}\cdot \theta^1\theta^2\bar{\theta}^{\dot{\beta}}			
		  \mbox{\Large $($}	
		   \sum_{\alpha, \mu}
            \vartheta_\alpha \sigma^{\mu\alpha}_{\;\;\;\;\dot{\beta}}		
			        \otimes f_{(0)}^\ast {\cal D}_Y^{[1,2]}						
	       + \vartheta_1\vartheta_2 \bar{\vartheta}_{\dot{\beta}}
		     \otimes f_{(0)}^\ast {\cal D}_Y^{[1,3]}
		  \mbox{\Large $)$}	 			 \\
   && \hspace{3em}		
	  +\,  \sum_\alpha
	       {\cal O}_X^{\,\Bbb C}\cdot
	      \theta^\alpha \bar{\theta}^{\dot{1}} \bar{\theta}^{\dot{2}} 		
		\mbox{\Large $($}
		  \sum_{\dot{\beta}, \mu}\sigma^{\mu\dot{\beta}}_\alpha
		      \otimes f_{(0)}^\ast {\cal D}_Y^{[1,2]}
		 +\vartheta_\alpha \bar{\vartheta}_{\dot{1}} \bar{\vartheta}_{\dot{2}}
	         \otimes f_{(0)}^\ast {\cal D}_Y^{[1,3]}
        \mbox{\Large $)$}			 			 			  \\
  &&
	  +\, {\cal O}_X^{\,\Bbb C}\cdot
	         \theta^1\theta^2 \bar{\theta}^{\dot{1}}\bar{\theta}^{\dot{2}}			 			
		\mbox{\Large $($}	 			
		f_{(0)}^\ast {\cal D}_Y^{[1,2]}		
		 + \sum_{\alpha, \dot{\beta}, \mu}
		       \vartheta_\alpha \bar{\vartheta}_{\dot{\beta}}
			      \sigma^\mu_{\alpha\dot{\beta}}
				   \otimes f_{(0)}^\ast {\cal D}_Y^{[1,3]}			
	     + \vartheta_1\vartheta_2 \bar{\vartheta}_{\dot{1}}  \bar{\vartheta}_{\dot{2}}
			     \otimes f_{(0)}^\ast {\cal D}_Y^{[1,4]}
	    \mbox{\Large $)$}		 			 			
  \end{eqnarray*}}in 
  $\widehat{\cal O}_X^{\,\widehat{\boxplus}, \smallscriptsize}\otimes_{{\cal O}_X}\!\!f_{(0)}^\ast{\cal D}_Y
     = \widehat{\cal O}_X^{\,\widehat{\boxplus}, \smallscriptsize}
	                    \otimes_{f_{(0)}^\sharp, {\cal O}_Y} {\cal D}_Y$,
  where ${\cal D}_Y$ is the sheaf of differential operators on $Y$ of smooth coefficients.
  
 In terms of $\breve{f}^\sharp: {\cal O}_Y\rightarrow  \widehat{\cal O}_X^{\,\widehat{\boxplus}, \smallscriptsize}$,
  a local chart $V$ of $Y$ with coordinates functions $(y^1,\,\cdots\,,\, y^n)$,
  and an open set $U\subset X$ such that $f_{(0)}(U)\subset V$,
 let
 {\small
  \begin{eqnarray*}
    \breve{f}^i  & :=\; &    \breve{f}^\sharp(y^i)   \\
	& = &
      f^\sharp_{(0)}(y^i)
      + \sum_\alpha \theta^\alpha\vartheta_\alpha f^\sharp_{(\alpha)}(y^i)
      + \sum_{\dot{\beta}} \bar{\theta}^{\dot{\beta}} \bar{\vartheta}_{\dot{\beta}}	
	        f^\sharp_{(\dot{\beta})}(y^i)  \\
   && \hspace{1em}						
	  +\, \theta^1\theta^2 \vartheta_1\vartheta_2 f^\sharp_{(12)}(y^i)
	  + \sum_{\alpha, \dot{\beta}} \theta^\alpha \bar{\theta}^{\dot{\beta}}
	       \mbox{\Large $($}
		    \sum_\mu
			   \sigma^\mu_{\alpha\dot{\beta}} f^\sharp_{[\mu]}(y^i)
			 + \vartheta_\alpha \bar{\vartheta}_{\dot{\beta}}   f^\sharp_{(\alpha\dot{\beta})}(y^i)
		   \mbox{\Large $)$}	
      +\, \bar{\theta}^{\dot{1}} \bar{\theta}^{\dot{2}}
	          \bar{\vartheta}_{\dot{1}} \bar{\vartheta}_{\dot{2}}
			      f^\sharp_{(\dot{1}\dot{2})}(y^i)   \\
   && \hspace{1em}			
	  +\, \sum_{\dot{\beta}} \theta^1\theta^2\bar{\theta}^{\dot{\beta}}
	       \mbox{\Large $($}
		     \sum_{\alpha, \mu}
			      \vartheta_\alpha \sigma^{\mu\alpha}_{\;\;\;\;\dot{\beta}} f^{\sharp\prime}_{[\mu]}(y^i)
	         + \vartheta_1\vartheta_2\bar{\vartheta}_{\dot{\beta}} f^\sharp_{(12\dot{\beta})}(y^i)
		   \mbox{\Large $)$}    \\
   && \hspace{3em}		
	  +\, \sum_\alpha \theta^\alpha \bar{\theta}^{\dot{1}} \bar{\theta}^{\dot{2}}
	      \mbox{\Large $($}
		    \sum_{\dot{\beta},\mu} \bar{\vartheta}_{\dot{\beta}}
			      \sigma^{\mu\dot{\beta}}_\alpha f^{\sharp\prime\prime}_{[\mu]}(y^i)
			+ \vartheta_\alpha \bar{\vartheta}_{\dot{1}} \bar{\vartheta}_{\dot{2}}
			        f^\sharp_{(\alpha\dot{1}\dot{2})}(y^i)
		  \mbox{\Large $)$}		   		    \\
   && \hspace{1em}		
	  +\, \theta^1\theta^2 \bar{\theta}^{\dot{1}}\bar{\theta}^{\dot{2}}
	       \mbox{\Large $($}
		    f^{\sharp\sim}_{(0)}(y^i)
			+ \sum_{\alpha, \dot{\beta}, \mu}
			     \vartheta_\alpha \bar{\vartheta}_{\dot{\beta}}
			       \bar{\sigma}^{\mu\dot{\beta}\alpha}  f^{\sharp\sim}_{[\mu]}(y^i)		
	        + \vartheta_1\vartheta_2 \bar{\vartheta}_{\dot{1}} \bar{\vartheta}_{\dot{2}}
           			f^\sharp_{(12\dot{1}\dot{2})}(y^i)	
		   \mbox{\Large $)$} \\			  		  		
   &\;=: & f_{(0)}^\sharp(y^i) + \breve{n}^i_{\breve{f}} \;\;\;\;
         \in\; C^\infty(\widehat{U}^{\widehat{\boxplus}})^\smallscriptsize\,, 		
  \end{eqnarray*}}where 
  $\breve{n}^i_{\breve{f}}$ is the nilpotent part of $\breve{f}^\sharp(y^i)$,
    and
  denote $\frac{\partial}{\partial y^i}$ by $\partial_i$, for $i=1,\,\cdots\,,\, n$.
  Then,
  for $h\in C^\infty(V)$,
 {\small
  \begin{eqnarray*}
   \lefteqn{
    \breve{f}^\sharp(h)\; =\;
	  f^\sharp_{(0)}(h)
     + \sum_\alpha \theta^\alpha\vartheta_\alpha
	         \sum_{i=1}^n f^\sharp_{(\alpha)}(y^i)\otimes \partial_i h
     + \sum_{\dot{\beta}} \bar{\theta}^{\dot{\beta}} \bar{\vartheta}_{\dot{\beta}}	
	         \sum_{j=1}^n f^\sharp_{(\dot{\beta})}(y^j) \otimes \partial_j h      }\\
    &&    			
	 +\, \theta^1\theta^2 \cdot  \vartheta_1\vartheta_2\,
	        \left\{\rule{0ex}{1.2em}\right.			
			  \sum_i   f^\sharp_{(12)}(y^i)\otimes \partial_i
			    + \mbox{\Large $\frac{1}{2}$}
				     \sum_{i,j} J_{\breve{f}, (12)}^{\,ij}\otimes \partial_i\partial_j
		    \left.\rule{0ex}{1.2em}\right\} h   \\
	&& \hspace{1em}
	 + \sum_{\alpha, \dot{\beta}} \theta^\alpha \bar{\theta}^{\dot{\beta}}	
		    \left\{\rule{0ex}{1.2em}\right.
			  \sigma^\mu_{\alpha\dot{\beta}}\,
			    \sum_i   f^\sharp_{[\mu]}(y^i)\otimes \partial_i							
			  +  \vartheta_\alpha \bar{\vartheta}_{\dot{\beta}}\,				
			      \mbox{\Large $($}
				   \sum_i f_{(\alpha\dot{\beta})}^\sharp(y^i)\otimes \partial_i
			        + \mbox{\Large $\frac{1}{2}$}\,				
				         \sum_{i,j} J_{\breve{f}, (\alpha\dot{\beta})}^{\,ij}\otimes \partial_i\partial_j
			     \mbox{\Large $)$}		 					
		    \left.\rule{0ex}{1.2em}\right\} h             \\
    && \hspace{1em}
     +\, \bar{\theta}^{\dot{1}} \bar{\theta}^{\dot{2}} \cdot
	        \bar{\vartheta}_{\dot{1}} \bar{\vartheta}_{\dot{2}}
			\left\{\rule{0ex}{1.2em}\right.
			  \sum_i   f^\sharp_{(\dot{1}\dot{2})}(y^i)\otimes \partial_i
			    + \mbox{\Large $\frac{1}{2}$}
				     \sum_{i,j} J_{\breve{f}, (\dot{1}\dot{2})}^{\,ij}\otimes \partial_i\partial_j
		    \left.\rule{0ex}{1.2em}\right\} h  \\
   &&		
	+\, \sum_{\dot{\beta}} \theta^1\theta^2\bar{\theta}^{\dot{\beta}}	   		
            \left\{\rule{0ex}{1.2em}\right.
			  \sum_{\alpha, \mu} \vartheta_\alpha \sigma^{\mu\alpha}_{\;\;\;\;\dot{\beta}}
			   \mbox{\Large $($}
			     \sum_i f^{\sharp\,\prime}_{[\mu]}(y^i) \otimes \partial_i\,
			      + \mbox{\Large $\frac{1}{2}$}
				        \sum_{i,j} J_{\breve{f}, ([\mu]}^{\,ij\,\prime}\otimes \partial_i\partial_j
               \mbox{\Large $)$}			   \\
      && \hspace{4em}			
		 +\, \vartheta_1\vartheta_2 \bar{\vartheta}_{\dot{\beta}}
			 \mbox{\Large $($}
			  \sum_i   f^\sharp_{(12\dot{\beta})}(y^i)\otimes \partial_i
			    + \mbox{\Large $\frac{1}{2}$}
				     \sum_{i,j} J_{\breve{f}, (12\dot{\beta})}^{\,ij}\otimes \partial_i\partial_j
			    + \mbox{\Large $\frac{1}{3!}$}
				     \sum_{i,j,k} J_{\breve{f}, (12\dot{\beta})}^{\,ijk}\otimes \partial_i\partial_j \partial_k
             \mbox{\Large $)$}					
		    \left.\rule{0ex}{1.2em}\right\} h    \\
   && \hspace{1em}
	  +\, \sum_\alpha
	          \theta^\alpha \bar{\theta}^{\dot{1}} \bar{\theta}^{\dot{2}}		  			
            \left\{\rule{0ex}{1.2em}\right. 			
			 \sum_{\dot{\beta}, \mu}
			   \bar{\vartheta}_{\dot{\beta}} \sigma^{\mu\dot{\beta}}_\alpha
			   \mbox{\Large $($}
			     \sum_i f^{\sharp\,\prime\prime}_{[\mu]}(y^i) \otimes \partial_i\,
			      + \mbox{\Large $\frac{1}{2}$}
				      \sum_{i,j}
					     J_{\breve{f}, [\mu]}^{\,ij\,\prime\prime}\otimes \partial_i\partial_j
               \mbox{\Large $)$}		\\			
      && \hspace{4em}
			 +\,  \vartheta_\alpha \bar{\vartheta}_{\dot{1}}\bar{\vartheta}_{\dot{2}}
			  \mbox{\Large $($}
			  \sum_i   f^\sharp_{(\alpha\dot{1}\dot{2})}(y^i)\otimes \partial_i
			    + \mbox{\Large $\frac{1}{2}$}
				     \sum_{i,j} J_{\breve{f}, (\alpha\dot{1}\dot{2})}^{\,ij}\otimes \partial_i\partial_j
			    + \mbox{\Large $\frac{1}{3!}$}
				     \sum_{i,j,k} J_{\breve{f}, (\alpha\dot{1}\dot{2})}^{\,ijk}
					       \otimes \partial_i\partial_j \partial_k
             \mbox{\Large $)$}						
		    \left.\rule{0ex}{1.2em}\right\} h    \\			
   &&		
	 +\, \theta^1\theta^2 \bar{\theta}^{\dot{1}}\bar{\theta}^{\dot{2}} \cdot
            \left\{\rule{0ex}{1.2em}\right.
			 \mbox{\Large $($}
			  \sum_i f^{\sharp\,\sim}_{[0]}(y^i)\otimes \partial_i
			   + \mbox{\Large $\frac{1}{2}$}
				     \sum_{i,j} J_{\breve{f}, (0)}^{\,ij\,\sim}\otimes \partial_i\partial_j
			 \mbox{\Large $)$} 	 \\
	  && \hspace{4em}	
			  + \sum_{\alpha, \dot{\beta}}
			       \vartheta_\alpha \bar{\vartheta}_{\dot{\beta}}\bar{\sigma}^{\mu\dot{\beta}\alpha}
				    \mbox{\Large $($}
				      \sum_i f^{\sharp\,\sim}_{[\mu]}(y^i)\otimes \partial_i
			          + \mbox{\Large $\frac{1}{2}$}
				           \sum_{i,j} J_{\breve{f}, [\mu]}^{\,ij\,\sim}\otimes \partial_i\partial_j
                      + \mbox{\Large $\frac{1}{3!}$}
				           \sum_{i,j,k} J_{\breve{f}, [\mu]}^{\,ijk\,\sim}\otimes \partial_i\partial_j						   						
				    \mbox{\Large $)$} \\
      && \hspace{4em}
             +\, \vartheta_1\vartheta_2\bar{\vartheta}_{\dot{1}}	  \bar{\vartheta}_{\dot{2}}
			   \mbox{\Large $($}
			    \sum_i   f^\sharp_{(12\dot{1}\dot{2})}(y^i)\otimes \partial_i
			      + \mbox{\Large $\frac{1}{2}$}
				     \sum_{i,j} J_{\breve{f}, (12\dot{1}\dot{2})}^{\,ij}\otimes \partial_i\partial_j\\[-1ex]
      && \hspace{12em}					
			    + \mbox{\Large $\frac{1}{3!}$}
				     \sum_{i,j,k} J_{\breve{f}, (12\dot{1}\dot{2})}^{\,ijk}
					     \otimes \partial_i\partial_j \partial_k
			    + \mbox{\Large $\frac{1}{4!}$}
				     \sum_{i,j,k,l} J_{\breve{f}, (12\dot{1}\dot{2})}^{\,ijkl}
					     \otimes \partial_i\partial_j \partial_k\partial_l
			  \mbox{\Large $)$}
		    \left.\rule{0ex}{1.2em}\right\} h    \\[1.2ex]
	 && \in\;
	        \widehat{\cal O}_U^{\,\widehat{\boxplus}, \smallscriptsize}	
			      \otimes_{f_{(0)}^\sharp, {\cal O}_V}{\cal D}_V\;
		    =\;  \widehat{\cal O}_U^{\,\widehat{\boxplus}, \smallscriptsize}
			          \otimes_{{\cal O}_U}f_{(0)}^\ast {\cal D}_V\,.
  \end{eqnarray*}}Here, 
  $J_{\breve{f}, (\tinybullet)}^{\,\tinybullet}\in C^\infty(U)^{\Bbb C}$ are
    the components in the following expansions:
 {\small
  \begin{eqnarray*}
   \lefteqn{
    \breve{n}^i_{\breve{f}}\, \breve{n}^j_{\breve{f}}\;
      =\;  \theta^1\theta^2 \vartheta_1\vartheta_2 J_{\breve{f}, (12)}^{\,ij}
	  + \sum_{\alpha, \dot{\beta}} \theta^\alpha \bar{\theta}^{\dot{\beta}}
	       \vartheta_\alpha\bar{\vartheta}_{\dot{\beta}} J_{\breve{f},(\alpha\dot{\beta})}^{\,ij}
      +\, \bar{\theta}^{\dot{1}} \bar{\theta}^{\dot{2}}
	           \bar{\vartheta}_{\dot{1}} \bar{\vartheta}_{\dot{2}}
	       J_{\breve{f},(\dot{1}\dot{2})}^{\,\ij}    } \\
   && \hspace{4em}
	+\, \sum_{\dot{\beta}} \theta^1\theta^2\bar{\theta}^{\dot{\beta}}	
		   \mbox{\Large $($}
		     \sum_{\alpha, \mu }
			   \vartheta_\alpha \sigma^{\mu\alpha}_{\;\;\;\;\dot{\beta}} J^{\,ij\,\prime}_{\breve{f}, [\mu]}			   		
	       + \vartheta_1\vartheta_2 \bar{\vartheta}_{\dot{\beta}}
		        J_{\breve{f}, (12\dot{\beta})}^{\,ij}
		   \mbox{\Large $)$} \\
   && \hspace{6em}		
	+\, \sum_\alpha \theta^\alpha \bar{\theta}^{\dot{1}} \bar{\theta}^{\dot{2}} 	
	   \mbox{\Large $($}
	     \sum_{\dot{\beta}, \mu}
		   \bar{\vartheta}_{\dot{\beta}} \sigma^{\mu\dot{\beta}}_\alpha
		      J^{\,ij\,\prime\prime}_{\breve{f}, [\mu]}
	    + \vartheta_\alpha \bar{\vartheta}_{\dot{1}} \bar{\vartheta}_{\dot{2}}
	       J_{\breve{f}, (\alpha\dot{1}\dot{2})}^{\,ij}
	   \mbox{\Large $)$}	   		    \\
   && \hspace{4em}	
	+\, \theta^1\theta^2 \bar{\theta}^{\dot{1}}\bar{\theta}^{\dot{2}}	
	   \mbox{\Large $($}
	     J^{\,ij\,\sim}_{\breve{f}, (0)}
		 + \sum_{\alpha, \dot{\beta}} \vartheta_\alpha \bar{\vartheta}_{\dot{\beta}}
		       \bar{\sigma}^{\mu\dot{\beta}\alpha} J^{\,ij\,\sim}_{\breve{f}, [\mu]}
	     +\vartheta_1\vartheta_2 \bar{\vartheta}_{\dot{1}} \bar{\vartheta}_{\dot{2}}
	      J_{\breve{f}, (12\dot{1}\dot{2})}^{\,ij}
	   \mbox{\Large $)$} \,, 		
  \end{eqnarray*}
  }
 {\small
  \begin{eqnarray*}
    \breve{n}^i_{\breve{f}}\,  \breve{n}^j_{\breve{f}}\, \breve{n}^k_{\breve{f}}
     & = &  \sum_{\dot{\beta}} \theta^1\theta^2\bar{\theta}^{\dot{\beta}}
	         \vartheta_1\vartheta_2 \bar{\vartheta}_{\dot{\beta}}
		     J_{\breve{f}, (12\dot{\beta})}^{\,ijk}
	     + \sum_\alpha   \theta^\alpha \bar{\theta}^{\dot{1}} \bar{\theta}^{\dot{2}}
		      \vartheta_\alpha \bar{\vartheta}_{\dot{1}} \bar{\vartheta}_{\dot{2}}
	            J_{\breve{f}, (\alpha\dot{1}\dot{2})}^{\,ijk}   \\
    &&				
	     + \theta^1\theta^2 \bar{\theta}^{\dot{1}}\bar{\theta}^{\dot{2}}		
			 \mbox{\Large $($}
			   \sum_{\alpha, \dot{\beta}} \vartheta_\alpha \bar{\vartheta}_{\dot{\beta}}
			     \bar{\sigma}^{\mu\dot{\beta}\alpha} J^{\,ij\,\sim}_{\breve{f}, [\mu]}			
	          + \vartheta_1\vartheta_2 \bar{\vartheta}_{\dot{1}}\bar{\vartheta}_{\dot{2}}
			        J_{\breve{f}, (12\dot{1}\dot{2})}^{\,ijk}
			 \mbox{\Large $)$} \,, 		
  \end{eqnarray*}
  }
 {\small
  $$
    \breve{n}^i_{\breve{f}}\, \breve{n}^j_{\breve{f}}
	   \breve{n}^k_{\breve{f}}\, \breve{n}^l_{\breve{f}}\;
     =\; \theta^1\theta^2 \bar{\theta}^{\dot{1}}\bar{\theta}^{\dot{2}}
	        \vartheta_1\vartheta_2 \bar{\vartheta}_{\dot{1}}\bar{\vartheta}_{\dot{2}}\,
	         J_{\breve{f}, (12\dot{1}\dot{2})}^{\,ijk}\,,
  $$
  }
 
 \medskip
 
 \noindent
 which can be read off explicitly from the expansion
 
 \medskip
 
 {\footnotesize
 \begin{eqnarray*}
  \lefteqn{
     \breve{n}^i_{\breve{f}}\,\breve{n}^j_{\breve{f}}\;
	  =\;   -\,\theta^1\theta^2\vartheta_1\bar{\vartheta}_{\dot{1}} \bar{\vartheta}_{\dot{2}}\cdot
               \mbox{\Large $($}
			      f^\sharp_{(1)}(y^i) f^\sharp_{(2)}(y^j)
				 + f^\sharp_{(2)}(y^i) f^\sharp_{(1)}(y^j)
               \mbox{\Large $)$}			      } \\
  &&			
   -\,  \sum_{\alpha,\dot{\beta}}
           \theta^\alpha \bar{\theta}^{\dot{\beta}} \vartheta_\alpha \bar{\vartheta}_{\dot{\beta}}\cdot
               \mbox{\Large $($}
			     f^\sharp_{(\alpha)}(y^i) f^\sharp_{(\dot{\beta})}(y^j)
				 + f^\sharp_{(\dot{\beta})}(y^i) f^\sharp_{(\alpha)}(y^j)
               \mbox{\Large $)$}				
   -  \bar{\theta}^{\dot{1}}\bar{\theta}^{\dot{2}}
         \bar{\vartheta}_{\dot{1}}\bar{\vartheta}_{\dot{2}}
             \cdot
               \mbox{\Large $($}
			    f^\sharp_{(\dot{1})}(y^i) f^\sharp_{(\dot{2})}(y^j)
				+ f^\sharp_{(\dot{2})}(y^i) f^\sharp_{(\dot{1})}(y^j)
               \mbox{\Large $)$}	\\
  &&
   +\, \sum_{\dot{\beta}}
          \theta^1\theta^2\bar{\theta}^{\dot{\beta}}
		  \left\{\rule{0ex}{1.2em}\right.
		   \sum_{\alpha, \mu} \vartheta_\alpha\sigma^{\mu\alpha}_{\;\;\;\;\dot{\beta}} \cdot
		     \mbox{\Large $($}
		      f^\sharp_{(\alpha)}(y^i) f^\sharp_{[\mu]}(y^j)
			   + f^\sharp_{[\mu]}(y^i) f^\sharp_{(\alpha)}(y^j)
		     \mbox{\Large $)$}  \\
      && \hspace{7em}		
		 +\, \vartheta_1\vartheta_2 \bar{\vartheta}_{\dot{\beta}}\cdot		  		
            \mbox{\Large $($}
			  f^\sharp_{(12)}(y^i) f^\sharp_{(\dot{\beta})}(y^j)
		      + f^\sharp_{(1\dot{\beta})}(y^i) f^\sharp_{(2)}(y^j)
		      + f^\sharp_{(2\dot{\beta})}(y^i) f^\sharp_{(1)}(y^j) \\[-2ex]
	  && \hspace{15em}
			 +\, f^\sharp_{(1)}(y^i) f^\sharp_{(2\dot{\beta})}(y^j)
			  + f^\sharp_{(2)}(y^i) f^\sharp_{(1\dot{\beta})}(y^j)
			  + f^\sharp_{(\dot{\beta})}(y^i) f^\sharp_{(12)}(y^j)
            \mbox{\Large $)$}	
		 \left.\rule{0ex}{1.2em}\right\}\\
  &&
   +\, \sum_\alpha
          \theta^\alpha\bar{\theta}^{\bar{1}}\bar{\theta}^{\dot{2}}
		  \left\{\rule{0ex}{1.2em}\right.
		  -\, \sum_{\dot{\beta}, \mu}
		         \bar{\vartheta}_{\dot{\beta}}\sigma^{\mu\dot{\beta}}_\alpha  \cdot
		     \mbox{\Large $($}
		      f^\sharp_{(\dot{\beta})}(y^i) f^\sharp_{[\mu]}(y^j)
			   + f^\sharp_{[\mu]}(y^i) f^\sharp_{(\dot{\beta})}(y^j)
			 \mbox{\Large $)$}	
		  \\
      && \hspace{7em}		
		+\,  \vartheta_\alpha  \bar{\vartheta}_{\dot{1}} \bar{\vartheta}_{\dot{2}}   \cdot
            \mbox{\Large $($}
			  f^\sharp_{(\dot{1}\dot{2})}(y^i) f^\sharp_{(\alpha)}(y^j)
		      + f^\sharp_{(\alpha\dot{1})}(y^i) f^\sharp_{(\dot{2})}(y^j)
		      + f^\sharp_{(\alpha\dot{2})}(y^i) f^\sharp_{(\dot{1})}(y^j) \\[-2ex]
	  && \hspace{15em}
			 +\, f^\sharp_{(\dot{1})}(y^i) f^\sharp_{(\alpha\dot{2})}(y^j)
			  + f^\sharp_{(\dot{2})}(y^i) f^\sharp_{(\alpha\dot{1})}(y^j)
			  + f^\sharp_{(\alpha)}(y^i) f^\sharp_{(\dot{1}\dot{2})}(y^j)
            \mbox{\Large $)$}
		  \left.\rule{0ex}{1.2em}\right\}  \\			
  && +\, \theta^1\theta^2\bar{\theta}^{\dot{1}}\bar{\theta}^{\dot{2}}
       \left\{\rule{0ex}{1.2em}\right.
	    2\,\sum_{\mu, \nu} \eta^{\mu\nu} f^\sharp_{[\mu]}(y^i) f^\sharp_{[\nu]} (y^j) \\
     && \hspace{4em}
	   +\, \sum_{\alpha, \dot{\beta}, \mu}
		      \vartheta_\alpha \bar{\vartheta}_{\dot{\beta}} \bar{\sigma}^{\mu\dot{\beta}\alpha}\cdot
			   \mbox{\Large $($}
			     - f^\sharp_{(\alpha)}(y^i) f^{\sharp\,\prime\prime}_{[\mu]}(y^j)
				+ f^\sharp_{(\dot{\beta})}(y^i) f^{\sharp\,\prime}_{[\mu]}(y^j)
				 - f^\sharp_{(\alpha\dot{\beta})}(y^i) f^\sharp_{[\mu]}(y^j)  \\[-2.4ex]
     && \hspace{15em}
              -\, f^{\sharp\,\prime\prime}_{[\mu]}(y^i)  	 f^\sharp_{(\alpha)}(y^j)
			    + f^{\sharp\,\prime}_{[\mu]}(y^i) f^\sharp_{(\dot{\beta})}(y^j)
			     - f^\sharp_{[\mu]}(y^i)f^\sharp_{(\alpha\dot{\beta})}(y^j)
		       \mbox{\Large $)$} \\[2ex]
	 && \hspace{4em}
	   +\, \vartheta_1\vartheta_2\bar{\vartheta}_{\dot{1}} \bar{\vartheta}_{\dot{2}}
        \cdot
         \mbox{\Large $($}
			 - f^\sharp_{(12\dot{1})}(y^i) f^\sharp_{(\dot{2})}(y^j)
		      - f^\sharp_{(12\dot{2})}(y^i) f^\sharp_{(\dot{1})}(y^j)
		      - f^\sharp_{(1\dot{1}\dot{2})}(y^i) f^\sharp_{(\dot{2})}(y^j)
			  - f^\sharp_{(2\dot{1}\dot{2})}(y^i) f^\sharp_{(1)}(y^j) \\
         && \hspace{7em}			
			+\, f^\sharp_{(12)}(y^i) f^\sharp_{(\dot{1}\dot{2})}(y^j)
			  + f^\sharp_{(1\dot{1})}(y^i) f^\sharp_{(2\dot{2})}(y^j)
			  + f^\sharp_{(1\dot{2})}(y^i) f^\sharp_{(2\dot{1})}(y^j)
		      + f^\sharp_{(2\dot{1})}(y^i) f^\sharp_{(1\dot{2})}(y^j) \\
         && \hspace{7em}			
		      +\, f^\sharp_{(2\dot{2})}(y^i) f^\sharp_{(1\dot{1})}(y^j)
			  +f^\sharp_{(\dot{1}\dot{2})}(y^i) f^\sharp_{(12)}(y^j)
			  - f^\sharp_{(1)}(y^i) f^\sharp_{(2\dot{1}\dot{2})}(y^j)
			  - f^\sharp_{(2)}(y^i) f^\sharp_{(1\dot{1}\dot{2})}(y^j)\\
         && \hspace{7em}
			 -\, f^\sharp_{(\dot{1})}(y^i) f^\sharp_{(12\dot{2})}(y^j)
		      - f^\sharp_{(\dot{2})}(y^i) f^\sharp_{(12\dot{1})}(y^j)		
            \mbox{\Large $)$}
		    \left.\rule{0ex}{1.2em}\right\}\,,
 \end{eqnarray*}
 } 
 {\footnotesize
 \begin{eqnarray*}
  \lefteqn{
     \breve{n}^i_{\breve{f}} \, \breve{n}^j_{\breve{f}} \, \breve{n}^k_{\breve{f}}\;
	 =\;
      -\, \sum_{\dot{\beta}}
             \theta^1\theta^2\bar{\theta}^{\dot{\beta}}
			    \vartheta_1 \vartheta_2 \bar{\vartheta}_{\dot{\beta}}\cdot
            \left\{\rule{0ex}{1.2em}\right.\!
			  f^\sharp_{(\dot{\beta})}(y^i)\cdot
			      \mbox{\large $($}
				     f^\sharp_{(1)}(y^j) f^\sharp_{(2)}(y^k)
				     + f^\sharp_{(2)}(y^j) f^\sharp_{(1)}(y^k)
				  \mbox{\large $)$} 				  } \\[-2.4ex]
      && \hspace{6em}
	      +\, f^\sharp_{(\dot{\beta})}(y^j)\cdot
			      \mbox{\large $($}
				     f^\sharp_{(1)}(y^i) f^\sharp_{(2)}(y^k)
				     + f^\sharp_{(2)}(y^i) f^\sharp_{(1)}(y^k)
				  \mbox{\large $)$} 			
           +  f^\sharp_{(\dot{\beta})}(y^k)\cdot
			      \mbox{\large $($}
				     f^\sharp_{(1)}(y^i) f^\sharp_{(2)}(y^j)
				     + f^\sharp_{(2)}(y^i) f^\sharp_{(1)}(y^j)
				  \mbox{\large $)$} 	
             \!\!\left.\rule{0ex}{1.2em}\right\}   \\
  &&
   -\, \sum_\alpha
          \theta^\alpha\bar{\theta}^{\bar{1}}\bar{\theta}^{\dot{2}}
		      \vartheta_\alpha \bar{\vartheta}_{\dot{1}} \bar{\vartheta}_{\dot{2}}  \cdot			  	
          \left\{\rule{0ex}{1.2em}\right.\!
			  f^\sharp_{(\alpha)}(y^i)\cdot
			      \mbox{\large $($}
				     f^\sharp_{(\dot{1})}(y^j) f^\sharp_{(\dot{2})}(y^k)
				     + f^\sharp_{(\dot{2})}(y^j) f^\sharp_{(\dot{1})}(y^k)
				  \mbox{\large $)$} 				   \\[-2.4ex]
      && \hspace{6em}
	      +\, f^\sharp_{(\alpha)}(y^j)\cdot
			      \mbox{\large $($}
				     f^\sharp_{(\dot{1})}(y^i) f^\sharp_{(\dot{2})}(y^k)
				     + f^\sharp_{(\dot{2})}(y^i) f^\sharp_{(\dot{1})}(y^k)
				  \mbox{\large $)$} 			
           +  f^\sharp_{(\alpha)}(y^k)\cdot
			      \mbox{\large $($}
				     f^\sharp_{(\dot{1})}(y^i) f^\sharp_{(\dot{2})}(y^j)
				     + f^\sharp_{(\dot{2})}(y^i) f^\sharp_{(\dot{1})}(y^j)
				  \mbox{\large $)$} 	
             \!\!\left.\rule{0ex}{1.2em}\right\}   \\	   					
  && +\, \theta^1\theta^2\bar{\theta}^{\dot{1}}\bar{\theta}^{\dot{2}} \cdot
     \left\{\rule{0ex}{1.2em}\right.\!
	   \sum_{\alpha, \dot{\beta}, \mu}
	     \vartheta_\alpha \bar{\vartheta}_{\dot{\beta}} \bar{\sigma}^{\mu\dot{\beta}\alpha}
		  \mbox{\Large $($}
		   f^\sharp_{[\mu]}(y^i)\cdot
			      \mbox{\large $($}
				     f^\sharp_{(\alpha)}(y^j) f^\sharp_{(\dot{\beta})}(y^k)
				     + f^\sharp_{(\dot{\beta})}(y^j) f^\sharp_{(\alpha)}(y^k)
				  \mbox{\large $)$} 		   \\[-2.4ex]
      && \hspace{15em}
           +  f^\sharp_{[\mu]}(y^j)\cdot
			      \mbox{\large $($}
				     f^\sharp_{(\alpha)}(y^i) f^\sharp_{(\dot{\beta})}(y^k)
				     + f^\sharp_{(\dot{\beta})}(y^i) f^\sharp_{(\alpha)}(y^k)
				  \mbox{\large $)$} 					  	   \\[-1ex]
      && \hspace{15em}
		   +  f^\sharp_{[\mu]}(y^k)\cdot
			      \mbox{\large $($}
				     f^\sharp_{(\alpha)}(y^i) f^\sharp_{(\dot{\beta})}(y^j)
				     + f^\sharp_{(\dot{\beta})}(y^i) f^\sharp_{(\alpha)}(y^j)
				  \mbox{\large $)$} 	
		  \mbox{\Large $)$}    \\
     && \hspace{7em}
       -\, \vartheta_1\vartheta_2 \bar{\vartheta}_{\dot{1}} \bar{\vartheta}_{\dot{2}}
	       \mbox{\Large $($}
             f^\sharp_{(12)}(y^i)\cdot
			      \mbox{\large $($}
				     f^\sharp_{(\dot{1})}(y^j) f^\sharp_{(\dot{2})}(y^k)
				     + f^\sharp_{(\dot{2})}(y^j) f^\sharp_{(\dot{1})}(y^k)
				  \mbox{\large $)$} 		   \\[-.4ex]
        && \hspace{16em}
           +\, f^\sharp_{(1\dot{1})}(y^i)\cdot
			      \mbox{\large $($}
				     f^\sharp_{(2)}(y^j) f^\sharp_{(\dot{2})}(y^k)
				     + f^\sharp_{(\dot{2})}(y^j) f^\sharp_{(2)}(y^k)
				  \mbox{\large $)$} 					  \\[-.4ex]
        && \hspace{16em}
		    +\, f^\sharp_{(1\dot{2})}(y^i)\cdot
			      \mbox{\large $($}
				     f^\sharp_{(2)}(y^j) f^\sharp_{(\dot{1})}(y^k)
				     + f^\sharp_{(\dot{1})}(y^j) f^\sharp_{(2)}(y^k)
				  \mbox{\large $)$} 	\\[-.4ex]
        && \hspace{16em}
           +\, f^\sharp_{(2\dot{1})}(y^i)\cdot
			      \mbox{\large $($}
				     f^\sharp_{(1)}(y^j) f^\sharp_{(\dot{2})}(y^k)
				     + f^\sharp_{(\dot{2})}(y^j) f^\sharp_{(1)}(y^k)
				  \mbox{\large $)$} 					  \\[-.4ex] 				
	    && \hspace{16em}
           +\, f^\sharp_{(2\dot{2})}(y^i)\cdot
			      \mbox{\large $($}
				     f^\sharp_{(1)}(y^j) f^\sharp_{(\dot{1})}(y^k)
				     + f^\sharp_{(\dot{1})}(y^j) f^\sharp_{(1)}(y^k)
				  \mbox{\large $)$} 					  \\[-.4ex]
		&& \hspace{16em}
           +\, f^\sharp_{(\dot{1}\dot{2})}(y^i)\cdot
			      \mbox{\large $($}
				     f^\sharp_{(1)}(y^j) f^\sharp_{(2)}(y^k)
				     + f^\sharp_{(2)}(y^j) f^\sharp_{(1)}(y^k)
				  \mbox{\large $)$} 					  \\
    && \hspace{13em}
           +\, f^\sharp_{(12)}(y^j)\cdot
			      \mbox{\large $($}
				     f^\sharp_{(\dot{1})}(y^i) f^\sharp_{(\dot{2})}(y^k)
				     + f^\sharp_{(\dot{2})}(y^i) f^\sharp_{(\dot{1})}(y^k)
				  \mbox{\large $)$} 		   \\[-.4ex]
        && \hspace{16em}
           +\, f^\sharp_{(1\dot{1})}(y^j)\cdot
			      \mbox{\large $($}
				     f^\sharp_{(2)}(y^i) f^\sharp_{(\dot{2})}(y^k)
				     + f^\sharp_{(\dot{2})}(y^i) f^\sharp_{(2)}(y^k)
				  \mbox{\large $)$} 					  \\[-.4ex]
        && \hspace{16em}
		    +\, f^\sharp_{(1\dot{2})}(y^j)\cdot
			      \mbox{\large $($}
				     f^\sharp_{(2)}(y^i) f^\sharp_{(\dot{1})}(y^k)
				     + f^\sharp_{(\dot{1})}(y^i) f^\sharp_{(2)}(y^k)
				  \mbox{\large $)$} 	\\[-.4ex]
        && \hspace{16em}
           +\, f^\sharp_{(2\dot{1})}(y^j)\cdot
			      \mbox{\large $($}
				     f^\sharp_{(1)}(y^i) f^\sharp_{(\dot{2})}(y^k)
				     + f^\sharp_{(\dot{2})}(y^i) f^\sharp_{(1)}(y^k)
				  \mbox{\large $)$} 					  \\[-.4ex] 				
	    && \hspace{16em}
           +\, f^\sharp_{(2\dot{2})}(y^j)\cdot
			      \mbox{\large $($}
				     f^\sharp_{(1)}(y^i) f^\sharp_{(\dot{1})}(y^k)
				     + f^\sharp_{(\dot{1})}(y^i) f^\sharp_{(1)}(y^k)
				  \mbox{\large $)$} 					  \\
		&& \hspace{16em}
           +\, f^\sharp_{(\dot{1}\dot{2})}(y^j)\cdot
			      \mbox{\large $($}
				     f^\sharp_{(1)}(y^i) f^\sharp_{(2)}(y^k)
				     + f^\sharp_{(2)}(y^i) f^\sharp_{(1)}(y^k)
				  \mbox{\large $)$} 					  \\
    && \hspace{13em}
           +\,f^\sharp_{(12)}(y^k)\cdot
			      \mbox{\large $($}
				     f^\sharp_{(\dot{1})}(y^i) f^\sharp_{(\dot{2})}(y^j)
				     + f^\sharp_{(\dot{2})}(y^i) f^\sharp_{(\dot{1})}(y^j)
				  \mbox{\large $)$} 		   \\[-.4ex]
        && \hspace{16em}
           +\, f^\sharp_{(1\dot{1})}(y^k)\cdot
			      \mbox{\large $($}
				     f^\sharp_{(2)}(y^i) f^\sharp_{(\dot{2})}(y^j)
				     + f^\sharp_{(\dot{2})}(y^i) f^\sharp_{(2)}(y^j)
				  \mbox{\large $)$} 					  \\[-.4ex]
        && \hspace{16em}
		    +\, f^\sharp_{(1\dot{2})}(y^k)\cdot
			      \mbox{\large $($}
				     f^\sharp_{(2)}(y^i) f^\sharp_{(\dot{1})}(y^j)
				     + f^\sharp_{(\dot{1})}(y^i) f^\sharp_{(2)}(y^j)
				  \mbox{\large $)$} 	\\[-.4ex]
        && \hspace{16em}
           +\, f^\sharp_{(2\dot{1})}(y^k)\cdot
			      \mbox{\large $($}
				     f^\sharp_{(1)}(y^i) f^\sharp_{(\dot{2})}(y^j)
				     + f^\sharp_{(\dot{2})}(y^i) f^\sharp_{(1)}(y^j)
				  \mbox{\large $)$} 					  \\[-.4ex] 				
	    && \hspace{16em}
           +\, f^\sharp_{(2\dot{2})}(y^k)\cdot
			      \mbox{\large $($}
				     f^\sharp_{(1)}(y^i) f^\sharp_{(\dot{1})}(y^j)
				     + f^\sharp_{(\dot{1})}(y^i) f^\sharp_{(1)}(y^j)
				  \mbox{\large $)$} 					  \\[-.4ex]
		&& \hspace{16em}
           +\, f^\sharp_{(\dot{1}\dot{2})}(y^k)\cdot
			      \mbox{\large $($}
				     f^\sharp_{(1)}(y^i) f^\sharp_{(2)}(y^j)
				     + f^\sharp_{(2)}(y^i) f^\sharp_{(1)}(y^j)
				  \mbox{\large $)$} 					  		
           \;\mbox{\Large $)$}		
     \!\!\left.\rule{0ex}{1.2em}\right\}\,,			
 \end{eqnarray*}
 } 
 {\footnotesize
 \begin{eqnarray*}
  \lefteqn{
   \breve{n}^i_{\breve{f}} \, \breve{n}^j_{\breve{f}}\,
      \breve{n}^k_{\breve{f}}\,  \breve{n}^l_{\breve{f}}\;
     =\; \theta^1\theta^2\bar{\theta}^{\dot{1}}\bar{\theta}^{\dot{2}}
	        \vartheta_1\vartheta_2\bar{\vartheta}_{\dot{1}} \bar{\vartheta}_{\dot{2}}
	     \cdot
		 \left\{\rule{0ex}{1.2em}\right.\!
            \mbox{\large $($}
				f^\sharp_{(1)}(y^i) f^\sharp_{(2)}(y^j)
				+ f^\sharp_{(2)}(y^i) f^\sharp_{(1)}(y^j)
	        \mbox{\large $)$}
            \mbox{\large $($}
				f^\sharp_{(\dot{1})}(y^k) f^\sharp_{(\dot{2})}(y^l)
				+ f^\sharp_{(\dot{2})}(y^k) f^\sharp_{(\dot{1})}(y^l)
	        \mbox{\large $)$} }\\
    && \hspace{8em}			
      +\, \mbox{\large $($}
				f^\sharp_{(\dot{1})}(y^i) f^\sharp_{(\dot{2})}(y^j)
				+ f^\sharp_{(\dot{2})}(y^i) f^\sharp_{(\dot{1})}(y^j)
	        \mbox{\large $)$}
            \mbox{\large $($}
				f^\sharp_{(1)}(y^k) f^\sharp_{(2)}(y^l)
				+ f^\sharp_{(2)}(y^k) f^\sharp_{(1)}(y^l)
	        \mbox{\large $)$} 			\\			
    && \hspace{8em}			
      +\, \sum_{\alpha\ne \gamma, \dot{\beta}\ne \dot{\delta}}
	        \mbox{\large $($}
				f^\sharp_{(\alpha)}(y^i) f^\sharp_{(\dot{\beta})}(y^j)
				+ f^\sharp_{(\dot{\beta})}(y^i) f^\sharp_{(\alpha)}(y^j)
	        \mbox{\large $)$}
            \mbox{\large $($}
				f^\sharp_{(\gamma)}(y^k) f^\sharp_{(\dot{\delta})}(y^l)
				+ f^\sharp_{(\dot{\delta})}(y^k) f^\sharp_{(\gamma)}(y^l)
	        \mbox{\large $)$} 			
	    \!\!\left.\rule{0ex}{1.2em}\right\}\!.
 \end{eqnarray*}
 } 
\end{proposition}

\smallskip

\begin{proof}
 By construction,
   $f^\sharp_{(0)}=f^\sharp$.
 Now let $h_1, h_2\in C^\infty(Y)$.
 Since $\breve{f}^\sharp$ is a ring-homomorphism, one has
  $$
   \breve{f}^\sharp(h_1h_2)\;=\; \breve{f}^\sharp(h_1) \breve{f}^\sharp(h_2).
  $$
 The expansion of the above identity in terms of $(\theta, \bar{\theta}, \vartheta, \bar{\vartheta})$-components
   gives
 {\footnotesize
  $$
   \begin{array}{rcl}
    f^\sharp_{(\alpha)}(h_1h_2)
	  & = &   f^\sharp_{(\alpha)}(h_1)\, f^\sharp_{(0)}(h_2)
	          + f^\sharp_{(0)}(h_1)\, f^\sharp_{(\alpha)}(h_2)\,, \\[1.6ex]
    f^\sharp_{(\dot{\beta})}(h_1h_2)\;
	  & = &  f^\sharp_{(\dot{\beta})}(h_1)\, f^\sharp_{(0)}(h_2)
	          + f^\sharp_{(0)}(h_1)\, f^\sharp_{(\dot{\beta})}(h_2)\,; \\[2ex]
    f^\sharp_{(12)}(h_1h_2)
	  & = & f^\sharp_{(12)}(h_1)\, f^\sharp_{(0)}(h_2)
	          - f^\sharp_{(1)}(h_1)\, f^\sharp_{(2)}(h_2)
			  - f^\sharp_{(2)}(h_1)\, f^\sharp_{(1)}(h_2)
	          + f^\sharp_{(0)}(h_1)\, f^\sharp_{(12)}(h_2)\,, \\[1.6ex] 	      	
    f^\sharp_{[\mu]}(h_1h_2)			
	  & = & f^\sharp_{[\mu]}(h_1)\, h^\sharp_{(0)}(h_2)
	             + f^\sharp_{(0)}(h_1)\, f^\sharp_{[\mu]}(h_2) \,, \\[1.6ex]
    f^\sharp_{(\alpha\dot{\beta})}(h_1h_2)
	   & = &  f^\sharp_{(\alpha\dot{\beta})}(h_1)\, f^\sharp_{(0)}(h_2)
	          - f^\sharp_{(\alpha)}(h_1)\, f^\sharp_{(\dot{\beta})}(h_2)
			  - f^\sharp_{(\dot{\beta})}(h_1)\, f^\sharp_{(\alpha)}(h_2)
	          + f^\sharp_{(0)}(h_1)\, f^\sharp_{(\alpha\dot{\beta})}(h_2)\,, \\[1.6ex]
    f^\sharp_{(\dot{1}\dot{2})}(h_1h_2)
	  & = &  f^\sharp_{(\dot{1}\dot{2})}(h_1)\, f^\sharp_{(0)}(h_2)
	          - f^\sharp_{(\dot{1})}(h_1)\, f^\sharp_{(\dot{2})}(h_2)
			  - f^\sharp_{(\dot{2})}(h_1)\, f^\sharp_{(\dot{1})}(h_2)
	          + f^\sharp_{(0)}(h_1)\, f^\sharp_{(\dot{1}\dot{2})}(h_2)\,; \\[2ex] 	
    f^{\sharp\,\prime}_{[\mu]} (h_1h_2)			
	  & = & f^{\sharp\,\prime}_{[\mu]}(h_1)\, f^\sharp_{(0)}(h_2)
	          + \mbox{\large $\frac{1}{2}$}
			      \sum_{\alpha, \dot{\beta}, \nu}
				    \sigma_{\mu\alpha}^{\;\;\;\;\dot{\beta}}\sigma^{\nu\alpha}_{\;\;\;\;\dot{\beta}}
				   \mbox{\large $($}
			         f^\sharp_{[\nu]}(h_1)\, f^\sharp_{(\alpha)}(h_2)
			         + f^\sharp_{(\alpha)}(h_1)\, f^\sharp_{[\nu]}(h_2)
			       \!\mbox{\large $)$}
	          + f^\sharp_{(0)}(h_1)\, f^{\sharp\,\prime}_{[\mu]}(h_2)\,, \\[1.6ex]			
    f^\sharp_{(12\dot{\beta})}(h_1h_2)
	  & = & f^\sharp_{(12\dot{\beta})}(h_1)\, f^\sharp_{(0)}(h_2)
	          + f^\sharp_{(12)}(h_1)\, f^\sharp_{(\dot{\beta})}(h_2)
			  + f^\sharp_{(1\dot{\beta})}(h_1)\, f^\sharp_{(2)}(h_2)
	          + f^\sharp_{(2\dot{\beta})}(h_1)\, f^\sharp_{(1)}(h_2) \\[1.2ex] 	
         && \hspace{2em}
	      +\,  f^\sharp_{(1)}(h_1)\, f^\sharp_{(2\dot{\beta})}(h_2)
	          + f^\sharp_{(2)}(h_1)\, f^\sharp_{(1\dot{\beta})}(h_2)
			  + f^\sharp_{(\dot{\beta})}(h_1)\, f^\sharp_{(12)}(h_2)
	          + f^\sharp_{(0)}(h_1)\, f^\sharp_{(12\dot{\beta})}(h_2)\,, \\[1.6ex]
  f^{\sharp\,\prime\prime}_{[\mu]}(h_1h_2)
    & = & f^{\sharp\,\prime\prime}_{[\mu]}(h_1)\, f^\sharp_{(0)}(h_2)	
	          -\,\mbox{\large $\frac{1}{2}$}
			      \sum_{\alpha, \dot{\beta}, \nu}
				   \sigma_{\mu\dot{\beta}}^\alpha \sigma^{\nu\dot{\beta}}_\alpha
   				    \mbox{\large $($}
	                  f^\sharp_{[\nu]}(h_1)\, f^\sharp_{(\dot{\beta})}(h_2)
			          + f^\sharp_{(\dot{\beta})}(h_1)\, f^\sharp_{[\nu]}(h_2)
			        \!\mbox{\large $)$}			  			
	          + f^\sharp_{(0)}(h_1)\, f^{\sharp\,\prime\prime}_{[\mu]}(h_2)\,, \\[1.6ex]
  f^\sharp_{(\alpha\dot{1}\dot{2})}(h_1h_2)
	 &  &   f^\sharp_{(\alpha\dot{1}\dot{2})}(h_1)\, f^\sharp_{(0)}(h_2)
	          + f^\sharp_{(\dot{1}\dot{2})}(h_1)\, f^\sharp_{(\alpha)}(h_2)
			  + f^\sharp_{(\alpha\dot{1})}(h_1)\, f^\sharp_{(\dot{2})}(h_2)
	          + f^\sharp_{(\alpha\dot{2})}(h_1)\, f^\sharp_{(\dot{1})}(h_2) \\[1.2ex] 	
         && \hspace{2em}
	        +\, f^\sharp_{(\dot{1})}(h_1)\, f^\sharp_{(\alpha\dot{2})}(h_2)
	          + f^\sharp_{(\dot{2})}(h_1)\, f^\sharp_{(\alpha\dot{1})}(h_2)
			  + f^\sharp_{(\alpha)}(h_1)\, f^\sharp_{(\dot{1}\dot{2})}(h_2)
	          + f^\sharp_{(0)}(h_1)\, f^\sharp_{(\alpha\dot{1}\dot{2})}(h_2)\,; \\[2ex]
   f^{\sharp\,\sim}_{(0)}(h_1h_2)
     & = &  f^{\sharp\,\sim}_{(0)}(h_1)\, f^\sharp_{(0)}(h_2)	
	             + 2\,\sum_{\mu, \nu} \eta^{\mu\nu}f^\sharp_{[\mu]}(h_1)\, f^\sharp_{[\nu]}(h_2)						  
	             + f^\sharp_{(0)}(h_1)\, f^{\sharp\,\sim}_{(0)}(h_2) \,,  \\[1.6ex]
   f^{\sharp\,\sim}_{[\mu]}(h_1h_2)
     & = & f^{\sharp\,\sim}_{[\mu]}(h_1)\, f^\sharp_{(0)}(h_2)	
	           + f^\sharp_{(0)}(h_1)\, f^{\sharp\,\sim}_{[\mu]}(h_2) \\[1.2ex]
	    && \hspace{-2em}
				-\,\mbox{\large $\frac{1}{2}$}
				   \sum_{\alpha, \dot{\beta}, \nu}
				     \sigma_{\mu\alpha\dot{\beta}} \bar{\sigma}^{\nu\dot{\beta}\alpha}
				\mbox{\large $($}
	             f^{\sharp\,\prime}_{[\nu]}(h_1)\, f^\sharp_{(\dot{\beta})}(h_2)
		        -  f^{\sharp\,\prime\prime}_{[\nu]}(h_1)\, f^\sharp_{(\alpha)}(h_2)
	            - f^\sharp_{(\alpha)}(h_1)\, f^{\sharp\,\prime\prime}_{[\nu]}(h_2)
	            + f^\sharp_{(\dot{\beta})}(h_1)\, f^{\sharp\,\prime}_{[\nu]}(h_2)
			    \!\mbox{\large $)$}\,,     \\[1.6ex]    	
   f^\sharp_{(12\dot{1}\dot{2})}(h_1h_2)
	 & = &  f^\sharp_{(12\dot{1}\dot{2})}(h_1)\, f^\sharp_{(0)}(h_2)
	          - f^\sharp_{(12\dot{1})}(h_1)\, f^\sharp_{(\dot{2})}(h_2)
			  - f^\sharp_{(12\dot{2})}(h_1)\, f^\sharp_{(\dot{1})}(h_2)
	          - f^\sharp_{(1\dot{1}\dot{2})}(h_1)\, f^\sharp_{(2)}(h_2) \\[1.2ex] 	
       && \hspace{2em}
	      -\,  f^\sharp_{(2\dot{1}\dot{2})}(h_1)\, f^\sharp_{(1)}(h_2)
	          + f^\sharp_{(12)}(h_1)\, f^\sharp_{(\dot{1}\dot{2})}(h_2)
			  + f^\sharp_{(1\dot{1})}(h_1)\, f^\sharp_{(2\dot{2})}(h_2)
	          + f^\sharp_{(1\dot{2})}(h_1)\, f^\sharp_{(2\dot{1})}(h_2) \\[1.2ex]
       && \hspace{2em}
	      +\, f^\sharp_{(2\dot{1})}(h_1)\, f^\sharp_{(1\dot{2})}(h_2)
	          + f^\sharp_{(2\dot{2})}(h_1)\, f^\sharp_{(1\dot{1})}(h_2)
			  + f^\sharp_{(\dot{1}\dot{2})}(h_1)\, f^\sharp_{(12)}(h_2)
	          - f^\sharp_{(1)}(h_1)\, f^\sharp_{(2\dot{1}\dot{2})}(h_2) \\[1.2ex] 	
       && \hspace{2em}
	        -\, f^\sharp_{(2)}(h_1)\, f^\sharp_{(1\dot{1}\dot{2})}(h_2)
	          - f^\sharp_{(\dot{1})}(h_1)\, f^\sharp_{(12\dot{2})}(h_2)
			  - f^\sharp_{(\dot{2})}(h_1)\, f^\sharp_{(12\dot{1})}(h_2)
	          + f^\sharp_{(0)}(h_1)\, f^\sharp_{(12\dot{1}\dot{2})}(h_2)\,.
   \end{array}
  $$}
 This shows that
  \begin{itemize}
   \item[\LARGE $\cdot$]
    $f^\sharp_{(\alpha)}$, $f^\sharp_{(\dot{\beta})}$, and $f^\sharp_{[\mu]}$
    satisfy the Leibniz rule and, hence, are $C^\infty(X)^{\Bbb C}$-valued derivations\\ on $C^\infty(Y)$;
  \end{itemize}
  other $f^{\sharp\,\tinybullet}_{\tinybullet}$ satisfy higher-order Leibniz rules and, hence,
     are $C^\infty(X)^{\Bbb C}$-valued differential operators on $C^\infty(Y)$:\;
   Inductively and recursively,
  \begin{itemize}
   \item[\LARGE $\cdot$]
     $f^\sharp_{(12)}$, $f^\sharp_{(\alpha\dot{\beta})}$, $f^\sharp_{(\dot{1}\dot{2})}$,
	 $f^{\sharp\,\prime}_{[\mu]}$, $f^{\sharp\,\prime\prime}_{[\mu]}$,   and
	 $f^{\sharp\,\sim}_{(0)}$ are
	 $C^\infty(X)^{\Bbb C}$-valued second-order differential operators on $C^\infty(Y)$;
	
   \item[\LARGE $\cdot$]	
     $f^\sharp_{(12\dot{\beta})}$, $f^\sharp_{(\alpha\dot{1}\dot{2})}$, and
	    $f^{\sharp\,\sim}_{[\mu]}$
	 are $C^\infty(X)^{\Bbb C}$-valued third-order differential operators on $C^\infty(Y)$;
	
   \item[\LARGE $\cdot$]	
    $f^\sharp_{(12\dot{1}\dot{2})}$
	  is a $C^\infty(X)^{\Bbb C}$-valued fourth-order differential operator on $C^\infty(Y)$.
  \end{itemize}
  
 Locally and explicitly, under the setting of the statement of the proposition,
  it follows from the $C^\infty$-hull structure of $C^\infty(\widehat{U}^{\widehat{\boxplus}, \smallscriptsize})$ that
 \begin{eqnarray*}
  \lefteqn{
   \breve{f}^\sharp (h)\;
    =\; h( \breve{f}^\sharp(y^1), \,\cdots\,,\, \breve{f}^\sharp(y^n))\;
	=\; h(f_{(0)}^\sharp(y_1)+ \breve{n}^1_{\breve{f}}, \,
	                 \cdots\,,\, f_{(0)}(y^n)+\breve{n}^n_{\breve{f}})	}\\
  &&
   =\; h(f_{(0)}^\sharp(y_1), \,\cdots\,,\, f_{(0)}(y^n))
         + \sum_i  (\partial_i h)(f_{(0)}^\sharp(y_1), \,\cdots\,,\, f_{(0)}(y^n))
        		   \cdot \breve{n}^i_{\breve{f}}    \\[-1ex]
  && \hspace{2em}				
        +\, \mbox{\Large $\frac{1}{2}$}
		      \sum_{i,j}  (\partial_i\partial_j h)(f_{(0)}^\sharp(y_1), \,\cdots\,,\, f_{(0)}(y^n))
        		   \cdot \breve{n}^i_{\breve{f}} \, \breve{n}^j_{\breve{f}}         	\\
  && \hspace{2em}				
        +\, \mbox{\Large $\frac{1}{3!}$}
		      \sum_{i,j,k}  (\partial_i\partial_j \partial_kh)
			                                        (f_{(0)}^\sharp(y_1), \,\cdots\,,\, f_{(0)}(y^n))
        		   \cdot \breve{n}^i_{\breve{f}}\,\breve{n}^j_{\breve{f}}\, \breve{n}^k_{\breve{f}} \\
  && \hspace{2em}				
        +\, \mbox{\Large $\frac{1}{4!}$}
		      \sum_{i,j,k,l}  (\partial_i\partial_j\partial_k\partial_l h)
			                      (f_{(0)}^\sharp(y_1), \,\cdots\,,\, f_{(0)}(y^n))
        		   \cdot \breve{n}^i_{\breve{f}} \, \breve{n}^j_{\breve{f}} \,
				              \breve{n}^k_{\breve{f}} \, \breve{n}^l_{\breve{f}}\,.  	
 \end{eqnarray*}
 After plugging in the $(\theta,\bar{\theta}, \vartheta, \bar{\vartheta})$-expansion
    of $\breve{n}^i_{\breve{f}}$'s and then simplifying the resulting expression,
 one obtains the explicit expression of $f^{\sharp\,\tinybullet}_{\tinybullet}$
  as differential operators of order as stated in the proposition and without the degree-zero term.
 This proves the proposition.
   
\end{proof}

\bigskip

\subsection{Chiral maps from $\widehat{X}^{\widehat{\boxplus}, \smallscriptsize}$ to a complex manifold $Y$}
	
In this subsection we study smooth maps from the small towered superspace
  $\widehat{X}^{\widehat{\boxplus}, \smallscriptsize}$ to a complex manifold $Y$ and
  introduce the notion of `chiral maps'.

\bigskip

\begin{flushleft}
{\bf Smooth functions on a complex manifold}
\end{flushleft}
Let $Y$ be a complex manifold of complex dimension $n$.
As a smooth real $2n$-manifold, its function-ring $C^\infty(Y)$ is a $C^\infty$-ring.
In this digression, we review how this structure is rephrased in terms of complex coordinates on $Y$;
cf.\: [L-Y4: Sec.\,4] (D(14.1)).

\bigskip

\begin{definition} {\bf [smooth function in complex coordinates]}\; {\rm
 For a local coordinate chart $V\subset Y$,
   the $C^\infty$-ring structure of $C^\infty(V)$ in terms of the complex coordinate functions
   $$
      (z^1,\,\cdots\,,\, z^n)=(y^1+\sqrt{-1}y^2, \,\cdots\,,\, y^{2n-1}+\sqrt{-1}y^{2n})\,.
  $$
 We will call $(y^1,\,y^2, \,\dots\,, y^{2n-1},\, y^{2n})$
   the {\it real coordinates functions} on $V$
      {\it associated to the complex coordinate functions} $(z^1,\,\cdots\,, z^n)$.
 Let $h=h(y^1,\,\cdots\,,\, y^{2n})\in C^\infty(V)$.
 Then, define
  $$
    h_{\Bbb C}\; :=\;  h_{\Bbb C}(z^1,\,\cdots\,,\, z^n\,,\,\bar{z}^1,\,\cdots\,,\,\bar{z}^n)\;
	 :=\;  h(y^1,\,\cdots\,,\, y^{2n})
  $$
  and call it {\it the presentation of $h\in C^\infty(V)$
   in complex coordinate functions $(z^1,\,\cdots\,,\, z^n)$ on $V$}.
 Define
   {\footnotesize
   $$
     \frac{\partial}{\partial z^i}\;
	    :=\;  \mbox{\Large $\frac{1}{2}$}
		          \mbox{\Large $($}
				     \frac{\partial}{\partial y^{2i-1}}- \sqrt{-1}\frac{\partial}{\partial z^{2i}}
		        \mbox{\Large $)$}\,,\hspace{2em}
     \frac{\partial}{\partial \bar{z}^i}\;
	    :=\;  \mbox{\Large $\frac{1}{2}$}
		          \mbox{\Large $($}
				     \frac{\partial}{\partial y^{2i-1}}+ \sqrt{-1}\frac{\partial}{\partial z^{2i}}
		        \mbox{\Large $)$}
   $$}and 
 $$
   \mbox{\Large $\frac{\partial}{\partial z^i}$}\,h_{\Bbb C}\;
     :=\; 	
	     \mbox{\Large
	       $\frac{1}{2}(
			\frac{\partial}{\partial y^{2i-1}}
			   \mbox{\normalsize $- \sqrt{-1}$}\frac{\partial}{\partial z^{2i}})$  }\! h\,,\;\;
   \mbox{\Large $\frac{\partial}{\partial \bar{z}^i}$}\,h_{\Bbb C}\;
     :=\; 	
	     \mbox{\Large
	       $\frac{1}{2}(
			\frac{\partial}{\partial y^{2i-1}}
			   \mbox{\normalsize $+ \sqrt{-1}$}\frac{\partial}{\partial z^{2i}})$  }\! h\;\;
     \in\; C^\infty(V)^{\Bbb C}\,. 			
 $$
}\end{definition}
 
\medskip

\begin{lemma} {\bf [Taylor's formula in complex coordinates]}\;
 Denote coordinate functions  on $V$ collectively by
  $\boldy:= (y^1,\,\cdots\,,\, y^{2n})$,
  $\boldz:= (z^1,\,\cdots\,, z^n) $, and $\bar{\boldz}:= (\bar{z}^1,\,\cdots\,,\, \bar{z}^n)$.
 Then,
   for
    $h\in C^\infty(V)$ and
    $q\in V$ of coordinates $\boldy$,  and
	$\bolda := (a^1,\,\cdots\,,\, a^{2n})\in {\Bbb R}^{2n}$ such that
        points $q_t$ of real coordinates $\boldy_t:=\boldy+t\cdot \bolda$ are contained in $V$ for all $t\in [0,1]$,
   the Taylor's formula
   $$
     h(\boldy+\bolda)\;
	   =\;  \sum_{d=0}^l
              \sum_{|\scriptsizeboldd|=d}	
			  \frac{1}{\boldd !}\,
	          \frac{\partial^{\,d} h}{\partial \boldy^{\scriptsizeboldd}}(\boldy)\,
			  \bolda^{\scriptsizeboldd}\;
			 +\;  \sum_{|\scriptsizeboldd|=l+1}	
			  \frac{1}{\boldd !}\,
	          \frac{\partial^{\,l+1} h}{\partial \boldy^{\scriptsizeboldd}}(\boldy_{t_0})\,
			  \bolda^{\scriptsizeboldd}\;
   $$
  for some $t_0\in [0,1]$  depending on $\bolda$
  has the following form in complex coordinates
   \begin{eqnarray*}
    \lefteqn{
      h_{\Bbb C}(\boldz+\boldu, \bar{\boldz}+\bar{\boldu})\;
	   =\;  \sum_{d=0}^l\,
              \sum_{|\scriptsizeboldd_1|+|\scriptsizeboldd_2|=d}	
			  \frac{1}{\boldd_1 ! \boldd_2 !}\,
	          \frac{\partial^{\,d} h_{\Bbb C}}{\partial \boldz^{\scriptsizeboldd_1}
			                                          \partial \bar{\boldz}^{\scriptsizeboldd_2}}(\boldz,\bar{\boldz})\,
			  \boldu^{\scriptsizeboldd_1}\bar{\boldu}^{\scriptsizeboldd_2}\;   }\\
     &&\hspace{7em}			
			 +\;  \sum_{|\scriptsizeboldd|=l+1}	
			  \frac{1}{\boldd_1 !\boldd_2 !}\,
	          \frac{\partial^{\,l+1} h_{\Bbb C}}
			      {\partial \boldz^{\scriptsizeboldd_1}\partial \bar{\boldz}^{\scriptsizeboldd_2}}
				  (\boldz_{t_0}, \bar{\boldz}_{t_0})\,
			  \boldu^{\scriptsizeboldd_1}\bar{\boldu}^{\scriptsizeboldd_2}
   \end{eqnarray*}
   for some $t_0\in [0,1]$  depending on $\boldu$.
  Here,
    $\boldd:= (d_1,\,\cdots\,,\, d_{2n})$ with $d_i\in {\Bbb Z}_{\ge 0}$,\\
	$|\boldd|:= d_1+\,\cdots\,+d_{2n}$,
	$\boldd !:= d_1!\,\cdots\,d_{2n}!$ with $0!:=1$,
	$\partial^{\,d}/{\partial\boldy}^{\scriptsizeboldd}
	   :=  (\partial/\partial y^1)^{d_1}
	           \cdots  (\partial/\partial y^{2n})^{d_{2n}} $ for $|\boldd|=d$,
    $\bolda^{\scriptsizeboldd}:= (a^1)^{d_1}\,\cdots\,(a^{2n})^{d_{2n}}$			   			
  and similarly
    $\boldu:= (u^1,\,\cdots\,,\, u^n)\in {\Bbb C}^n$	
	such that
        points $q_t$ of complex coordinates $\boldz_t:=\boldz+t\cdot \boldu$
		are contained in $V$ for all $t\in [0,1]$, 	
	$\boldd_i=(d_{i,1},\,\cdots\,,\, d_{i, n})$, $i=1,2$, with $d_{i,j}\in {\Bbb Z}_{\ge 0}$,
    $|d_i|:= d_{i,1}+\,\cdots\,+ d_{i,n}$,	
    $\boldd_i !:= d_{i,1}!\,\cdots\,d_{i,n}!$, \\	
    $\partial^{\,d}/(\partial\boldz^{\scriptsizeboldd_1}\partial\bar{\boldz}^{\scriptsizeboldd_2})
	   :=  (\partial/\partial z^1)^{d_{1,1}} \cdots  (\partial/\partial z^n)^{d_{1,n}}
	        (\partial/\partial \bar{z}^1)^{d_{2,1}} \cdots  (\partial/\partial \bar{z}^n)^{d_{2,n}}
	 $ for $|\boldd_1|+|\boldd_2|=d$, \\
    $\boldu^{\scriptsizeboldd}:= (u^1)^{d_{1,1}}\,\cdots\,(u^{1,n})^{d_{1,n}}$,
	$\bar{\boldu}^{\scriptsizeboldd}
	    := (\bar{u}^1)^{d_{2,1}}\,\cdots\,(\bar{u}^{2,n})^{d_{2,n}}$.
\end{lemma}

\bigskip

\begin{flushleft}
{\bf Chiral and antichiral maps from $\widehat{X}^{\widehat{\boxplus}, \smallscriptsize}$ to a complex manifold $Y$}
\end{flushleft}

\begin{definition} {\bf [chiral/antichiral map]}\; {\rm
 Let
  $\breve{f}: \widehat{X}^{\widehat{\boxplus}, \smallscriptsize} \rightarrow Y$ be a smooth map    and
  $$
     \breve{f}^\sharp\;:\;  {\cal O}_Y^{\,\Bbb C}\;
       \longrightarrow\;  \widehat{\cal O}_X^{\,\widehat{\boxplus}, \smallscriptsize}
  $$	
   be the associated equivalence class of gluing systems of ring-homomorphisms.
 Then
  $\breve{f}$ is called {\it chiral} if it satisfies
  (1)
	 $\breve{f}^\sharp(\bar{h})= \breve{f}^\sharp(h)^\dag$   and
  (2)
     $\breve{f}^\sharp({\cal O}_Y^{\,{\Bbb C}, \hol})
        \subset \widehat{\cal O}_X^{\,\widehat{\boxplus}, \smallscriptsize, \scriptsizech}$.
   
 Similarly,
   $\breve{f}$ is called {\it antichiral} if it satisfies
  (1)
	 $\breve{f}^\sharp(\bar{h})= \breve{f}^\sharp(h)^\dag$   and
  (2$^\prime$)
     $\breve{f}^\sharp({\cal O}_Y^{\,{\Bbb C}, \hol})
        \subset \widehat{\cal O}_X^{\,\widehat{\boxplus}, \smallscriptsize, \scriptsizeach}$.
}\end{definition}

\bigskip

Note that
 if   $\breve{f}: \widehat{X}^{\widehat{\boxplus}, \smallscriptsize} \rightarrow Y$ is chiral
 (resp.\:antichiral) then
   $\breve{f}^\sharp({\cal O}_Y^{\,{\Bbb C}, \ahol})
        \subset \widehat{\cal O}_X^{\,\widehat{\boxplus}, \smallscriptsize, \scriptsizeach}$.
  (resp.\ $\breve{f}^\sharp({\cal O}_Y^{\,{\Bbb C}, \ahol})
        \subset \widehat{\cal O}_X^{\,\widehat{\boxplus}, \smallscriptsize, \scriptsizech}$.)

\bigskip

\begin{remark} $[$naturality of chiral/antichiral map$]$\; {\rm
 Recall that the most natural class of maps from a complex manifold to another complex manifold is
   the class of holomorphic maps or antiholomorphic maps.
 One should think the same for chiral or antichiral maps from
  $\widehat{X}^{\widehat{\boxplus}, \smallscriptsize}$ to a complex manifold.
}\end{remark}

\bigskip

\begin{flushleft}
{\bf Local presentation of the components of $\breve{f}^\sharp$ that defines a chiral map $\breve{f}$}
\end{flushleft}
Proposition~5.1.3
 can be adapted for a chiral map from $\widehat{X}^{\widehat{\boxplus}, \smallscriptsize}$ to a complex manifold $Y$.
For the simplicity of presentation, we let $Y={\Bbb C}^n$ in a single complex coordinate chart and
work out the explicit form of the action functional for chiral maps to a complex manifold
 $\breve{f}:\widehat{X}^{\widehat{\boxplus}, \smallscriptsize} \rightarrow Y$
 specified by the independent components
 $(f^\sharp_{(0)}, (f^\sharp_{(\alpha)})_\alpha, f^\sharp_{(12)})$  of the associated
 $\breve{f}^\sharp:C^\infty(Y)^{\Bbb C}\rightarrow C^\infty(\widehat{X}^{\widehat{\boxplus}})$.

\bigskip

\noindent
$(a)$\;{\it $h\in C^\infty(Y)^{\Bbb C}$}

\medskip

\noindent
Continuing the discussion and the notations in the previous theme.
Let
 $$
   h_{\Bbb C}(z^1,\,\cdots\,,\,z^n\,,\,\bar{z}^1,\,\cdots\,,\, \bar{z}^n)\;
   := h(y^1,\, y^2,\,\cdots\,,\, y^{2n-1}, y^{2n}) \; \in\; C^\infty(Y)
 $$
 be a K\"{a}hler potential of the K\"{a}hler metric on $Y={\Bbb C}^n$,
   expressed in terms of the complex coordinates on ${\Bbb C}$,
   --- here, for simplicity, we assume that
        the K\"{a}hler metric on ${\Bbb C}^n$ admits a K\"{a}hler potential that is defined on all ${\Bbb C}^n$ ---
 and
 {\small
 \begin{eqnarray*}
  \breve{f}^i & :=\:  & \breve{f}^\sharp(z^i) \\
  & = &
    f^i_{(0)}(x)
	+ \sum_\alpha \theta^\alpha\vartheta_\alpha f^i_{(\alpha)}(x)
    + \sqrt{-1} \sum_{\alpha,\dot{\beta},\mu}
         \theta^\alpha\bar{\theta}^{\dot{\beta}}
		   \sigma_{\alpha\dot{\beta}}^\mu\, \partial_\mu f^i_{(0)}(x)
	+ \theta^1\theta^2 \vartheta_1\vartheta_2 f^i_{(12)}(x)    \\
  && \hspace{3em}
    + \sqrt{-1} \sum_{\dot{\beta},\alpha, \mu}
	     \theta^1\theta^2\bar{\theta}^{\dot{\beta}}  \vartheta_{\alpha}
		   \sigma^{\mu\alpha}_{\;\;\;\;\dot{\beta}}\,\partial_\mu f^i_{(\alpha)}(x)
	- \theta^1\theta^2\bar{\theta}^{\dot{1}}\bar{\theta}^{\dot{2}}\,
	    \square f^i_{(0)}(x)  \\
  & \:=: & f^i_{(0)}(x) + \breve{n}^{\prime\, i}_{\breve{f}}\,,		              \\[1.2ex]
  \breve{f}^{i\dag} & :=\:  & \breve{f}^\sharp(\bar{z}^i) \\
  & = &
    \overline{f^i_{(0)}(x)}
	- \sum_{\dot{\beta}} \bar{\theta}^{\dot{\beta}}\bar{\vartheta}_{\dot{\beta}}\,
      	\overline{f^i_{(\beta)}(x)}
    - \sqrt{-1} \sum_{\alpha,\dot{\beta},\mu}
         \theta^\alpha\bar{\theta}^{\dot{\beta}}
		   \sigma_{\alpha\dot{\beta}}^\mu\, \partial_\mu  \overline{f^i_{(0)}(x)}
	+ \bar{\theta}^{\dot{1}}\bar{\theta}^{\dot{2}}
	     \bar{\vartheta}_{\dot{1}} \bar{\vartheta}_{\dot{2}}\,
		                                \overline{f^i_{(12)}(x)}    \\
   && \hspace{3em}
    - \sqrt{-1} \sum_{\alpha,\dot{\beta}\mu}
	     \theta^\alpha\bar{\theta}^{\dot{1}}\bar{\theta}^{\dot{2}}
		     \bar{\vartheta}_{\dot{\beta}}			
		   \sigma^{\mu\dot{\beta}}_\alpha\,\partial_\mu \overline{f^i_{(\beta)}(x)}
	- \theta^1\theta^2\bar{\theta}^{\dot{1}}\bar{\theta}^{\dot{2}}\,
	    \square \overline{f^i_{(0)}(x)}\\
  & \:=: &		
     \overline{f^i_{(0)}}(x) + \breve{n}^{\prime\prime\, i}_{\breve{f}}\,. \\		
 \end{eqnarray*}}
Here,
 $\breve{n}^{\prime\, i}_{\breve{f}}$  (resp.\ $\breve{n}^{\prime\prime\, i}_{\breve{f}}$)
   is the nilpotent component of $\breve{f}^i$ (resp.\ $\breve{f}^{i,\dag}$).
They commute among themselves and satisfy
 $$
   \breve{n}^{\prime\, i_1}_{\breve{f}}\, \breve{n}^{\prime\, i_2}_{\breve{f}} \,
        \breve{n}^{\prime\, i_3}_{\breve{f}}\;
   =\;   \breve{n}^{\prime\prime\, i_1}_{\breve{f}}\,
               \breve{n}^{\prime\prime\, i_2}_{\breve{f}} \,
			       \breve{n}^{\prime\prime\, i_3}_{\breve{f}}\;
   =\; 0\,,
 $$
 for $1\le i_1, i_2, i_3\le n$.
Denote $(f^1_{(0)}(x),\,\cdots\,,\, f^n_{(0)}(x))$  collectively by $f_{(0)}(x)$ and\\
  $(\overline{f^1_{(0)}(x)}\,,\, \cdots\,,\, \overline{f^n_{(0)}(x)})$
     collectively by $\overline{f_{(0)}(x)}$.
Then it follows from Lemma~5.2.2
		and the $C^\infty$-hull structure of $C^\infty(\widehat{X}^{\widehat{\boxplus}})^\smallscriptsize$
 that
{\small
\begin{eqnarray*}
 \breve{f}^\sharp(h) & = &  \breve{f}^\sharp(h_{\Bbb C}) \\
 & =   & h_{\Bbb C}(\breve{f}^1,\,\cdots\,,\, \breve{f}^n\,,\,
                                           \breve{f}^{1\dag},\,\cdots\,,\, \breve{f}^{n\dag}) \\
 &= &
   h_{\Bbb C}(f_{(0)}(x), \overline{f_{(0)}(x)})
   + \sum_{i=1}^n
      (\partial_{z^i}h_{\Bbb C})(f_{(0)}(x), \overline{f_{(0)}(x)})
		  \cdot \breve{n}^{\prime\, i}_{\breve{f}}		
   + \sum_{j=1}^n
      (\partial_{\bar{z}^j}h_{\Bbb C})(f_{(0)}(x), \overline{f_{(0)}(x)})
		    \cdot \breve{n}^{\prime\prime\, j}_{\breve{f}}		   \\
 &&
  + \mbox{\Large $\frac{1}{2}$}\,
       \sum_{i,j}(\partial_{z^i}\partial_{z^j}
	        h_{\Bbb C})(f_{(0)}(x), \overline{f_{(0)}(x)})
	    \cdot \breve{n}^{\prime\, i}_{\breve{f}}\, \breve{n}^{\prime\, j}_{\breve{f}}
  +  \sum_{i,j}(\partial_{z^i}\partial_{\bar{z}^j}h_{\Bbb C})
	            (f_{(0)}(x), \overline{f_{(0)}(x)})
	    \cdot \breve{n}^{\prime\, i}_{\breve{f}}\, \breve{n}^{\prime\prime\, j}_{\breve{f}}   \\
 && \hspace{3em}
  + \mbox{\Large $\frac{1}{2}$}\,
       \sum_{i,j}(\partial_{\bar{z}^i}\partial_{\bar{z}^j}h_{\Bbb C})
	            (f_{(0)}(x), \overline{f_{(0)}(x)})
	    \cdot \breve{n}^{\prime\prime\, i}_{\breve{f}}\, \breve{n}^{\prime\prime\, j}_{\breve{f}} \\
 &&
  + \mbox{\Large $\frac{1}{2}$}\,
       \sum_{i,j, k}(\partial_{z^i}\partial_{z^j}\partial_{\bar{z}^k}h_{\Bbb C})
	            (f_{(0)}(x), \overline{f_{(0)}(x)})
	    \cdot \breve{n}^{\prime\, i}_{\breve{f}}\, \breve{n}^{\prime\, j}_{\breve{f}}\,
		             \breve{n}^{\prime\prime\, k}_{\breve{f}}  \\ 		
 && \hspace{3em}
  + \mbox{\Large $\frac{1}{2}$}\,
       \sum_{i,j, k}(\partial_{z^i}\partial_{\bar{z}^j}\partial_{\bar{z}^k}h_{\Bbb C})
	            (f_{(0)}(x), \overline{f_{(0)}(x)})
	    \cdot \breve{n}^{\prime\, i}_{\breve{f}}\,
		              \breve{n}^{\prime\prime\, j}_{\breve{f}} \, \breve{n}^{\prime\prime\, k}_{\breve{f}}  \\
 &&
  + \mbox{\Large $\frac{1}{2!\,2!}$}\,
       \sum_{i,j, k, l}(\partial_{z^i}\partial_{z^j}\partial_{\bar{z}^k}\partial_{\bar{z}^l}h_{\Bbb C})
	            (f_{(0)}(x), \overline{f_{(0)}(x)})
	    \cdot \breve{n}^{\prime\, i}_{\breve{f}}\, \breve{n}^{\prime\, j}_{\breve{f}}\,
		              \breve{n}^{\prime\prime\, k}_{\breve{f}} \, \breve{n}^{\prime\prime\, l}_{\breve{f}}\,.
\end{eqnarray*}}  

By definition,
{\small
 \begin{eqnarray*}
  \breve{n}^{\prime\, i}_{\breve{{f}}}
  & = &
	 \sum_\alpha \theta^\alpha\vartheta_\alpha f^i_{(\alpha)}(x)
    + \sqrt{-1} \sum_{\alpha,\dot{\beta},\mu}
         \theta^\alpha\bar{\theta}^{\dot{\beta}}
		   \sigma_{\alpha\dot{\beta}}^\mu\, \partial_\mu f^i_{(0)}(x)
	+ \theta^1\theta^2 \vartheta_1\vartheta_2 f^i_{(12)}(x)    \\
  && \hspace{3em}
    + \sqrt{-1} \sum_{\dot{\beta},\alpha, \mu}
	     \theta^1\theta^2\bar{\theta}^{\dot{\beta}}  \vartheta_\alpha
		   \sigma^{\mu\alpha}_{\;\;\;\;\dot{\beta}}\,\partial_\mu f^i_{(\alpha)}(x)		
	- \theta^1\theta^2\bar{\theta}^{\dot{1}}\bar{\theta}^{\dot{2}}\,
	    \square f^i_{(0)}(x)  \,,		              \\[1.2ex]
  \breve{n}^{\prime\prime\, i} & =  & (\breve{n}^{\prime\, i})^\dag \\
  & = &
	- \sum_{\dot{\beta}} \bar{\theta}^{\dot{\beta}}\bar{\vartheta}_{\dot{\beta}}\,
      	\overline{f^i_{(\beta)}(x)}
    - \sqrt{-1} \sum_{\alpha,\dot{\beta},\mu}
         \theta^\alpha\bar{\theta}^{\dot{\beta}}
		   \sigma_{\alpha\dot{\beta}}^\mu\, \partial_\mu  \overline{f^i_{(0)}(x)}
	+ \bar{\theta}^{\dot{1}}\bar{\theta}^{\dot{2}}
	      \bar{\vartheta}_{\dot{1}} \bar{\vartheta}_{\dot{2}}\, \overline{f^i_{(12)}(x)}    \\
   && \hspace{3em}
    - \sqrt{-1} \sum_{\alpha,\dot{\beta}, \mu}
	     \theta^\alpha \bar{\theta}^{\dot{1}}\bar{\theta}^{\dot{2}}  \bar{\vartheta}_{\dot{\beta}}		        
		   \sigma^{\mu\dot{\beta}}_\alpha\,\partial_\mu \overline{f^i_{(\beta)}(x)}		 		
	- \theta^1\theta^2\bar{\theta}^{\dot{1}}\bar{\theta}^{\dot{2}}\,
	    \square \overline{f^i_{(0)}(x)}\,.		
 \end{eqnarray*}}A 
straightforward computation gives the following explicit formulae:

{\small
\begin{eqnarray*}
 \lefteqn{\breve{n}^{\prime\, i}_{\breve{f}}\, \breve{n}^{\prime\, j}_{\breve{f}}  } \\
  && =\;\;
   -\,\theta^1\theta^2 \vartheta_1\vartheta_2
        \,\mbox{\large $($}
		  f^i_{(1)}(x) f^j_{(2)}(x) + f^i_{(2)}(x) f^j_{(1)}(x)
		\mbox{\large $)$}\\
  && \hspace{2em}
   +\,\sqrt{-1}\, \sum_{\dot{\beta}, \alpha, \mu}
         \theta^1\theta^2\bar{\theta}^{\dot{\beta}}  \vartheta_\alpha
		   \sigma^{\mu\alpha}_{\;\;\;\;\dot{\beta}}
		  \,\mbox{\large $($}		
		     f^i_{(\alpha)}(x)\,\partial_\mu f^j_{(0)}(x)
			 +  \partial_\mu f^i_{(0)}(x) f^j_{(\alpha)}(x)
		  \mbox{\large $)$}   		  		  		  \\
  && \hspace{2em}
   -\,2\,\theta^1\theta^2\bar{\theta}^{\dot{1}}\bar{\theta}^{\dot{2}}
            \sum_{\mu,\nu}\eta^{\mu\nu}\,\partial_\mu f^i_{(0)}(x)\,\partial_\nu f^j_{(0)}(x) \,;
\end{eqnarray*}} 
{\small
\begin{eqnarray*}
 \lefteqn{\breve{n}^{\prime\prime\, i}_{\breve{f}}\, \breve{n}^{\prime\prime\, j}_{\breve{f}}   } \\
  && =\;\;
   -\, \bar{\theta}^{\dot{1}}\bar{\theta}^{\dot{2}}
          \bar{\vartheta}_{\dot{1}}  \bar{\vartheta}_{\dot{2}}
        \,\mbox{\large $($}
		  \overline{f^i_{(1)}(x)}\, \overline{f^j_{(2)}(x)}
		   + \overline{f^i_{(2)}(x)}\, \overline{f^j_{(1)}(x)}
		\mbox{\large $)$}\\
  && \hspace{2em}
   -\,\sqrt{-1}\, \sum_{\alpha, \dot{\beta}, \mu}
         \theta^\alpha \bar{\theta}^{\dot{1}}\bar{\theta}^{\dot{2}}\bar{\vartheta}_{\dot{\beta}}
		   \sigma^{\mu\dot{\beta}}_\alpha
		  \,\mbox{\large $($}
		     \overline{f^i_{(\beta)}(x)}\,\partial_\mu \overline{f^j_{(0)}(x)}
			 + \partial_\mu \overline{f^i_{(0)}(x)}\, \overline{f^j_{(\beta)}(x)}					
		  \mbox{\large $)$}   		  \\
  && \hspace{2em}
   -\,2\,\theta^1\theta^2\bar{\theta}^{\dot{1}}\bar{\theta}^{\dot{2}}
            \sum_{\mu,\nu}\eta^{\mu\nu}\,\partial_\mu \overline{f^i_{(0)}(x)}\,
			           \partial_\nu \overline{f^j_{(0)}(x)} \,;
\end{eqnarray*}} 
{\small
\begin{eqnarray*}
 \lefteqn{\breve{n}^{\prime\, i}_{\breve{f}}\, \breve{n}^{\prime\prime\, j}_{\breve{f}}    } \\
 && =\;\;
   \sum_{\alpha, \dot{\beta}}
      \theta^\alpha \bar{\theta}^{\dot{\beta}}  \vartheta_{\alpha} \bar{\vartheta}_{\dot{\beta}}\,      
          f^i_{(\alpha)}(x)\,\overline{f^j_{(\beta)}(x)}    \\
 && \hspace{2em}
   -\, \sum_{\dot{\beta}}    \theta^1\theta^2\bar{\theta}^{\dot{\beta}}
         \,\mbox{\Large $($}
               \sqrt{-1} \sum_{\alpha, \mu}
			    \vartheta_\alpha \sigma^{\mu\alpha}_{\;\;\;\;\dot{\beta}}
				     f^i_{(\alpha)}(x)\,\partial_\mu \overline{f^j_{(0)}(x)}
		  + \vartheta_1\vartheta_2 \bar{\vartheta}_{\dot{\beta}}\,
 		            f^i_{(12)}\,\overline{f^j_{(\beta)}(x)}
	     \mbox{\Large $)$}    \\	
 && \hspace{2em}
   +\, \sum_\alpha \theta^\alpha \bar{\theta}^{\dot{1}} \bar{\theta}^{\dot{2}}
          \,\mbox{\Large $($}	
		     \sqrt{-1}\sum_{\dot{\beta}, \mu}
			    \bar{\vartheta}_{\dot{\beta}}
			      \sigma^{\mu\dot{\beta}}_\alpha\,
				     \partial_\mu f^i_{(0)}(x)\,\overline{f^j_{(\beta)}(x)}
			+ \vartheta_\alpha \bar{\vartheta}_{\dot{1}} \bar{\vartheta}_{\dot{2}}\,
			      f^i_{(\alpha)}(x)\,\overline{f^j_{(12)}(x)}  			
           \mbox{\Large $)$}		\\
 && \hspace{2em}
  +\, \theta^1\theta^2\bar{\theta}^{\dot{1}}\bar{\theta}^{\dot{2}}
         \,\mbox{\LARGE $($}
		     2 \sum_{\mu, \nu} \eta^{\mu\nu}
		          \partial_\mu f^i_{(0)}(x)\, \partial_\nu \overline{f^j_{(0)}(x)}      \\
     && \hspace{4em}
        +\, \sqrt{-1}\, \sum_{\alpha, \dot{\beta}, \mu}
		        \vartheta_\alpha \bar{\vartheta}_{\dot{\beta}} \bar{\sigma}^{\mu\dot{\beta}\alpha}
		       \mbox{\large $($}
			    f^i_{(\alpha)}(x)\, \partial_\mu \overline{f^j_{(\beta)}(x)}
				-   \partial_\mu f^i_{(\alpha)}(x)     \overline{f^j_{(\beta)}(x)}			
			   \mbox{\large $)$} \\
     && \hspace{4em}
      	  +\, \vartheta_1\vartheta_2 \bar{\vartheta}_{\dot{1}} \bar{\vartheta}_{\dot{2}}\,
  		           f^i_{(12)}(x)\,\overline{f^j_{(12)}(x)}
		 \mbox{\Large $)$}\,;
\end{eqnarray*}
} 
{\small
\begin{eqnarray*}
 \lefteqn{\breve{n}^{\prime\, i}_{\breve{f}}\, \breve{n}^{\prime\, j}_{\breve{f}} \,
    \breve{n}^{\prime\prime\, k}_{\breve{f}}                } \\
 && =\;\;
 \sum_{\dot{\beta}}
   \theta^1\theta^2 \bar{\theta}^{\dot{\beta}}
       \vartheta_1 \vartheta_2 \bar{\vartheta}_{\dot{\beta}}
      \,\mbox{\Large $($}
	    f^i_{(1)}(x)f^j_{(2)}(x) + f^i_{(2)}(x) f^j_{(1)}(x)
	  \mbox{\Large $)$}\cdot \overline{f^k_{(\beta)}(x)}    \\
 && \hspace{2em}
  -\, \theta^1\theta^2\bar{\theta}^{\dot{1}}\bar{\theta}^{\dot{2}}
         \,\mbox{\LARGE $($}
        \sqrt{-1}\,\sum_{\alpha, \dot{\beta}, \mu}
          \vartheta_\alpha \bar{\vartheta}_{\dot{\beta}} \bar{\sigma}^{\mu\dot{\beta}\alpha}		
            \mbox{\large $($}
		      f^i_{(\alpha)}(x)\, \partial_\mu f_{(0)}^j(x)
		      + \partial_\mu f_{(0)}^i(x) f^j_{(\alpha)}(x)
		    \mbox{\large $)$}\cdot  \overline{f^k_{(\beta)}(x)}		 		  \\		  		
    && \hspace{10em}
	    +\, \vartheta_1\vartheta_2 \bar{\vartheta}_{\dot{1}} \bar{\vartheta}_{\dot{2}}\,
             \mbox{\large $($}
	           f^i_{(1)}(x)f^j_{(2)}(x) + f^i_{(2)}(x) f^j_{(1)}(x)
	         \mbox{\large $)$}\cdot \overline{f^k_{(12)}(x)}  	
		  \mbox{\LARGE $)$}\,;
\end{eqnarray*}
} 
{\small
\begin{eqnarray*}
 \lefteqn{\breve{n}^{\prime\, i}_{\breve{f}}\, \breve{n}^{\prime\prime\, j}_{\breve{f}} \,
                      \breve{n}^{\prime\prime\, k}_{\breve{f}}             }\\
 && =\;\;
  -\, \sum_\alpha
        \theta^\alpha \bar{\theta}^{\dot{1}} \bar{\theta}^{\dot{2}}
            \vartheta_\alpha \bar{\vartheta}_{\dot{1}} \bar{\vartheta}_{\dot{2}}  \,
          f^i_{(\alpha)}(x) \cdot
             \mbox{\large $($}
	             \overline{f^j_{(1)}(x)}\, \overline{f^k_{(2)}(x)}
		           + \overline{f^j_{(2)}(x)}\cdot \overline{f^k_{(1)}(x)}
	         \mbox{\large $)$} \\
 && \hspace{2em}
  +\, \theta^1\theta^2\bar{\theta}^{\dot{1}}\bar{\theta}^{\dot{2}}
         \,\mbox{\LARGE $($}		 		
          \sqrt{-1}\, \sum_{\alpha, \dot{\beta}, \mu}
		     \vartheta_\alpha \bar{\vartheta}_{\dot{\beta}} \bar{\sigma}^{\mu\dot{\beta}\alpha}\,
			   f^i_{(\alpha)}(x) \cdot
                 \mbox{\large $($}
		          \overline{f^j_{(\beta)}(x)}\, \partial_\mu \overline{f_{(0)}^k(x)}
		            + \partial_\mu \overline{f_{(0)}^j(x)}\,\overline{f^k_{(\beta)}(x)}
				 \mbox{\large $)$}			  		 		 		  \\
    && \hspace{10em}		
	       -\, \vartheta_1\vartheta_2 \bar{\vartheta}_{\dot{1}}  \bar{\vartheta}_{\dot{2}}\,
		     f^i_{(12)}(x) \cdot
		     \mbox{\large $($}
	           \overline{f^j_{(1)}(x)}\, \overline{f^k_{(2)}(x)}
			    + \overline{f^j_{(2)}(x)}\, \overline{f^k_{(1)}(x)}
	         \mbox{\large $)$}
         \mbox{\LARGE $)$} \,;
\end{eqnarray*}
} 
{\small
\begin{eqnarray*}
 \lefteqn{\breve{n}^{\prime\, i}_{\breve{f}}\, \breve{n}^{\prime\, j}_{\breve{f}}\,
                     \breve{n}^{\prime\prime\, k}_{\breve{f}} \, \breve{n}^{\prime\prime\, l}_{\breve{f}}    }\\
 && =\;\;
    \theta^1\theta^2\bar{\theta}^{\dot{1}}\bar{\theta}^{\dot{2}}
	 \vartheta_1\vartheta_2 \bar{\vartheta}_{\dot{1}} \bar{\vartheta}_{\dot{2}}
        \,\mbox{\large $($}
		  f^i_{(1)}(x)f^j_{(2)}(x)  + f^i_{(2)}(x)f^j_{(1)}(x)
		 \mbox{\large $)$}
		   \cdot
		  \mbox{\large $($}
		   	\overline{f^k_{(1)}(x)}\,  \overline{f^l_{(2)}(x)}
			  + \overline{f^k_{(2)}(x)}\, \overline{f^l_{(1)}(x)}
		 \mbox{\large $)$}\,.
\end{eqnarray*}
} 

\vspace{6em}
Substituting these expressions into the expansion of
 $h_{\Bbb C}(\breve{f}^1, \,\cdots\,,\,\breve{f}^n\,,\,
                                \breve{f}^{1\dag},\,\dots\,,\, \breve{f}^{n\dag})$,
one obtains

{\footnotesize
\begin{eqnarray*}
 \lefteqn{
  \breve{f}^\sharp(h) \;\; =\;\;   \breve{f}^\sharp(h_{\Bbb C})\;\;
    =\;\;  h_{\Bbb C}(\breve{f}^1,\,\cdots\,,\, \breve{f}^n\,,\,
                                           \breve{f}^{1\dag},\,\cdots\,,\, \breve{f}^{n\dag})     } \\
 && =\;\;
   h_{\Bbb C}(f_{(0)}(x), \overline{f_{(0)}(x)})\,
    +  \sum_{\alpha, i}\theta^\alpha \vartheta_\alpha\,
	       f^i_{(\alpha)}(x)\cdot
	           (\partial_{z^i}   h_{\Bbb C})(f_{(0)}(x), \overline{f_{(0)}(x)})
	 - \sum_{\dot{\beta}, i}\bar{\theta}^{\dot{\beta}} \bar{\vartheta}_{\dot{\beta}}\,
	       \overline{f^i_{(\beta)}(x)}\cdot
	           (\partial_{\bar{z}^i}   h_{\Bbb C})(f_{(0)}(x), \overline{f_{(0)}(x)})    \\
 && \hspace{2em}			
    +\,  \theta^1\theta^2 \vartheta^1 \vartheta^2\,
       \left\{\rule{0ex}{1.2em}\right.
          \sum_i f^i_{(12)}(x)\cdot
		     (\partial_{z^i} h_{\Bbb C})(f_{(0)}(x), \overline{f_{(0)}(x)})\\[-2ex]
    && \hspace{10em} 		
		   -\mbox{\large $\frac{1}{2}$}\,\sum_{i,j}
		       \mbox{\large $($}
			     f^i_{(1)}(x) f^j_{(2)}(x) + f^i_{(2)}(x) f^j_{(1)} (x)
			   \mbox{\large $)$}\cdot
			    (\partial_{z^i} \partial_{z^j}   h_{\Bbb C})(f_{(0)}(x), \overline{f_{(0)}(x)})
       \left.\rule{0ex}{1.2em}\right\}     \\
 && \hspace{2em}
   +\, \sum_{\alpha,\dot{\beta}} \theta^\alpha \bar{\theta}^{\dot{\beta}} \,
          \left\{\rule{0ex}{1.2em}\right.
	       \sqrt{-1}\sum_{\mu, i} \sigma^\mu_{\alpha\dot{\beta}}
		     \mbox{\LARGE $($}
			  \partial_\mu f^i_{(0)}(x)\cdot
			      (\partial_{z^i}h_{\Bbb C})(f_{(0)}(x), \overline{f_{(0)}(x)})
				- \partial_\mu \overline{f^i_{(0)}(x)}\cdot
				      (\partial_{\bar{z}^i} h_{\Bbb C})(f_{(0)}(x), \overline{f_{(0)}(x)})			
	         \mbox{\LARGE $)$} \\
   && \hspace{10em}
		+\,  \vartheta_\alpha \bar{\vartheta}_{\dot{\beta}}
		       \sum_{i,j} f^i_{(\alpha)}(x) \overline{f^j_{(\beta)}(x)}\cdot
			     (\partial_{z^i}\partial_{\bar{z}^j} h_{\Bbb C})(f_{(0)}(x), \overline{f_{(0)}(x)})
	      \left.\rule{0ex}{1.2em}\right\}   \\
 && \hspace{2em}			
    +\,  \bar{\theta}^{\dot{1}}\bar{\theta}^{\dot{2}}
	          \bar{\vartheta}^{\dot{1}} \bar{\vartheta}^{\dot{2}}\,
       \left\{\rule{0ex}{1.2em}\right.
          \sum_i \overline{f^i_{(12)}(x)}\cdot
		       (\partial_{\bar{z}^i}   h_{\Bbb C})(f_{(0)}(x), \overline{f_{(0)}(x)}) \\[-2ex]
    && \hspace{10em}			
		   - \mbox{\large $\frac{1}{2}$}\,\sum_{i,j}
		       \mbox{\large $($}
			     \overline{f^i_{(1)}(x)}\, \overline{ f^j_{(2)}(x)}
				   + \overline{f^i_{(2)}(x)}\, \overline{f^j_{(1)} (x)}
			   \mbox{\large $)$}\cdot
			    (\partial_{\bar{z}^i} \partial_{\bar{z}^j}h_{\Bbb C})
				                                                                    (f_{(0)}(x), \overline{f_{(0)}(x)})
       \left.\rule{0ex}{1.2em}\right\}     \\ 		
 && \hspace{2em}
   + \sum _{\dot{\beta}} \theta^1\theta^2\bar{\theta}^{\dot{\beta}}
       \left\{\rule{0ex}{1.2em}\right.
	     \sqrt{-1}\sum_{\alpha, \mu}
		   \vartheta_\alpha   \sigma^{\mu\alpha}_{\;\;\;\;\dot{\beta}}
	       \mbox{\LARGE $($}
		    \sum_i \partial_\mu f^i_{(\alpha)}(x)\cdot
			   (\partial_{z^i}h_{\Bbb C})(f_{(0)}(x), \overline{f_{(0)}(x)})  \\
      && \hspace{14em}			
		    +\, \mbox{\large $\frac{1}{2}$}\sum_{i,j}
			      \mbox{\large $($}
				    f^i_{(\alpha)}(x)\,\partial_\mu f^j_{(0)}(x)
					+ \partial_\mu f^i_{(0)}(x) f^j_{(\alpha)}(x)
				  \mbox{\large $)$} \cdot
				   (\partial_{z^i}\partial_{z^j} h_{\Bbb C})(f_{(0)}(x), \overline{f_{(0)}(x)})   \\
      && \hspace{14em}				
		    -\, \sum_{i,j} f^i_{(\alpha)}(x)\,\partial_\mu \overline{f^j_{(0)}(x)}\cdot
			      (\partial_{z^i}\partial_{\bar{z}^j} h_{\Bbb C})
				                                                            (f_{(0)}(x), \overline{f_{(0)}(x)})
		   \mbox{\LARGE $)$}    \\
   && \hspace{6em}
      +\, \vartheta_1\vartheta_2 \bar{\vartheta}_{\dot{\beta}}
	          \mbox{\LARGE $($}
			    -\sum_{i,j} f^i_{(12)}(x)\overline{f^j_{(\beta)}(x)}\cdot
				   (\partial_{z^i}\partial_{\bar{z}^j} h_{\Bbb C})
				                                                            (f_{(0)}(x), \overline{f_{(0)}(x)})    \\
      && \hspace{12em}																			
                + \mbox{\large $\frac{1}{2}$}	
				     \sum_{i,j,k}
					   \mbox{\large $($}
					     f^i_{(1)}(x) f^j_{(2)}(x) +   f^i_{(2)}(x) f^j_{(1)}(x)
					   \mbox{\large $)$}\,
					   \overline{f^k_{(\beta)}(x)} \cdot
				     (\partial_{z^i}\partial_{z^j}\partial_{\bar{z}^k} h_{\Bbb C})
				                                                            (f_{(0)}(x), \overline{f_{(0)}(x)})
			  \mbox{\LARGE $)$}
	   \left.\rule{0ex}{1.2em}\right\} \\
 && \hspace{2em}
   + \sum _{\alpha} \theta^\alpha \bar{\theta}^{\dot{1}}\bar{\theta}^{\dot{2}}
       \left\{\rule{0ex}{1.2em}\right.
	     \sqrt{-1}\sum_{\dot{\beta}, \mu}
		   \bar{\vartheta}_{\dot{\beta}}   \sigma^{\mu\dot{\beta}}_\alpha
	       \mbox{\LARGE $($}
		    - \sum_i \partial_\mu \overline{f^i_{(\beta)}(x)}\cdot
			   (\partial_{\bar{z}^i}h_{\Bbb C})(f_{(0)}(x), \overline{f_{(0)}(x)})  \\
      && \hspace{14em}			
		    -\, \mbox{\large $\frac{1}{2}$}\sum_{i,j}
			      \mbox{\large $($}
				    \overline{f^i_{(\beta)}(x)}\,\partial_\mu \overline{f^j_{(0)}(x)}
					+ \partial_\mu \overline{f^i_{(0)}(x)}\, \overline{f^j_{(\beta)}(x)}
				  \mbox{\large $)$} \cdot
				   (\partial_{\bar{z}^i}\partial_{\bar{z}^j}
				         h_{\Bbb C})(f_{(0)}(x), \overline{f_{(0)}(x)})   \\
      && \hspace{14em}				
		    +\, \sum_{i,j}\partial_\mu f^i_{(0)}(x)\,\overline{f^j_{(\beta)}(x)}\cdot
			      (\partial_{z^i}\partial_{\bar{z}^j} h_{\Bbb C})
				                                                            (f_{(0)}(x), \overline{f_{(0)}(x)})
		   \mbox{\LARGE $)$}    \\
   && \hspace{6em}
      +\, \vartheta_\alpha \bar{\vartheta}_{\dot{1}} \bar{\vartheta}_{\dot{2}}
	          \mbox{\LARGE $($}
			    \sum_{i,j} f^i_{(\alpha)}(x)\overline{f^j_{(12)}(x)}\cdot
				   (\partial_{z^i}\partial_{\bar{z}^j} h_{\Bbb C})
				                                                            (f_{(0)}(x), \overline{f_{(0)}(x)})    \\
      && \hspace{12em}			
                -\, \mbox{\large $\frac{1}{2}$}	
				     \sum_{i,j,k}
					   f^i_{(\alpha)}(x)\,
					   \mbox{\large $($}
					    \overline{f^j_{(1)}(x)}\, \overline{f^k_{(2)}(x)}
						   +  \overline{f^j_{(2)}(x)}\, \overline{f^k_{(1)}(x)}
					   \mbox{\large $)$}\,
					   \cdot
				     (\partial_{z^i}\partial_{\bar{z}^j}\partial_{\bar{z}^k} h_{\Bbb C})
				                                                            (f_{(0)}(x), \overline{f_{(0)}(x)})
			  \mbox{\LARGE $)$}
	   \left.\rule{0ex}{1.2em}\right\}   	
 \end{eqnarray*}
 }

 \bigskip
 \centerline{(formula continued to the next page)}

{\footnotesize
 \begin{eqnarray*}
 && \hspace{2em}
   +\, \theta^1\theta^2\bar{\theta}^{\dot{1}}\bar{\theta}^{\dot{2}}\,
   \left\{\rule{0ex}{1.2em}\right.\!
    - \sum_{i=1}^n
	     \square f^i_{(0)}(x)\cdot
        (\partial_{z^i}h_{\Bbb C})(f_{(0)}(x), \overline{f_{(0)}(x)})
	- \sum_i \square \overline{f^i_{(0)}(x)}\cdot		
              (\partial_{\bar{z}^j}h_{\Bbb C})(f_{(0)}(x), \overline{f_{(0)}(x)})		  \\
    && \hspace{10em}
      - \, \sum_{i,j, \mu, \nu}
          \eta^{\mu\nu}\,\partial_\mu f^i_{(0)}(x)\,\partial_\nu f^j_{(0)}(x)	\cdot
         (\partial_{z^i}\partial_{z^j}h_{\Bbb C})(f_{(0)}(x), \overline{f_{(0)}(x)})  \\				 		
    && \hspace{10em}
        +\,2  \sum_{i,j, \mu, \nu}
          \eta^{\mu\nu}\,\partial_\mu f^i_{(0)}(x)\,\partial_\nu \overline{f^j_{(0)}(x)}	\cdot
         (\partial_{z^i}\partial_{\bar{z}^j}h_{\Bbb C})(f_{(0)}(x), \overline{f_{(0)}(x)})  \\
    && \hspace{10em}
      -\, \sum_{i,j, \mu, \nu}
         \eta^{\mu\nu}\,
		    \partial_\mu \overline{f^i_{(0)}(x)}\,\partial_\nu \overline{f^j_{(0)}(x)} \cdot
             (\partial_{\bar{z}^i}\partial_{\bar{z}^j}h_{\Bbb C})
	                                           (f_{(0)}(x), \overline{f_{(0)}(x)}) \\				
 && \hspace{4em}
   +\, \sqrt{-1} \sum_{\alpha, \dot{\beta}, \mu}
           \vartheta_\alpha \bar{\vartheta}_{\dot{\beta}}\, \bar{\sigma}^{\mu\dot{\beta}\alpha}
        \mbox{\LARGE $($}
          \sum_{i,j}
		      \mbox{\large $($}
		       f^i_{(\alpha)}(x)\,\partial_\mu \overline{f^j_{(\beta)}(x)}	
			    -   \partial_\mu f^i_{(\alpha)}(x)\,\overline{f^j_{(\beta)}(x)}	
			  \mbox{\large $)$} \cdot
			(\partial_{z^i}\partial_{\bar{z}^j}h_{\Bbb C})(f_{(0)}(x), \overline{f_{(0)}(x)}) \\
      && \hspace{10em}			
	      -\,\mbox{\large $\frac{1}{2}$}
		      \sum_{i,j,k}
			    \mbox{\large $($}
				  f^i_{(\alpha)}(x)\, \partial_\mu f^j_{(0)}(x)
				   + \partial_\mu f^i_{(0)}(x) f^j_{(\alpha)}(x)
				\mbox{\large $)$}\,
				\overline{f^k_{(\beta)}(x)}\cdot
			   (\partial_{z^i}\partial_{z^j}\partial_{\bar{z}^k}h_{\Bbb C})
			               (f_{(0)}(x), \overline{f_{(0)}(x)})  \\
      && \hspace{10em}						
	     +\,\mbox{\large $\frac{1}{2}$}
		      \sum_{i,j,k}
			   f^i_{(\alpha)}(x)
			    \mbox{\large $($}
				  \overline{f^j_{(\beta)}(x)}\, \partial_\mu \overline{f^k_{(0)}(x)}
				   + \partial_\mu \overline{f^j_{(0)}(x)} \, \overline{f^k_{(\beta)}(x)}
				\mbox{\large $)$}\cdot
			   (\partial_{z^i}\partial_{\bar{z}^j}\partial_{\bar{z}^k}h_{\Bbb C})
			               (f_{(0)}(x), \overline{f_{(0)}(x)})	  	
        \mbox{\LARGE $)$}
		\\
  && \hspace{4em}
   +\, \vartheta_1\vartheta_2 \bar{\vartheta}_{\dot{1}}\bar{\vartheta}_{\dot{2}}
    \mbox{\LARGE $($}
      \sum_{i,j}
           f^i_{(12)}(x)\,\overline{f^j_{(12)}(x)}    \cdot
           (\partial_{z^i}\partial_{\bar{z}^j}h_{\Bbb C})
	                  (f_{(0)}(x), \overline{f_{(0)}(x)})   \\
    && \hspace{10em}
     -\, \mbox{\large $\frac{1}{2}$}\,
       \sum_{i,j,k}	   	   				
		     \mbox{\large $($}
	           f^i_{(1)}(x)f^j_{(2)}(x) + f^i_{(2)}(x) f^j_{(1)}(x)
	         \mbox{\large $)$}\,
			 \overline{f^k_{(12)}(x)}      \cdot
            (\partial_{z^i}\partial_{z^j}\partial_{\bar{z}^k}h_{\Bbb C})
	            (f_{(0)}(x), \overline{f_{(0)}(x)})			 \\			 			 			 		 		 
    && \hspace{10em}
     -\, \mbox{\large $\frac{1}{2}$}\,
       \sum_{i,j, k}
		   f^i_{(12)}(x) \,
		     \mbox{\large $($}
	           \overline{f^j_{(1)}(x)}\, \overline{f^k_{(2)}(x)}
			    + \overline{f^j_{(2)}(x)}\, \overline{f^k_{(1)}(x)}
	         \mbox{\large $)$}      \cdot
           (\partial_{z^i}\partial_{\bar{z}^j}\partial_{\bar{z}^k}h_{\Bbb C})
	             (f_{(0)}(x), \overline{f_{(0)}(x)})			 \\			 			 		 		 
    && \hspace{10em}
     +\, \mbox{\large $\frac{1}{2!\,2!}$}\,
       \sum_{i,j, k, l}
         \mbox{\large $($}
		  f^i_{(1)}(x)f^j_{(2)}(x)  + f^i_{(2)}(x)f^j_{(1)}(x)
		 \mbox{\large $)$}\,
		 \mbox{\large $($}
		   	\overline{f^k_{(1)}(x)}\,  \overline{f^l_{(2)}(x)}
			  + \overline{f^k_{(2)}(x)}\, \overline{f^l_{(1)}(x)}
		 \mbox{\large $)$}     \\[-2ex]
     && \hspace{30em}		
		     \cdot(\partial_{z^i}\partial_{z^j}\partial_{\bar{z}^k}\partial_{\bar{z}^l}h_{\Bbb C})
	            (f_{(0)}(x), \overline{f_{(0)}(x)})
    \mbox{\LARGE $)$}		
  \left.\rule{0ex}{1.2em}\right\}			
\end{eqnarray*}
 }

 \bigskip
 \centerline{(formula continued to the next page)}
 
{\footnotesize
\begin{eqnarray*}
 && =\;\;
   h_{\Bbb C}(f_{(0)}(x), \overline{f_{(0)}(x)})\,
    +  \sum_{\alpha, i}\theta^\alpha \vartheta_\alpha\,
	       f^i_{(\alpha)}(x)\cdot
	           (\partial_{z^i}   h_{\Bbb C})(f_{(0)}(x), \overline{f_{(0)}(x)})
	 - \sum_{\dot{\beta}, i}\bar{\theta}^{\dot{\beta}} \bar{\vartheta}_{\dot{\beta}}\,
	       \overline{f^i_{(\beta)}(x)}\cdot
	           (\partial_{\bar{z}^i}   h_{\Bbb C})(f_{(0)}(x), \overline{f_{(0)}(x)})    \\
 && \hspace{2em}			
    +\,  \theta^1\theta^2 \vartheta^1 \vartheta^2\,
       \left\{\rule{0ex}{1.2em}\right.
          \sum_i f^i_{(12)}(x)\cdot
		     (\partial_{z^i} h_{\Bbb C})(f_{(0)}(x), \overline{f_{(0)}(x)})
		   -\,\sum_{i,j} f^i_{(1)}(x) f^j_{(2)}(x)   \cdot
			    (\partial_{z^i} \partial_{z^j}   h_{\Bbb C})(f_{(0)}(x), \overline{f_{(0)}(x)})
       \left.\rule{0ex}{1.2em}\right\}     \\
 && \hspace{2em}
   +\, \sum_{\alpha,\dot{\beta}} \theta^\alpha \bar{\theta}^{\dot{\beta}} \,
          \left\{\rule{0ex}{1.2em}\right.
	       \sqrt{-1}\sum_{\mu, i} \sigma^\mu_{\alpha\dot{\beta}}
		     \mbox{\LARGE $($}
			  \partial_\mu f^i_{(0)}(x)\cdot
			      (\partial_{z^i}h_{\Bbb C})(f_{(0)}(x), \overline{f_{(0)}(x)})
				- \partial_\mu \overline{f^i_{(0)}(x)}\cdot
				      (\partial_{\bar{z}^i} h_{\Bbb C})(f_{(0)}(x), \overline{f_{(0)}(x)})			
	         \mbox{\LARGE $)$} \\
   && \hspace{10em}
		+\,  \vartheta_\alpha \bar{\vartheta}_{\dot{\beta}}
		       \sum_{i,j} f^i_{(\alpha)}(x) \overline{f^j_{(\beta)}(x)}\cdot
			     (\partial_{z^i}\partial_{\bar{z}^j} h_{\Bbb C})(f_{(0)}(x), \overline{f_{(0)}(x)})
	      \left.\rule{0ex}{1.2em}\right\}   \\
 && \hspace{2em}			
    +\,  \bar{\theta}^{\dot{1}}\bar{\theta}^{\dot{2}}
	          \bar{\vartheta}^{\dot{1}} \bar{\vartheta}^{\dot{2}}\,
       \left\{\rule{0ex}{1.2em}\right.
          \sum_i \overline{f^i_{(12)}(x)}\cdot
		       (\partial_{\bar{z}^i}   h_{\Bbb C})(f_{(0)}(x), \overline{f_{(0)}(x)})
		   - \sum_{i,j}\overline{f^i_{(1)}(x)}\, \overline{ f^j_{(2)}(x)} \cdot
			    (\partial_{\bar{z}^i} \partial_{\bar{z}^j}h_{\Bbb C})
				                                                                    (f_{(0)}(x), \overline{f_{(0)}(x)})
       \left.\rule{0ex}{1.2em}\right\}     \\ 		
 && \hspace{2em}
   + \sum _{\dot{\beta}} \theta^1\theta^2\bar{\theta}^{\dot{\beta}}
       \left\{\rule{0ex}{1.2em}\right.
	     \sqrt{-1}\sum_{\alpha, \mu}
		   \vartheta_\alpha   \sigma^{\mu\alpha}_{\;\;\;\;\dot{\beta}}
	       \mbox{\LARGE $($}
		    \sum_i \partial_\mu f^i_{(\alpha)}(x)\cdot
			   (\partial_{z^i}h_{\Bbb C})(f_{(0)}(x), \overline{f_{(0)}(x)})  \\
      && \hspace{14em}			
		    +\,\sum_{i,j}
			      \mbox{\large $($}
				    f^i_{(\alpha)}(x)\,\partial_\mu f^j_{(0)}(x)
				  \mbox{\large $)$} \cdot
				   (\partial_{z^i}\partial_{z^j} h_{\Bbb C})(f_{(0)}(x), \overline{f_{(0)}(x)})   \\
      && \hspace{14em}				
		    -\, \sum_{i,j} f^i_{(\alpha)}(x)\,\partial_\mu \overline{f^j_{(0)}(x)}\cdot
			      (\partial_{z^i}\partial_{\bar{z}^j} h_{\Bbb C})
				                                                            (f_{(0)}(x), \overline{f_{(0)}(x)})
		   \mbox{\LARGE $)$}    \\
   && \hspace{6em}
      +\, \vartheta_1\vartheta_2 \bar{\vartheta}_{\dot{\beta}}
	          \mbox{\LARGE $($}
			    -\sum_{i,j} f^i_{(12)}(x)\overline{f^j_{(\beta)}(x)}\cdot
				   (\partial_{z^i}\partial_{\bar{z}^j} h_{\Bbb C})
				                                                            (f_{(0)}(x), \overline{f_{(0)}(x)})    \\
      && \hspace{12em}																			
                +  \sum_{i,j,k} f^i_{(1)}(x) f^j_{(2)}(x) \,\overline{f^k_{(\beta)}(x)} \cdot
				     (\partial_{z^i}\partial_{z^j}\partial_{\bar{z}^k} h_{\Bbb C})
				                                                            (f_{(0)}(x), \overline{f_{(0)}(x)})
			  \mbox{\LARGE $)$}
	   \left.\rule{0ex}{1.2em}\right\} \\
 && \hspace{2em}
   + \sum _{\alpha} \theta^\alpha \bar{\theta}^{\dot{1}}\bar{\theta}^{\dot{2}}
       \left\{\rule{0ex}{1.2em}\right.
	     \sqrt{-1}\sum_{\dot{\beta}, \mu}
		   \bar{\vartheta}_{\dot{\beta}}   \sigma^{\mu\dot{\beta}}_\alpha
	       \mbox{\LARGE $($}
		    - \sum_i \partial_\mu \overline{f^i_{(\beta)}(x)}\cdot
			   (\partial_{\bar{z}^i}h_{\Bbb C})(f_{(0)}(x), \overline{f_{(0)}(x)})  \\
      && \hspace{14em}			
		    -\, \sum_{i,j} \overline{f^i_{(\beta)}(x)}\,\partial_\mu \overline{f^j_{(0)}(x)}\cdot
				   (\partial_{\bar{z}^i}\partial_{\bar{z}^j}
				         h_{\Bbb C})(f_{(0)}(x), \overline{f_{(0)}(x)})   \\
      && \hspace{14em}				
		    +\, \sum_{i,j}\partial_\mu f^i_{(0)}(x)\,\overline{f^j_{(\beta)}(x)}\cdot
			      (\partial_{z^i}\partial_{\bar{z}^j} h_{\Bbb C})
				                                                            (f_{(0)}(x), \overline{f_{(0)}(x)})
		   \mbox{\LARGE $)$}    \\
   && \hspace{6em}
      +\, \vartheta_\alpha \bar{\vartheta}_{\dot{1}} \bar{\vartheta}_{\dot{2}}
	          \mbox{\LARGE $($}
			    \sum_{i,j} f^i_{(\alpha)}(x)\overline{f^j_{(12)}(x)}\cdot
				   (\partial_{z^i}\partial_{\bar{z}^j} h_{\Bbb C})
				                                                            (f_{(0)}(x), \overline{f_{(0)}(x)})    \\
      && \hspace{12em}			
                -\,\sum_{i,j,k}
				     f^i_{(\alpha)}(x)\, \overline{f^j_{(1)}(x)}\, \overline{f^k_{(2)}(x)} \cdot
				     (\partial_{z^i}\partial_{\bar{z}^j}\partial_{\bar{z}^k} h_{\Bbb C})
				                                                            (f_{(0)}(x), \overline{f_{(0)}(x)})
			  \mbox{\LARGE $)$}
	   \left.\rule{0ex}{1.2em}\right\}   \\
 \end{eqnarray*}
 }

 \bigskip
 \centerline{(formula continued to the next page)}
 
{\footnotesize
 \begin{eqnarray*}
 && \hspace{2em}
   +\, \theta^1\theta^2\bar{\theta}^{\dot{1}}\bar{\theta}^{\dot{2}}\,
   \left\{\rule{0ex}{1.2em}\right.\!
    - \sum_{i=1}^n
	     \square f^i_{(0)}(x)\cdot
        (\partial_{z^i}h_{\Bbb C})(f_{(0)}(x), \overline{f_{(0)}(x)})
	- \sum_i \square \overline{f^i_{(0)}(x)}\cdot		
              (\partial_{\bar{z}^j}h_{\Bbb C})(f_{(0)}(x), \overline{f_{(0)}(x)})		  \\
    && \hspace{10em}
      - \, \sum_{i,j, \mu, \nu}
          \eta^{\mu\nu}\,\partial_\mu f^i_{(0)}(x)\,\partial_\nu f^j_{(0)}(x)	\cdot
         (\partial_{z^i}\partial_{z^j}h_{\Bbb C})(f_{(0)}(x), \overline{f_{(0)}(x)})  \\				 		
    && \hspace{10em}
        +\,2  \sum_{i,j, \mu, \nu}
          \eta^{\mu\nu}\,\partial_\mu f^i_{(0)}(x)\,\partial_\nu \overline{f^j_{(0)}(x)}	\cdot
         (\partial_{z^i}\partial_{\bar{z}^j}h_{\Bbb C})(f_{(0)}(x), \overline{f_{(0)}(x)})  \\
    && \hspace{10em}
      -\, \sum_{i,j, \mu, \nu}
         \eta^{\mu\nu}\,
		    \partial_\mu \overline{f^i_{(0)}(x)}\,\partial_\nu \overline{f^j_{(0)}(x)} \cdot
             (\partial_{\bar{z}^i}\partial_{\bar{z}^j}h_{\Bbb C})
	                                           (f_{(0)}(x), \overline{f_{(0)}(x)}) \\				
 && \hspace{4em}
   +\, \sqrt{-1} \sum_{\alpha, \dot{\beta}, \mu}
           \vartheta_\alpha \bar{\vartheta}_{\dot{\beta}}\, \bar{\sigma}^{\mu\dot{\beta}\alpha}
        \mbox{\LARGE $($}
          \sum_{i,j}
		      \mbox{\large $($}
		       f^i_{(\alpha)}(x)\,\partial_\mu \overline{f^j_{(\beta)}(x)}	
			    -   \partial_\mu f^i_{(\alpha)}(x)\,\overline{f^j_{(\beta)}(x)}	
			  \mbox{\large $)$} \cdot
			(\partial_{z^i}\partial_{\bar{z}^j}h_{\Bbb C})(f_{(0)}(x), \overline{f_{(0)}(x)}) \\
      && \hspace{10em}			
	      -\,\sum_{i,j,k} \partial_\mu f^i_{(0)}(x) f^j_{(\alpha)}(x)\,
				\overline{f^k_{(\beta)}(x)}\cdot
			   (\partial_{z^i}\partial_{z^j}\partial_{\bar{z}^k}h_{\Bbb C})
			               (f_{(0)}(x), \overline{f_{(0)}(x)})  \\
      && \hspace{10em}						
	     +\,\sum_{i,j,k}
			   f^i_{(\alpha)}(x)
				  \overline{f^j_{(\beta)}(x)}\, \partial_\mu \overline{f^k_{(0)}(x)} \cdot
			   (\partial_{z^i}\partial_{\bar{z}^j}\partial_{\bar{z}^k}h_{\Bbb C})
			               (f_{(0)}(x), \overline{f_{(0)}(x)})	  	
        \mbox{\LARGE $)$}  \\
 && \hspace{4em}
   +\, \vartheta_1\vartheta_2 \bar{\vartheta}_{\dot{1}}\bar{\vartheta}_{\dot{2}}
    \mbox{\LARGE $($}
      \sum_{i,j}
           f^i_{(12)}(x)\,\overline{f^j_{(12)}(x)}    \cdot
           (\partial_{z^i}\partial_{\bar{z}^j}h_{\Bbb C})
	                  (f_{(0)}(x), \overline{f_{(0)}(x)})   \\
    && \hspace{10em}
     -\, \sum_{i,j,k}	 f^i_{(1)}(x)f^j_{(2)}(x) \,\overline{f^k_{(12)}(x)}   \cdot
            (\partial_{z^i}\partial_{z^j}\partial_{\bar{z}^k}h_{\Bbb C})
	            (f_{(0)}(x), \overline{f_{(0)}(x)})			 \\			 			 			 		 		 
    && \hspace{10em}
     -\, \sum_{i,j, k}
		   f^i_{(12)}(x) \,\overline{f^j_{(1)}(x)}\, \overline{f^k_{(2)}(x)} \cdot
           (\partial_{z^i}\partial_{\bar{z}^j}\partial_{\bar{z}^k}h_{\Bbb C})
	             (f_{(0)}(x), \overline{f_{(0)}(x)})			 \\			 			 		 		 
    && \hspace{10em}
     +\,\sum_{i,j, k, l}
		  f^i_{(1)}(x)f^j_{(2)}(x)\,
		   	\overline{f^k_{(1)}(x)}\,  \overline{f^l_{(2)}(x)}
		     \cdot(\partial_{z^i}\partial_{z^j}\partial_{\bar{z}^k}\partial_{\bar{z}^l}h_{\Bbb C})
	            (f_{(0)}(x), \overline{f_{(0)}(x)})
    \mbox{\LARGE $)$}		
  \left.\rule{0ex}{1.2em}\right\}\,,
\end{eqnarray*}
} 

\medskip

\noindent
after relabeling the $i, j, k, l$ indices of some of the terms and collecting the like terms.
In particular,
 as in Proposition~5.1.3,
 the components of $\breve{f}-f_{(0)}$
     in the $(\theta, \bar{\theta}, \vartheta, \bar{\vartheta})$-expansion
     are $C^\infty(X)^{\Bbb C}$-valued differential operators on $C^\infty(Y)^{\Bbb C}$    and
 the equality $(\breve{f}^\sharp(h) )^\dag = \breve{f}^\sharp(h)$ holds explicitly
     for $\breve{f}$ chiral and $h$ (real-valued) smooth function on $Y$.

\bigskip

\noindent
$(b)$\;{\it $h$ holomorphic on $Y={\Bbb C}^n$}

\medskip

\noindent
When $h\in C^\infty(Y)^{\Bbb C}$ is holomorphic on $Y$,
 $$
   \partial_{\bar{z}^i}h_{\Bbb C}\;
    =\; \partial_{z^i}\partial_{\bar{z}^j}h_{\Bbb C}\;
	=\; \partial_{z^i}\partial_{z^j}\partial_{\bar{z}^k}h_{\Bbb C}\;
	=\; \partial_{z^i}\partial_{\bar{z}^j}\partial_{\bar{z}^k}h_{\Bbb C}\;
	=\;	\partial_{z^i}\partial_{z^j}\partial_{\bar{z}^k}\partial_{\bar{z}^l}h_{\Bbb C}\;	
	=\; 0
 $$
 and
$\breve{f}^\sharp(h)$	 reduces to

{\footnotesize
\begin{eqnarray*}
 \lefteqn{
  \breve{f}^\sharp(h) \;\; =\;\;   \breve{f}^\sharp(h_{\Bbb C})\;\;
    =\;\;  h_{\Bbb C}(\breve{f}^1,\,\cdots\,,\, \breve{f}^n)     } \\
 && =\;\;
   h_{\Bbb C}(f_{(0)}(x))\,
    +  \sum_{\alpha, i}\theta^\alpha \vartheta_\alpha\,
	       f^i_{(\alpha)}(x)\cdot
	           (\partial_{z^i}   h_{\Bbb C})(f_{(0)}(x))  	 \\
 && \hspace{2em}			
    +\,  \theta^1\theta^2 \vartheta^1 \vartheta^2\,
       \left\{\rule{0ex}{1.2em}\right.
          \sum_i f^i_{(12)}(x)\cdot
		     (\partial_{z^i} h_{\Bbb C})(f_{(0)}(x))
		   - \sum_{i,j}
			     f^i_{(1)}(x) f^j_{(2)}(x)\cdot			
			    (\partial_{z^i} \partial_{z^j}   h_{\Bbb C})(f_{(0)}(x))
       \left.\rule{0ex}{1.2em}\right\}	      \\
 && \hspace{2em}
   +\, \sqrt{-1}\sum_{\alpha,\dot{\beta}, \mu, i} \theta^\alpha \bar{\theta}^{\dot{\beta}} \,
	       \sigma^\mu_{\alpha\dot{\beta}}\,\partial_\mu f^i_{(0)}(x)\cdot
			      (\partial_{z^i}h_{\Bbb C})(f_{(0)}(x))\\		
 && \hspace{2em}
   +\, \sqrt{-1}\sum _{\dot{\beta}, \alpha, \mu} \theta^1\theta^2\bar{\theta}^{\dot{\beta}}
		   \vartheta_\alpha   \sigma^{\mu\alpha}_{\;\;\;\;\dot{\beta}}
	       \left\{\rule{0ex}{1.2em}\right.
		    \sum_i \partial_\mu f^i_{(\alpha)}(x)\cdot
			  (\partial_{z^i}h_{\Bbb C})(f_{(0)}(x))
		    +\, \sum_{i,j}
				    f^i_{(\alpha)}(x)\,\partial_\mu f^j_{(0)}(x)	\cdot
			      (\partial_{z^i}\partial_{z^j} h_{\Bbb C})(f_{(0)}(x))   	   			
	   \left.\rule{0ex}{1.2em}\right\} \\   	
 && \hspace{2em}
   -\, \theta^1\theta^2\bar{\theta}^{\dot{1}}\bar{\theta}^{\dot{2}}\,
   \left\{\rule{0ex}{1.2em}\right.\!
    \sum_{i=1}^n
	     \square f^i_{(0)}(x)\cdot
        (\partial_{z^i}h_{\Bbb C})(f_{(0)}(x))			
      + \, \sum_{i,j, \mu, \nu}
          \eta^{\mu\nu}\,\partial_\mu f^i_{(0)}(x)\,\partial_\nu f^j_{(0)}(x)	\cdot
         (\partial_{z^i}\partial_{z^j}h_{\Bbb C})(f_{(0)}(x))
  \left.\rule{0ex}{1.2em}\right\}.			
\end{eqnarray*}
} 

This is the chiral superfield in $C^\infty(\widehat{X}^{\widehat{\boxplus}})^{\smallscriptsize}$
 determined by the four components
 {\footnotesize
 $$
  \mbox{\large $($}
   h_{\Bbb C}(f_{(0)}(x))\,,\;
   f^i_{(\alpha)}(x)\cdot
	           (\partial_{z^i}   h_{\Bbb C})(f_{(0)}(x))\,,\;  	
    \sum_i f^i_{(12)}(x)\cdot
		     (\partial_{z^i} h_{\Bbb C})(f_{(0)}(x))
		   - \sum_{i,j}
			     f^i_{(1)}(x) f^j_{(2)}(x)\cdot			
			    (\partial_{z^i} \partial_{z^j}   h_{\Bbb C})(f_{(0)}(x))
  \mbox{\large $)$}_{\alpha, =1, 2}\,.
 $$} 

Its twisted complex conjugate $\breve{f}^\sharp(h)^\dag$ is antichiral and is given by

{\footnotesize
\begin{eqnarray*}
 \lefteqn{  (\breve{f}^\sharp(h))^\dag\;\;
  =\;\;    \bar{h}_{\Bbb C}( (\breve{f}^\sharp(z^1)^\dag,\,\cdots\,,\,(\breve{f}^\sharp(z^n)^\dag)   \;\;
  =\;\;    \bar{h}_{\Bbb C}(\breve{f}^{1\dag},\,\cdots\,,\,\breve{f}^{n\dag})}\\[1.2ex]  		
 && =\;\;
   \bar{h}_{\Bbb C}(\overline{f_{(0)}(x)})
   - \sum_{\dot{\beta}}\bar{\theta}^{\dot{\beta}} \bar{\vartheta}_{\dot{\beta}}
        \sum_i  \overline{f^i_{(\beta)}(x)}\cdot
		                 (\partial_{\bar{z}^i}\bar{h}_{\Bbb C})(\overline{f_{(0)}(x)})\\
 && \hspace{1.6em}
   +\, \bar{\theta}^{\dot{1}}\bar{\theta}^{\dot{2}}
           \bar{\vartheta}_{\dot{1}} \bar{\vartheta}_{\dot{2}}
        \,\mbox{\Large $($}
		    \sum_i  \overline{f^i_{(12)}(x)}\cdot
			                 (\partial_{\bar{z}^i}\bar{h}_{\Bbb C})(\overline{f_{(0)}(x)})\,
		     -\,\sum_{i,j}
				    \overline{f^i_{(1)}(x)}\, \overline{f^j_{(2)}(x)}  \cdot
           		      (\partial_{\bar{z}^i}\partial_{\bar{z}^j}\bar{h}_{\Bbb C})
			              (\overline{f_{(0)}(x)})
	   \!\mbox{\Large $)$}  \\
 && \hspace{1.6em}
  -\, \sqrt{-1}\sum_{\alpha,\dot{\beta}; i,\mu}
         \theta^\alpha \bar{\theta}^{\dot{\beta}} \sigma^\mu_{\alpha\dot{\beta}}
			  \partial_\mu \overline{f^i_{(0)}(x)}  \cdot
			   (\partial_{\bar{z}^i}\bar{h}_{\Bbb C})(\overline{f_{(0)}(x)})\\
 && \hspace{1.6em}
  -\, \sqrt{-1}\sum_{\alpha; \dot{\beta}, \mu}
        \theta^\alpha \bar{\theta}^{\dot{1}} \bar{\theta}^{\dot{2}}
         \bar{\vartheta}_{\dot{\beta}} \sigma^{\mu\dot{\beta}}_\alpha  		
		\mbox{\Large $($}
		  \sum_i 		
		     \partial_\mu \overline{f^i_{(\beta)}(x)}	\cdot
			 (\partial_{\bar{z}^i}\bar{h}_{\Bbb C})(\overline{f_{(0)}(x)})
         + \sum_{i,j}
		        \overline{f^i_{(\beta)}(x)}\,\partial_\mu \overline{f^j_{(0)}(x)}\cdot
                 (\partial_{\bar{z}^i}\partial_{\bar{z}^j}\bar{h}_{\Bbb C})
				              (\overline{f_{(0)}(x)})		
		\!\mbox{\Large $)$}\\
 && \hspace{1.6em}
  -\, \theta^1\theta^2 \bar{\theta}^{\dot{1}} \bar{\theta}^{\dot{2}}
       \mbox{\Large $($}
        \sum_i
		   \square \overline{f^i_{(0)}(x)} \cdot
		   (\partial_{\bar{z}^i}\bar{h}_{\Bbb C})(\overline{f_{(0)}(x)})
        + \sum_{i,j,\mu,\nu}		
			  \eta^{\mu\nu}\,
			   \partial_\mu \overline{f^i_{(0)}(x)}\,\partial_\nu \overline{f^j_{(0)}(x)}\cdot
			(\partial_{\bar{z}^i}\partial_{\bar{z}^j}\bar{h}_{\Bbb C})(\overline{f_{(0)}(x)})
       \mbox{\Large $)$}\,.
\end{eqnarray*}
}

\bigskip

\subsection{The action functional for chiral maps\\ --- $d=3+1$, $N=1$ nonlinear sigma models with a superpotential}

With the mathematical background/preparation in Sec.\:5.1 and Sec.\:5.2,
 we can now reconstruct
  \begin{itemize}
   \item[\LARGE $\cdot$]
    [Wess \& Bagger: Chapter XXII.\:{\it Chiral models and K\"{a}hler geometry}\,]
  \end{itemize}
 in the complexified ${\Bbb Z}/2$-graded $C^\infty$-Algebraic Geometry setting of Sec.\:1.

\bigskip

\begin{definition}{\bf [action functional for chiral maps]}\;  {\rm
 Let
  \begin{itemize}
   \item[\LARGE $\cdot$]
    $Y=({\Bbb C}^n, h)$ be the K\"{a}hler manifold ${\Bbb C}^n$
     with complex coordinates\\
 	 $(z^1, \,\cdots\,, z^n)=(y^1+\sqrt{-1}y^2\,, \,\cdots\,,\, y^{2n-1}+\sqrt{-1}y^{2n})$  and 	
	 a K\"{a}hler metric determined by\\ a K\"{a}hler potential:
	  a pluri-subharmonic function $h\in C^\infty({\Bbb C}^n)$,
	
    \item[\LARGE $\cdot$]	
	  $W$ be a holomorphic function on $Y$,   and
	
    \item[\LARGE $\cdot$]	
    $\breve{f}: \widehat{X}^{\widehat{\boxplus}, \smallscriptsize} \rightarrow Y$ be a chiral map,
     with the underlying $C^\infty$-ring-homomorphism
     $$
      \breve{f}^\sharp\,:\; C^\infty(Y)\;
	    \longrightarrow\; C^\infty(\widehat{X}^{\widehat{\boxplus}})^\smallscriptsize\,.
     $$
  \end{itemize}
 Then, as a $4$-dimensional sigma model on $Y$,
  the $N=1$ supersymmetric action functional for $\breve{f}$'s  is given by the action functional
  \marginpar{\vspace{.6em}\raggedright\tiny
         \raisebox{-1ex}{\hspace{-2.4em}\LARGE $\cdot$}Cf.\,[\,\parbox[t]{8em}{Wess
		\& Bagger:\\ Eq.\:(22.1)].}}
  $$
   S^{(h, W)}(\breve{f})\;=\;
    \int_{X}\!d^{\,4}x\!
    \left(-\frac{1}{4}
    \int\! d\bar{\theta}^{\dot{2}}\,d\bar{\theta}^{\dot{1}}\,d\theta^2\, d\theta^1
      \breve{f}^\sharp(h)
 	  + \mbox{\LARGE $[$}\!
	        \int\! d\theta^2 d\theta^1
               \breve{f}^\sharp(W)		  	     	
           - \int\! d\bar{\theta}^{\bar{2}}d\bar{\theta}^{\dot{1}}	
		        \mbox{\large $($}\breve{f}^\sharp(W)\!\mbox{\large $)$}^\dag
   	      \mbox{\LARGE $]$}	
	  \!\right)\!.
  $$
 Here the factor $\,-\frac{1}{4}\,$ is added
  so that in the end the kinetic term for $f_{(0)}:X\rightarrow Y$ would have the correct/conventional coefficient.\;\;
 (Cf.\:[Wess \& Bagger: Eq.\:(22.1)].)
}\end{definition}

\vspace{10em}
\bigskip
 
\begin{flushleft}
{\bf The K\"{a}hler potential part}
\end{flushleft}
{\small
\begin{eqnarray*}
 \lefteqn{
  \breve{f}^\sharp(h)\;\;
   =\;\;
     \mbox{{\Large $($}terms of total $(\theta,\bar{\theta})$-degree = 0, 1, 2, 3{\Large $)$}} } \\
 && \hspace{2em}
   +\, \theta^1\theta^2\bar{\theta}^{\dot{1}}\bar{\theta}^{\dot{2}}\,
   \left\{\rule{0ex}{1.2em}\right.
    - \sum_i  (\partial_{z^i}h_{\Bbb C})(f_{(0)}(x), \overline{f_{(0)}(x)})
	     \cdot  \square f^i_{(0)}(x)    	
    - \sum_i (\partial_{\bar{z}^i}h_{\Bbb C})(f_{(0)}(x), \overline{f_{(0)}(x)})
	     \cdot \square\,\overline{ f^i_{(0)}(x)}        \\
     && \hspace{10em}		
	   -   \sum_{i,j}(\partial_{z^i}\partial_{z^j}h_{\Bbb C})
	            (f_{(0)}(x), \overline{f_{(0)}(x)})  \cdot
		     \sum_{\mu, \nu} \eta^{\mu\nu}
		          \partial_\mu f^i_{(0)}(x)\, \partial_\nu f^j_{(0)}(x)				    \\
     && \hspace{10em}
       - \sum_{i,j}(\partial_{\bar{z}^i}\partial_{\bar{z}^j}h_{\Bbb C})
	            (f_{(0)}(x), \overline{f_{(0)}(x)})  \cdot
		  \sum_{\mu, \nu} \eta^{\mu\nu}
		          \partial_\mu \overline{f^i_{(0)}(x)}\, \partial_\nu \overline{f^j_{(0)}(x)}     \\	
 &&  \hspace{3em}
    + \sum_{i,j}(\partial_{z^i}\partial_{\bar{z}^j}h_{\Bbb C})
	            (f_{(0)}(x), \overline{f_{(0)}(x)})
	    \cdot  \mbox{\LARGE $($}
		  2 \sum_{\mu, \nu} \eta^{\mu\nu}
		          \partial_\mu f^i_{(0)}(x)\, \partial_\nu \overline{f^j_{(0)}(x)}      \\
     && \hspace{8em}
        +\, \sqrt{-1}\, \sum_{\alpha, \dot{\beta},\mu}
		        \vartheta_\alpha \bar{\vartheta}_{\dot{\beta}} \, \bar{\sigma}^{\mu\dot{\beta}\alpha}
				  \mbox{\Large $($}
				   f^i_{(\alpha)}(x)\, \partial_\mu \overline{f^j_{(\beta)}(x)}
				   -  \partial_\mu f^i_{(\alpha)}(x)\,\overline{f^j_{(\beta)}(x)}				
				  \mbox{\Large $)$}   \\
     && \hspace{8em}
          	 +\,\vartheta_1\vartheta_2 \bar{\vartheta}_{\dot{1}}  \bar{\vartheta}_{\dot{2}}
 			        f^i_{(12)}(x)\,\overline{f^j_{(12)}(x)}			
		    \mbox{\LARGE $)$}				 \\[1.2ex]
 && \hspace{3em}
  -\, \sum_{i,j,k}(\partial_{z^i}\partial_{z^j}\partial_{\bar{z}^k}h_{\Bbb C})
	            (f_{(0)}(x), \overline{f_{(0)}(x)})
	    \cdot    \\[-1.2ex]
     && \hspace{7em}		
         \mbox{\LARGE $($}		
          \sqrt{-1}\,\sum_{\alpha, \dot{\beta}, \mu}
		    \vartheta_\alpha \bar{\vartheta}_{\dot{\beta}} \, \bar{\sigma}^{\mu\dot{\beta}\alpha}			            
		       \partial_\mu f^i_{(0)}(x) f^j_{(\alpha)}(x)\, \overline{f^k_{(\beta)}(x)}		
	      +\,\vartheta_1\vartheta_2 \bar{\vartheta}_{\dot{1}} \bar{\vartheta}_{\dot{2}}\,
	           f^i_{(1)}(x)f^j_{(2)}(x) \, \overline{f^k_{(12)}(x)}
        \mbox{\LARGE $)$}     \\[1.2ex] 		
 && \hspace{3em}
  +\,
       \sum_{i,j,k}(\partial_{z^i}\partial_{\bar{z}^j}\partial_{\bar{z}^k}h_{\Bbb C})
	            (f_{(0)}(x), \overline{f_{(0)}(x)})
		\cdot  \\[-1.2ex]
   && \hspace{7em}		
         \mbox{\LARGE $($}		
         \sqrt{-1}\sum_{\alpha, \dot{\beta}, \mu}
		   \vartheta_\alpha \bar{\vartheta}_{\dot{\beta}}\,\bar{\sigma}^{\mu\dot{\beta}\alpha}\,
		    f^i_{(\alpha)}(x)\,
		     \overline{f^j_{(\beta)}(x)}\,\partial_\mu \overline{f_{(0)}^k(x)}
          -\, \vartheta_1\vartheta_2 \bar{\vartheta}_{\dot{1}} \bar{\vartheta}_{\dot{2}}\,
    	     f^i_{(12)}(x) \cdot
	           \overline{f^j_{(1)}(x)}\, \overline{f^k_{(2)}(x)}			  	
         \!\mbox{\LARGE $)$}   \\[1.2ex] 						
 && \hspace{3em}
  +\, \sum_{i,j, k, l}
	    (\partial_{z^i}\partial_{z^j}\partial_{\bar{z}^k}\partial_{\bar{z}^l}h_{\Bbb C})
	            (f_{(0)}(x), \overline{f_{(0)}(x)})\,
		  \vartheta_1 \vartheta_2 \bar{\vartheta}_{\dot{1}} \bar{\vartheta}_{\dot{2}}\cdot
		  f^i_{(1)}(x)f^j_{(2)}(x)\,  		
		   	\overline{f^k_{(1)}(x)}\,  \overline{f^l_{(2)}(x)}
  \left.\rule{0ex}{1.2em}\right\}.			
\end{eqnarray*}
} 
 
\smallskip
 
First observe that the sum of the first four summations has a conversion
{\small
\begin{eqnarray*}
 \lefteqn{
  - \sum_{i=1}^n
      (\partial_{z^i}h_{\Bbb C})(f_{(0)}(x), \overline{f_{(0)}(x)})
		  \cdot  \square f^i_{(0)}(x)  		
  - \sum_{j=1}^n
      (\partial_{\bar{z}^j}h_{\Bbb C})(f_{(0)}(x), \overline{f_{(0)}(x)})
		    \cdot   \square \overline{f^i_{(0)}(x)}		     } \\
 && \hspace{3em}
  - \, \sum_{i,j}(\partial_{z^i}\partial_{z^j}h_{\Bbb C})(f_{(0)}(x), \overline{f_{(0)}(x)})
	    \cdot
         \sum_{\mu,\nu}\eta^{\mu\nu}\,\partial_\mu f^i_{(0)}(x)\,\partial_\nu f^j_{(0)}(x)		\\		
 && \hspace{3em}
   -\, \sum_{i,j}(\partial_{\bar{z}^i}\partial_{\bar{z}^j}h_{\Bbb C})
	            (f_{(0)}(x), \overline{f_{(0)}(x)})
	    \cdot  \sum_{\mu,\nu}\eta^{\mu\nu}\,\partial_\mu \overline{f^i_{(0)}(x)}\,
			                                         \partial_\nu \overline{f^j_{(0)}(x)} 		  \\[2ex] 		
 && =\;\;
  -\,\sum_{\mu,\nu} \eta^{\mu\nu} \partial_\mu
       \mbox{\LARGE $($}
	    \sum_i
		 (\partial_{z^i}h_{\Bbb C})(f_{(0)}(x), \overline{f_{(0)}(x)})
		      \cdot \partial_\nu f^i_{(0)}(x)
	    + \sum_j
		    (\partial_{\bar{z}^j}h_{\Bbb C})(f_{(0)}(x), \overline{f_{(0)}(x)})
		      \cdot \partial_\nu \overline{f^j_{(0)}(x)}		
	   \mbox{\LARGE $)$}   	   \\
 && \hspace{2em}
  +\, 2\,\sum_{i,j}(\partial_{z^i}\partial_{\bar{z}^j}h_{\Bbb C})
	            (f_{(0)}(x), \overline{f_{(0)}(x)})
		 \cdot
		 \sum_{\mu, \nu} \eta^{\mu\nu}
		          \partial_\mu f^i_{(0)}(x)\, \partial_\nu \overline{f^j_{(0)}(x)}\,.
\end{eqnarray*}The
} 
first summation $-\sum_{\mu,\nu}\eta^{\mu\nu}\partial_\mu(\cdots_{\,\nu})$
 gives rise to a boundary-term of the action functional $S^h(\breve{f})$ on the space-time $X$.
Second and optionally,
 one may choose to present the kinetic term for spinor fields with only $\partial_\mu f^i_{(\alpha)}$'s
 or only $\partial\mu \overline{f^i_{(\beta)}}$'s.
Suppose the former, then
{\small
\begin{eqnarray*}
 \lefteqn{
   \sqrt{-1}\,\sum_{i,j}(\partial_{z^i}\partial_{\bar{z}^j}h_{\Bbb C})
	            (f_{(0)}(x), \overline{f_{(0)}(x)})
		\sum_{\alpha, \dot{\beta}, \mu}
		  \vartheta_\alpha\bar{\vartheta}_{\dot{\beta}}\,\bar{\sigma}^{\mu\dot{\beta}\alpha}\,
		   f^i_{(\alpha)}(x)\,\partial_\mu \overline{f^j_{(\beta)}(x)}
     } \\
 && =\;\;
   \sqrt{-1}\,\sum_\mu \partial_\mu
     \mbox{\LARGE $($}
	   \sum_{i,j}
	     (\partial_{z^i}\partial_{\bar{z}^j}h_{\Bbb C})(f_{(0)}(x), \overline{f_{(0)}(x)})
		 \sum_{\alpha, \dot{\beta}}
		   \vartheta_\alpha\bar{\vartheta}_{\dot{\beta}}\,\bar{\sigma}^{\mu\dot{\beta}\alpha}\,
		   f^i_{(\alpha)}(x)\,\overline{f^j_{(\beta)}(x)}
	 \mbox{\LARGE $)$}
   \\
 && \hspace{3em}
   -\,\sqrt{-1}\,\sum_{i,j, k}(\partial_{z^i}\partial_{z^j}\partial_{\bar{z}^k}h_{\Bbb C})
	              (f_{(0)}(x), \overline{f_{(0)}(x)})
		\sum_{\alpha, \dot{\beta}, \mu}
		  \vartheta_\alpha\bar{\vartheta}_{\dot{\beta}}\,\bar{\sigma}^{\mu\dot{\beta}\alpha}\,
		   \partial_\mu f^j_{(0)}(x) f^j_{(\alpha)}(x)\, \overline{f^k_{(\beta)}(x)}
    \\
 && \hspace{3em}
   -\,\sqrt{-1}\,\sum_{i,j, k}(\partial_{z^i}\partial_{\bar{z}^j}\partial_{\bar{z}^k}h_{\Bbb C})
	              (f_{(0)}(x), \overline{f_{(0)}(x)})
		\sum_{\alpha, \dot{\beta}, \mu}
		  \vartheta_\alpha\bar{\vartheta}_{\dot{\beta}}\,\bar{\sigma}^{\mu\dot{\beta}\alpha}\,
		    f^j_{(\alpha)}(x) f^j_{(\beta)}(x)\, \partial_\mu\, \overline{f^k_{(0)}(x)}
    \\	
 && \hspace{3em}
   -\, \sqrt{-1}\,\sum_{i,j}(\partial_{z^i}\partial_{\bar{z}^j}h_{\Bbb C})
	            (f_{(0)}(x), \overline{f_{(0)}(x)})
		\sum_{\alpha, \dot{\beta}, \mu}
		  \vartheta_\alpha\bar{\vartheta}_{\dot{\beta}}\,\bar{\sigma}^{\mu\dot{\beta}\alpha}\,
		    \partial_\mu f^i_{(\alpha)}(x)\,\overline{f^j_{(\beta)}(x)}  \,,
\end{eqnarray*}which
} 
contribute another boundary-term to the action functional $S^h(\breve{f})$ on the space-time $X$.
 
\smallskip
 
Thus,
setting
{\small
\begin{eqnarray*}
  \lefteqn{
   \mbox{(space-time boundary terms)}
    }\\
  && =\;\;
    -\,\sum_{\mu,\nu} \eta^{\mu\nu} \partial_\mu
       \mbox{\LARGE $($}
	    \sum_i
		 (\partial_{z^i}h_{\Bbb C})(f_{(0)}(x), \overline{f_{(0)}(x)})
		      \cdot \partial_\nu f^i_{(0)}(x)
	    + \sum_j
		    (\partial_{\bar{z}^j}h_{\Bbb C})(f_{(0)}(x), \overline{f_{(0)}(x)})
		      \cdot \partial_\nu \overline{f^j_{(0)}(x)}		
  	   \mbox{\LARGE $)$}   	   \\
  && \hspace{3em} 	
  +\, \sqrt{-1}\,\sum_\mu \partial_\mu
     \mbox{\LARGE $($}
	   \sum_{i,j}
	     (\partial_{z^i}\partial_{\bar{z}^j}h_{\Bbb C})(f_{(0)}(x), \overline{f_{(0)}(x)})
		 \sum_{\alpha, \dot{\beta}}
		   \vartheta_\alpha\bar{\vartheta}_{\dot{\beta}}\,\bar{\sigma}^{\mu\dot{\beta}\alpha}\,
		   f^i_{(\alpha)}(x)\,\overline{f^j_{(\beta)}(x)}
	 \mbox{\LARGE $)$}             	
\end{eqnarray*}}then, 
in a final form
{\small
\begin{eqnarray*}
 \lefteqn{
  \breve{f}^\sharp(h)\;\;
   =\;\;
     \mbox{{\Large $($}terms of total $(\theta,\bar{\theta})$-degree = 0, 1, 2, 3{\Large $)$}}
	 + \mbox{(space-time boundary terms)}  } \\
 && \hspace{2em}
   +\, \theta^1\theta^2\bar{\theta}^{\dot{1}}\bar{\theta}^{\dot{2}}\,
   \left\{\rule{0ex}{1.2em}\right.\!
    \sum_{i,j}(\partial_{z^i}\partial_{\bar{z}^j}h_{\Bbb C})
	            (f_{(0)}(x), \overline{f_{(0)}(x)})
	    \cdot  \mbox{\LARGE $($}\;
		  4 \sum_{\mu, \nu} \eta^{\mu\nu}
		          \partial_\mu f^i_{(0)}(x)\, \partial_\nu \overline{f^j_{(0)}(x)}      \\
     && \hspace{8em}
        -\,2\, \sqrt{-1}\, \sum_{\alpha, \dot{\beta},\mu}
		        \vartheta_\alpha \bar{\vartheta}_{\dot{\beta}} \, \bar{\sigma}^{\mu\dot{\beta}\alpha}				 		
				      \partial_\mu f^i_{(\alpha)}(x)\,\overline{f^j_{(\beta)}(x)}				
          	 + \vartheta_1\vartheta_2 \bar{\vartheta}_{\dot{1}}  \bar{\vartheta}_{\dot{2}}\,
 			        f^i_{(12)}(x)\,\overline{f^j_{(12)}(x)}			
		    \mbox{\LARGE $)$}				 \\[1.2ex]
 && \hspace{3em}
  -\,\sum_{i,j,k}(\partial_{z^i}\partial_{z^j}\partial_{\bar{z}^k}h_{\Bbb C})
	            (f_{(0)}(x), \overline{f_{(0)}(x)})
	    \cdot    \\[-1.2ex]
     && \hspace{7em}		
         \mbox{\LARGE $($}		
          2\,\sqrt{-1}\,\sum_{\alpha, \dot{\beta}, \mu}
		    \vartheta_\alpha \bar{\vartheta}_{\dot{\beta}} \, \bar{\sigma}^{\mu\dot{\beta}\alpha}			
              \partial_\mu f^i_{(0)}(x) f^j_{(\alpha)}(x)\, \overline{f^k_{(\beta)}(x)}		
	      + \vartheta_1\vartheta_2 \bar{\vartheta}_{\dot{1}} \bar{\vartheta}_{\dot{2}}\,	
	           f^i_{(1)}(x)f^j_{(2)}(x)\, \overline{f^k_{(12)}(x)}
        \mbox{\LARGE $)$}     \\[1.2ex] 		
 && \hspace{3em}
  -\,
       \sum_{i,j,k}(\partial_{z^i}\partial_{\bar{z}^j}\partial_{\bar{z}^k}h_{\Bbb C})
	            (f_{(0)}(x), \overline{f_{(0)}(x)})\cdot
        \vartheta_1\vartheta_2 \bar{\vartheta}_{\dot{1}} \bar{\vartheta}_{\dot{2}}\,
    	     f^i_{(12)}(x) \, \overline{f^j_{(1)}(x)}\, \overline{f^k_{(2)}(x)}			   	
         \!\mbox{\LARGE $)$}   \\[1.2ex] 						
 && \hspace{3em}
  +\, \sum_{i,j, k, l}
	      (\partial_{z^i}\partial_{z^j}\partial_{\bar{z}^k}\partial_{\bar{z}^l}h_{\Bbb C})
	            (f_{(0)}(x), \overline{f_{(0)}(x)})\cdot
		  \vartheta_1 \vartheta_2 \bar{\vartheta}_{\dot{1}} \bar{\vartheta}_{\dot{2}}\, 		
		  f^i_{(1)}(x)f^j_{(2)}(x)\, \overline{f^k_{(1)}(x)}\,  \overline{f^l_{(2)}(x)}			
  \left.\rule{0ex}{1.2em}\right\},			
\end{eqnarray*}
} 
 \marginpar{\vspace{-18em}\raggedright\tiny
         \raisebox{-1ex}{\hspace{-2.4em}\LARGE $\cdot$}Cf.\,[\,\parbox[t]{8em}{Wess
		\& Bagger:\\ Eq.\:(22.7)].}}
\bigskip

\noindent
which is the re-writing of [Wess \& Bagger: Eq.\:(22.7)] in the setting of current notes:

 {\center\footnotesize
 $$
 \hspace{-1em}\begin{array}{ccl}
  \underline{\raisebox{-1ex}{\hspace{2em}}
     \mbox{[W-B: Ch.\:XXII, Eq.\,(22.7), in $\theta\theta\bar{\theta}\bar{\theta}$]}\hspace{2em}}
    & \underline{\raisebox{-1ex}{\hspace{2em}}
	        \mbox{final explicit formula for $\breve{f}^\sharp(h)$
			               in $\theta^1\theta^2\bar{\theta}^{\dot{1}}\bar{\theta}^{\dot{2}}$}\hspace{2em}}
						   \\[2ex]
  \mbox{\large $\frac{\partial^2 K_{NM}|}{\partial A^i \partial A^{\ast j}}$}F^i F^{\ast j}	
    &
	   \sum_{i,j}\partial_{z^i}\partial_{\bar{z}^j}h_{\Bbb C}\cdot
	    \vartheta_1\vartheta_2\bar{\vartheta}_{\dot{1}} \bar{\vartheta}_{\dot{2}}\cdot
		 f^i_{(12)}\,\overline{f^j_{(12)}}  \\[2ex]
    \raisebox{.2em}{
     $-\,\mbox{\large $\frac{1}{2}
            \frac{\partial^3 K_{NM}|}{\partial A^i \partial A^{\ast j}\partial A^{\ast k}}$}\,
		   F^i\bar{\chi}^j \bar{\chi}^k
      -\,\mbox{\large $\frac{1}{2}
        \frac{\partial^3 K_{NM}|}{\partial A^{\ast i} \partial A^j \partial A^k}$}\,
		   F^{\ast i}\chi^j \chi^k$
		     }
    &
	      \begin{array}{l}	
	       -\,\sum_{i,j,k}\partial_{z^i}\partial_{z^j}\partial_{\bar{z}^k}h_{\Bbb C}\cdot
			   \vartheta_1\vartheta_2\bar{\vartheta}_{\dot{1}} \bar{\vartheta}_{\dot{2}}\,
		      f^i_{(1)}f^j_{(2)} \overline{f^k_{(12)}}     \\[.6ex]
		  \hspace{2em}
          -\,\sum_{i,j,k}\partial_{z^i}\partial_{\bar{z}^j}\partial_{\bar{z}^k}h_{\Bbb C} \cdot
	              \vartheta_1\vartheta_2\bar{\vartheta}_{\dot{1}} \bar{\vartheta}_{\dot{2}}\,
		        f^i_{(12)}(x)\,\overline{f^j_{(1)}}\, \overline{f^k_{(2)}}			       	
	      \end{array}   \\[3ex]
   \raisebox{-.2em}{
     $+\, \mbox{\large $\frac{1}{4}
           \frac{\partial^4 K_{NM}|}{\partial A^i\partial A^j\partial A^{\ast k}\partial A^{\ast l}}$}\,
		    \chi^i\chi^j\bar{\chi}^k\bar{\chi}^l$
			}
	&
      \rule{0ex}{2em}
	   +\,\sum_{i,j, k, l}
		      \partial_{z^i}\partial_{z^j}\partial_{\bar{z}^k}\partial_{\bar{z}^l}h_{\Bbb C}\cdot
	             \vartheta_1\vartheta_2\bar{\vartheta}_{\dot{1}} \bar{\vartheta}_{\dot{2}}\,
		        f^i_{(1)}f^j_{(2)}\,\overline{f^k_{(1)}}\,  \overline{f^l_{(2)}}		
	   \\[1.2ex]	
  \raisebox{0em}{
   $-\,\mbox{\large $\frac{\partial^2 K_{NM}|}{\partial A^i \partial A^{\ast j}}$}\,
      \partial_mA^i \partial^m A^{\ast j}$
	  }
   &
      \rule{0ex}{2em}
      +\, 4\, \sum_{i,j}\partial_{z^i}\partial_{\bar{z}^j}h_{\Bbb C}
		 \cdot
		 \sum_{\mu, \nu} \eta^{\mu\nu}
		          \partial_\mu f^i_{(0)}\, \partial_\nu \overline{f^j_{(0)}}
       \\[2ex]
   \raisebox{0em}{
   $-\,\sqrt{-1}\,\mbox{\large
       $\frac{\partial^2 K_{NM}|}{\partial A^i \partial A^{\ast j}}$}\,
	   \bar{\chi}^j \bar{\sigma}^m \partial_m \chi^i$
	   }
	&
        \rule{0ex}{2em}
	   -\,2\, \sqrt{-1}\sum_{i,j}\partial_{z^i}\partial_{\bar{z}^j}h_{\Bbb C}\cdot
		   \sum_{\alpha,\dot{\beta},\mu} \vartheta_\alpha \bar{\vartheta}_{\dot{\beta}}\,
		    \bar{\sigma}^{\mu\dot{\beta}\alpha}\,  		
				\partial_\mu f^i_{(\alpha)}  \overline{f^j_{(\beta)}}		
	   \\[1.2ex]		
  \raisebox{0em}{
   $-\,\sqrt{-1}\, \mbox{\large
         $\frac{\partial^3 K_{NM}|}{\partial A^i \partial A^j \partial A^{\ast k}}$}\,
	      \bar{\chi}^k \bar{\sigma}^m\chi^i \partial_m A^j $
	}
  &
  \hspace{-2em}
   \rule{0ex}{2em}
   - \,2 \sqrt{-1}\,
       \sum_{i,j,k}\partial_{z^i}\partial_{z^j}\partial_{\bar{z}^k}h_{\Bbb C}
	    \cdot
		 \sum_{\alpha, \dot{\beta}, \mu}\vartheta_\alpha \bar{\vartheta}_{\dot{\beta}}\,
		    \bar{\sigma}^{\mu\dot{\beta}\alpha}\,
		  \partial_\mu f^i_{(0)} f_{(\alpha)}^j\, \overline{f^k_{(\beta)}}
 \end{array}
 $$
 } 

\bigskip

\begin{flushleft}
{\bf The superpotential part}
\end{flushleft}
Let $W=W(z^1,\,\cdots\,,\,z^n)$ be a holomorphic function on $Y$.
It follows from Sec.\:5.2 that
 
{\footnotesize
\begin{eqnarray*}
 \lefteqn{\breve{f}^\sharp(W)\;\;
  =\;\;    W(\breve{f}^\sharp(z^1),\,\cdots\,,\,\breve{f}^\sharp(z^n))   \;\;
  =\;\;    W(\breve{f}^1,\,\cdots\,,\,\breve{f}^n)}\\[1.2ex]  		
 && =\;\;
   W(f_{(0)}(x))
   + \sum_\alpha\theta^\theta\vartheta_\alpha
        \sum_i (\partial_{z^i}W)(f_{(0)}(x))\cdot f^i_{(\alpha)}(x)\\
 && \hspace{1.6em}
   +\, \theta^1\theta^2\vartheta_1\vartheta_2
        \,\mbox{\Large $($}
		    \sum_i (\partial_{z^i}W)(f_{(0)}(x))\cdot f^i_{(12)}(x)\,
		     -\,\sum_{i,j}
			     (\partial_{z^i}\partial_{z^j}W)(f_{(0)}(x))
			 	 \cdot  f^i_{(1)}(x)f^j_{(2)}(x)
	   \!\mbox{\Large $)$}  \\
 && \hspace{1.6em}
  +\, \sqrt{-1}\sum_{\alpha,\dot{\beta}; i,\mu}
         \theta^\alpha \bar{\theta}^{\dot{\beta}} \sigma^\mu_{\alpha\dot{\beta}}
			\cdot (\partial_{z^i}W)(f_{(0)}(x))\cdot \partial_\mu f^i_{(0)}(x)  \\
 && \hspace{1.6em}
  +\, \sqrt{-1}\sum_{\dot{\beta}, \alpha, \mu}
        \theta^1 \theta^2 \bar{\theta}^{\dot{\beta}}
         \vartheta_\alpha \sigma^{\mu\alpha}_{\;\;\;\;\dot{\beta}}   		
		\mbox{\Large $($}
		  \sum_i (\partial_{z^i}W)(f_{(0)}(x))\cdot \partial_\mu f^i_{(\alpha)}(x)				
         + \sum_{i,j}(\partial_{z^i}\partial_{z^j}W)(f_{(0)}(x))
                 \cdot   f^i_{(\alpha)}(x)\,\partial_\mu f^j_{(0)}(x)		
		\!\mbox{\Large $)$}\\
 && \hspace{1.6em}
  -\, \theta^1\theta^2 \bar{\theta}^{\dot{1}} \bar{\theta}^{\dot{2}}
       \mbox{\Large $($}
        \sum_i(\partial_{z^i}W)(f_{(0)}(x))\cdot \square f^i_{(0)}(x)
        + \sum_{i,j,\mu,\nu}
		     (\partial_{z^i}\partial_{z^j}W)(f_{(0)}(x))\cdot \eta^{\mu\nu}\,
			   \partial_\mu f^i_{(0)}(x)\,\partial_\nu f^j_{(0)}(x)
       \mbox{\Large $)$}
\end{eqnarray*}
}
 \marginpar{\vspace{-14em}\raggedright\tiny
         \raisebox{-1ex}{\hspace{-2.4em}\LARGE $\cdot$}Cf.\,[\,\parbox[t]{8em}{Wess
		\& Bagger:\\ Eq.\:(22.3)].}}

\noindent
and
 \marginpar{\vspace{2.4em}\raggedright\tiny
         \raisebox{-1ex}{\hspace{-2.4em}\LARGE $\cdot$}Cf.\,[\,\parbox[t]{8em}{Wess
		\& Bagger:\\ Eq.\:(22.4)].}}

{\footnotesize
\begin{eqnarray*}
 \lefteqn{  (\breve{f}^\sharp(W))^\dag\;\;
  =\;\;    \overline{W}( (\breve{f}^\sharp(z^1)^\dag,\,\cdots\,,\,(\breve{f}^\sharp(z^n)^\dag)   \;\;
  =\;\;    \overline{W}(\breve{f}^{1\dag},\,\cdots\,,\,\breve{f}^{n\dag})}\\[1.2ex]  		
 && =\;\;
   \overline{W}(\overline{f_{(0)}(x)})
   - \sum_{\dot{\beta}}\bar{\theta}^{\dot{\beta}} \bar{\vartheta}_{\dot{\beta}}
        \sum_i (\partial_{\bar{z}^i}\overline{W})(\overline{f_{(0)}(x)})
		      \cdot \overline{f^i_{(\beta)}(x)}\\
 && \hspace{1.6em}
   +\, \bar{\theta}^{\dot{1}}\bar{\theta}^{\dot{2}}
           \bar{\vartheta}_{\dot{1}} \bar{\vartheta}_{\dot{2}}
        \,\mbox{\Large $($}
		    \sum_i (\partial_{\bar{z}^i}\overline{W})
			          (\overline{f_{(0)}(x)})\cdot \overline{f^i_{(12)}(x)}\,
		     -\,\sum_{i,j}
			     (\partial_{\bar{z}^i}\partial_{\bar{z}^j}\overline{W})(\overline{f_{(0)}(x)})
			 	 \cdot
				    \overline{f^i_{(1)}(x)}\, \overline{f^j_{(2)}(x)}
	   \!\mbox{\Large $)$}  \\
 && \hspace{1.6em}
  -\, \sqrt{-1}\sum_{\alpha,\dot{\beta}; i,\mu}
         \theta^\alpha \bar{\theta}^{\dot{\beta}} \sigma^\mu_{\alpha\dot{\beta}}
			\cdot (\partial_{\bar{z}^i}\overline{W})(\overline{f_{(0)}(x)})
			  \cdot \partial_\mu \overline{f^i_{(0)}(x)}  \\
 && \hspace{1.6em}
  -\, \sqrt{-1}\sum_{\alpha; \dot{\beta}, \mu}
        \theta^\alpha \bar{\theta}^{\dot{1}} \bar{\theta}^{\dot{2}}
         \bar{\vartheta}_{\dot{\beta}} \sigma^{\mu\dot{\beta}}_\alpha  		
		\mbox{\Large $($}
		  \sum_i (\partial_{\bar{z}^i}\overline{W})(\overline{f_{(0)}(x)})
		    \cdot \partial_\mu \overline{f^i_{(\beta)}(x)}				
         + \sum_{i,j}(\partial_{\bar{z}^i}
		         \partial_{\bar{z}^j}\overline{W})(\overline{f_{(0)}(x)})
                 \cdot   \overline{f^i_{(\beta)}(x)}\,\partial_\mu \overline{f^j_{(0)}(x)}		
		\!\mbox{\Large $)$}\\
 && \hspace{1.6em}
  -\, \theta^1\theta^2 \bar{\theta}^{\dot{1}} \bar{\theta}^{\dot{2}}
       \mbox{\Large $($}
        \sum_i(\partial_{\bar{z}^i}\overline{W})(\overline{f_{(0)}(x)})\cdot
		   \square \overline{f^i_{(0)}(x)}
        + \sum_{i,j,\mu,\nu}
		     (\partial_{\bar{z}^i}\partial_{\bar{z}^j}\overline{W})(\overline{f_{(0)}(x)})
			  \cdot \eta^{\mu\nu}\,
			   \partial_\mu \overline{f^i_{(0)}(x)}\,\partial_\nu \overline{f^j_{(0)}(x)}
       \mbox{\Large $)$}\,.
\end{eqnarray*}
}

\vspace{10em}
\bigskip

\begin{flushleft}
{\bf The action functional of $d = 3 + 1$, $N = 1$ sigma models with a superpotential}
\end{flushleft}
\noindent
Combining the previous two themes,
 one has the action functional of $d = 3 + 1$, $N = 1$ sigma models with a superpotential
 given explicitly by
{\small
\begin{eqnarray*}
 \lefteqn{
  S^{(h, W)}(\breve{f})\;=\;
   \int_X\!d^{\,4}x
   \left(-\frac{1}{4}
   \int\! d\bar{\theta}^{\dot{2}}\,d\bar{\theta}^{\dot{1}}\,d\theta^2\, d\theta^1
     \breve{f}^\sharp(h)
	 + \mbox{\LARGE $[$}\!
	       \int\! d\theta^2 d\theta^1
              \breve{f}^\sharp(W)		  	     	
          - \int\! d\bar{\theta}^{\dot{2}}d\dot{\theta}^{\dot{1}}	
		       \mbox{\large $($}\breve{f}^\sharp(W)\!\mbox{\large $)$}^\dag
   	     \mbox{\LARGE $]$}	
	 \!\right) }\\
 && =\;\;
  \int_X\! d^{\,4}x
   \left\{\rule{0ex}{1.2em}\right.
    -\,\sum_{i,j, \mu,\nu}
       (\partial_{z^i} \partial_{\bar{z}^j}h_{\Bbb C})(f_{(0)}(x), \overline{f_{(0)}(x)})\,
	    \eta^{\mu\nu}\partial_\mu f^i_{(0)}(x)\, \partial_\nu \overline{f^j_{(0)}(x)}  	
		\\
   && \hspace{3em}
	 +\, \mbox{\Large $\frac{\sqrt{-1}}{2}$}
	    \sum_{i,j,\alpha, \dot{\beta}, \mu}
		 (\partial_{z^i} \partial_{\bar{z}^j}h_{\Bbb C})(f_{(0)}(x), \overline{f_{(0)}(x)})
		  \cdot \vartheta_\alpha \dot{\vartheta}_{\dot{\beta}}\cdot
		  \bar{\sigma}^{\mu\dot{\beta}\alpha}\,
		   \partial_\mu f^i_{(\alpha)}(x)\,\overline{f^j_{(\beta)}(x)}		 	
		\\
   && \hspace{3em}
     -\, \mbox{\Large $\frac{1}{4}$}
	    \sum_{i,j}
		 (\partial_{z^i} \partial_{\bar{z}^j}h_{\Bbb C})(f_{(0)}(x), \overline{f_{(0)}(x)})
		 \cdot \vartheta_1\vartheta_2 \bar{\vartheta}_{\dot{1}} \bar{\vartheta}_{\dot{2}}\cdot
		 f^i_{(12)}(x)\,\overline{f^j_{(12)}(x)}	
		\\
   && \hspace{3em}
     +\, \sum_{i,j, k}
		 (\partial_{z^i}\partial_{z^j} \partial_{\bar{z}^k}h_{\Bbb C})
		      (f_{(0)}(x), \overline{f_{(0)}(x)}) \cdot
		  \\[-3ex]
		  && \hspace{7em}
		 \mbox{\Large $($}
		  \mbox{\Large $\frac{1}{4}$}\,
		    \vartheta_1\vartheta_2 \bar{\vartheta}_{\dot{1}} \bar{\vartheta}_{\dot{2}}\cdot
			f^i_{(1)}(x) f^j_{(2)}(x)\,\overline{f^k_{(12)}(x)}
			+ \mbox{\Large $\frac{\sqrt{-1}}{2}$}
			   \sum_{\alpha, \dot{\beta}, \mu}
			    \vartheta_\alpha \bar{\vartheta}_{\dot{\beta}}\cdot
				\bar{\sigma}^{\mu\dot{\beta}\alpha} \cdot
				\partial_\mu f^i_{(0)}(x)\,f^j_{(\alpha)}(x) \,\overline{f^k_{(\beta)}(x)}
		 \mbox{\Large $)$ }      		
		 \\
	&& \hspace{3em}
	 +\, \mbox{\Large $\frac{1}{4}$}
	    \sum_{i,j, k}
		 (\partial_{z^i}\partial_{\bar{z}^j} \partial_{\bar{z}^k}h_{\Bbb C})
		      (f_{(0)}(x), \overline{f_{(0)}(x)}) \cdot
		  \vartheta_1\vartheta_2 \bar{\vartheta}_{\dot{1}} \bar{\vartheta}_{\dot{2}}\cdot
		  f^i_{(12)}(x)\, \overline{f^j_{(1)}(x)}\,\overline{f^k_{(2)}(x)}		
         \\		
    && \hspace{3em}
	 -\, \mbox{\Large $\frac{1}{4}$}
	    \sum_{i,j, k, l}
		 (\partial_{z^i}\partial_{z^j}\partial_{\bar{z}^k} \partial_{\bar{z}^l}h_{\Bbb C})
		      (f_{(0)}(x), \overline{f_{(0)}(x)}) \cdot
		  \vartheta_1\vartheta_2 \bar{\vartheta}_{\dot{1}} \bar{\vartheta}_{\dot{2}}\cdot
		  f^i_{(1)}(x) f^j_{(2)}(x)\, \overline{f^k_{(1)}(x)}\,\overline{f^l_{(2)}(x)}		
         \\		  		
    && \hspace{3em}
	  +\, \sum_i (\partial_{z^i}W)(f_{(0)})\cdot \vartheta_1\vartheta_2 \cdot f^i_{(12)}(x)
	    - \sum_{i,j}  (\partial_{z^i}\partial_{z^j} W)(f_{(0)}(x))\cdot
		    \vartheta_1\vartheta_2\cdot f^i_{(1)}(x) f^j_{(2)}(x)
      \\	
    && \hspace{3em}
	  -\, \sum_i (\partial_{\bar{z}^i}\overline{W})(\overline{f_{(0)}})  \cdot
	        \bar{\vartheta}_{\dot{1}} \bar{\vartheta}_{\dot{2}} \cdot  \overline{f^i_{(12)}(x)}
	    + \sum_{i,j}
		     (\partial_{\bar{z}^i}\partial_{\bar{z}^j}\overline{W})(\overline{f_{(0)}(x)})\cdot
		    \bar{\vartheta}_{\dot{1}}\bar{\vartheta}_{\dot{2}}\cdot
			   \overline{f^i_{(1)}(x)}\, \overline{f^j_{(2)}(x)}
   \left.\rule{0ex}{1.2em}\right\}\,,
\end{eqnarray*}up} 
to boundary terms.
(Caution that
   $(d\theta^2 d\theta^1)^\dag    = d\bar{\theta}^{\dot{1}}d\bar{\theta}^{\dot{2}}
	   =\,-\,d\bar{\theta}^{\dot{2}}d\bar{\theta}^{\dot{1}}$.)

Recall that in terms of the K\"{a}hler potential $h_{\Bbb C}$,
   the K\"{a}hler metric $g$ on $Y={\Bbb C}^n$,
   the Christoffel symbols for the Levi-Civita connection associated to the metric,   and
   the curvature tensor
 are given by the {\it Basic Identities}:
 \begin{eqnarray*}
  & g_{i\bar{j}}\;=\; g_{\bar{j}i}\;=\; \partial_i\partial_{\bar{j}}h_{\Bbb C}\,,\hspace{2em}
  \sum_s g_{s\bar{k}}\,\Gamma^s_{ji}\;=\; \partial_i g_{j\bar{k}}\,,\hspace{2em}
  \sum_s g_{\bar{s}k}\,\Gamma^{\bar{s}}_{\bar{j}\bar{i}}\;
     =\; \partial_{\bar{i}} g_{\bar{j}k}\,,    & \\[1.2ex]
  & \begin{array}{rcl}
       R_{i\bar{k}j\bar{l}} & =
	    & \partial_j\partial_{\bar{l}} g_{i\bar{k}}
	        - \sum_{s,t} g^{\bar{s}t}\partial_j g_{i\bar{s}}\,\partial_{\bar{l}} g_{\bar{k}t}\;\;
		   =\;\; \partial_j\partial_{\bar{l}} g_{i\bar{k}}
		           - \sum_{s, t }g_{s\bar{t}} \Gamma^s_{ij} \Gamma^{\bar{t}}_{\bar{k}\bar{l}}
				                                     \\[1ex]
       & =
  	      & \partial_i\partial_j\partial_{\bar{k}}\partial_{\bar{l}} h_{\Bbb C}
	     - \sum_{s, t} g^{\bar{s}t}\,\partial_{\bar{s}}\partial_i\partial_j h_{\Bbb C}
		                              \cdot \partial_t \partial_{\bar{k}} \partial_{\bar{l}} h_{\Bbb C}\,,
     \end{array}									
 \end{eqnarray*}
(e.g.\ [K-N: vol.\:II: Sec.\:IX.5]).
It follows that, up to boundary terms and with some re-arrangement of terms and relabelling,

{\small
\begin{eqnarray*}
 \lefteqn{
  S^{(h, W)}(\breve{f})\;=\;
   \int_X\!d^{\,4}x
   \left(-\frac{1}{4}
   \int\! d\bar{\theta}^{\dot{2}}\,d\bar{\theta}^{\dot{1}}\,d\theta^2\, d\theta^1
     \breve{f}^\sharp(h)
	 + \mbox{\LARGE $[$}\!
	       \int\! d\theta^2 d\theta^1
              \breve{f}^\sharp(W)		  	     	
          - \int\! d\bar{\theta}^{\dot{2}}d\dot{\theta}^{\dot{1}}	
		       \mbox{\large $($}\breve{f}^\sharp(W)\!\mbox{\large $)$}^\dag
   	     \mbox{\LARGE $]$}	
	 \!\right) }\\
 && =\;\;
  \int_X\! d^{\,4}x
   \left\{\rule{0ex}{1.2em}\right.
    -\,\sum_{i,j, \mu,\nu}
        g_{i\bar{j}}(f_{(0)}(x), \overline{f_{(0)}(x)})\,
	    \eta^{\mu\nu}\partial_\mu f^i_{(0)}(x)\, \partial_\nu \overline{f^j_{(0)}(x)}  	
		\\
   && \hspace{3em}
	 +\, \mbox{\Large $\frac{\sqrt{-1}}{2}$}
	    \sum_{i,j,\alpha, \dot{\beta}, \mu}
		 \vartheta_\alpha \dot{\vartheta}_{\dot{\beta}}\cdot
		     g_{i\bar{j}}(f_{(0)}(x), \overline{f_{(0)}(x)})\,
		  \bar{\sigma}^{\mu\dot{\beta}\alpha}\,
		   D_\mu f^i_{(\alpha)}(x)\,\overline{f^j_{(\beta)}(x)}		 	
		\\
   && \hspace{3em}
	 -\, \mbox{\Large $\frac{1}{4}$}\,
	    \vartheta_1\vartheta_2 \bar{\vartheta}_{\dot{1}} \bar{\vartheta}_{\dot{2}}\cdot
	    \sum_{i,j, k, l}
		  g_{i\bar{k}, j\bar{l}}(f_{(0)}(x), \overline{f_{(0)}(x)}) \,	
		  f^i_{(1)}(x) f^j_{(2)}(x)\, \overline{f^k_{(1)}(x)}\,\overline{f^l_{(2)}(x)}		
         \\		  			
   && \hspace{3em}
     -\, \mbox{\Large $\frac{1}{4}$}\,
	    \vartheta_1\vartheta_2 \bar{\vartheta}_{\dot{1}} \bar{\vartheta}_{\dot{2}}\cdot
	    \sum_{i,j}
		 g_{i\bar{j}}(f_{(0)}(x), \overline{f_{(0)}(x)})\,
		 f^i_{(12)}(x)\,\overline{f^j_{(12)}(x)}	
		\\
   && \hspace{3em}
      -\,\vartheta_1\vartheta_2\cdot
   		\sum_{i,j}
		 (\partial_{z^i}\partial_{z^j} W)(f_{(0)}(x))\, f^i_{(1)}(x) f^j_{(2)}(x)
        +  \bar{\vartheta}_{\dot{1}}\bar{\vartheta}_{\dot{2}}\cdot
  	          \sum_{i,j}
		     (\partial_{\bar{z}^i}\partial_{\bar{z}^j}\overline{W})(\overline{f_{(0)}(x)})\,		
			    \overline{f^i_{(1)}(x)}\, \overline{ f^j_{(2)}(x)}			
       \\
    && \hspace{3em}
	  +\, \sum_i
	      f^i_{(12)}(x)\,
       \mbox{\Large $($}\,		
		  \vartheta_1\vartheta_2\cdot (\partial_{z^i}W)(f_{(0)})
		  \\[-2ex]
		&& \hspace{10em}
	      + \mbox{\Large $\frac{1}{4}$}
		     \vartheta_1\vartheta_2 \bar{\vartheta}_{\dot{1}} \bar{\vartheta}_{\dot{2}}\cdot
	         \sum_{j, k, l}
			   g_{i\bar{l}}(f_{(0)}(x), \overline{f_{(0)}(x)}) \,
		       \Gamma^{\bar{l}}_{\bar{j}\bar{k}}(f_{(0)}(x), \overline{f_{(0)}(x)}) \,
		       \overline{f^j_{(1)}(x)}\,\overline{f^k_{(2)}(x)}				
       \mbox{\Large $)$}		
      \\	
    && \hspace{3em}
	  +\,\sum_i
	      \overline{f^i_{(12)}(x)}	
		  \mbox{\Large $($}	
	       - \bar{\vartheta}_{\dot{1}} \bar{\vartheta}_{\dot{2}} \cdot
		        (\partial_{\bar{z}^i}\overline{W})(\overline{f_{(0)}})
				\\[-2ex]
		&& \hspace{10em}
		    + \mbox{\Large $\frac{1}{4}$}\,
			    \vartheta_1\vartheta_2 \bar{\vartheta}_{\dot{1}} \bar{\vartheta}_{\dot{2}}\cdot
       		     \sum_{j, k, l}
				   g_{l\bar{i}}(f_{(0)}(x), \overline{f_{(0)}(x)}) \,			
		           \Gamma^l_{jk}(f_{(0)}(x), \overline{f_{(0)}(x)}) \,			
			       f^j_{(1)}(x) f^k_{(2)}(x)		
          \mbox{\Large $)$}		
   \left.\rule{0ex}{1.2em}\right\}\,,
\end{eqnarray*}where} 
 \marginpar{\vspace{-26.4em}\raggedright\tiny
         \raisebox{-1ex}{\hspace{-2.4em}\LARGE $\cdot$}Cf.\,[\,\parbox[t]{8em}{Wess
		\& Bagger:\\ Eq.\:(22.10)] \\
		before purging\\ nilpotent\\ Grassmann\\ factors.}}
 $$
   D_\mu f^i_{(\alpha)}(x) \;
    :=\;   \partial_\mu f^i_{(\alpha)}(x)
	          + \sum_{j,k} \Gamma^i_{jk}(f_{(0)}(x), \overline{f_{(0)}(x)})\,
			       \partial_\mu f^j_{(0)}(x) f^k_{(\alpha)}(x)
 $$
 is the covariant derivative along $\partial_\mu$ from the induced connection on
  $({\cal S}^\prime\oplus {\cal S}^{\prime\prime})
       \otimes_{\,{\cal O}_X^{\,\Bbb C}}\!\! f_{(0)}^\ast{\cal T}_Y$.
This reproduces [Wess \& Bagger: Eq.\:(22.10)]
  completely via the function ring of towered superspace as defined in the current notes
without imposing the purge-evaluation map to remove the even nilpotent factors
  $$
    \vartheta_1\vartheta_2\,,\;\;\;\;
	\vartheta_\alpha\bar{\vartheta}_{\dot{\beta}}\,,\;\;\;\;
    \bar{\vartheta}_{\dot{1}}	\bar{\vartheta}_{\dot{2}}\,,\;\;\;\;
	 \vartheta_1\vartheta_2 \bar{\vartheta}_{\dot{1}}	\bar{\vartheta}_{\dot{2}}
  $$	
 in the expression.

In the form above
the action functional $S^{(h, W)}$ remains to take values in a Grassmann algebra, albeit even.
To render it real-valued, one needs to impose a purge-evaluation map.
Here we will take\footnote{Note
                                  that $\vartheta_\alpha \bar{\vartheta}_{\dot{\beta}}$'s
             				       appear only in the kinetic term
   				                   $\sum_{\alpha, \dot{\beta}, \mu, i,j}
								               \vartheta_\alpha \bar{\vartheta}_{\dot{\beta}}\,
				                               \bar{\sigma}^{\mu\dot{\beta}\alpha}\, g_{i\bar{j}}\,
					                          D_\mu f^i_\alpha\, \overline{f^j_{(\beta)}}$
           				            for the mappino fields $(f^i_{(\alpha)})_{i,\alpha}$.
								   One is free to set them to be either all $1$ or all $-1$.
								   The choice here is to match [Wess \& Bagger] even with the sign.
								   For the other three. once setting $\vartheta_1\vartheta_2\rightsquigarrow 1$, then
								     the purge-evaluation rule,
									    $\bar{\vartheta}_{\dot{1}}	\bar{\vartheta}_{\dot{2}}\rightsquigarrow -1$ and
	                                    $\vartheta_1\vartheta_2 \bar{\vartheta}_{\dot{1}}	\bar{\vartheta}_{\dot{2}}
										   \rightsquigarrow -1$,
									   is governed by the invariance under the twisted complex conjugation and
									    the wish to make the rule as close to product-preserving as possible.										
								   It turns out that this is also the choice that makes the final form of $S^{(h,W)}(\breve{f})$
								     term-by-term sign-identical with [Wess \& Bagger].
				                    } 
 $$
  \Pev\;:\;
    \vartheta_1\vartheta_2\;\rightsquigarrow\; 1\,,\;\;\;\;
	\vartheta_\alpha\bar{\vartheta}_{\dot{\beta}}\; \rightsquigarrow\; -1\,,\;\;\;\;
    \bar{\vartheta}_{\dot{1}}	\bar{\vartheta}_{\dot{2}}\; \rightsquigarrow\; -1\,,\;\;\;\;
	 \vartheta_1\vartheta_2 \bar{\vartheta}_{\dot{1}}	\bar{\vartheta}_{\dot{2}}\;
	      \rightsquigarrow\; -1\,.
 $$
Then, in the next-to-final form and up to boundary terms,
%
%
{\footnotesize
\begin{eqnarray*}
 \lefteqn{
  S^{(h, W)}(\breve{f})\;=\;
   \int_X\!d^{\,4}x\,
    \Pev
   \!\!\left(-\frac{1}{4}
   \int\! d\bar{\theta}^{\dot{2}}\,d\bar{\theta}^{\dot{1}}\,d\theta^2\, d\theta^1
     \breve{f}^\sharp(h)
	 + \mbox{\LARGE $[$}\!
	       \int\! d\theta^2 d\theta^1
              \breve{f}^\sharp(W)		  	     	
          - \int\! d\bar{\theta}^{\dot{2}}d\dot{\theta}^{\dot{1}}	
		       \mbox{\large $($}\breve{f}^\sharp(W)\!\mbox{\large $)$}^\dag
   	     \mbox{\LARGE $]$}	
	 \right) }\\
 && =\;\;
  \int_X\! d^{\,4}x
   \left\{\rule{0ex}{1.2em}\right.
    -\,\sum_{i,j, \mu,\nu}
        g_{i\bar{j}}(f_{(0)}(x), \overline{f_{(0)}(x)})\,
	    \eta^{\mu\nu}\partial_\mu f^i_{(0)}(x)\, \partial_\nu \overline{f^j_{(0)}(x)}  	
		\\
   && \hspace{3em}
	 -\, \mbox{\large $\frac{\sqrt{-1}}{2}$}
	    \sum_{i,j,\alpha, \dot{\beta}, \mu}
		     g_{i\bar{j}}(f_{(0)}(x), \overline{f_{(0)}(x)})\,
		  \bar{\sigma}^{\mu\dot{\beta}\alpha}\,
		   D_\mu f^i_{(\alpha)}(x)\,\overline{f^j_{(\beta)}(x)}		 	
		\\
   && \hspace{3em}
	 +\, \mbox{\large $\frac{1}{4}$}\,
	    \sum_{i,j, k, l}
		  g_{i\bar{k}, j\bar{l}}(f_{(0)}(x), \overline{f_{(0)}(x)}) \,	
		  f^i_{(1)}(x) f^j_{(2)}(x)\, \overline{f^k_{(1)}(x)}\,\overline{f^l_{(2)}(x)}		
     +\, \mbox{\large $\frac{1}{4}$}\,
	    \sum_{i,j}
		 g_{i\bar{j}}(f_{(0)}(x), \overline{f_{(0)}(x)})\,
		 f^i_{(12)}(x)\,\overline{f^j_{(12)}(x)}	
		\\
   && \hspace{3em}
      -\,\sum_{i,j}
		 (\partial_{z^i}\partial_{z^j} W)(f_{(0)}(x))\, f^i_{(1)}(x) f^j_{(2)}(x)
        - \sum_{i,j}
		     (\partial_{\bar{z}^i}\partial_{\bar{z}^j}\overline{W})(\overline{f_{(0)}(x)})\,		
			    \overline{f^i_{(1)}(x)}\, \overline{ f^j_{(2)}(x)}			
       \\
    && \hspace{3em}
	  +\, \sum_i
	      f^i_{(12)}(x)\,
       \mbox{\Large $($}\,		
		(\partial_{z^i}W)(f_{(0)})
	      - \mbox{\large $\frac{1}{4}$}
	         \sum_{j, k, l}
			   g_{i\bar{l}}(f_{(0)}(x), \overline{f_{(0)}(x)}) \,
		       \Gamma^{\bar{l}}_{\bar{j}\bar{k}}(f_{(0)}(x), \overline{f_{(0)}(x)}) \,
		       \overline{f^j_{(1)}(x)}\,\overline{f^k_{(2)}(x)}				
       \mbox{\Large $)$}		
      \\	
    && \hspace{3em}
	  +\,\sum_i
	      \overline{f^i_{(12)}(x)}	
		  \mbox{\Large $($}	
		     (\partial_{\bar{z}^i}\overline{W})(\overline{f_{(0)}})
		    - \mbox{\large $\frac{1}{4}$}\,
       		     \sum_{j, k, l}
				   g_{l\bar{i}}(f_{(0)}(x), \overline{f_{(0)}(x)}) \,			
		           \Gamma^l_{jk}(f_{(0)}(x), \overline{f_{(0)}(x)}) \,			
			       f^j_{(1)}(x) f^k_{(2)}(x)		
          \mbox{\Large $)$}		
   \left.\rule{0ex}{1.2em}\right\}\,.
\end{eqnarray*}}(Cf.\:[Wess \& Bagger: Eq.\:(22.10)].)  
 \marginpar{\vspace{-18em}\raggedright\tiny
         \raisebox{-1ex}{\hspace{-2.4em}\LARGE $\cdot$}Cf.\,[\,\parbox[t]{8em}{Wess
		\& Bagger:\\ Eq.\:(22.10)].}}
 
Observe next that
  $S^{(h, W)}(\breve{f})$ contains no kinetic terms
       for $f^i_{(12)}$'s and $\overline{f^i_{(12)}}$'s and, hence,
 these components are nondynamical and 	
 the equations of motion for them from the first variation of $S^{(h, W)}(\breve{f})$
  with respect to $f^i_{(12)}$ or $\overline{f^i_{(12)}}$ are purely algebraic in
     $f^i_{(12)}$, $\overline{f^j_{(12)}}$:
 {\footnotesize
 \begin{eqnarray*}
   \mbox{\large $\frac{1}{4}$}
      \sum_i g_{i\bar{j}}(f_{(0)}(x), \overline{f_{(0)}(x)}) f^i_{(12)}(x)	
	+\, (\partial_{\bar{z}^j}\overline{W})(\overline{f_{(0)}(x)})
	   \hspace{10em}
      \\[-1.2ex]
 	\hspace{3em}
     - \mbox{\large $\frac{1}{4}$}
	     \sum_{i, k,l}
	      g_{i\bar{j}}(f_{(0)}(x), \overline{f_{(0)}(x)})
		    \Gamma^i_{kl}(f_{(0)}(x), \overline{f_{(0)}(x)}) f^k_{(1)}(x) f^l_{(2)}(x)
	 & = & 0\,,\;\;\;\; \mbox{for all $\bar{j}$}\,;
	 \\[1.2ex]
	\mbox{\large $\frac{1}{4}$}
      \sum_j g_{i\bar{j}}(f_{(0)}(x), \overline{f_{(0)}(x)}) \,\overline{f^j_{(12)}(x)}	
	+\, (\partial_{z^i}W)(f_{(0)}(x))
	   \hspace{10em}
      \\[-1.2ex]
 	\hspace{3em}
     - \mbox{\large $\frac{1}{4}$}
	     \sum_{j, k, l}
	      g_{i\bar{j}}(f_{(0)}(x), \overline{f_{(0)}(x)})
		    \Gamma^{\bar{j}}_{\bar{k}\bar{l}}(f_{(0)}(x), \overline{f_{(0)}(x)})\,
			   \overline{f^k_{(1)}(x)}\, \overline{ f^l_{(2)}(x)}
	 & = & 0\,,\;\;\;\; \mbox{for all $i$}\,.    	
 \end{eqnarray*}}Which  
gives
 {\footnotesize
 \begin{eqnarray*}
   f^i_{(12)}(x)    & =
    & -\,4 \sum_j g^{i\bar{j}}(f_{(0)}(x), \overline{f_{(0)}(x)})\,
	             (\partial_{\bar{z}^j}\overline{W})(\overline{f_{(0)}(x)})
       +  \sum_{k,l}
	        \Gamma^i_{kl}(f_{(0)}(x), \overline{f_{(0)}(x)})\,
			    f^k_{(1)}(x) f^l_{(2)}(x)\,, \\
   \overline{f^j_{(12)}(x)}   & =
    & -\,4 \sum_i
	             g^{i\bar{j}}(f_{(0)}(x), \overline{f_{(0)}(x)})
	              (\partial_{z^i} W)(f_{(0)}(x))
       + \sum_{k,l}
	        \Gamma^{\bar{j}}_{\bar{k}\bar{l}}(f_{(0)}(x), \overline{f_{(0)}(x)})
		       \overline{f^k_{(1)}(x)}\, \overline{f^l_{(2)}(x)}\,.    	
 \end{eqnarray*}
Plugging this into $S^{(h,W)}(\breve{f})$
   to remove\footnote{\makebox[12em][l]{\it Note for mathematicians}
                        This is called by physicists
			              ``{\it integrating out the non-dynamical $f^i_{(12)}$ and $\overline{f^j_{(0)}}$}"
					      from the perspective of path-integrals in Quantum Field Theory.
													   }
 the nondynamical component fields	$f^i_{(12)}$ and $\overline{f^j_{(12)}}$     and
 employing the Basic Identities,
one obtains the final form of the action functional after two sets of cancellations
--- one leading to the curvature term and the other involving superpotential terms ---

{\small
\begin{eqnarray*}
 \lefteqn{
  S^{(h,W)}(f_{(0)}, (f_{(\alpha)})_{\alpha=1,2})\;\;
     =\;\;    S^{(g,W)}(f_{(0)}, (f_{(1)})_{\alpha=1,2})        }\\
 && =\;\;
  \int_X\! d^{\,4}x
    \left\{\rule{0ex}{1.2em}\right.\!
	  -\,\sum_{i,j, \mu,\nu}
        g_{i\bar{j}}(f_{(0)}(x), \overline{f_{(0)}(x)})\,
	    \eta^{\mu\nu}\partial_\mu f^i_{(0)}(x)\, \partial_\nu \overline{f^j_{(0)}(x)}  	
		\\
 && \hspace{3em}
	 -\, \mbox{\large $\frac{\sqrt{-1}}{2}$}
	    \sum_{i,j,\alpha, \dot{\beta}, \mu}
		     g_{i\bar{j}}(f_{(0)}(x), \overline{f_{(0)}(x)})\,
		  \bar{\sigma}^{\mu\dot{\beta}\alpha}\,
		   D_\mu f^i_{(\alpha)}(x)\,\overline{f^j_{(\beta)}(x)}		 	
		\\		
  && \hspace{3em}
	 +\, \mbox{\large $\frac{1}{4}$}\,
	    \sum_{i,j, k, l}
		  R_{i\bar{k} j\bar{l}}(f_{(0)}(x), \overline{f_{(0)}(x)}) \,	
		  f^i_{(1)}(x) f^j_{(2)}(x)\, \overline{f^k_{(1)}(x)}\,\overline{f^l_{(2)}(x)}	
		\\
  && \hspace{3em}
     -\,4\, \sum_{i,j}
	      g^{i\bar{j}}(f_{(0)}(x), \overline{f_{(0)}(x)})\,
		   (\partial_{z^i} W)(f_{(0)}(x))\,
		   (\partial_{\bar{z}^j}\overline{W})(\overline{f_{(0)}(x)})			
	    \\
   && \hspace{3em}
    -\, \sum_{i,j}
	      (D_{z^i}\partial_{z^j}W)(f_{(0)}(x), \overline{f_{(0)}(x)})\,
		     f^i_{(1)}(x) f^j_{(2)}(x)
     - \sum_{i,j}
	      (D_{\bar{z}^i}\partial_{\bar{z}^j}\overline{W})
		              (f_{(0)}(x), \overline{f_{(0)}(x)})\,
		     \overline{f^i_{(1)}(x)}\, \overline{f^j_{(2)}(x)}			 				
    \left.\rule{0ex}{1.2em}\right\},
\end{eqnarray*}where} 
 \marginpar{\vspace{-16em}\raggedright\tiny
         \raisebox{-1ex}{\hspace{-2.4em}\LARGE $\cdot$}Cf.\,[\,\parbox[t]{20em}{Wess
		\& Bagger:\\ Eq.\:(22.12)].}}
 $$
   D_{z^i}\partial_{z^j} W\;
    =\;   \partial_{z^i}\partial_{z^j} W
	          - \sum_k \Gamma^s_{ij}\partial_{z^k} W
    \hspace{2em}\mbox{and}\hspace{2em}			
   D_{\bar{z}^i}\partial_{\bar{z}^j}\overline{W}\;
    =\;   \partial_{\bar{z}^i}\partial_{\bar{z}^j}\overline{W}
	          - \sum_k \Gamma^{\bar{k}}_{\bar{i}\bar{j}}\partial_{\bar{k}} \overline{W}
 $$
 are the induced covariant derivative on the complexified cotangent bundle $T^{\ast\,{\Bbb C}} Y$ of $Y$
 from the Levi-Civita connection on $T_{\ast}Y$.
Up to some conventional coefficients, this is precisely the manifestly real expression for
 [Wess \& Bagger: Eq.\:(22.12)] in the setting of the current notes.
 
\bigskip

\begin{remark} $[$\,for $Y$ a general K\"{a}hler manifold\,$]$\; {\rm
 The final expression of $S^{(h, W)}(f_{(0)}, (f_{(\alpha)})_{\alpha=1,2})$
   implies that\\
   $S^{(h, W)}(f_{(0)}, (f_{(\alpha)})_{\alpha=1,2})
     = S^{(g, W)}(f_{(0)}, (f_{(\alpha)})_{\alpha=1,2})$
   depends only on the K\"{a}hler metric $g$,
     not the choice of the K\'{a}hler potential $h$ that gives the metric.
 It follows that
  the $d=3+1$, $N=1$ supersymmetric sigma model
  $S^{(g, W)}(f_{(0)}, (f_{(\alpha)})_{\alpha=1,2})$ is defined
  for all K\"{a}hler manifold $Y$.
(When $Y$  is compact, the superpotential $W=0$.)
}\end{remark}

\bigskip

Once a mathematical presentation of superspace and supersymmetry
     that matches the particle physicists' standard language of the topic as presented in [Wess \& Bagger]
   is completed,
 the immediate next question is:
  \begin{itemize}
   \item[{\bf Q.}]
    \parbox[t]{45em}{\it What is the intrinsic description of the same,
     without resorting to a choice of a trivialization of the spinor bundles by covariantly constant sections?}
  \end{itemize}
 That is the theme for the sequel.

\newpage
\baselineskip 13pt
{\footnotesize

\vspace{1em}

\noindent
chienhao.liu@gmail.com,
chienliu@cmsa.fas.harvard.edu; \\
yau@math.harvard.edu

}

\end{document}